\documentclass[a4paper,12pt]{amsart}
\usepackage{color}
\usepackage{a4wide}
\usepackage{geometry}
\usepackage{hyperref}
\usepackage{amsmath}
\usepackage{amssymb}
\usepackage{amsthm}
\usepackage[active]{srcltx}
\usepackage[dvips]{graphicx}

\numberwithin{equation}{section}

\def\jet{(r,p,A)}
\def\ss{\subset}

\def\half{\hbox{${1\over 2}$}}

\def\tr{{\rm tr}}
\def\max{{\rm max}}
\def\min{{\rm min}}

\def\det{{\rm det}}

\def\Symn{\mathcal{S}(n)}
\def\Sym1{\mathcal{S}(1)}

\def\cE{{\mathcal E}}

\def\cG{{\mathcal G}}

\def\cH{{\mathcal H}}

\def\cD{{\mathcal D}}
\def\cDt{{\widetilde \cD}}
\def\cM{{\mathcal M}}
\def\cMt{{\widetilde \cM}}
\def\cP{{\mathcal P}}
\def\cPt{{\widetilde \cP}}
\def\cN{{\mathcal N}}

\def\cQ{{\mathcal Q}}
\def\cQt{{\widetilde \cQ}}
\def\cR{{\mathcal R}}
\def\cRt{{\widetilde \cR}}
\def\cS{{\mathcal S}}

\newtheorem{thm}{\textbf{Theorem}}[section]
\newtheorem{lem}[thm]{\textbf{Lemma}}
\newtheorem{prop}[thm]{\textbf{Proposition}}
\theoremstyle{remark}
\newtheorem{rem}[thm]{\textbf{Remark}}
\newtheorem{sumrem}[thm]{\textbf{Summary Remark}}
\newtheorem{cor}[thm]{\textbf{Corollary}}

\newtheorem{exe}[thm]{\textbf{Example}}
\newtheorem{exes}[thm]{\textbf{Examples}}
\theoremstyle{definition}
\newtheorem{defn}[thm]{{Definition}}
\newtheorem{conv}[thm]{{Convention}}

\newtheoremstyle{Claim}{}{}{\itshape}{}{\itshape\bfseries}{:}{ }{#1}
\theoremstyle{Claim}

\newcommand{\N}{{\mathbb N}}

\newcommand{\R}{{\mathbb R}}
\newcommand{\CF}{{\mathbb C}}
\newcommand{\HF}{{\mathbb H}}

\newcommand{\veps}{\varepsilon}

\newcommand{\J}{\mathcal{J}}

\newcommand{\F}{\mathcal{F}}
\newcommand{\G}{\mathcal{G}}
\newcommand{\cU}{\mathcal{U}}
\newcommand{\wt}{\widetilde}
\newcommand{\FD}{\wt{\F}}

\newcommand{\USC}{\mathrm{USC}}
\newcommand{\LSC}{\mathrm{LSC}}
\newcommand{\UCP}{\mathrm{UCP}}
\newcommand{\Int}{\mathrm{Int}}
\newcommand{\SA}{\mathrm{SA}}
\newcommand{\EC}[1]{\overrightarrow{#1}}

\newcommand{\A}{\mathrm{Aff}}
\newcommand{\Ap}{\mathrm{Aff}^+}
\newcommand{\SAp}{\mathrm{SA}^+}
\newcommand{\pol}{\mathfrak{g}}
\newcommand{\polh}{\mathfrak{h}}

\newcommand{\transv}{\mathrel{\text{\tpitchfork}}}
\makeatletter
\newcommand{\tpitchfork}{%
	\vbox{
		\baselineskip\z@skip
		\lineskip-.52ex
		\lineskiplimit\maxdimen
		\m@th
		\ialign{##\crcr\hidewidth\smash{$-$}\hidewidth\crcr$\pitchfork$\crcr}
	}%
}
\makeatother

\newcommand{\changelocaltocdepth}[1]{%
	\addtocontents{toc}{\protect\setcounter{tocdepth}{#1}}%
	\setcounter{tocdepth}{#1}%
}

\title[Comparison principles by monotonicity and duality]{Comparison principles by monotonicity and duality for constant coefficient nonlinear potential theory and PDEs}
\author[M. Cirant]{Marco Cirant}
\address{Dipartimento di Matematica\\  Universit\`a di Padova\\ Via Trieste, 63 \\ 35121 -- Padova, Italy}
\email{marco.cirant@unipd.it (Marco Cirant)}
\author[F.R. Harvey]{F. Reese Harvey}
\address{Department of Mathematics\\ Rice University\\ P.O. Box 1892\\ Houston, TX 77005-1892, USA}
\email{harvey@rice.edu (F. Reese Harvey)}
\author[H.B. Lawson, Jr.]{H. Blaine Lawson, Jr.}
\address{Department of Mathematics\\ Stony Brook University\\ Stony Brook, NY 11794-3651, USA}
\email{blaine@math.stonybrook.edu (H.\ Blaine Lawson, Jr.)}
\author[K.R. Payne]{Kevin R. Payne}
\address{Dipartimento di Matematica ``F. Enriques''\\ Universit\`a di Milano\\ Via C. Saldini 50\\ 20133--Milano, Italy}
\email{kevin.payne@unimi.it (Kevin R.\ Payne)}

\date{\today} 


\begin{document}

\maketitle

\begin{abstract}

    One main purpose of this paper is to prove {\em comparison principles} for nonlinear potential theories in $\R^n$
in a very straightforward manner from {\em duality} and {\em monotonicity}.  We shall also show how to deduce comparison
principles for nonlinear differential operators -- a program seemingly different from the first. However,
we shall marry these two points of view, for a wide variety of equations, under something called the
{\em correspondence principle}.  

In potential theory one is given a constraint set $\F$ on the 2-jets of a function, and the boundary of $\F$ gives a differential equation. There are many differential operators, suitably organized around $\F$, which give the same equation. So potential theory gives a great strengthening and simplification to the operator theory. Conversely, the set of operators associated to $\F$ can have much to say about the potential theory.

An object of central interest here is that of monotonicity, which explains and unifies much of the theory. We shall always assume that the maximal monotonicity cone for a potential theory has interior. This is automatic for gradient-free equations where monotonicity is simply the standard  degenerate ellipticity (positivity) and properness (negativity) assumptions. 

We show that for each such potential theory $\F$ there is an associated {\em canonical operator} $F$, defined on the entire 2-jet space and having all the desired properties. Furthermore, comparison holds for this $F$ on any domain $\Omega \subset\subset \R^n$ which admits a $C^2$ strictly $\cM$-subharmonic function, where $\cM$ is a monotonicity subequation for $\F$. For example, for the potential theory corresponding to convex functions, the canonical operator is the minimal eigenvalue of $D^2 u$  in the $C^2$-case.

On the operator side there is an important dichotomy into the {\em unconstrained cases} and {\em constrained cases},
where the operator must be restricted to a proper subset of 2-jet space
The unconstrained  case is best illustrated by the canonical operators, whereas the constrained case
is best illustrated by Dirichlet-G\aa rding  operators. 

The article gives many, many examples from pure and applied mathematics, and also from theoretical physics.

\end{abstract}

\makeatletter
\def\l@subsection{\@tocline{2}{0pt}{2.5pc}{5pc}{}}
\makeatother

 \setcounter{tocdepth}{1}
\tableofcontents

\section{Introduction}\label{sec:intro} 

One of the main purposes of this paper is to prove {\em comparison principles} with respect 
to a constant coefficient nonlinear potential theory in a very straightforward manner from {\em duality} and 
{\em monotonicity}.  We shall also show how to deduce comparison principles for nonlinear differential 
operators, a program  which seems somewhat different from the first.  However, we shall marry these
two points of view, for a wide variety of equations, under something we call the {\em correspondence principle}.
This turns out to be interesting for several reasons.  In potential theory one is given a constraint set 
$\F$ on the 2-jets of a function, and the boundary of $\F$ gives a differential
equation.  There are many differential operators, suitably organized around $\F$, which give the same equation.
So  potential theory gives a great strengthening and simplification to the operator  theory.
Conversely, the set of operators associated to $\F$ can have much to say about the potential theory.

\def\wt{\widetilde}

One motivation for this study comes from the following consideration (other motivations will be given below). 
For the Dirichlet Problem (DP) on a bounded domain $\Omega$ in Euclidian space, it is proved in \cite{HL11}  
that {\em existence always holds}  in the constant coefficient case \footnote{The conclusion
	``existence always holds''  is  precisely defined  in Theorem \ref{thm:existence} in Appendix A.} 
(assuming that $\Omega$ has a smooth $C^2$  boundary satisfying the appropriate strict boundary convexity conditions).  This leaves uniqueness, which the comparison principle obviously implies. Interestingly, in our constant coefficient case, one can show that uniqueness and comparison are actually equivalent using the fact that existence always holds (see Theorem \ref{thm:comparion_uniqueness}).

An object of central interest here is that of  {\em monotonicity}.
It is monotonicity that explains and unifies much of the theory.
In simpler cases, such as pure second order equations, or gradient-free equations,
monotonicity comes down to the standard degenerate ellipticity and negativity assumptions. To explain this in more detail we need some notation.

\subsection{The Potential Theory Setting}

Set $\J^2 := \R \times\R^n\times \Symn$ the space
of $2$-jets with standard jet coordinates  $(r,p, A)$,
where $\Symn$ is the space of symmetric  $n\times n$-matrices with real entries,
and consider a  set $\F\ss \J^2$.  Then $\F$ is called a 
{\em constant coefficient subequation constraint set} (or simply {\em subequation}, or {\em constraint set})
if  $\F$ is not $\emptyset$ or $\Symn$, and 
\begin{equation}\label{SE_defn}
\F +  \cP_0 \ \ss\ \F, \qquad \F  +   \cN_0 \ \ss\  \F  \qquad {\rm and} \qquad \F \ =\ \overline{\Int \, \F},
\end{equation}
where $\cP_0 := \{0\} \times \{0\}  \times \cP$ and $\cN_0 := \cN \times \{0\} \times \{0\}$
in $\J^2 = \R \times\R^n\times \Symn$,
with
\begin{equation}\label{PN_defn}
\cP \ := \ \{ A \in  \Symn : A\geq 0\}
\qquad{\rm and}\qquad
\cN \ := \ \{r \in \R : r\leq 0\}.
\end{equation}
Associated  to a constraint set $\F$ is its {\em dual} constraint  set \footnote{Throughout the paper, $\Int \, \F$ is the interior of $\F$ and $\sim \F = \J^2 \setminus \F$ the complement of $\F$ with respect to $\J^2$.}
\begin{equation}\label{DD_defn}
\wt \F \ := \ \sim\{-\Int \, \F\} \ = \ -\{\sim \Int \, \F\}.
\end{equation}

Now, each constraint   set $\F$ determines a potential theory of {\em $\F$-subharmonic functions}.  For a $C^2$-function $u$ on an open subset $X\ss \R^n$ is {\em $\F$-subharmonic} on $X$
if 
\begin{equation}\label{SH_defn_1}
J^2_{x_0}  u := (u(x_0), Du(x_0), D^2 u(x_0) ) \in \F \ \ \text{for all} \ \ x_0 \in X.
\end{equation}
Using viscosity theory, this condition can be transferred pointwise from the $2$-jet $J^2_{x_0}u$ to the set of upper test jets (see Defintion \ref{defn:ULTF} by requiring 
\begin{equation}\label{SH_defn-2}
J^{2}_{x_0}  \varphi  \in \F \ \ \text{for all upper test functions} \ \varphi \ \text{for $u$ at $x_0 \in X$},
\end{equation}
thereby extending the notion of $\F$-subharmonic from $C^2$-functions to the space $\USC(X)$  of all upper  semi-continuous, $[-\infty, \infty)$-valued functions on $X$.

In addition   to the notion of duality \eqref{DD_defn}, the other fundamental concept for this paper is monotonicity.

\begin{defn}\label{defn_MC}   A {\bf{\em monotonicity cone}} for  a subequation $\F$  is a  cone $\cM\ss \J^2$
	(with vertex at the origin) such that 
	\begin{equation}\label{MC_defn}
	\F + \cM\ \ss\ \F,
	\end{equation}
	and in this case we say that $\F$ {\em is} {\bf {\em $\cM$-monotone}}.
\end{defn}
\noindent Note that since $\cM$ contains the origin the inclusion \eqref{MC_defn} is an equality $\F + \cM = \F$.

Since $\F$ is a subequation, one can always enlarge a monotonicity cone to one where
\begin{equation}\label{MC_enlarge_1}
\cM \ \supset \ \cN\times \{0\}\times \cP
\qquad {\rm and}\qquad
\cM\ \ \text{is a closed convex cone}.
\end{equation}
In fact,  the closed convex cone hull of a monotonicity cone is also a monotonicity cone.
For each $\F$ there is a {\em maximal} monotonicity cone.
Moreover, in this paper we are interested in subequations $\F$ which have monotonicity cones
$\cM$ which satisfy \eqref{MC_enlarge_1} and 
\begin{equation}\label{MCS_defn}
\Int \, \cM \ \neq \ \emptyset,
\end{equation}
so that $\cM$ is itself a subequation. (To see this note that for a closed convex  cone $\cM$,
we have $\Int \, \cM \neq \emptyset \iff \cM = \overline {\Int \,  \cM}$.)
From this assumption \eqref{MCS_defn}, which holds for 
many constraint sets (including all second-order, in fact, all gradient-free subequations), many important things 
follow: 

\qquad $\bullet$ \ \ The Correspondence Principle,

\qquad $\bullet$ \ \ Comparison Theorems,  

\qquad $\bullet$ \ \ The Existence of Canonical Operators,

\qquad $\bullet$ \ \ The Existence of Unique Solutions to the Dirichlet Problem.

\qquad $\bullet$ \ \ and Much More, See Below.

\subsection{The Differential Operator Setting.}

We now address the companion setting of differential operators.  There are two cases: the unconstrained case
and the constrained case.

\begin{defn}\label{defn_COSP}  A {\bf {\em compatible operator-subequation pair $(F, \F)$}} consists of either
	
	\noindent
	{\bf {\em The Unconstrained Case:}} \quad  $\F \ =\ \J^2 \quad \text{ and} \quad F\in C(\F)$
	
	\noindent    or

	\noindent
	{\bf {\em The Constrained Case:}}  \ \ $\F \ \ss  \J^2 \ \text{is a subequation} \quad \text{and} \quad F\in C(\F)$, \\
	with the properties
	\begin{equation}\label{compatibility_defn}
	c_0 \ \equiv \ \inf_{\F} F \ >\ -\infty  
	\qquad{\rm and}\qquad
	\partial \F \ =\ \{J\in \F : F(J)= c_0\}.
	\end{equation}
	In both cases $F(\F)$ is called the set of  {\em admissible levels} of the pair.
\end{defn}

Let $\cM\ss\J^2$ be a convex cone with vertex at the origin.  We say that a 
compatible operator-subequation pair $(F, \F)$ is $\cM$-{\bf {\em monotone}} if $\F$ is $\cM$-monotone
and 
\begin{equation}\label{MM_pair}
F(J+J') \  \geq\ F(J) \qquad \forall\, J\in \F\quad{\rm and}\quad \forall J'\in \cM.
\end{equation}
If $\cM \supset \cN\times \{0\}\times \cP$, then $(F, \F)$ is called {\bf {\em proper elliptic}}\footnote{Often in the viscosity literature (such as \cite{CIL92}), one uses the simpler term {\em proper} for the $\cN \times \{0\} \times \cP$-monotonicity, but we prefer the phrase {\em proper elliptic} to recall both the $\cP$-monotonicity (degenerate ellipticity or positivity) and $\cN$-monotonicity (properness or negativity) for us.}
These are the only operators we consider, because of our focus on comparison.

Next these proper elliptic operators are divided into two classes: those which are {\bf {\em  topologically pathological}} meaning, as a function on the 2-jet space, the operator has a level set with interior; with the remaining case being refered to as {\bf {\em  topologically tame}}. The topologically pathological case is discarded here because uniqueness of solutions (and hence comparison) is trivially impossible. Various equivalent formulations of topological tameness appear in Theorem \ref{thm:tameness}. For this topologically tame case, we shall establish a rigorous {\em correspondence principle} between potential theory  and PDE's.

\subsection{The Correspondence Principle.}
This result builds a bridge between nonlinear potential theory (subharmonics for a subequation and its dual) and nonlinear PDE's (admissible viscosity sub/supersolutions of PDEs); in particular, it represents the part of the theory using monotonicty and duality for which the two approaches are equivalent.

We begin with a definition from viscosity theory.

\begin{defn}\label{defn:ULTF}  Let $x_0\in X \subset \R^n$ (an open subset) and $u\in\USC(X)$.
	An {\em upper test function} for $u$ at $x_0$ is a $C^2$ function $\varphi$ defined 
	near $x_0$ with 
	$$ 
	u(x)\leq \varphi(x) \quad \text{and} \quad u(x_0) =  \varphi(x_0).
	$$
	A {\em  lower test function} for $u$ at $x_0$ is a $C^2$ function $\varphi$
	such that $-\varphi$ is an upper test function for $-u$ at $x_0$
	We will denote by $J^{2, \pm}_{x_0}u \subset \J^2$ the spaces of {\em (upper/lower) test jets for $u$ at $x_0$}; that is, the set of all $J = J^2_{x_0} \varphi$ where $\varphi$ is a $C^2$ (upper/lower) test function for $u$ at $x_0$.
\end{defn}

For compatible pairs $(F, \F)$ which admit a   monotonicity cone subequation, there is a potential theory at each admissible level $c \in F(\F)$.

\begin{defn}\label{defn_F_c-sub}  Let  $(F, \F)$ be a compatible operator-subequation pair,
	which admits a monotonicity subequation $\cM$,
	and let $c \in F(\F)$ be an admissible level. Consider the subequation
	$\F_c \equiv \{J\in \F : F(J)\geq c\}$.
	Let $u\in \USC(X)$ where $X\subset \R^n$ is an open subset.
	Then $u$ is said to be {\em $\F_c$-subharmonic} on $X$ if 
	\begin{equation}\label{F_c-sub}
	J^{2, +}_{x_0}u \subset \F_c \ \ \text{for all}  \ x_0 \in X.   
	\end{equation}
	A function $v\in \LSC(X)$ is said to be {\em $\F_c$-superharmonic} on $X$ if
	$-v$ is $\widetilde \F_c$-subharmonic on $X$. By duality,  this is equivalent to asking that $J^{2,-}_{x_0}v \subset \sim \Int \, \F_c$ for each $x_0 \in X$.
\end{defn}

We now present  an essential notion for the constrained case; that is, of admissible viscosity subsolutions and supersolutions where the subequation $\F$ places a constraint on the upper/lower test jets that compete in the defintion. In particular, part (b) in the definition below makes systematic what is often done in an ad hoc way in the literature.

\begin{defn}\label{defn_AVSS}  
	Let  $(F, \F)$ be a compatible operator-subequation pair as above. Let $\Omega$ be a domain in $\R^n$ and let $c \in F(\F)$ be an admissible level.
	
	(a)\ \  A function $u \in \USC(\Omega)$ is said to be an  {\em $\F$-admissible viscosity subsolution of $F(u,Du,D^2u) = c$ in $\Omega$} if for every $x_0 \in \Omega$ one has
	\begin{equation}\label{AVsub_defn}	 
	\mbox{$J \in J^{2, +}_{x_0}u \ \ \Rightarrow \ \   J \in \F$ \ \ \text{and} \ \ $F(J) \geq c$.}
	\end{equation}
	
	(b)\ \  A function $w \in \LSC(\Omega)$ is said to be an {\em $\F$-admissible viscosity supersolution  of $F(u,Du,D^2u) = c$ in $\Omega$} if
	\begin{equation}\label{AVsuper_defn}		
	\mbox{$J \in J^{2, -}_{x_0} w  \ \ \Rightarrow$ \ \ either [ $J \in \F$ and \ $F(J) \leq c$\, ] \ or \ $J \not\in \F$.}
	\end{equation}	
\end{defn}		

A main result of this paper is the following theorem. 

\noindent
{\bf Thoerem \ref{cor:AVSolns}}.\ ({\bf The Correspondence Principle for Compatible Pairs}): {\em Let $(F,\F)$ be a compatible proper elliptic operator-subequation pair,
	which is $\cM$-monotone for a convex cone subequation $\cM$.
	Suppose also that   $F$ is topologically tame.
	Let $c\in F(\F)$ be an admissible value, and set $\F_c \equiv \{J\in \F : F(J)\geq c\}$ as above.  Fix a domain} $\Omega\ss\R^n$. Then:
\begin{itemize}
	\item[(a)] {\em $u \in \USC(\Omega)$ is an $\F$-admissible viscosity subsolution of $F(u,Du,D^2u) = c$ in $\Omega$ if and only if $u$ is $\F_c$-subharmonic on $\Omega$;}
	\item[(b)] {\em $u \in \LSC(\Omega)$ is an $\F$-admissible viscosity supersolution of $F(u,Du,D^2u) = c$ in $\Omega$ if and only if $u$ is $\F_c$-superharmonic on $\Omega$; }
	\item[(c)] {\em comparison for the subequation $\F_c$ on a domain $\Omega$ is valid if and only if comparison for the equation $F(u, Du, D^2 u) = c$ on $\Omega$ is valid.}
\end{itemize}

The correspondence principle is a very general and powerful tool, which needs to be ``unpacked'' in order to fully appreciate it. First, there an important {\bf dichotomy} between the {\em unconstrained case} $(F, \J^2)$ in which the operator $F$ is proper elliptic on all of $\J^2$ and the {\em constrained case} $(F, \F)$ where $F$ is proper elliptic only when restricted to some compatible subequation $\F$. 
Note that  in the constrained case, the constraint set  $\F$  on the domain of the operator $F$ is used in the Definition\ref{defn_AVSS} of  {\em $\F$-admissible sub/supersolutions}, while in the unconstrained case sub/supersolutions are in the standard viscosity sense.

Second, using this principle, one can reduce PDE comparison to potential theoretic comparison, in order to free the operator from its particular form, retaining only the need to analyze its maximal monotonicity cone $\cM$. This is done in section \ref{sec:CP_operators} for many classes of operators.

\subsection{\bf Canonical  Operators.}
This collection of operators gives some of the best illustrations of the unconstrained case.
The construction starts with a subequation $\F$ which admits a 
monotonicity subequation $\cM$.  One then chooses an element $J_0\in \Int \, \cM$. 
Associated to this is a {\bf {\em canonical operator}}  $F\in C(\J^2)$, defined on all of $\J^2$,
with very nice properties.
It is canonically  defined via  the Structure Theorem \ref{thm:structure} which 
says that for each $J\in \J^2$, the set
\begin{equation}\label{Interval_J}
I_J := \{ t \in \R: \ J + tJ_0 \in \F \} 
\end{equation}
is a closed interval of the form $[t_J, +\infty)$ with $t_J \in \R$ (finite). Moreover
\begin{itemize}
	\item[(a)] $J + t J_0 \not\in \F \ \iff \ t < t_J$;
	\item[(b)] $J + t_J J_0 \in \partial \F$;
	\item[(c)] $ J + t J_0 \in \Int \, \F \ \iff \ t > t_J$;
\end{itemize}

The canonical operator $\F : \J^2 \to \R$ is then defined by
\begin{equation}\label{CO_defn} 
F(J)  \ =\ - t_J
\end{equation}	
and it has the following properties. $F$ decomposes $\J^2$ into three disjoint pieces
\begin{equation}\label{CO_properties}
\partial \F = \{F(J) = 0\}, \ \  \Int \, \F = \{  F(J) > 0\} \ \	 \text{and} \ \	\J^2 \setminus \F = \{ F(J) < 0\},
\end{equation}
and  $F$ is strictly increasing in the direction $J_0$.  In fact,
$F(J + t J_0) = F(J) + t$ for each $t \in \R$.
Furthermore,	 $F$ is proper elliptic on $\J^2$,  and in fact, it is $\cM$-monotone.	It is also Lipschitz.
(See Propositions \ref{prop:canonical_structure}, \ref{prop:graphing_function} and \ref{prop:Lipschitz_g}.)

Interestingly, there exist important cases where this is the only construction of a  good operator,
that is, there are no polynomial operators that we know.  Examples come from the geometrical  potential theories
for Slag, G(2) and Spin(7).

We have the following two results.

\noindent
{\bf Theorem \ref{thm:CanOp_ComPair} .}\ ({\bf Canonical operators and compatible pairs}).    {\em
	Suppose that a subequation $\F$ admits a monotonicity cone subequation $\cM$. Let $F \in C(\J^2)$ be the canonical operator for $\F$ determined by any fixed $J_0 \in \Int \, \cM$. Then: 
	\begin{itemize}
		\item[(a)] $(F, \J^2)$ is an unconstrained proper elliptic operator-subequation pair;
		\item[(b)] $F(\J^2) = \R$ and the operator $F$ is topologically tame; 
		\item[(c)] for each $c \in \R$, the set $\F_c:= \{ J \in \J^2: \ F(J) \geq c \}$ is a subequation constraint set with $\F_0 = \F$ and the pair $(F, \F_c)$ satisfies the compatibility conditions 
		$$
		\inf_{\F_c} F = c \quad \text{and} \quad \partial \F_c = \{J \in \F_c: \ F(J) = c \}.
		$$
	\end{itemize} 
	In addition, the canonical operator (determined by $J_0 \in \Int \, \cM$) for the dual subequation $\wt{\F}$ is given by 
	$$
	\wt{F}(J) := - F(-J) \ \ \text{for all} \ J \in \J^2. \qquad \text{Also note that } \ \wt{\F}_c = \F_{-c}.
	$$
	The analogous statements of (a), (b) and (c) hold for $(\wt{F}, \J^2)$ and $(\wt{F}, \wt{\F}_c)$.
}

\noindent
{\bf Theorem \ref{thm:CP_canonical_ops}.}\  ({\bf Comparison for canonical operators}).	{\em
	Let $\F \subset \J^2$ be a subequation constraint set which admits a monotonicity cone subequation $\cM$. 
	Further, suppose that $\cM$ admits a strict approximator $\psi$ on a bounded domain $\Omega$; that is, $\psi \in C(\overline{\Omega}) \cap C^2(\Omega)$ such that $J^2_x \psi \in \Int \, \cM$ for each $x \in \Omega$. Then, for each $J_0 \in \Int \, \cM$ fixed, the canonical operator $F$ for $\F$ determined by $J_0$ satisfies the comparison principle at every level $c \in \R$; that is,
	$$
	\mbox{$u \leq w$ on $\partial \Omega \ \ \Rightarrow \ \  u \leq w$ on $\Omega$}
	$$
	for $u \in \USC(\overline{\Omega})$ and $w \in \LSC(\overline{\Omega})$ which are respectively 
	viscosity subsolutions and supersolutions to $F(u,Du,D^2u) = c$ on $\Omega$.
}

This gives rise to many beautiful  operator subequation pairs, starting simply with just the subequation $\F$ itself.

\begin{exe}[{\bf {\em Minimal eigenvalue operator}}]\label{exe:MEO}  For the simplest example of a canonical operator,
	let $\F =  \R\times \R^n\times \cP$ (the real convexity  subequation),
	and take $J_0 =  (0,0,{1\over n} I)$.   Then
	$$
	F(J) \ =\ F(r,p,A) \ =\ \lambda_1(A) \qquad\text{(the smallest eigenvalue of $A$)}.
	$$
	Of course, there are many other operators which are compatible with $\F$ and are zero on $\partial \F$, such as 
	$\det(A)$ or $\det(A)^{1\over  n}$.  However, for any such operator we know
	from Theorem \ref{thm:CP_canonical_ops} above  that {\bf  comparison always holds.}
\end{exe}

It is interesting to note that all linear operators are canonical (see Lemma \ref{lem:LCO}). In addition, the concave operator $F$ which is the infimum over a suitably renormalized {\em pointed family}  $\mathfrak{F} = \{F_{\sigma}\}_{\sigma \in \Sigma}$ of linear operators is also the canonical operator for the convex cone subequation $\F$ which is the intersection the associated half-space constraint sets $\{\F_{\sigma}\}_{\sigma \in \Sigma}$. (See Theorem \ref{thm:HJB_inf}). Similar considerations hold for the canonical supremum operator associated to the closure of the union of the $\F_{\sigma}$ (see Remark \ref{rem:HJB_sup}).The precise notion of being pointed is given in Definition \ref{defn:pointed_set} and is a geometrical hypothesis (see Remark \ref{rem:pointed}) on the set of coeffient vectors 
$$S = \{ J_{\sigma} = (a_{\sigma}, b_{\sigma}, E_{\sigma})\}_{\sigma \in \Sigma} \subset \J^2
$$ 
defining the operators in the family by
\begin{equation}\label{linear_intro}
F_{\sigma}(J) = F_{\sigma}(r, p, A):= {\rm tr}(E_{\sigma} A) + \langle b_{\sigma}, p \rangle + a_{\sigma} r = \langle J_{\sigma} , J \rangle, \ \ J \in \J^2. 
\end{equation}
In the proper elliptic case, where each $(a_{\sigma}, E_{\sigma}) \in \cN \times \cP$, one also has the validity comparison principle (see Theorem \ref{thm:CP_HJB}) which depends on the interesting fact that a necessary and sufficient condition for the canonical operator for $\F$ to be $\cM$-monotone is that $S$ is contained in the convex polar cone $\cM^{\circ}$ of $\cM$. To facilitate the application of Theorem \ref{thm:CP_HJB}, the polars of many monotonicity cones are listed in Propsoition \ref{prop:polar_cones}. The following example application comes from optimal control and is discussed in Example \ref{exe:HJB2}.

\begin{exe}[{\bf {\em  Hamilton-Jacobi-Bellman operators with directed drift}}]\label{exe:HJBO}   
	In optimal control, one problem concerns an agent who seeks to minimize  an infinite-horizon discounted cost functional by acting on its drift and volatility parameters. The relavant operator to consider is the infimum over a familly of linear operators like \eqref{linear_intro}, where we will specialize to
	\begin{equation}\label{OCO}
	F_\sigma(J) = F_\sigma (r,p,A) = {\rm tr}(E_\sigma A) + \langle
	b_\sigma, p \rangle + cr = \langle J_{\sigma}, J \rangle, \sigma \in \Sigma
	\end{equation}
	where $\delta:= -c > 0$ is the {\em discount factor}, $b_\sigma$ is the {\em drift term}
	and $E_\sigma$ is the {\em (squared) volatility}. Under the assumptions that $E_\sigma$ is allowed to vary in bounded sets and the set of drifts $S_d:= \{b_\sigma\}_{\sigma \in \Sigma}$ share a ``preferred'' direction $b_0$ (the family is pointed with axis $b_0 \in \R^n \setminus \{0\}$), Theorem \ref{thm:CP_HJB} shows that the comparison principle holds on arbitrary bounded domains for the equation $F(u, Du, D^2u) = c$ for each $c \in \R$. 
\end{exe} 

We now consider an important example of an unconstrained operator that is {\em not} a canonical operator.
This particular equation has received much attention in recent years from quite varied points of view.
There is  some history in \cite{HL20}.

\begin{exe}[{\bf {\em Special Lagrangian potential operator}}]\label{exe:SLPO}  This pure second order operator was introduced along with special Lagrangian geometry in \cite{HL82}. It takes the form
	$$
	F(A):= \sum_{k=1}^n \arctan(\lambda_k(A))
	$$
	and is $\cP$-monotone on all of $\J^2$. Comparison on arbitrary bounded domains 
	holds for the equation $F(D^2u) = c$ at all admissible levels $c \in (-n\pi/2, n \pi/2)$.
\end{exe}

\subsection{\bf Gradient-Free Operators.}
Given a subequation $\F$  our results apply if  the maximal monotonicity cone  of $\F$  has interior. 
However, notice that {\sl this is true
	for every pure second-order subequation} $\F = \R\times \R^n\times \F_0$ since $\R\times \R^n\times \cP$ 
is always a monotonicity subequation for $\F$.   In fact,  {\sl this is true
	for every pure gradient-free subequation} $\F$, since $\cN\times \R^n\times \cP$ 
is always a monotonicity subequation for $\F$ by Definition \ref{defn:GFSE1} of {\em gradient-free}. We have the following result.

\noindent
{\bf Theorem \ref{thm:E0}.}  ({\bf Comparison in the  Gradient-Free Case}).
{\em Suppose that $(F, \F)$ is a compatible, gradient-free pair.  Then for every bounded domain
	$\Omega$ and every $c\in F(\F)$, one has the comparison principle:
	$$
	u\ \leq\ w\ on\ \partial \Omega \quad \Rightarrow\qquad 
	u\ \leq\ w\ on\  \Omega
	$$
	for  $u\in \USC(\overline\Omega)$ and $w\in \LSC(\overline\Omega)$ where
	$u$ is $\F_c$-subharmonic and $w$ is $\F_c$-super-harmonic (i.e., $-w$ is $\widetilde\F_c$-subharmonic). }

We have seen that  the unconstrained case is best illustrated by canonical operators.
The constrained case  is best illustrated by operators involving G\aa rding hyperbolic polynomials, 
which we examine next.

\subsection{\bf Operators Involving Dirichlet-G\aa rding Polynomials.}
G\aa rding's theory \cite{Ga59} provides  a unified approach to studying many of the most important subequations.  The reader should look at section \ref{subsec:garding} for more details and to \cite{HL13a} and \cite{HL10}
for a modern self-contained treatment. A  {\bf {\em Dirichlet-G\aa rding polynomial}} is a homogeneous polynomial $\pol$ of 
degree $m$ on $\Symn$ with the following properties.  

(1)\ \ ($I$-Hyperbolicity). For each $A\in \Symn$, the  
polynomial $p_A(t) \equiv\pol(tI +A)$ has all real roots. The negatives of these real roots are called the {\bf {\em G\aa rding $I$-eigenvalues}} of $A$ 
and up to permutation can be written in increasng order as  $\lambda^\pol_1(A) \leq 
\lambda^\pol_2(A) \leq  \cdots \leq \lambda^\pol_m(A)$.

(2)  \ \ (Positivity). We assume $\pol(I)>0$ and define the {\bf {\em G\aa rding cone}} $\Gamma$  to be the connected component of 
$\Symn \setminus \{\pol=0\}$ which contains the identity $I$.  This is a convex cone (see Theorem 11.30), given by those $A$ with
$\lambda^\pol_1(A) > 0$.
We assume the positivity property 
\begin{equation}\label{DG_Cone_1}
\overline\Gamma+\cP \subset \overline\Gamma, \ \ \text{which is equivalent to either} \ \overline\Gamma+\cP = \overline\Gamma \ \text{or} \  \cP \subset \overline\Gamma,
\end{equation}
since $\cP$ contains the origin and $\overline\Gamma$ is a convex cone.

We normalize so that $\pol(I)=1$. Then we have 
\begin{equation}\label{DG_Cone_2}
\pol(tI + A) \ =\ \prod_{k=1}^m (t+\lambda^\pol_k(A)),
\end{equation}
which  when evaluated in $t = 0$ shows that each Dirichlet-G\aa rding operator is a {\em generalized Monge-Amp\`{e}re operator}, where the G\aa rding $I$-eigenvalues of $A$ take the place of the standard eigenvalues of $A$ in the special case $\pol = {\rm det}$.

If $\pol(A)$ be a  Dirichlet-G\aa rding polynomial on $\Symn$
with closed G\aa rding cone $\overline\Gamma$, then this    gives rise to a pure second-order polynomial
operator $\pol(D^2 u)$ constrained  to the 
pure second-order subequation $\R\times \R^n \times \overline\Gamma$.
These are discussed 
at length in Section 11.6.  Simple examples are given by the elementary symmetric functions
of the eigenvalues of $A$ (the so-called {\em Hessian equations}). 
There are many  more interesting examples, including the 
Lagrangian Monge-Amp\`ere operator (see Examples \ref{exes:DG_polys} (4)), the   {\em geometric  $k$-convexity operator} (see Example \ref{exe:GKCO}  below)). Moreover, each of these {\em universal examples}  (defined in terms of the standard eigenvalues $\lambda_k(A)$), gives rise to a huge family of examples by simply replacing the standard eigenvalues $\lambda_k(A)$ by the  G\aa rding $I$-eigenvalues $\lambda_k^{\pol}(A)$ of $A$ for any G\aa rding $I$-hyperbolic polynomial $\pol$ on $\cS(n)$. 

There are many interesting equations which involve Dirichlet-G\aa rding polynomials $\pol(A)$.
We now look at some of the examples.

\begin{exe}[{\bf {\em $k$-Plurisubharmonicity, the truncated Laplacian and the geometric $k$-convexity operator}}]\label{exe:GKCO} 
	These  examples were introduced in \cite{HL09} (see page 39). Here they  illustrate the
	general fact that given a G\aa rding polynomial, there are two natural operators, the G\aa rding operator defined directly by $\pol$ and the canonical operator for the G\aa rding cone $\overline{\Gamma}$ determined by $\pol$. 
We discuss an interpolation of operators between them.
	First we define the  potential theory, which is quite interesting.  Fix an integer $k$,  $1\leq k\leq n$.
	A $k$-{\em plurisubharmonic function} is defined by requiring that its restriction
	to every affine $k$-plane is classically Laplacian subharmonic (or $\equiv -\infty$).
	The subequation $\cP(k)$ is defined by requiring that $A\in \Symn$
	restricts, as a quadratic form, to have a positive  trace on all affine $k$-planes.
	The $k$-plurisubharmonic functions are exactly the $\cP(k)$-subharmonics.
	
	The canonical operator is the {\bf {\em truncated Laplacian}}
	\begin{equation}\label{TLO}
	\Delta_{\cP(k)}(A) \ \equiv\ \lambda_1(A) + \cdots + \lambda_k(A)
	\qquad (\lambda_1 \leq \cdots\leq \lambda_n).
	\end{equation}
	There is also a polynomial Dirichlet-G\aa rding operator 
	\begin{equation}\label{GkCO}
	T_k(A) \ =\ \prod_ {i_1< \cdots<i_k} (\lambda_{i_1}(A) + \cdots + \lambda_{i_k}(A)),
	\end{equation}
	which we call the {\bf {\em geometric $k$-convexity operator}}. This yields two compatible 
	operator-subequation pairs using the canonical operator and a Dirichlet-G\aa rding operator; namely 
	$$
	(\Delta_{\cP(k)}, \cP(k))
	\qquad {\rm and} \qquad
	(T_k, \cP(k)),
	$$
	and yields a new interpolated sequence between the pairs  $(\lambda_1,\cP)$ and $(\det, \cP)$ at the $k=1$ end,
	and the identical  pairs $(\Delta, \{\tr\geq 0\})$ and $(\Delta, \{\tr\geq 0\})$ at the $k=n$ end.  
	The canonical operator has been studied in \cite{BGI18} where the 
	terminology {\em truncated Laplacian} 
	was introduced.
\end{exe}

We point out that $\cP(k)$-subharmonic functions restrict to be subharmonic on all 
$k$-dimensional  minimal submanifolds \cite{HL14}.

\begin{exe}\label{exe:LPPTO} The reader might enjoy the article \cite{HL17}  where one has a full blown Lagrangain plurisubharmonic potential theory complete with an operator of  ``Monge-Amp\`ere type'' in  Lagrangian geometry. 
\end{exe}

\begin{exe}[{\bf {\em Branches of a G\aa rding-Dirichlet operator}}]\label{exe:BGDO}
	In  subsection \ref{subsec:branches} we discuss the general notion of branches.   A  {\bf {\em  branch}} is a closed subset
	of $\J^2$ which is the boundary of a subequation. Given a Dirichlet-G\aa rding polynomial $\pol$
	of degree $m$, there are $m$ distinct branches
	$$
	\Lambda_1^\pol \ \subset\ \Lambda_2^\pol \ \subset\  \cdots \ \subset\ \Lambda_m^\pol 
	\ \ \ {\rm where} \ \ 
	\Lambda_k^\pol \ =\ \{\lambda^\pol_k \geq 0\}.
	$$
	Our theory applies to all of these branches, because they are pure second-order.
	(The only natural operator for these branches is the canonical operator $\lambda^\pol_k$ unless $k=1$.)
\end{exe}

\begin{exe}[{\bf {\em Gradient-Free operators with a Dirichlet-G\aa rding factor}}]\label{exe:GFDGO} 
	Let $\pol(A)$ be a  Dirichlet-G\aa rding polynomial as above, and 
	let $h\in C((-\infty, 0])$ be non-negative, non-increasing and with $h(r)=0 \iff r=0$.
	Consider the operator 
	\begin{equation}\label{GFDGO}
	F(r,p,A) \ =\ h(r) \pol(A).
	\end{equation}
	Restricting $F$ to the subequation $\F = \cN \times \R^n \times  \overline\Gamma$
	gives a compatible gradient-free pair, and hence comparison holds at every admissible level
	on every bounded domain.
\end{exe}

An interesting special case comes from affine hyperbolic geometry, as presented in Example \ref{exe:cheng_yau}.

\begin{exe}[{\bf {\em The hyperbolic affine sphere equation}}]\label{exe:HASE} The partial differential equation 
	\begin{equation}\label{HASE}
	{\rm det}(D^2 u) = \left( \frac{L}{u} \right)^{n+2}, \ \   L \leq 0,   \qquad  {\rm i.e., } \ \ 
	(-r)^{n+2} \det(A) = (-L)^{n+2}
	\end{equation}
	arises in the study of {\em hyperbolic affine spheres} with mean curvature $L$ where $u < 0$ is convex and vanishes on the boundary of $\Omega \subset \R^n$ convex (see Cheng-Yau \cite{CY86}). This equation  is covered by the Example above if one takes $\pol(A) = {\rm det}(A)$ and $h(r) = (-r)^{n+2}$ and $c = (-L)^{n+2} \geq 0$ are the admissible levels. 
\end{exe}

The next example   illustrates a new construction in subsection \ref{subsec:garding} (see Lemma \ref{lem:build_DG} which produces a gradient-free Dirichlet-G\aa rding operator from a pure second order Dirichlet-G\aa rding operator.

\begin{exe}\label{exe:POS_GF} For each Dirichlet-G\aa rding polynomial $\pol$ of degree $m$ on $\cS(n)$ with G\aa rding $I$-eigenvalues of $A$ given by $\lambda_k^{\pol}(A)$, $k = 1, \ldots, m$, the operator 
	\begin{equation}\label{POS_GF}
	\polh(r,A) \ =\ \prod_{k=1}^m (\lambda_k^{\pol}(A) - r) \ =\ \pol(A-r I)
	\end{equation}
	is a $\left(-\frac{1}{2}, \frac{1}{2} I\right)$-hyperbolic Dirichlet-G\aa rding polynomial of degree $m$ on $\R \times \cS(n)$ (normalized to have $\polh \left(-\frac{1}{2}, \frac{1}{2} I\right) = 1$) with G\aa rding eigenvalues $\lambda^\polh_k(A) = \lambda_k^{\pol}(A) - r$.
\end{exe}

Now we consider  an example with gradient dependence which requires an additional {\em directionality property} (D) with respect to a {\em directional cone} $\cD$ (see Defintion 2.2).

\begin{exe}[{\bf {\em Example \ref{exe:GFDGO} with a directional cone}}]\label{exe:DGD} Let $\pol$ and $h$ be as in Example \ref{exe:GFDGO} above,
	and consider a continuous $d: \cD \to \R$,  where $\cD \subsetneq \R^n$ is a directional cone, with 
	\begin{equation}\label{dD1}
	d\ \geq\ 0 \ \ {\rm and } \ \ d(p)=0   \iff p\in\partial \cD
	\end{equation}
	\begin{equation}\label{dD2}
	d(p+q) \ \geq\ d(p) \ \text{for each $p,q \in\cD$}.
	\end{equation}
	Then the operator 
	\begin{equation}\label{DGD_operator}
	F(r,p,A) \ =\ h(r) d(p) \pol(A)
	\end{equation}
	with restricted domain
	\begin{equation}\label{DGD_domain}
	\F \ =\  \cN \times \cD\times \overline{\Gamma}
	\end{equation}
	defines a compatible $\cN \times \cD \times \cP$-monotone operator-subequation pair $(F, \F)$. Hence, comparison holds on arbitrary bounded domains at every admissible level of $F$.	
	
	\noindent{\bf Note:} Some examples of such pairs $(d, \cD)$ are:
	\begin{equation}\label{dD_pair_1}
	d(p)= p_n \ \ \text{and} \ \ \cD = \{ (p', p_n) \in \R^n: \ p_n \geq 0 \} \ \ \text{(a half-space)},
	\end{equation}
	and for  $k \in \{1, \ldots, n \}$
	\begin{equation}\label{dD_pair_2}
	d(p)= \prod_{j = 1}^k p_j \ \ \text{and} \ \ \cD = \{ (p_1, \ldots , p_n) \in \R^n: \ p_j \geq 0 \ \text{for each} \ j= 1, \ldots  k \}.
	\end{equation}
\end{exe}

An interesting special case of Example \ref{exe:DGD} concerns parabolic operators, which are discussed in subsection \ref{subsec:parabolic} in both constrained (Theorem \ref{thm:PCP_CC}) and unconstrained cases (Theorem \ref{thm:CP_P1})

\begin{exe}[{\bf {\em Parabolic operators}}]\label{exe:PO} In the case where the gradient pair $(d, \cD)$ is defined by \eqref{dD_pair_1}, 
	$h\equiv 1$, and $\pol(A)$ is replaced by $G(A')$  which depends only on $A' \in \cS(n-1)$ (second order derivatives only in the {\em spatial variables} $x' \in \R^{n-1}$), one has a fully nonlinear {\em parabolic} operator 
	\begin{equation}\label{PO}
	F(r,p,A) := p_n G(r, A').
	\end{equation}
	of the kind considered by Krylov in his extension of Alexandroff's methods to parabolic equations in \cite{Kv76}. The compatible subequation is described in formula \eqref{FFpair} of the paper.
\end{exe}

Another interesting special case of Example \ref{exe:DGD} comes from a very particular form of optimal transport with quadratic cost, as presented in Example \ref{exe:OTE}. 

\begin{exe}[{\bf {\em Potential equation for optimal transport with uniform source density and directed target density}}]\label{exe:OTE_intro} 
	Equations of the form
	\begin{equation}\label{OTO_defn}
	d(Du) \, {\rm det}(D^2u) = c, \ \ c \geq 0
	\end{equation}
	arise in the theory of optimal transport, under some restrictive assumptions. In general, there would be a function $f = f(x)$ in place of the constant $c$, where $f$ represents the mass density in the source configuration and $d$ represents the mass density of the target configuration (with he mass balance $||f||_{L^1} = ||d||_{L^1}$). 
	One seeks to transport the mass with density $f$ onto the mass with density $d$ at minimal transportation cost (which is quadratic respect to transport distance). The solution of this minimization problem is given by the gradient of a convex function $u$, which turns out to be a generalized solution of the equation \eqref{OTO_defn}. In the special case of uniform source density $f \equiv c$ and with target density $d$ having some directionality, comparison principles can be obtained as a special case of Example \ref{exe:DGD} with $h(r) := 1$ and $\pol(A) := {\rm det} \, A$.

	

	Thus we see  that  seemingly diverse equations can be established from a surprisingly {\bf unified point of view}.
	It frees the theory from any particular form of the operator. Given a potential theory; that is, given a subequation constraint set $\F$, there are many natural choices for an associated operator. If $\F$ has {\bf sufficient monotonicity}; that is, if $\F$  admits a monotonicity cone subequation $\cM$ (i.e., the maximal monotonicity cone has interior), there is always one choice that is ``canonical'', but for proving useful estimates, other choices may be better choices. For instance, a polynomial operator, if there is one, may be preferable. Restricting attention, as we do here, to the continuous version of the Dirichlet Problem (DP), the correspondence principle enables a single potential theory/subequation result to be applied to all of the many compatible operators $F$ associated to the subequation $\F$. 
\end{exe}


\subsection{General Potential Theoretic Comparison Theorems.}
One   of the important parts of this paper is   understanding convex  cone subequations $\cM\subset \J^2$ and the comparison results for subequations $\F$ which are $\cM$-monotone. 

By comparison results, we mean the validity of the comparison principle on bounded domains $\Omega \subset \R^n$; that is,
\begin{equation}\label{CP1_intro}
u \leq w \ \ \text{on} \ \ \partial \Omega \ \ \Rightarrow \ \ u \leq w \ \ \text{on} \ \ \Omega 
\end{equation}
for all $u \in \USC(\overline{\Omega})$ and $w \in \LSC(\overline{\Omega})$ which are respectively $\F$-subharmonic and $\F$-superharmonic on $\Omega$. By duality, this is equivalent to showing
\begin{equation}\label{CP2_intro}
u + v \leq 0 \ \ \text{on} \ \ \partial \Omega \ \ \Rightarrow \ \ u + v \leq 0 \ \ \text{on} \ \ \Omega 
\end{equation}
for all $u,v \in \USC(\overline{\Omega})$ which are respectively $\F, \wt{\F}$-subharmonic on $\Omega$. Our method of proof for $\cM$-monotone subequations $\F$ makes use of this second formulation.

Here is a guide to the method. There are four steps.

\noindent
{\bf Step 1.\ (Jet addition)} We have the following elementary but important fact concerning constraint sets,  monotonicity and duality:
\begin{equation}\label{SAT_intro}
\F + \cM \ \subset\ \F
\qquad\iff\qquad
\F + \widetilde \F \ \subset\ \widetilde \cM.
\end{equation}
So the monotonicity condition on the left is equivalent to  the condition on the right which is perfect for comparison, as one sees from \eqref{CP2_intro}.

Showing that this infinitesimal statement passes to a potential theoretic statement is the hard analysis step in the method.

\noindent
{\bf Step 2.\ (Subharmonic addition)} We prove the following potential theoretic result.

\noindent 
{\bf Theorem \ref{thm:SAT_MD}.\ (Subharmonic addition, monotonicity and duality)}
{\em Suppose that $\cM\subset\J^2$ is a monotonicity cone subequation and that $\F\subset \J^2$
	is an $\cM$-monotone subequation constraint set.  Then for every open set $X\subset \R^n$, one has}
\begin{equation}\label{SAT_intro}
\F(X) + \widetilde \F(X) \ \subset \ \widetilde \cM(X).
\end{equation}
(where $\F(X)$ is the set of $u\in \USC(X)$ which are $\F$-subharmonic on $X$).

\noindent
{\bf Step 3.\  (Reduce comparison to the  Zero Maximum Principle  for $\widetilde\cM$)}  Armed with Theorem \ref{thm:SAT_MD}, it is clear from \eqref{CP2_intro} that comparison for $\F$ on $\Omega$ will hold if we can prove the
\noindent
{\bf {\em Zero Maximum Principle (ZMP) for $\widetilde \cM$}} on a bounded domain $\Omega \subset \R^n$; that is,
\begin{equation}\label{ZMP_intro}
z\ \leq \ 0  \ \ {\rm on}\ \ \partial \Omega 
\qquad\Rightarrow \qquad 
z\ \leq \ 0  \ \ {\rm on}\  \ \Omega 
\end{equation}
for all $z\in \USC(\overline \Omega)$ which are $\widetilde \cM$-subharmonic on $\Omega$.

\noindent
{\bf Step 4.\ (Prove the Zero Maximum Principle for $\widetilde\cM$)} A key concept in the proof is the following.

\begin{defn}\label{defn:SA_intro} Suppose that $\cM$ is a convex cone subequation. 
	Given a domain $\Omega \subset \subset \R^n$, we say $\cM$ {\em admits a strict approximator on $\Omega$} if there exists $\psi$ with
	\begin{equation}\label{SA_intro}
	\psi \in C(\overline{\Omega}) \cap C^2(\Omega)  \ \text{and} \ \ J^2_x \psi \in  \Int \, \cM \ \ \text{for each} \ x \in \Omega.
	\end{equation}
\end{defn}
This important notion gives a sufficient condition for proving the (ZMP) for $\wt{\cM}$ and hence comparison for $\F$. 

\noindent
{\bf Theorem \ref{thm:ZMP}.\ (The Zero Maximum Principle)} 
{\em  Suppose   that $\cM$ is a convex cone subequation that admits a 
	strict approximator on $\Omega$. Then the zero maximum principle (ZMP) 
	holds for $\wt{\cM}$ on $\overline{\Omega}$ }

Putting these four steps  together gives the following. 

\noindent
{\bf Theorem \ref{thm:CP_general}\ (The General Comparison Theorem)}  {\em
	Suppose that $\cM\subset\J^2$ is a monotonicity cone subequation and that $\F\subset \J^2$
	is an $\cM$-monotone subequation constraint set.  Suppose   that $\cM$ is a convex cone subequation that admits a 
	strict approximator on $\Omega$.  Then comparison holds for $\F$ on $\overline \Omega$.  That is, 
	given $u,v\in\USC(\overline{\Omega})$ where $u$ is $\F$-subharmonic on $\Omega$ and 
	$v$ is $\widetilde \F$-subharmonic on $\Omega$, then
	$$
	u+v \ \leq\ 0 \ \ {\rm on}\ \ \partial \Omega
	\qquad\Rightarrow\qquad
	u+v \ \leq\ 0 \ \ {\rm on}\ \ \Omega
	$$
	
	The conclusion here can be restated as follows.  
	Given $u\in\USC(\overline{\Omega})$ and  $w\in\LSC(\overline{\Omega})$  
	where $u$ is $\F$-subharmonic on $\Omega$ and 
	$w$ is an  $\F$-superharmonic on $\Omega$, then}
$$
u \ \leq\ w \ \ {\rm on}\ \ \partial \Omega
\qquad\Rightarrow\qquad
u \ \leq\ w \ \ {\rm on}\ \ \Omega
$$

To see this last statement we only need to know that $w$ is  $\F$-superharmonic on $\Omega$
if and only if $v\equiv -w$ is $\widetilde \F$-subharmonic on $\Omega$.

Now in Section \ref{sec:monotonicity} we present and study a list  of fundamental monotonicity cone subequations
$\cM(\gamma, \cD, R)$ where $\cD\subset \R^n$ is a directional cone, $\gamma\in [0, \infty)$ 
and $R\in (0,\infty]$.  In  the Fundamental Family Theorem \ref{thm:fundamental}, it is shown that:
\begin{equation}\label{FF}
\text{Every monotonicity cone subequation contains one of these.}
\end{equation}
We note that
$$
\cM(\gamma):= \left\{ (r,p,A) \in \J^2: \ r \leq - \gamma |p| \right\}
\qquad
\cM(R):= \left\{ (r,p,A) \in \J^2: \ A \geq \frac{|p|}{R}I \right\}
$$	
and 
$$
\cM(\gamma, \cD, R) \ :=\  \cM(\gamma) \cap \cM(\cD) \cap     \cM(R).
$$
where  $\cM(\cD) := \R\times \cD\times \Symn$. 

The fundamental nature of this family of monotonicity cones, together with the general comparison principle of Theorem \ref{thm:CP_general} leads to a main comparison result, Theorem \ref{thm:comparison} (The Fundamental Family Comparison Theorem) which depends on
the cone $\cM(\gamma, \cD, R)$.   For some of these cones, comparison
holds on all bounded domains.  For the others comparison holds only on domains
$\Omega \subset \R^n$ which are subsets of a translation of the truncated
cone $\cD\cap B_R(0)$. 
This is  semi-local comparison with explicit parameters.
Note that by the Fundamental Families result \eqref{FF}, local comparison always holds  (see Theorem \ref{thm:LC}).

Concerning the  applicability of the fundamental comparison result of Theorem \ref{thm:comparison}, it is worth mentioning that larger monotonicity cones $\cM$ for a given subequation $\F$ give a better chance of proving comparison (one more likely to be able to construct a strict approximator) but smaller monotonicity cones $\cM$ will apply to larger families of subequations. In particular, if one would like to know if comparison holds on arbitrary bounded domains, one should search for the largest possible $\cM$, which is perhaps not in the list of the fundamental family. For example, in Theorems \ref{thm:CP_improvents} and \ref{thm:CP_improvents2} we present enlargements of the cones with $R$ finite for which comparison holds on all bounded domains.

On the other hand, the  search for sufficient monotonicity to have comparison on arbitrary bounded domains may be futile. In particular, for $\F:= \cM(R)$ which is its own maximal monotonicity cone, it is shown that the (ZMP) fails for $\wt{\cM}(R)$ on large balls in Proposition \ref{prop:ZMP_failure} and hence comparison also fails for $\F = \cM(R)$ on large balls. This failure of comparison on large balls is extended to interesting subequations $\F$ with maximal monotonicity cone equal to $\cM(R)$ in Proposition \ref{prop:CE1_CP}. The situation can be even worse. 

\begin{rem}[{\bf {\em Failure of local comparison with insufficient monotonicity}}]\label{rem:failure_small}    In Theorem \ref{thm:failure}, we show that comparison can fail on arbitrarily small balls (even if both (P) and (N) hold), if there is insufficient monotonicity  In the examples the maximal monotonicity cone $\cM_{\F}$ has empty interior, hence no strict approximators on any ball,
	no matter how small.
\end{rem}

Concerning Step 3  of our method (in which comparison resuces to the validity of the (ZMP) for the dual $\wt{\cM}$ of the monotonicity cone), the following observation is of interest.

\begin{rem}[{\bf {\em Strong Comparison from the Strong (ZMP)}}]\label{rem:SMP} 
	The monotonicity and duality method can be used to prove a strong comparison principle which, by the Subharmonic Addition Theorem, reduces to proving a a strong (ZMP) for $\wt{\cM}$ on $\overline{\Omega}$; that is,
	\begin{equation}\label{SMP_intro}
	z\ \leq \ 0  \ \ {\rm on}\ \ \partial \Omega 
	\ \ \Rightarrow \ \ 
	z \equiv 0  \ \ \text{or} \ \ z < 0 \ {\rm on}\  \ \Omega 
	\end{equation}
	for all $z\in \USC(\overline \Omega)$ which are $\widetilde \cM$-subharmonic on $\Omega$. This method was used in \cite{HL16c} to prove strong comparison for pure second order subequations. We will not attempt to extend this to the general constant coefficient case in this paper. There is, of course, a rich literature on the strong maximum principle for nonlinear operators including the important work of Bardi and Da Lio initiated in \cite{BD99}, along with recent papers of Birindelli-Galise-Ishhi \cite{BGI18} and Goffi-Pediconi \cite{GP20}. 
	
\end{rem} 

A few additional potential theoretic ingredients are worth mentioning. First, an elaboration on the potential theory underlying Example \ref{exe:HJBO}.

\begin{rem}[{\bf {\em Canonical operators, duality, intersections and unions}}] \label{rem:IUSI} For families $\{\F_{\sigma}\}_{\sigma \in \Sigma}$ of subequations with a common monotonicty cone subequation $\cM$, by using unions, intersections and duality, four interesting $\cM$-monotone subequations are constructed together with their canonical operators (see Theorem \ref{thm:canonical_inf}).
\end{rem}

Next, an elaboration on the gradient-free case.

\begin{rem}[{\bf {\em Subaffine plus functions}}] \label{rem:SAPF}   Subaffine plus theory concerns the potential theory of the gradient free subequation
	\begin{equation}\label{dQ_intro}
	\cQt:= \{(r,A) \in \R \times \cS(n):  \ r \leq 0 \ \ \text{or} \ \ A \in \cP \},
	\end{equation} where $\cPt$ is the pure second order {\em subaffine} subequation.	This $\cQt$ is the dual of the fundamental gradient-free monotoncity cone $\cM = \cQ := \cN \times \cP$ and this potential theory is develped in detail (see Theorems \ref{thm:SAPChar} and \ref{thm:SAPT}). In particular, we extend the elegant method of using subaffine functions (the $\wt{\cP}$-subharmonics) to prove that ``comparison always holds''  for pure second order subequations. Subaffine plus functions (the $(\cQt$-subharmonics) are used to prove that  ``comparison always holds'' for the larger family of gradient-free subequations.  
\end{rem}

\subsection{Limitations of the method and comparison with the literature.}

The monotonicity and duality method presented here applies to a vast array of constant coefficient potential theories and operators, but not all of them. There are many interesting and important examples with insufficient monotonicity to be treated by our method. For example quasilinear operators such as the minimal surface operator
$$
F(p,A):= {\rm tr}(A) - \frac{\langle Ap, p \rangle}{1 + |p|^2}
$$
the $q$-Laplacian (with $1 < q < 2$ or $2 < q < \infty$)
$$
F(p,A):= |p|^{q -2} {\rm tr}(A) + (q-2)|p|^{q -4} \langle Ap, p \rangle
$$
and the infinite Laplacian
$$
F(p,A):= \langle Ap, p \rangle
$$
do not have monotonicity cones $\cM$ with interior, which we require. Such examples (and others) have been treated by Barles and Busca \cite{BB01}. On the other hand, for these reduced (no explicit dependence on the jet variable $r \in \R$) operators $F$, they do require structural assumptions such as their condition (F2) of being {\em strictly elliptic in the gradient direction}. This condition is {\bf not} satisfied by an operator such as 
$$
F(p,A):= d(p) \, {\rm det} \, A,
$$
which is Example \ref{exe:GFDGO} with $h(r)\equiv 1$ and $\pol(A) = {\rm det} \, A$. Such examples can be treated by our method.

Next, we discuss a prototype operator which surprisingly creates difficulty for any method. The operator looks particularly attractive for comparison since it is strictly decreasing in the solution variable $r \in \R$ and is increasing in the hessian variable $A$ when restricted to $\cP \subset \cS(n)$.  Namely, consider the seemingly innocuous operator
\begin{equation}\label{bad_example}
F(r,A) = {\rm det} \, A - r,
\end{equation}
(which is further discussed in Remark 12.7). The operator $F$ is gradient-free and proper elliptic on $\R \times \cP$; that is, it is $\cQ = \cN \times \cP$-monotone on $\F:= \R \times \cP$. However, the potential theory equation $\partial \F$ is not contained in the zero locus $\{ (r,A) \in \F: \ F(r,A) = 0 \}$; that is, $F$ and $\F$ are not compatible. This cannot be remedied by another choice of $\F$, creating a major obstacle to the study of this operator. This incompatibility means that $\F$-superharmonics will not correcpond to $\F$-admissible supersolutions to the equation $F(u,Du,D^2u) = 0$. In order to formulate a notion of admissible supersolution, one could make use of the {\em generalized equation} approach initiated in \cite{HL19} for pure second order equations in which one looks for a second constraint set $\cG$ (different from $\F$) such that 
$$
	\F \cap (-\wt{\cG}) = \{ (r,A) \in \F: \ F(r,A) = 0 \}.
$$
The admissible supersolutions are those $w \in \LSC(\Omega)$ which are $-\wt{\cG}$-subharmonic. We will not pursue this program here. 

In addition to the paper \cite{BB01} discussed above, earleir pioneering work in the constant coefficient case was done by Jensen \cite{Je88}. The equations treated by him are all unconstrained in our language, where the monotoncity properties (P) and (N) do {\bf not} require restricting the domain $F$ to a constraint set $\F$. In subsections \ref{subsec:strict_r} and \ref{subsec:strict_ell}, we recover Jensen's results in this unconstrained setting (see Remark \ref{rem:Jensen1}). Of course, we also treat many constrained cases in the present paper, which is an important motivation for us.

Concerning the constrained case and our notion of compatible pairs $(F, \F)$, we should mention that the special case of Monge-Amp\`{e}re-type equations with the convexity constriant $\cP$ is given in Ishii-Lions \cite{IL90} together with a notion of admissible supersolutions in our language. A similar admissibility notion was also given by Trudinger \cite{Tr90} for prescribed curvature equations and later by Trudinger and Wang for the so-called Hessian equations in a series of papers beginning with \cite{TW97}. As noted previously, another motivation of ours is to treat constrained cases in a robust and general way. The potential theoretic approach initiated in \cite{HL09} was influenced by the important paper of Krylov \cite{Kv95} on the general notion of ellipticity, who championed the idea of freeing a given differential operator $F$ from its particular form by looking instead at the constraint that is imposes on the $2$-jets of subsolutions to the equation.

Finally, we wish to comment on our choice to focus on the constant coefficient case. The most basic reason is that in this situation, monotonicity and duality alone suffice to produce comparison for compatible pairs $(F, \F)$. On Euclidian spaces, where $x$-dependence is added into the pair; that is,
$$
F: \Omega \times \J^2 \to \R \quad \text{and} \quad \F: \Omega \to \{ {\rm subequations \ in} \ \J^2 \}
$$
one also needs at least the continuity of the subequation-valued map $\F(\cdot)$. This, together with monotonicity and duality has been shown to be sufficient for comparison in the pure second order and gradient-free cases in \cite{CP17} and \cite{CP21}. See also the recent paper of Brustad \cite{Br20}. Moreover, constant coefficient subequations on Euclidian space generate a rich and interesting class of subequations on manifolds $X$, as developped in \cite{HL11}. These subequations on $X$ are those which are {\em locally jet-equivalent to a constant coefficient subequation}.
Any riemannian G-subequation on a riemmanian manifold $X$ with topological $G$-structure is such a subequation. For simple examples, let $\mathfrak{p}: \Symn \to \R $ be a continuous function which is invariant under the action of O$(n)$ (such as the determinant or the trace).
Applying $\mathfrak{p}$ to the riemannian hessian gives an operator (real Monge-Amp\`ere or Laplace-Beltrami) on $X$, which has the jet-equivalence property above. 
These notions are discussed in the introduction of \cite{HL11} (see pages 398-402) along with much, much more.

\section{Constant Coefficient Constraint Sets and their Subharmonics}\label{sec:constraints}

In this section, we will discuss nonlinear potential theory. Two definitions are of fundamental importance; that of {\em subequation constraint sets} and their {\em subharmonics} (see Definitions \ref{defn:constraint} and \ref{defn:FSH}). In all that follows, $X$ will denote an open subset of $\R^n$, $\Symn$ the space of symmetric $n \times n$ matrices with real entries (with its partial ordering given by the associated quadratic forms) and
\begin{equation}\label{2jets}
\J^2 := \R \times \R^n \times \Symn
\end{equation}
will denote the space of {\em 2-jets} with coordinates $J = (r,p,A)$. The spaces of upper, lower semi-continuous functions on $X$ taking values in $[-\infty, +\infty), (-\infty, +\infty]$ will be denoted by $\USC(X), \LSC(X)$ respectively.

\begin{defn}\label{defn:constraint}
	A {\em subequation (constraint set)} is a non empty proper subset \footnote{Somewhat surprisingly, the non-empty and proper subset hypothesis is rarely needed in the proofs; for example, one always has (trivially) comparison if $\F = \emptyset$ or $\F = \J^2$.} $\F \subset \J^2$ which satisfies the {\em Positivity Condition}
	$$ \mbox{(P) \ \ \ \ $(r,p,A) \in \F$\ \ and $P\geq 0 \ \ \Rightarrow\ \ (r,p,A+P)\in \F$,}$$
	the {\em Negativity Condition}
	$$ \mbox{(N) \ \ \ \ $(r,p,A) \in \F$ \ \ and $s \leq 0 \ \ \Rightarrow\ \ (r + s,p,P)\in \F$}$$
	and the {\em Topological Condition}
	$$ \mbox{(T) \ \ \ \ $\F = \overline{\Int \, \F}$,}$$
	which implies that $\F$ is closed and has non empty interior.
\end{defn}

Denoting by $\cP := \{ P \in \Symn: P \geq 0\}$ and $\cN := \{ s \in \R: s \leq 0\}$, the monotonicity conditions are 
$$ \mbox{(P) \ \ \ \ $\F+ \left( \{0\} \!\times\!  \{0\} \!\times\!  \cP \right) \ \ss \ \F$ \quad and \quad (N) \ \ \ \ $\F+ \left( \cN \!\times\!  \{0\} \!\times\!  \{0\} \right)  \ \ss \ \F$.}$$
Also, it is useful to note that for closed convex sets $\F$
\begin{equation}\label{convex_T}
\mbox{the topological conditon (T) holds  \ \ $\Longleftrightarrow \ \ \Int \, \F \neq \emptyset$.} 
\end{equation}
Hence a closed convex set $\F \subset \J^2$ is a subequation if and only if $\F$ satisfies (P) and (N) and has non-empty interior.

In addition to the monotonicity properties (P) and (N), a third monotonicity condition plays an important role in comparison and is introduced here. It depends on the choice of a suitable cone $\cD \subseteq \R^n$.

	\begin{defn}\label{defn:property_D} A closed convex cone $\cD \subseteq \R^n$ (possibly all of $\R^n$) with vertex at the origin which satisfies the topological condition (T) (equivalently $\Int \, \cD \neq \emptyset$) will be called a {\em directional cone}. The {\em Directionality Condition} on a subequation $\F \subset \J$ is
	$$ \mbox{(D) \ \ \ \ $(r,p,A) \in \F$\ \ and $q \in \cD \ \ \Rightarrow\ \ (r,p + q,A)\in \F$,}$$
	or equivalently 
	$$
	\mbox{(D) \ \ \ \ $\F+ \left( \{0\} \!\times\!  \cD \!\times\!  \{0\} \right) \ \ss \ \F$ .}$$
	\end{defn}

\begin{rem}\label{rem:subequation_terminology} In previous works, the sets $\mathcal{F}$ described in Definition \ref{defn:constraint}  have been called both a {\em Dirichlet set} and a {\em subequation}. Here we will often use the shortened form {\em subequation} in place of the longer phrase {\em subequation constraint set} introduced in Definition \ref{defn:constraint}.
	\end{rem}

A function $u \in C^2(X)$ {\em satisfies the subequation constraint $\F$} if
\begin{equation}\label{C2FSH}
J^2_x u := (u(x), Du(x), D^2u(x)) \in \F \ \ \ \text{for each} \ x \in X,
\end{equation}
and will be called {\em $\F$-subharmonic on $X$}. If
\begin{equation}\label{C2SFSH}
J^2_x u \in \Int \, \F \ \ \text{for each} \ x \in X,
\end{equation} 
we will say that $u$ is {\em strictly $\F$-subharmonic on $X$}. 

For upper semi-continuous functions $u \in \USC(X)$, one can define the differential inclusion \eqref{C2FSH} in the {\em viscosity sense}. 

\begin{defn}\label{defn:FSH} Given a function $u\in\USC(X)$:
	\begin{itemize}
		\item[(a)] a function $\varphi$ which is $C^2$ near $x_0 \in X$ is said to be a {\em ($C^2$ upper) test function for $u$ at $x_0$} if
		\begin{equation}\label{TF}
		u-\varphi \leq \ 0  \ \ \text{near}\ x_0 \ \ \ \text{and} \ \ \  
		u-\varphi = 0  \ \ \text{at} \ x_0;  
		\end{equation}	
		\item[(b)] the function $u$ is said to be {\em $\F$-subharmonic at $x_0$} if $J^2_{x_0} \varphi \in \F$ for all upper $C^2$ test functions $\varphi$ for $u$ at $x_0$; 
		\item[(c)] the function $u$ is said to be {\em $\F$-subharmonic on $X$} if $u$ is $\F$-subharmonic at each $x_0 \in X$. 
	\end{itemize} 

	The space of all $\F$-subharmonic functions on $X$ will be denoted by $\F(X)$.
\end{defn}

A pair or remarks concerning the $\F$-subharmonicity are in order.

\begin{rem}\label{rem:FSH_equiv}
	There are several equivalent ways of defining $u$ to be $\F$-subharmonic at $x_0 \in X$, which can all be formulated as
	\begin{equation} \label{FSH_equiv}
	J = (r,p,A) \in \F \ \ \text{for all ``upper test jets'' for $u$ at $x_0$}.
	\end{equation}
	To complete the definition \eqref{FSH_equiv}, there are several natural choices for defining the concept ``$J = (r,p,A)$ is an upper test jet for $u$ at $x_0$'' which all yield the same notion of $u$ being $\F$-subharmonic at $x_0$. This is, of course, well known to specialists. For the convenience of the reader, we present four equivalent reformulations in Lemma \ref{lem:A}. The statement includes nomenclature for each formulation that we believe to be useful. (Also, somewhat surprisingly, all of the equivalences of Lemma \ref{lem:A} are valid for any closed set $\F$ without mention of properties (P) or (N)). Being $\F$-subharmonic at $x_0$ is the differential inclusion 
	\begin{equation}\label{FSH-DI}
	\mbox{$J_{\ast}(x_0,u) \subset \F$ for one of the four choices of upper test jets in Lemma \ref{lem:A}.}
	\end{equation}
	It  is important to note that the four sets of upper test jets for $u$ at $x_0$ are nested in the sense that (see \eqref{A2}):
	\begin{equation}\label{UTJ_chain}
	J_1(x_0, u) \subset J_2(x_0, u) \subset J_3(x_0, u) \subset J_4(x_0, u).
	\end{equation}
	Using the smallest set $J_1(x_0,u)$ of {\em strict quadratic test jets} gives the best choice for showing that $u$ is subharmonic by a contradiction argument. This is done in the Bad Test Jet Lemma \ref{lem:nonFSH}, which concludes that there is a bad test jet $J \in J_1(x_0,u)$. The set $J_3(x_0,u)$ of {\em $C^2$-test jets} is the set of upper test functions used in our Definition \ref{defn:FSH} (b) above. This definition is local, but can be reformulated as a form of the comparison principle which always holds globally.
	See Lemma \ref{lem:DCP} on {\em definitional comparison}. Finally, the largest set $J_4(x_0,u)$ of {\em little-o quadratic test jets} yields one of the standard definitions in terms of the second order superdifferential since  $J_4(x_0, u) = J_{x_0}^{2,+} u$. 
\end{rem}

\begin{rem}\label{rem:FSH_silly}
	The pointwise notion of Definition \ref{defn:FSH} (b) that $u$ is $\F$-subharmonic at $x_0$ is not without its ``pathology''. For example, the function $u(x) = |x|$ is $\F$-subharmonic at $x_0 = 0$ for every subequation $\F$ since there are no upper test functions $\varphi$ for $u$ at $x_0 = 0$. However, the concept of $u$ being $\F$-subharmonic on an open set $X$ mitigates this pathology in the following way. For a fixed $u \in \USC(X)$, consider the set of {\em upper contact points}
	\begin{equation}\label{UCP}
		\UCP(X):= \{ x_0 \in X: \ J_{\ast}(x_0,u) \neq \emptyset \},
	\end{equation}
	where $J_{\ast}(x_0,u)$ is any one of set of test jets in  \eqref{UTJ_chain}, as defined by (J1)-(J4) in Lemma \ref{lem:A}. One can show that $u \equiv - \infty$ on the open set $X \setminus \overline{\UCP(X)}$,
	by using the proof of Lemma 6.1 of \cite{HL16b}. In particular, if $u(x_0) > -\infty$ ($u$ is finite at $x_0$), then $x_0$ is the limit of a sequence $\{x_k\}_{k \in \N} \subset \UCP(X)$ \footnote{This fact strengthens Proposition 2.5(2) of \cite{Ko04} where only $\USC$ functions are considered which are everywhere finite.} 
	\end{rem}

The natural notion of $u$ being {\em $\F$-harmonic on $X$} will be recorded in Definition \ref{defn:FH}, in terms of {\em Dirichlet duality} (along with the notion of $u$ being {\em $\F$-superharmonic on $X$}). 

A pair of elementary examples of $\F$-subharmonic functions are worth mentioning to help fix the idea. 

\begin{exes}[Convex functions and classical Laplacian subharmonics]\label{exes:SHs} For the {\em convexity subequation}  $\F = \R \times \R^n \times \cP$, one has (see Proposition 4.5 of \cite{HL09}):
	\begin{equation}\label{Convex_char}
	\mbox{$u \in \F(X) \ \Leftrightarrow \ u$ is convex  or $u \equiv - \infty$ on connected components of $X$.}
	\end{equation}

	For $\F = \R \times \R^n \times \F_{\Delta}$ with $\F_{\Delta}:= \{ A \in \Symn: {\rm{tr}}(A) \geq 0\}$, one can show that for $u \in \USC(X)$ 
	\begin{equation}\label{Subharmonic_char}
	\mbox{$\displaystyle{u \in \F(X) \ \Leftrightarrow \ u(x_0) \leq \frac{1}{|B_r(x_0)|} \int_{B_r(x_0)} u(x) \, dx}$ for each $B_r(x_0) \subset \subset X$,}
	\end{equation}
which is the classical definition of a subharmonic function. 

	In both examples, there is a canonical \footnote{See Proposition 6.11 of \cite{HL18b} for the definition of the {\em canonical operator}, its construction and many additional examples. See also subsection \ref{subsection:canonical}  below.} choice (but not the only choice) of a differential operator $F(u, Du, D^2u)$ with 
	$$
	\F = \{ (r,p,A) \in \J^2: \ F(r,p,A) \geq 0\}
	$$
	and hence $\F(X)$ is the space of viscosity subsolutions of  $F(u, Du, D^2u) = 0$. The first is the {\em minimum eigenvalue operator}
	$$
		F(r,p,A) = \lambda_{\rm min}(A)
	$$
	and the second is the {\em Laplacian}
	 $F(r,p,A) = {\rm tr}(A)$.
\end{exes}

As noted in Remark \ref{rem:FSH_equiv}, if one takes the contrapositive of Definition \ref{defn:FSH}(b) using the {\em strict quadratic test jet} formulation (J1) of Lemma \ref{lem:A},  one immediately obtains the following very useful tool for establishing $\F$-subharmonicity by providing the existence of a ``bad test jet'' at a point where $\F$-subharmonicity fails.

\begin{lem}[The Bad Test Jet Lemma]\label{lem:nonFSH} 
Given $u \in \USC(X)$ and $\F \neq \emptyset$, if $u$ is not $\F$-subharmonic at $x_0$ then there exist $\veps >0$ and $J=(r,p,A) \notin \F$ such that the quadratic function $Q_J$ with 2-jet $J$ at $x_0$ is a test function for $u$ at $x_0$ in the following  $\veps$-strict sense:
	\begin{equation}\label{nonFSH}
\mbox{$u(x)-Q_J(x) \leq \ -\veps|x-x_0|^2$ \ \ near $x_0$, with equality at $x_0$.} 
	\end{equation}
\end{lem}

The converse of Lemma \ref{lem:nonFSH} is one of the many equivalent definitions of $\F$-subharmonicity (using the strict quadratic test jets $J_1(x_0, u)$). 

\begin{rem}[Coherence and the Positivity Condition]\label{rem:coherence}
	A fundamental observation concerning the Definition \ref{defn:FSH}(b) of  $u \in \USC(X)$ being $\F$-subarmonic at $x_0$ with $\F \neq \emptyset$ is the following {\em coherence property}: if $u$ is twice differentiable at $x_0$ then 
	\begin{equation}\label{coherence}
	\mbox{$u$ is $\F$-subharmonic at $x_0 \ \ \Longleftrightarrow \ \ J_{x_0}^2 u \in \F$.}
	\end{equation}
\end{rem}
The forward half of the equivalence $(\Rightarrow)$ is an immediate consequence of the little-o formulation \eqref{A3} of Lemma \ref{lem:A}, where the second order Taylor expansion of $u$ in $x_0$ with Peano remainder is an upper test function for $u$ at $x_0$. The reverse half of the equivalence $(\Leftarrow)$ is the first instance where the positivity condition (P) for $\F$ is required. Indeed, assume $J^2_{x_0}u \in \F$ and that $\varphi$ is an upper test function satisfying \eqref{TF}. By elementary calculus, one has $\varphi(x_0) = u(x_0), D \varphi(x_0) = D u(x_0)$ and $D^2 u(x_0) - D^2 \varphi(x_0) = -P$ with $P \geq 0$ and hence
$$
J^2_{x_0} \varphi = J^2_{x_0} u + P \ \ \text{with} \ \ P \geq 0,
$$
which yields $J^2_{x_0} \varphi \in \F$ if and only if property (P) holds for $\F$.

Another result which makes use of property (P) is the following useful lemma, which has been given in a more general form on manifolds in Theorem 4.1 of \cite{HL16b}. For the convenience of the reader we will give a sketch of the proof. We recall that a function $u$ is {\em locally quasi-convex (locally semi-convex) on an open set $X \subset \R^n$} if for each $x \in X$ and for some $\lambda > 0$ the function $u(\cdot) + \frac{\lambda}{2} | \cdot|$ is convex on a neighborhood of $x$.

\begin{lem}[The Almost Everywhere Theorem]\label{lem:AET}
	Suppose that $\F \subset \J^2$ is a subequation constraint set and that $u$ is a locally quasi-convex function on the open set $X \subset \R^n$. Then $u$ is twice differentiable almost everywhere on $X$ by Alexandroff's theorem and 
\begin{equation}\label{AET}
	\mbox{ $J^2_xu \in \F$ \ for almost every $x \in X \ \ \Rightarrow \ \ u$ \ is $\F$-subharmonic on $X$.}
	\end{equation}
	\end{lem}

\begin{proof} The proof is by contradiction. If $u$ fails to be $\F$-subharmonic at $x_0 \in X$, then the Bad Test Jet Lemmma \ref{lem:nonFSH} yields a strict upper contact jet $J = (r,p,A) \notin \F$ for which \eqref{nonFSH} holds in a neighborhood of $x_0$. From this strict upper contact jet, the Jensen-Slodkowski Lemma (see Theorem 3.6 of \cite{HL16b}) applied to the subset $E \subset X$ of full measure on which $u$ is twice differentiable at $x \in E$ and $J^2_x u \in \F$ yields a sequence $\{x_k\}_{k \in \N} \subset E$ such that for $k \to +\infty$ one has
$$
x_k \to x_0 \ \ \text{and} \ \ J_{x_k}^2u \to (r,p,A - P) \ \text{with} \ P \geq 0.
$$
Since $\F$ is closed, one has $(r,p,A - P) \in \F$ and then, by property (P), $J = (r,p,A) = (r,p,A - P + P) \in \F$, which contradicts the choice of $J$.

\end{proof}

\begin{rem}\label{rem:AET} Of course, by Alexandroff's theorem and the coherence property of Remark \ref{rem:coherence}, the implication \eqref{AET} is an equivalence.

\end{rem}

We conclude this section with a few additional remarks on the conditions (P), (N) and (T) of Definition \ref{defn:constraint} and condition (D) of Definition \ref{defn:property_D}. Condition (P) corresponds to {\em degenerate ellipticity} of the differential inclusion \eqref{FSH-DI} for each $x_0 \in X$ and, as shown above, proves to be crucial in showing one half of the coherence property \eqref{coherence}. Property (N), when combined with (P), corresponds to {\em properness} of \eqref{FSH-DI} and is well known to play a role in the validity of maximum and comparison principles on arbitrary bounded domains. For example, the linear equation $u^{\prime \prime} + u = 0$ in one variable will satisfy the maximum principle only on intervals of length less than $\pi$. Property (D) plays an essential role in parabolic equations. Condition (N) yields the following elementary but useful fact.

\begin{rem}[Translations and the Negativity Condition]\label{rem:N-translates} For any $\F$ satisfying (N), one has that $u - m$ is $\F$-subharmonic on $X$ for each $m \in (0, +\infty)$ if $u \in \USC(X)$ is $\F$-subharmonic on $X$. Indeed, for each $x_0 \in \Omega$ and each $\varphi$ which is $C^2$ near $x_0$, one has
	\begin{equation*}\label{TTF}
	\mbox{$\varphi$ is a test function for $u$ at $x_0 \ \ \Leftrightarrow \ \ \varphi - m$  is a test function for $u - m$ at $x_0$.}
	\end{equation*}
\end{rem}

Conditions (T) and (P) imply the following fact.

\begin{rem}[Local existence of smooth subharmonics]\label{rem:LBS}
	For any $\F$ satisfying (T), about each point $x_0 \in \R^n$ one can construct strictly $\F$-subharmonic functions which are smooth (and hence bounded) near $x_0$. Indeed, pick any 2-jet $J_0:= (r_0, p_0, A_0) \in \Int \, \F$, which is nonempty by condition (T). The quadratic polynomial 
	$$
	Q(x):= r_0 + \langle p_0, x - x_0 \rangle + \frac{1}{2} \langle A_0(x - x_0), x - x_0 \rangle,
	$$ 
	which has 2-jet $J^2_{x_0} Q = J_0$ at $x_0$, has 2-jet
	$$
	J^2_x Q = J_0 + (Q(x) - r_0, A_0(x-x_0), 0) \in \Int \, \F ,
	$$
	for each $x$ sufficiently near to $x_0$. If, in addition, $\F$ satisfies property (P) then the coherence property of Remark \ref{rem:coherence} implies that $Q$ is $\F$-subharmonic.
\end{rem}

\begin{rem}[The subequations determined by $\cP, \cN$ and $\cD$]\label{rem:building_blocks} Each of our three monotonicity conditions (P), (N) and (D) on a subequation $\F$ were phrased as $\F + \cM \subset \F$ using the {\em monotonicity sets} $\cM$ defined by
	$$
	 \{0\} \!\times\!  \{0\} \!\times\!  \cP, \ \ \   \cN \!\times \! \{0\} \!\times\!  \{0\} \ \ \text{and} \ \  \{0\} \!\times\!  \cD \!\times \! \{0\} 
	$$
respectively. None of these monotonicity sets are subequations as they violate at least two of the conditions (P), (N) and (T). However, there are three naturally associated subequations for which the sets $\cP, \cN$ and $\cD$ give the {\em active constraint} while the constraint coming from the other factors is silent. Namely,
\begin{equation}\label{PND}
	\cM(\cP):= \R \times \R^n \times \cP, \ \ \cM(\cN):= \cN \times \R^n \times \Symn \ \  \text{and} \ \ \cM(\cD):= \R \times \cD \times \Symn.
\end{equation}
The second subequation, in isolation, is not very interesting since no derivatives are constrained. However, taken all together these three subequations will be used as the basic building blocks in Section \ref{sec:cones}.  It is convenient to refer to $\cP$ as the {\em convexity subequation}, while actually meaning $\cM(\cP) = \R \times \R^n \times \cP$. Similarly, it will be convenient to use $\cP$ in place of the monotonicity set $ \{0\} \!\times\!  \{0\} \!\times\!  \cP$ used in the definition of property (P). Similarly, we will refer to $\cN$ as the {\em negativity subequation} and $\cD$ as the {\em $\cD$-directional subequation} while actually meaning $\cM(\cN)$ and $\cM(\cD)$ as in \eqref{PND}. This convention that involves reduced constraints will be expanded on and formalized in Convention \ref{conv:reduction} and plays a simplifying role in Section \ref{sec:reductions}.
	
	\end{rem}

The importance of properties (N) and (T) become even more transparent in conjunction with Dirichlet duality for subsets $\F \subset \J^2$, which will be recalled in the next section.

\section{Dirichlet Duality and $\F$-subharmonic Functions}\label{subsec:duality}

We begin by recalling the notion of duality introduced in \cite{HL09} and \cite{HL11}.

\begin{defn}\label{defn:duality}
Suppose $\F \subset \J^2 = \R \times \R^n \times\Symn$ is an arbitrary subset.
The {\em Dirichlet dual of $\F$} is defined to be
\begin{equation}\label{Fdual}
\FD \ =\ \sim(-\Int \, \F)\ =\ -(\sim \Int \, \F),
\end{equation}
where $\sim$ is the complement relative to $\J^2$.
\end{defn}

Several elementary properties will be of constant use, and some illustrate the importance of the topological property (T). 

\begin{prop}[Elementary Properties of the Dirichlet Dual]\label{prop:duality} For $\F, \F_1$ and $\F_2$ arbitrary subsets of $\J^2$, one has:
\begin{itemize}
\item[(1)] $\F_1\ \ss\ \F_2 \ \Rightarrow\ \wt{\F_2}\ \ss\ \wt{\F_1}$;
\item[(2)] $\F + J \subset \F \ \Rightarrow \  \wt{\F} + J \subset \wt{\F}$ for each $J = \jet \in \J^2$;
\item[(3)] $\wt{\F + J} =  \wt{\F} - J$ for each  $J = \jet \in \J^2$.
\end{itemize}
By property (2), one has:
\begin{itemize}
	\item[(4)] $\F$ satisfies property (P) $ \ \Rightarrow \ \wt{\F}$ satisfies property (P);
\item[(5)] $\F$ satisfies property (N) $ \ \Rightarrow \ \wt{\F}$ satisfies property (N);
\end{itemize}
The topological property (T) is equivalent to reflexivity; that is, 
\begin{itemize}
	\item[(6)] $\F = \overline{ \Int \F}  \ \iff \  \wt{ \wt{\F}} = \F$.
\end{itemize}
Moreover
\begin{itemize}
	\item[(7)] $\F$ satisfies property (T) $ \ \Rightarrow \ \wt{\F}$ satisfies property (T)
\end{itemize}
and hence for $\F \subset \J^2$ a proper subset one has
\begin{itemize}
	\item[(8)] $\F$ is a subequation constraint set $ \ \Rightarrow \ \wt{\F}$ is a subequation constraint set.
\end{itemize}
\end{prop}

\begin{proof}
Property (1) follows from the definition \eqref{Fdual} of the dual since 
$\F_1 \subset \F_2 \ \Leftrightarrow \ \Int \, \F_1 \subset \Int \, \F_2$. 

For property (2), note that $\F + J \subset \F$ implies that $\Int \, \F + J$ is an open subset of $\F$ and hence $\Int \, \F + J \subset \Int \, \F$ which yields
$$
	\sim \left( \Int \, \F \right) \subset \,  \sim \left( \Int \, \F + J \right) = \left( \sim  \Int \F \right) + J \ \text{and hence} \ \wt{\F} \subset \wt{\F} - J.
$$
Thus $\wt{\F} \supset \wt{\F} + J$, as desired. The proof of property (3) is similar, using the definitions of $\wt{ \F + J}$ and $\wt{\F}$. Properties (4) and (5) are immediate consequences of (2).

To prove properties (6) and (7), we will use the fact that
\begin{equation}\label{comp_closure}
	\Int \, \cS = \, \sim \overline{\left( \sim \cS \right)} \ \ \text{for any subset} \ \cS \subset \J^2.
\end{equation}
For the equivalence (6), using the definition of the dual twice and canceling the minus signs, one has
\begin{equation}\label{reflexivity1}
	\wt{\wt{\F}} = \, \sim \Int \, \left( \sim \Int \, \F \right).
\end{equation}
Taking $\cS := \, \sim \Int \, \F$ in \eqref{comp_closure} transforms \eqref{reflexivity1} into
$$
	\wt{\wt{\F}}  = \overline{\Int \, \F},
$$
so that $\wt{\wt{\F}} = \F$ if and only if property (T) holds for $\F$.

For property (7), first take $\cS := \wt{\F}$ in \eqref{comp_closure} to obtain
$$
	\Int \, \wt{\F} = \, \sim \overline{\left( \sim \wt{\F} \right)} = \, \sim \overline{\left( - \Int \, \F \right)},
$$
which, by property (T) for $\F$, proves that $\Int \, \wt{\F} = \, \sim \left( - \F \right)$. Therefore 
\begin{equation}\label{PropT1}
\overline{\Int \, \wt{\F}} = \overline{ - \left( \sim \F \right)} = - \overline{ \left( \sim \F \right)}. 
\end{equation}
Next, take $\cS := \F$ in \eqref{comp_closure} to obtain $\overline{\sim \F} = \, \sim \left( \Int \F \right)$; that is,
\begin{equation}\label{PropT2}	
	\overline{- \left( \sim \F \right)} = - \left( \sim \Int \, \F \right) = \wt{\F},
	\end{equation}
Combining \eqref{PropT1} and \eqref{PropT2} yields $\overline{ \Int \, \wt{\F}} = \wt{\F}$. 

Finally, property (8) is an immediate consequence of properties (4), (5) and (7) and  Definition \ref{defn:constraint}. 
\end{proof}

It is worth noting that the reverse implications in (4) and (7) may be false. The following is a simple one dimensional pure second order example. 

\begin{exe}\label{exe:no_inverse} If $\F = \{-1\} \cup [0, + \infty) \subset \Sym1$ then $\wt{\F} = [0, +\infty)$ satisfies properties (P) and (T), but $\F$ satisfies neither.
	\end{exe} 

Similarly, the reverse implication in (5) may be false. Next, we turn to the behavior of duality under unions and intersections. 

\begin{prop}[Duality, unions and intersections]\label{prop:algebra_duals} For $\F_1$ and $\F_2$ arbitrary subsets of $\J^2$ one has
	\begin{equation}\label{dual_intersection}
	\wt{\F_1 \cap \F_2} = \wt{\F}_1 \cup \wt{\F}_2
	\end{equation}
and if, in addition $\Int \, (\F_1 \cup \F_2) = \Int \, F_1 \cup \Int \, \F_2$, then	
	\begin{equation}\label{dual_union}
\wt{\F_1 \cup \F_2} = \wt{\F}_1 \cap \wt{\F}_2
\end{equation}
	\end{prop}

\begin{proof}
	For the dual of an intersection, note that $\Int \, (\F_1 \cap \F_2) = (\Int \, \F_1) \cap (\Int \, \F_2)$ and use the definition of duality to arrive at \eqref{dual_intersection}. For the dual of a union, by the definition of duality and the hypothesis on $\Int \left( \F_1 \cup \F_2 \right)$ one has
	$$
	J \in \wt{ \F_1 \cup \F_2} \ \iff \ -J \not\in \Int \, \left( \F_1 \cup \F_2 \right) = \left( \Int \F_1 \right) \cup \left( \Int \, \F_2 \right),
	$$
	which yields $J \in \wt{F}_1 \cap \wt{F}_2$ as desired.
\end{proof}

It is worth noting that property \eqref{dual_union} can fail if one does not have the hypothesis  $\Int \, (\F_1 \cup \F_2) = \Int \, F_1 \cup \Int \, \F_2$. The following is a simple pure first order one dimensional counterexample.

\begin{exe}\label{exe:prop11} In dimension $n = 1$, consider
	$$
	\F_1 := [-1,0] \quad \text{and} \quad \F_2 := [0,1].
	$$
	One easily verifies that 
	$$ \wt{\F_1 \cup \F_2} = (-\infty, -1] \cup [1, + \infty) \quad \text{while} \quad \wt{F}_1 \cap \wt{F}_2 = \wt{\F_1 \cup \F_2} \cup \{0\}.
	$$
\end{exe}

Additional properties of the dual, the validity of property (T) and algebra of subequations will be addressed in the next section on monotonicity (see Propositions \ref{prop:MMC2}, \ref{prop:subequation_cones}  and \ref{prop:algebra}). 

Notice that property (T) ensures that there is a true duality in the form of the reflexivity property given in (7), and since $\F$ is closed, one has also
\begin{equation}\label{Fclosed}
\partial \F = \F \cap (\sim \Int \, \F) = \F \cap (-\wt{\F}),
\end{equation}
which leads to the following definition.

\begin{defn}\label{defn:FH} Let $\F$ be a subequation constraint set. A function $u$ is said to be {\em $\F$-harmonic in $X$} if 
	\begin{equation}\label{FH}
	u \in \F(X) \quad \text{and} \quad -u \in \wt{\F}(X).
	\end{equation}
	A function $u$ is said to be {\em $\F$-superharmonic in $X$} if  $-u \in \wt{\F}(X)$.
\end{defn}

\begin{rem}\label{rem:Fsuper}
	An equivalent formulation of $u \in \LSC(\Omega)$ being $\F$-superharmonic in $X$ is that for each $x_0 \in \Omega$, one has
	\begin{equation}\label{Fsuper}
		J^{2,-}_{x_0}u \subset \,  \sim \Int \, \F,
	\end{equation}
	where $J^{2,-}_{x_0}u$ is the set of {\em (lower) test jets for $u$ at $x_0$}; that is, the set of 2-jets $J^2_{x_0} \varphi$ with $\varphi$ a {\em $C^2$ (lower) test function for $u$ at $x_0$} ($\varphi$ is $C^2$ near $x_0$ and $u - \varphi$ has a local minimum value of zero in $x_0$). 
	\end{rem}

Notice that $\F$-harmonic functions are automatically continuous and that $\F$-superharmonic functions are automatically lower semi-continuous. Moreover, $u \in C^2(X)$ is $\F$-harmonic on $X$ if and only if
\begin{equation}\label{FHC2}
    J_xu \in \partial \F, \ \ \text{for each} \ x \in X,
\end{equation}
which follows from the coherence property for $\F, \wt{\F}$ of Remark \ref{rem:coherence} coupled with \eqref{FH} and \eqref{Fclosed}.
\begin{exe}[Classical subharmonic and harmonic functions]\label{exe:Harmonic} If $\F:= \R \times \R^n \times \F_{\Delta}$ with $\F_{\Delta} = \{ A \in \Symn: {\rm tr} A \geq 0 \}$, then $\wt{\F} = \F$ is self dual and $\F$-harmonics are characterized by satisfying the mean value property; that is, by having equality in \eqref{Subharmonic_char}.
\end{exe}

Pairs of $\F$ and $\wt{\F}$-subharmonics are the key players in the comparison principle when making use of Dirichlet duality and hence a few additional pairs are worth mentioning now. 

\begin{exe}[Convex and subaffine functions]\label{exe:SA} If $\F:= \R \times \R^n \times \cP$ is the {\em convexity subequation} , then $\wt{\F} = \R \times \R^n \times \cPt$
where $\cPt = \, \sim (-\Int \, \cP))$ is the compliment of the negative definite quadratic forms. Thus $A \in \cPt \iff \lambda_{\rm max}(A) \geq 0$, where $\lambda_{\rm max}$ is the maximal eigenvalue. The class of $\wt{\F}$-subharmonic functions are characterized by: $u \in \USC(X)$ belongs to $\wt{\F}(X)$ if and only if for every open subset $\Omega$ with $\Omega \subset \subset X$ and each affine function $a$, one  has 
\begin{equation}\label{SAC}
  \mbox{$u \leq a$ on $\partial \Omega \Rightarrow u \leq a$ on $\Omega$.}
\end{equation}
Since they are equal (see Proposition 4.5 of \cite{HL09}),
we will denote this space of {\em subaffine functions on $X$} by $\SA(X)$ as well as $\wt{\F}(X)$.
\end{exe}

\begin{exe}[Negative functions]\label{exe:NN} If $\F:= \cN \times \R^n \times \Symn$, then $\wt{\F} = \F$ is self-dual and  zeroth order.
\end{exe}

\begin{conv}\label{conv:reduction} If the constraint $\F$ places restrictions only on the matrix variable; that is, if $\F := \R \times \R^n \times \F_2$ with $\F_2$ satisfying (P): $\F_2 + \cP \subset \F_2$, then ``$u$ is $\F$-subharmonic on $X$'' and ``$u$ is $\F_2$-subharmonic on $X$'' will mean the same thing. For example, we can denote by $\cP(X)$ the convex functions of \eqref{Convex_char} and  $\cPt(X)$ the subaffine functions of Example \ref{exe:SA}. Similarly, the negative functions of Example \ref{exe:NN} can be denoted by $\cN(X)$. In addition, when it is clear from the context, we will refer to the {\em subequation} $\F_2$ meaning $\F$.
\end{conv}

Additional justification for this convention is provided in the discussion before Remark \ref{rem:PSO}. 

\begin{exe}[Negative convex and subaffine-plus functions]\label{exe:SAP} If $\F:= \cN \times \R^n \times \cP$ is the {\em negativity and convexity subequation}, then 
\begin{equation}
\wt{\F} = (\cN \times \R^n \times \Symn) \cup (\R \times \R^n \times \cPt)
\end{equation}
and the dual subharmonics are characterized by 
\begin{equation}\label{SAPC}
  w \in \wt{\F}(X) \iff w^+ := {\rm max}\{w,0\} \in \cPt(X) = \SA(X)
\end{equation}
and referred to as the {\em subaffine-plus functions on $X$.}
This characterization is discussed in Theorem \ref{thm:SAPChar}. Denoting by $\cQ = \cN \times \cP$ and $\cQt:= \{(r,A) \in \R \times \Symn : r \in \cN \ \text{or} \ A \in \cPt \}$, we will also say that the subaffine-plus functions are $\cQt$-subharmonic on $X$ in the spirit of Convention \ref{conv:reduction}.
\end{exe}

Additional examples include $\F = \cM$ where $\cM$ is a {\em monotonicity cone}, as will be discussed in the next section. For future reference we record the following example of duals to the elementary monotonicity cones introduced in \eqref{PND}.

\begin{exe}\label{exe:PND_duals} The duals of
	$$
	\cM(\cP):= \R \times \R^n \times \cP, \ \ \cM(\cN):= \cN \times \R^n \times \Symn \ \  \text{and} \ \ \cM(\cD):= \R \times \cD \times \Symn.
	$$
	are
	\begin{equation}\label{PND_duals}
	\wt{\cM}(\cP) = \R \times \R^n \times \wt{\cP}, \ \ \wt{\cM}(\cN)= \cN \times \R^n \times \Symn \ \  \text{and} \ \ \wt{\cM}(\cD)= \R \times \wt{\cD} \times \Symn
	\end{equation}
	as follows easily from applying the definition \eqref{Fdual} of duality.
	\end{exe}

We conclude this section with a very useful result which illustrates the importance of the negativity condition (N), saying that the comparison principle \eqref{CP} always holds if the function $v$ is $C^2$ \underline{smooth} and \underline{strictly} $\wt{\F}$-subharmonic.

\begin{lem}[Definitional Comparison]\label{lem:DCP} Suppose that $\F$ is a subequation constraint set and that $u \in \USC(X)$.
	\begin{itemize}
		\item[(a)] If $u$ is $\F$-subharmonic on $X$, then the following form of the comparison principle holds for each bounded domain $\Omega \subset \subset X$:
		\begin{equation}\label{DCP}
		\left\{ \begin{array}{c}
		\mbox{ $u + v \leq 0$ on $\partial \Omega \Longrightarrow u + v \leq 0$ on $\Omega$} \\
		 \\
		 	\mbox{if $v \in \USC(\overline{\Omega}) \cap C^2(\Omega)$ is strictly $\wt{\F}$-subharmonic on $X$.} \end{array} \right.
		\end{equation}
		With $w:= -v$ one has the equivalent statement
		\begin{equation}\label{DCP2}
			\left\{ \begin{array}{c} \mbox{ $u \leq w$ on $\partial \Omega \Longrightarrow u \leq w$ on $\Omega$} \\
			\\
			\mbox{if $w \in \LSC(\overline{\Omega}) \cap C^2(\Omega)$ with $J_x^2 w \not\in \F$ for each $x \in \Omega$.} \end{array} \right.
			\end{equation}
(That is, for $w$ which are regular and strictly $\F$-superharmonic in $X$.)
		\item[(b)] Conversely, suppose that for each $x \in X$ there is a neighborhood $\Omega \subset \subset X$ where the form of comparison of part (a) holds. Then $u$ is $\F$-subharmonic on $X$.
	\end{itemize} 
	
\end{lem}

\begin{proof} Suppose that \eqref{DCP2} fails for some domain $\Omega \subset \subset X$ and some regular strictly $\F$-superharmonic function $w$. Then $u - w \in \USC(\overline{\Omega})$ will have a positive maximum value $m > 0$ at an interior point $x_0 \in \Omega$ and hence $\varphi := w + m$ is $C^2$ near $x_0$ and satisfies
$$
   \mbox{ $u - \varphi \leq 0$ \ near $x_0$ \ with equality at $x_0$. }
   $$
Since $u$ is $\F$-subharmonic at $x_0$, by Definition \ref{defn:FSH} one has
$$
    J^2_{x_0}(w + m) = J^2_{x_0}w  + (m,0,0) \in \F,
$$
which contradicts property (N) since $m > 0$ and $J^2_{x_0}w  \not\in \F$. This completes the proof of part (a).

For part (b), suppose that $u$ fails to be $\F$-subharmonic at some $x_0 \in X$. By the Bad Test Jet Lemma \ref{lem:nonFSH} there exist $\rho, \veps > 0$ and $J = (u(x_0), p, A) \not\in \F$ such that
\begin{equation}\label{BTJ}
\mbox{$u(x)-Q_J(x) \leq \ -\veps|x-x_0|^2$ \ \ on $\overline{B}_{\rho}(x_0)$, with equality at $x_0$,} 
\end{equation}
where $Q_J$ is the quadratic with $J^2_{x_0} Q_J = J$. The function $-Q_J$ is smooth and strictly $\wt{\F}$-subharmonic in $x_0$ since $J^2_{x_0}(-Q_J) = -J \in - (\sim {\F}) = \Int \, \wt{\F}$.
Since $\Int \, \wt{\F}$ is open, by choosing $\rho > 0$ sufficiently small, the function $v:= -Q_J + \veps \rho^2$ will be strictly $\wt{\F}$-subharmonic in $B_{\rho}(x_0)$. Moreover, reducing $\rho$ preserves the validity of \eqref{BTJ} and hence the comparison \eqref{DCP} fails for $v$ on $\overline{B}_{\rho}(x_0)$ since 
$$
u + v = 0 \ \ \text{on} \ \ \partial B_{\rho}(x_0) \ \ \text{and} \ \ u(x_0) + v(x_0) = \veps \rho^2 > 0.
$$
\end{proof}

\begin{rem}\label{rem:DCP1}
Note that the above proof is very general. The result holds verbatim on any manifold $X$ and any subequation $\F$ on $X$. More precisely, it holds if $\F$ satifies properties (P), (N) and $\F = \overline{\Int \, \F}$. Replacing $\F$ by its dual $\wt{\F}$, since $\wt{\wt{\F}} = \F$ (the proof of reflexivity makes use of the defintion of duality and \eqref{comp_closure}, whihch holds in any topological space), one also has the comparison principle \eqref{DCP} if the function $u$ is $C^2$ \underline{smooth} and \underline{strictly} $\F$-subharmonic and $v \in \USC(X)$ is $\wt{\F}$-subharmonic.
\end{rem}

We note that Lemma \ref{lem:DCP} will be useful in the proof of the {\em Zero Maximum Principle} for $\cMt(X)$-subharmonic functions of section \ref{sec:ZMP} where $\cM$ is a {\em monotonicity cone}, which is the subject of the next section.

We also note that our basic ``tool kit'' of viscosity solution techniques include, amonsgst other things the bad test jet lemma (Lemma \ref{lem:nonFSH}), the almost everywhere theorem (Lemma \ref{lem:AET}) and definitional comparison (Lemma \ref{lem:DCP}).

We conclude this section by adding  a useful classical computational lemma to our tool kit, for the reader's convenience and future reference. We combine the statement and the proofs as a remark. First, we need some notation. For $x \neq 0$ in $\R^n$, we denote orthogonal projection onto the one dimensional linear subspace $[x]$ through $x$ by 
\begin{equation}\label{projection_x}
P_{x} := \frac{1}{|x|^2} x \otimes x
\end{equation}
and similarly by
\begin{equation}\label{projection_xperp}
P_{x^{\perp}} := I - \frac{1}{|x|^2} x \otimes x = I - P_x,
\end{equation}
the projection onto the orthogonal complement $[x]^{\perp}$.

\begin{rem}[Radial calculations and examples]\label{rem:radial_calculus} First note that for $C^2$ functions $v: \Omega \to \R$ and  $\psi: v(\Omega) \subset \R\to \R$, then the chain rule applied to $u = \psi \circ v \in C^2(\Omega)$ gives (for each $x \in \Omega$):
	\begin{equation}\label{chain rule_1}
Du(x) = D_x(\psi(v(x)) = \psi'(v(x)) Dv(x)
\end{equation}
and
\begin{equation}\label{chain rule_2}
D^2u(x) =  D_x^2(\psi(v(x)) = \psi'(v(x)) D^2v(x) + \psi^{\prime \prime}(v(x)) Dv(x) \otimes Dv(x).
\end{equation}	
In particular, with $v(x) = |x|$ and $\psi$ of class $C^2$ on an interval in $(0,+\infty)$, since
	\begin{equation}\label{radial_calc_2}
Dv(x) = D_x(|x|) = \frac{x}{|x|} \quad \text{and} \quad D^2v(x) = D^2_x(|x|) = \frac{1}{|x|} P_{x^{\perp}},
\end{equation}
the radial function defined by $u(x) := \psi(|x|)$ has reduced $2$-jet $(Du(x), D^2u(x))$ given by
	\begin{equation}\label{radial_calc_3}
D_x \psi(|x|) = \psi'(|x|) \frac{x}{|x|} \quad \text{and} \quad 		D^2_x\psi(|x|) = \frac{\psi'(|x|)}{|x|} P_{x^{\perp}} + \psi^{\prime \prime}(|x|) P_x,
\end{equation}
on the corresponding annular domain. In addition to the example $\psi(t):=t$ in \eqref{radial_calc_2}, here are some other useful examples. 

\noindent{\bf Example 1:} For $u(x):= \frac{1}{2} |x|^2$; that is, $\psi(t):= \frac{1}{2} t^2$, one has
	\begin{equation}\label{radial_calc_E1}
(Du(x), D^2u(x)) = \left( x, I \right).
\end{equation}
The these two examples \eqref{radial_calc_2} and \eqref{radial_calc_E1} generalize from $m = 0$ and $m = 1$ to

\noindent{\bf Example 2:} For $u(x):= \frac{1}{m + 1} |x|^{m + 1}$; that is, $\psi(t):= \frac{1}{m+1} t^{m+1}$, with $m \in \R$, $m \neq 1$, one has
\begin{equation}\label{radial_calc_E2}
	(Du(x), D^2u(x)) = |x|^{m-1} \left( x, P_{x^{\perp}} + m P_x \right) =  |x|^{m-1} \left(x, I + (m-1)P_x \right).
\end{equation}

\noindent{\bf Example 3:} For $u(x):= e^{\alpha |x|^2/2}$; that is, $\psi(t):= e^{\alpha t^2/2}$, with $\alpha \in \R$, one has
\begin{equation}\label{radial_calc_E3}
(Du(x), D^2u(x)) = \alpha e^{\alpha \frac{|x|^2}{2}} \left( x, P_{x^{\perp}} + (\alpha + 1) P_x \right) =  \alpha e^{\alpha \frac{|x|^2}{2}} \left(x, I + \alpha P_x \right).
\end{equation}
These two examples $\frac{1}{m + 1} |x|^{m + 1}$ and $e^{\alpha |x|^2/2}$ can usually be used interchangeably, since up to a positve scalar multiple, the reduced $2$-jet is $\left(x, I + (m-1)P_x \right)$ or $\left(x, I + \alpha P_x \right)$.

\noindent{\bf Example 4:} For $u(x):= e^{\alpha |x|}$; that is, $\psi(t):= e^{\alpha t}$ with $\alpha \in \R$, one has
\begin{eqnarray}\label{radial_calc_E4}
(Du(x), D^2u(x)) &=& \frac{\alpha}{|x|} e^{\alpha |x|}  \left(x, P_{x^{\perp}} + \alpha|x| P_x \right)\\
& = &\frac{\alpha}{|x|} e^{\alpha |x|}  \left(x, I + (\alpha|x| - 1) P_x \right).
\end{eqnarray}

	\end{rem}

\changelocaltocdepth{2}

\section{Monotonicity Cones for Constant Coefficient Subequations}\label{sec:cones}

\begin{defn}\label{defn:M_for_F} Given a subset $\F \subset \J^2$, a set $\cM \subset \J^2$ is called a {\em monotonicity set for $\F$} and we say {\em $\F$ is $\cM$-monotone} if
	\begin{equation}\label{M_monotone}
	\F + \cM \subset \F.
	\end{equation}
	Since
	\begin{equation}\label{PN_monotone}
	\mbox{ $\F$ satisfies (P) and (N) $\iff \F$ is $\cN \times \{0\} \times \cP$-monotone, }
	\end{equation}
	the set
	\begin{equation}\label{M0}
	\cM_0 := \cN \times \{0\} \times \cP.
	\end{equation}
	will be called the {\em minimal monotonicity set for all subequation constraint sets $\F$}.
\end{defn}

This set is a closed convex cone which satisfies properties (P) and (N), but it is \underline{not} a subequation since property (T) fails, or equivalently $\Int \, \cM_0 = \emptyset$.

A larger monotonicity set $\cM$ for $\F$ provides more information about $\F$. Note that the sum $\cM_1 + \cM_2$ of two monotonicity sets for $\F$ is again amonotonicity set for $\F$. Hence, if $\F$ is a subequation, we can always add $\cM_0$ to $\cM$, since $\F + \cM_0 \subset \F$. Also, we can replace $\cM$ by its closure (assuming that $\F$ is closed) and the resulting set will still be a monotonicity set for $\F$. Said differently, we need only consider monotonicity sets $\cM$ for subsequations $\F$ with the properties:
\begin{equation}\label{M_requests}
\mbox{ (i) $\cM_0 \subset \cM$ ; \ \ (ii) $\cM + \cM \subset \cM$ ; \ \ and (iii) $\cM$ is closed.}
\end{equation}
For simplicity, in this paper, we restrict attention to sets $\cM$ which are {\em cones}; that is,
\begin{equation}\label{M_cone}
	t \cM \subset \cM \ \ \text{for each} \ t \geq 0.
\end{equation}
All cones are taken to have vertex at the origin, unless stated otherwise. Note that
\begin{equation}\label{cone_sums}
\mbox{ if $\cM$ is a cone then $\cM + \cM$ is a convex cone.}
\end{equation}
As a consequence of \eqref{M_requests} and \eqref{M_cone}, if $\cM$ is a monotonicity cone for a subequation $\F$ then $\overline{{\rm C}(\cM)}$, the closed convex cone hull of $\cM$
is also a monotonicity cone for $\F$. Said differently, given a subsequation $\F$ a monotonicity cone $\cM$ for $\F$ can always be enlarged to one where
\begin{equation}\label{M_requests2}
\mbox{ (i) $\cM_0 \subset \cM$  \quad and \quad (ii)  $\cM$ is closed convex cone.}
\end{equation}

\subsection{The maximal monotonicity cone subequation}\label{subsec:MMC}

Among all sets $\cM$ which are cones and for which a given subequation $\F$ is $\cM$-monotone, there is a unique largest or maximal monotonicity cone for $\F$, defined as follows. 

\begin{defn}\label{defn:MMC} Suppose that $\F \subset \J^2$ is a subequation. Associated to $\F$ is its {\em maximal monotonicity cone} $\cM_{\F}$ defined by
	\begin{equation}\label{MMC}
	\cM_{\F} := \{ J \in \J^2: \ \F + tJ \subset \F \ \ \text{for each} \ t \in [0,1] \}.
	\end{equation}
	\end{defn}
	This invariant satisfies \eqref{M_requests2} and is characterized in the following result.
	
\begin{prop}\label{prop:MMC} Suppose that $\F \subset \J^2$ is a subequation. Its maximal monotonicity cone $\cM_{\F}$ is a closed convex cone containing the minimal monotonicity set $\cM_0 = \cN \times \{0\} \times \cP$ and hence satisfies properties (N) and (P). Therefore
\begin{equation}\label{MF_T} \mbox{ $\cM_{\F}$ is a subequation \ \ $\Leftrightarrow$ \ \   $\cM_{\F}$ satisfies the topological property (T).}
\end{equation}
 Moreover, since $\cM_{\F}$ is a closed convex cone, (T) is satisfied if and only if $\Int \, \cM_{\F} \neq \emptyset$. Finally, $\cM_{\F}$ is maximal in the sense that if $\cM$ is a cone and is a monotonicity set for $\F$, then 
	\begin{equation}\label{MMC_maximality}
		\cM \subset \cM_{\F}.
	\end{equation}
	\end{prop}

\begin{proof}
	It is easy to see that $\F$ is closed implies that $\cM_{\F}$ is closed. By \eqref{PN_monotone} $\F$ satisfies (P) and (N) if and only if $\cN \times \{0\} \times \cP \subset \cM_{\F}$. Obviously, $\cM_{\F} + \cM_{\F} \subset \cM_{\F}$ and hence if $J \in \cM_{\F}$ then each integer multiple $kJ \in \cM_{\F}$. By the definition of $\cM_{\F}$ this implies that $tJ \in \cM_{\F}$ for each $t \in [0, k]$, which proves that $\cM_{\F}$ is a cone. Since $\cM_{\F} + \cM_{\F} \subset \cM_{\F}$, this imples that $\cM_{\F}$ is a convex cone. 
	
	The statements regarding property (T) are straightforward. The final maximality claim for $\cM_{\F}$ is immediate from the definition of $\cM_{\F}$.
\end{proof}

\begin{rem}\label{rem:MMC} In light of Proposition \ref{prop:MMC}, given a subsequation $\F \subset \J^2$, the question of whether or not $\F$ has a monotonicity cone which is a subequation reduces to (is equivalent to) the question:
	\begin{equation}\label{MMC_rem}
\mbox{	Does the maximal monotonicity cone $\cM_{\F}$ for $\F$ have interior?}
	\end{equation}
As we will see in what follows, in order to apply the techniques of this paper to a subequation $\F$, it must have monotonicity $\cM$ which is a subequation.
	
	\end{rem}

Some of the additional properties of the maximal monotonicity cone $\cM_{\F}$ of a subequation $\F \subset \J^2$ are as follows. 

\begin{prop}\label{prop:MMC2}
	\begin{itemize}
		\item[(a)] If a subequation $\F$ is a convex cone, then $\cM_{\F} = \F$. In particular, if $\cM_{\F}$ is also a subequation, then $\cM_{\cM_{\F}} = \cM_{\F}$.
		\item[(b)] $\F$ and its dual $\wt{\F}$ have the same maximal monotonicity cones; that is, $\cM_{\wt{\F}} = \cM_{\F}$.
	\end{itemize}
	\end{prop}

\begin{proof} Part (a) follows from the fact that if $\mathcal{C}$ is any convex cone in $\R^N$, then for each $v \in \R^N$
	$$
	\mathcal{C} + v \subset \mathcal{C} \ \ \Leftrightarrow \ \ v \in \mathcal{C}.
	$$
	Part (b) follows from the fact for any subequation $\F$, one has: 
	\begin{equation}\label{J_invariance}
	\text{for each} \ J \in \J^2: \ \ \F + J \subset \F \ \ \Leftrightarrow \ \ \wt{\F} + J \subset \wt{\F}.
	\end{equation}
	Indeed, the implication $(\Rightarrow)$ is part (2) of Proposition \ref{prop:duality}. By the same property, if $\wt{\F} + J \subset \wt{\F}$, then $\wt{\wt{\F}} + J \subset \wt{\wt{\F}}$, but $\wt{\wt{\F}} = \F$ by the reflexivity in part (6) of Proposition \ref{prop:duality}, since any subequation $\F$ satisfies property (T).
	\end{proof}

We record the following basic definitions. 

\begin{defn}\label{defn:MCS} A closed convex cone  $\cM \subset \J^2$ with vertex at the origin that contains the minimal monotonicity set $\cM_0 := \cN \times \{0\} \times \cP$ will be referred to as a {\em monotonicity cone}. A monotonicity cone $\cM$ which satisfies property (T) will be called a {\em monotonicity cone subequation}. 
\end{defn}

In this case,  for any closed set  $\F \subset \J$ which is $\cM$-monotone, property (T) for $\F$ follows.

\begin{prop}\label{prop:subequation_cones} Let $\cM$ be a monotonicity cone subequation. Then, for any closed subset $\F \subset \J^2$ which is $\cM$-monotone one has the following set identities
\begin{equation}\label{SI1}
	 \F + \Int \, \cM = \Int \, \F
\end{equation}
and
\begin{equation}\label{SI2}
	\F = \overline{\Int \, \F} ,
\end{equation}
and hence $\F$ satisfies properties (P), (N) and (T), so that $\F$ is a subequation constraint set (if $\F \neq \emptyset$ and  $\F \neq \J^2)$. 
\end{prop}

\begin{proof}
	For the identity \eqref{SI1}, first note that  $\F + \Int \, \cM$ is an open set being the union over $J \in \F$ of the open sets $J + \Int \, \cM$. By monotonicity, this open set is contained in $\F$ and hence are contained in $\Int \, \F$. For the reverse inclusion, if $J \in \Int \, \F$ then, picking $J_0 \in \Int \, \cM$, one has
	$$
		J = (J - \veps J_0) + \veps J_0 \in \Int \, \F + \Int \, \cM
	$$
	if $\veps > 0$ is chosen sufficiently small since $\Int \, \F$ is open and $\Int \, \cM$ is a cone.
	 
	For the identity \eqref{SI2}, with $J_0 \in \Int \, \cM$, each $J \in \F$ can be approximated by $J + \veps J_0$, which belongs to $\Int \, \F$ by the identity  \eqref{SI1}.
	
	The subequation claim for $\F$ is immediate, since \eqref{SI2} is the topological property (T) and the properties (P) and (N) follow from the $\cM_0$-monotonicity of $\F$.
\end{proof}

As a corollary of the fact that monotonicity implies the topological property (T), we have the following result. 

\begin{prop}[Intersections and unions]\label{prop:algebra} 	Suppose that $\{ \F_{\sigma} : \sigma \in \Sigma \}$ is an arbitrary family of subequations which are all $\cM$-monotone for a given monotonicity subequation cone $\cM$. Then one has the following statements.
\begin{itemize}
		\item[(a)] The intersection $\F  := \bigcap_{\sigma \in \Sigma} \F_{\sigma}$ (if non empty) is an $\cM$-monotone subequation.
		\item[(b)] The closure of the union $\F  := \overline{\bigcup_{\sigma \in \Sigma} \F_{\sigma}}$ (if not equal to all of $\J^2$) is an $\cM$-monotone subequation .
\end{itemize}
\end{prop}

\begin{proof}
	First note that arbitrary  intersections and arbitrary unions of sets that are $\cM$-monotone are again $\cM$-monotone for any set $\cM$. In particular, property (P), as well as property (N) are preserved under arbitrary  intersections and unions. In addition, if $\cM$ is a monotonicity cone subequation then $\cM$-monotonicity for a set $\F$ implies property (T) for $\F$ as long as $\F$ is closed (see Proposition \ref{prop:subequation_cones}). Arbitrary intersections of closed sets are closed, while finite unions are (which is why in general one must take the closure of the union). Finally, subequations $\F \subset \J^2$ are by defintion non empty and proper subsets of $\J^2$. Intersections of proper subsets are proper, but may be empty. Unions of non empty sets are non empty but may have closure equal to all of $\J^2$. This completes the proof.
\end{proof}

It is worth noting that easy examples (including pure first order subequations where (P) and (N) are automatic) illustrate the role monotonicity plays in Propsoition \ref{prop:algebra}. More precisely, for arbitrary closed sets, property (T) is not preserved under finite intersections and the interior of finite unions may be larger than the (finite) union of the interiors. For example, consider two closed cubes of the same size which intersect along a common face. 

Now we turn to examples of monotonicty cone subequations.

\subsection{Product monotonicity cone subequations}

Consider a product set 
$$
\cM := \cM^0 \times \cM^1 \times \cM^2 \subset \J^2 = \R \times \R^n \times \Symn.
$$
Note that $\cM$ is a convex cone if and only if each factor of $\cM$ is a convex cone, and $\cM$ satisfies (T) if and only if each factor does; that is,
\begin{equation}\label{T_for_M}
\cM = \overline{ \Int \, \cM} \iff \cM^k = \overline{ \Int \, \cM^k}, \ \ k = 0,1,2.
\end{equation}
In this case, $\cM^1$ is a convex cone satisfying property (T). Thus
\begin{equation}\label{T_for_M}
\mbox{$\cM^1 := \cD$ is a directional cone as in Definition \ref{defn:property_D};}
\end{equation}
that is, a closed convex cone in $\R^n$ with vertex at the origin and non-empty interior. Note that $\cD = \R^n$ is allowed.
Also note that
\begin{equation}\label{PMS2}
\mbox{ $\cM$ satisfies (P) $\iff \cP \subset \cM^2$}
\end{equation}
and that
\begin{equation}\label{PMS0}
\mbox{ $\cM$ satisfies (N) $\iff \cN \subset \cM^0$}
\end{equation}
This can be summarized as follows.

\begin{prop}\label{prop:MCS} A product set $\cM = \cM^0 \times \cM^1 \times \cM^2 \subset \J^2$ is a monotonicity cone subequation if and only if the factors satisfy
	\begin{itemize}
		\item[(0)] $\cM^0 = \cN$ or $\cM^0 = \R$;
		\item[(1)] $\cM^1 = \cD$ is a directional cone in $\R^n$ as in Definition \ref{defn:property_D};
		\item[(2)] $\cM^2$ is a closed convex cone in $\Symn$ which contains $\cP$.
	\end{itemize}
\end{prop}

In particular, important examples of product monotonicity cone subequations include those determined by $\cN, \cD$ and $\cP$ as introduced in \eqref{PND}; that is,
$$
\cM(\cP):= \R \times \R^n \times \cP, \ \ \cM(\cN):= \cN \times \R^n \times \Symn \ \  \text{and} \ \ \cM(\cD):= \R \times \cD \times \Symn.
$$
Moreover, each product monotonicity cone subequation contains the intersection of these basic examples; that is, these intersections form a fundamental neighborhood system for $\cM_0$ among all product monotonicity subequations. The proof is omitted. 

\begin{defn}\label{defn:FPMC} A {\em fundamental product monotonicity subequation} is a subset $\cM \subset \J^2$ of the form
	\begin{equation}\label{FPMC}
	\cM := \cN \times \cD \times \cP = \cM(N) \cap \cM(\cD) \cap \cM(\cP)
	\end{equation}
	determined by the choice of a directional cone $\cD$  in $\R^n$.
\end{defn}

An arbitrary subset $\F \subset \J^2$ is $\cM$-monotone for a fundamental product monotonicity subequation $\cM$ as in \eqref{FPMC} if and only if $\F$ satisfies (P), (N) and the {\em Directionality Property}
\begin{equation*}\label{Directionality}
{\rm (D)}\ \ \ \ (r,p,A) \in \F \ \Longrightarrow \ (r, p + q, A) \in \F \ \  \text{for all} \ q \in \cD.
\end{equation*}

\begin{lem}\label{lem:FPMCDual}
	The Dirichlet dual of a fundamental product monotonicity subequation $\cM$ is
	\begin{equation}\label{FPMCDual}
		\begin{split}
	\cMt  & = \wt{\cM}(N) \cup  \wt{\cM}(\cD) \cup \wt{\cM}(\cP) \\
	& = \left\{ (r,p,A) \in \J^2: \ \text{either} \ r \leq 0, \ \text{or} \ p \in \cDt, \ \text{or} \ A \in \cPt \right\}.
	\end{split}
	\end{equation}
\end{lem}

\begin{proof} The first line of \eqref{FPMCDual} just uses the formula for the dual of intersections \eqref{dual_intersection} and the second line uses the formulas \eqref{PND_duals} for the duals of the elementary cones. 
\end{proof}

\section{A Fundamental Family of Monotonicity Cone Subequations}\label{sec:monotonicity}

In this section, we will construct a family of monotonicity cone subequations which is fundamental in the the sense that if a given subequation $\F$ is $\cM$-monotone for a monotonicity subequation $\cM$, then there exists an element $\cM_{\ast}$ of the fundamental family with $\cM_{\ast} \subset \cM$. That is, the family provides a fundamental neighborhood system for $\cM_0$ among all convex cone subequations. Note that $\cM_{\ast} \subset \cM$ implies $\wt{\cM} \subset \wt{\cM}_{\ast}$ so that if the (ZMP) holds for $\wt{\cM}_{\ast}$ it will also hold for $\wt{\cM}$.

In order to construct the fundamental family, we will use intersections of five elementary monotonicity cone subequations to build up a family of seventeen distinct monotonicity cone subequations for use in the main comparison theorem of section \ref{sec:CP}. Some of these will also have a product structure and in some cases there will be silent factors which are suitable for the results on reduction of section \ref{sec:reductions}.

\subsection{Construction of the fundamental family} 

We begin by adding two one parameter families of monotonicity cones to the elementary building blocks $\cM(\cP)$, $\cM(\cN)$ and $\cM(\cD)$ given in \eqref{PND} as
$$
	\cM(\cP):= \R \times \R^n \times \cP, \ \ \cM(\cN):= \cN \times \R^n \times \Symn \ \  \text{and} \ \ \cM(\cD):= \R \times \cD \times \Symn,
$$
where 
$$
	\cP = \{ A \in \Symn: \ A \geq 0 \}, \ \ \cN =  (-\infty, 0] \subset \R \ \ \text{and} \ \ \cD \subset \R^N\
$$
and $\cD$ is a closed convex cone with vertex in $0$ and $\Int \, \cD \neq \emptyset$.

\begin{defn}\label{defn:GRCone} Let $\J = \R \times \R^n \times \Symn$. Given a real number $\gamma \in (0, + \infty)$, the set
\begin{equation}\label{GCone}
\cM(\gamma):= \left\{ (r,p,A) \in \J^2: \ r \leq - \gamma |p| \right\}
\end{equation}	
will be called the {\em $\gamma$-monotonicity cone subequation}	and given a real number $R \in (0, + \infty)$, the set
\begin{equation}\label{RCone}
\cM(R):= \left\{ (r,p,A) \in \J^2: \ A \geq \frac{|p|}{R}I \right\}
\end{equation}	
will be called the {\em $R$-monotonicity cone subequation}
	\end{defn}

Besides helping to complete a fundamental family, these new building blocks are interesting in their own right. They are all closed convex cones with vertex at the origin. It is also clear that they satisfy the subequation constraint conditions so that each is a monotonicity cone subequation. We will see that the cone $\cM(\gamma)$ arises naturally in equations with strict monotonicity in $r$ and a Lipschitz bound in $p$ (see Theorem \ref{thm:CP_DECr}). Similarly $\cM(R)$ arises in zero order free order equations with some degree of strict ellipticity and a Lipschitz bound in $p$ (see Theorem \ref{thm:CP_SE}). 

Additional monotonicity cone subequations are generated by taking intersections of the five basic cones $\cM(\cN), \cM(\cP),\cM(\cD), \cM(\gamma)$ and $\cM(R)$. This provides us with our list of seventeen distinct monotonicity cone subequations.  
 
\begin{defn}[Our list of monotonicity cone subequations]\label{defn:cone_zoo}
	
	Part I) First on the list are four of the five basic examples defined above:
\begin{itemize}
	\item [(1)] $\cM(\cN)$:
	\item[(2)] $\cM(\cP)$
	\item[(3)] $\cM(\gamma)$ \ with $\gamma \in (0, +\infty)$;
	\item[(4)] \ $\cM(R)$ \ with $\R \in (0, +\infty)$. 
\end{itemize}

	Part II) Double intersections of those in Part I) give rise to four new monotonicity cone subequations:
	\begin{itemize}
	\item[(5)] $\cM(\gamma, R)  := \cM(\gamma) \cap \cM(R) = \left\{ (r,p,A) \in \J^2: r \leq - \gamma |p| \   \text{and} \ A \geq \frac{|p|}{R}I     \right\}$:
	\item[(6)] $	\cM(\gamma, \cP)  := \cM(\gamma) \cap \cM(\cP) = \left\{ (r,p,A) \in \J^2: \ r \leq - \gamma |p| \   \text{and} \ A \geq 0    \right\}$;
	\item[(7)] $
	\cM(\cN, R)  := \cM(\cN) \cap \cM(R) = \left\{ (r,p,A) \in \J^2: r \leq 0 \   \text{and} \  A \geq \frac{|p|}{R}I    \right\}$;
	\item[(8)] $\cM(\cN, \cP)  := \cM(\cN) \cap \cM(\cP) = \left\{ (r,p,A) \in \J^2: r \leq 0 \   \text{and} \ A \geq 0    \right\}$.
	\end{itemize}

Part III) To complete the list, first add our last basic example: 
\begin{itemize}
	\item[(9)] $\cM(\cD)$ with $\cD \subsetneq \R^n$ a proper directional cone.
\end{itemize}

Part IV) Intersecting such a $\cM(\cD)$ with the first eight examples completes the list:
\begin{itemize}
	\item[(10)]  $\cM(\cN, \cD)  := \cM(\cN) \cap \cM(\cD) = \left\{ (r,p,A) \in \J^2: \ r \leq 0 \   \text{and} \ p \in \cD    \right\}$ ;
	\item[(11)] $\cM(\cD, \cP)  := \cM(\cD) \cap \cM(\cP) = \left\{ (r,p,A) \in \J^2: \ p \in \cD  \    \text{and} \  A \geq 0 \right\}$;
	\item[(12)] $\cM(\gamma, \cD)  := \cM(\gamma) \cap \cM(\cD) = \left\{ (r,p,A) \in \J^2: \ r \leq - \gamma |p| \   \text{and} \ p \in \cD    \right\}$;
	\item[(13)] 	$\cM(\cD, R)  := \cM(\cD) \cap \cM(R) = \left\{ (r,p,A) \in \J^2: \ p \in \cD  \    \text{and} \  A \geq \frac{|p|}{R}I \right\}$;
		\item[(14)] $	\cM(\gamma, \cD, R)  := \cM(\gamma) \cap \cM(\cD) \cap \cM(R)$; that is, 
		$$ \cM(\gamma, \cD, R) = \left\{ (r,p,A) \in \J^2: \ r \leq - \gamma |p|, \  p \in \cD \ \text{and} \ A \geq \frac{|p|}{R}I     \right\}; $$
			\item[(15)] $
			\cM(\gamma, \cD, \cP) := \cM(\gamma, \cP) \cap \cM(\cD) = \cM(\gamma) \cap \cM(\cD) \cap \cM(\cP)$;
				\item[(16)] $
				\cM(\cN, \cD, R) := \cM(\cN, R) \cap \cM(\cD) = \cM(\cN) \cap \cM(\cD) \cap \cM(R) $;
					\item[(17)] $\cM(\cN, \cD, \cP) := \cM(\cN, \cP) \cap \cM(\cD) = \cM(\cN) \cap \cM(\cD) \cap \cM(\cP) $.
\end{itemize}
The cone $\cM(\gamma,R)$ will be called the {\em $(\gamma, R)$-monotonicity cone subequation}. The same nomenclature will be used for the other cones above; for example, $\cM(\gamma, \cD,R)$ will be called the {\em  $(\gamma, \cD, R)$-monotonicty cone subequation}. 
\end{defn}

About this family of cones, a few remarks are in order.

\begin{rem}\label{rem:zoo1} While taking all possible intersections of the basic five cones will produce more than twelve additional sets, many of the intersections are not distinct since for each $\gamma, R \in (0, + \infty)$ one has
\begin{equation}\label{GNRP1}
	\cM(\gamma) \subsetneq \cM(\cN) \quad \text{and} \quad \cM(R) \subsetneq \cM(\cP) \ \text{are proper subsets.}
\end{equation}
Consequently, in Part II), the double intersections  $\cM(\gamma) \cap \cM(\cN) = \cM(\gamma)$ and $\cM(R) \cap \cM(\cP) = \cM(R)$ can be ignored, along with the triple intersections involving them and $\cM(\cN) \cap \cM(\cP) \cap \cM(\gamma) \cap \cM(R) = \cM(\gamma, R)$.
\end{rem}

Some of the cones have been seen before.

\begin{rem}[Fundamental products]\label{rem:cone_zoo_fundamental_products} The elementary cones $\cM(\cN)$, $\cM(\cD)$ and $\cM(\cP)$ are all examples of fundamental product monotonicity subequations in the sense of Definition \ref{defn:FPMC} as are their intersections $\cM(\cN, \cP), \cM(\cN, \cD), \cM(\cD, \cP)$ and $\cM(\cN, \cD, \cP)$.
	\end{rem}

Some  of the cones have silent factors which provide useful simplifications in the reductions of section \ref{sec:reductions}.

\begin{rem}[Reduction by suppressing trivial factors]\label{rem:cone_zoo_reduction} With the exception of $\cM(\cN, \cD, \cP)$ in the case $\cD \neq \R^n$, all of the cones in Remark \ref{rem:cone_zoo_fundamental_products} have a trivial factor $\R, \R^n$ or $\Symn$. A particularly important case is the {\em negativity-convexity subequation} $\cM(\cN, \cP) = \cN \times \R^n \times \cP$, which will de discussed at length in subsection \ref{subsec:GFSE} on comparison for {\em gradient free} subequations $\F$.  In this case, one can {\em reduce} to the monotonicity cone $\cQ:= \cN \times \cP \subset \R \times \Symn$. 
	
The other reducible cones are $\cM(\gamma), \cM(R), \cM(\gamma, \cD)$ and $\cM(\cD, R)$. If one eliminates the trivial factor, one can define {\em reduced cones} such as 
\begin{equation}\label{GDPrime}                 \cM^{\prime}(\gamma, \cD) := \left\{ (r,p) \in \R \times \R^n: \ r \leq -\gamma |p| \ \text{and} \ p \in \cD \right\}
\end{equation}
and
\begin{equation}\label{DRPrime}
\cM^{\prime}(\cD, R) := \left\{ (p, A) \in \R^n \times \Symn: \  p \in \cD \  \text{and} \ A \geq \frac{|p|}{R}I \right\}
\end{equation}
so that
\begin{equation}\label{M_Products}
\cM(\gamma, \cD) = \cM^{\prime}(\gamma, \cD) \times \Symn \ \ \text{and} \ \ 
\cM(\cD, R) = \R \times \cM^{\prime}(\cD, R).
\end{equation}
The cones in \eqref{DRPrime} will be used for zero order free subequations in subsection \ref{subsec:ZOFSE}. The other pair reducible cones $\cM(\gamma)$ and $\cM(R)$ correspond to $\cD = \R^n$ and we will write
\begin{equation}\label{GPrime}           \cM^{\prime}(\gamma) := \left\{ (r,p) \in \R \times \R^n: \ r \leq -\gamma |p| \right\} 
\end{equation}
and
\begin{equation}\label{RPrime}           \cM^{\prime}(R) := \left\{ (p, A) \in \R^n \times \Symn: \ A \geq \frac{|p|}{R}I \right\}.
\end{equation}
\end{rem}
The utility of $\cM^{\prime}(\gamma)$ and $\cM^{\prime}(R)$ was noted following Definition \ref{defn:GRCone}.

Finally, for future use, we record the following formulas for strictness and duality.
		
\begin{rem}\label{rem:dual_zoo} If $\cM = \cM(\gamma, \cD, R)$ then
\begin{equation}\label{Int_GDRCone}
 \Int \, \cM =  \left\{ (r,p,A) \in \J^2: \ r < - \gamma |p|, \  p \in  \Int \, \cD, \ \text{and} \ A > \frac{|p|}{R}I     \right\};
\end{equation}
\begin{equation}\label{GDRCD}
\cMt = \left\{ (r,p,A) \in \J^2: \ \text{either} \ r \leq \gamma |p|, \ \text{or} \ p \in \cDt, \ \text{or} \ A + \frac{|p|}{R}I \in \cPt \right\};
\end{equation}
\begin{equation}\label{notGDRD}
\sim \cMt = \left\{ (r,p,A) \in \J^2: \  r > \gamma |p|, \ -p \in  \cD, \ \text{and} \ A < - \frac{|p|}{R}I \right\}.
\end{equation}
Similar formulas for the other cones defined above are easily deduced from these.
\end{rem}

\subsection{Nesting, limit cases and simplifying the family of cones}

In this subsection, we look at the limiting cases of $\cM(\gamma, \cD, R)$ cones when the parameters $\gamma$ and $R$ in $(0, + \infty)$ tend to their limiting values. This analysis will allow us to express every element of the fundamental family in terms of a triple $(\gamma, \cD, R)$ with $\gamma \in [0, +\infty)$, $\cD \subseteq \R^N$ a proper directional cone or all of $\R^n$ and $R \in (0, + \infty]$. This simplification is carried out in  Remark \ref{rem:cone_zoo_compress} below. With  respect to the partial ordering by set inclusion, $\cM(\gamma)$ and $\cM(R)$ are nested families. It is easy to see from the definitions that
\begin{equation}\label{GNest}
	\mbox{$\cM(\gamma)$ is decreasing in $\gamma \in (0, +\infty)$}
\end{equation}
and
\begin{equation}\label{RNest}
\mbox{$\cM(R)$ is increasing in $R \in (0, +\infty)$}.
\end{equation}
Hence  
\begin{equation}\label{GRNest}
\mbox{$\cM(\gamma, R)$ decreases as $\gamma$ increases and $R$ decreases}
\end{equation}
and $\cM(\gamma, R)$ increases as $\gamma$ decreases and $R$ increases. Moreover, these  monotonicity properties pass to intersections with $\cM(\cD)$.

\begin{prop}[Decreasing limits]\label{prop:decreasing_cones} For the family of $(\gamma, \cD, R)$-cones defined in formula (14) of Part IV of Definition  \ref{defn:cone_zoo}, the decreasing limits \footnote{From the monotoncity of \eqref{GNest}, one means, of course, the intersection over $\gamma \in (0, + \infty)$ of $\cM(\gamma, \cD, R)$ for the first decreasing limit and similar intersections for the last two.}
		\begin{equation}\label{decr_limits_cones}
		\lim_{\gamma \nearrow + \infty} \cM(\gamma, \cD, R), \  \lim_{R \searrow 0} \cM(\gamma, \cD, R) \ \text{and} \ \lim_{\substack{\gamma \nearrow + \infty\\ R \searrow 0}} \cM(\gamma, \cD, R),
		\end{equation} 
		are all equal  to the minimal monotonicity set $\cM_0 = \cN \times \! \{0\} \! \times \cP$.
	\end{prop}

\begin{proof} First note that
\begin{equation*}\label{DL0}
 \lim_{\gamma \nearrow +\infty} \cM(\gamma) = \bigcap_{\gamma > 0} \cM(\gamma) = \cN \times \{0\} \times \Symn 
\end{equation*}
and
\begin{equation*}\label{DL1}
\lim_{R \searrow 0} \cM(R) = \bigcap_{R > 0} \cM(R) = \R \times \{0\} \times \cP 
\end{equation*}
since
\begin{equation*}\label{DL2}
	\cM(\gamma) = \left\{ (r,p,A) \in \J^2:  |p| \leq - \frac{r}{\gamma} \right\} \ \ \text{and} \ \ \cM(R) = \left\{ (r,P,A) \in \J^2: |p| I \leq RA \right\}.
\end{equation*}
It follows easily that each of the three decreasing limits 
\begin{equation}\label{DL3}
\lim_{\gamma \nearrow + \infty} \cM(\gamma, R), \  \lim_{R \searrow 0} \cM(\gamma, R) \ \text{and} \ \lim_{\substack{\gamma \nearrow + \infty\\ R \searrow 0}} \cM(\gamma, R) \ \text{all decrease to} \ \cM_0.
\end{equation}
The role of $\cD$ (a proper cone in $\R^n$) is innocuous and \eqref{decr_limits_cones} follws immediately from \eqref{DL3} by interesting with $\cM(\cD)$.
\end{proof}

For the increasing limits, first note that
\begin{equation}\label{IL1}
\lim_{\gamma \searrow 0} \cM(\gamma) = \bigcup_{\gamma > 0} \cM(\gamma) = \cM(\cN) \backslash \left( \{0\} \times (\R^n \setminus \{0\}) \times \Symn \right),
\end{equation}
which has closure $\cM(\cN)$, and
\begin{equation}\label{IL2}
\lim_{R \nearrow +\infty} \cM(R) = \bigcup_{R > 0} \cM(R) =  \cM(\cP) \backslash \left( \R \times (\R^n \setminus \{0\}) \times \{0\} \right),
\end{equation}
which has closure $\cM(\cP)$.

From these two facts we leave it to the reader to prove the following result.
 
\begin{prop}[Increasing limits]\label{prop:incr_limits_cones} For the family of $(\gamma, \cD, R)$-cones defined in formula (14) of Part IV of Definition  \ref{defn:cone_zoo}, the following hold.
	\begin{itemize}
		\item[(a)] $\cM^{\prime}(\gamma, \cD) \times \cP$ is the closure of the increasing limit $\displaystyle{\lim_{R \nearrow + \infty} \cM(\gamma, \cD, R)}$.
		\item[(b)] $\cN \times \cM^{\prime}(\cD, R)$ is the closure of the increasing limit $\displaystyle{\lim_{\gamma \searrow 0} \cM(\gamma, \cD, R)}$.
		\item[(c)] $\cN \times \cD \times \cP$ is the closure of the increasing limit $\displaystyle{\lim_{\substack{\gamma \searrow 0\\ R \nearrow + \infty}} \cM(\gamma, \cD, R)}$.
	\end{itemize} 
\end{prop}

\begin{rem}[Simplifying the list of cones]\label{rem:cone_zoo_compress} In light of \eqref{IL1} and \eqref{IL2}, it is natural to extend the parameters to allow $\gamma = 0$ and $R = +\infty$ and to define
	\begin{equation}\label{GNRP2}
	\cM(\gamma = 0) := \cM(\cN) \quad \text{and} \quad \cM(R = +\infty) := \cM(\cP), 
	\end{equation}
	Also note that in the definition of $\cM(\gamma)$ if one simply sets $\gamma = 0$, one obtains $\cM(\cN)$. Similarly, one obtains $\cM(\cP)$ by setting $R = + \infty$ in the definition of $\cM(R)$.
	 With such a choice, the five basic cones can be simplified to three
	 $$ 
	 \mbox{$\cM(\gamma)$ with $\gamma \in [0, +\infty), \cM(R)$ with $R \in (0, +\infty]$ and
	 	$\cM(\cD)$.}
 	$$
 	The remaining cones all take the form 
	\begin{equation}\label{GNRP3}
	\cM(\gamma, R), \ \cM(\gamma, \cD), \ \cM(\cD, R) \ \text{and} \ \cM(\gamma, \cD, R).
	\end{equation}
	Moreover, when $\cD = \R^n$, one has $\cM(\cD) = \J^2$ and hence $\cM(\gamma, \R^n, R) = \cM(\gamma, R)$ and so on. Adopting these conventions/definitions, each of the seventeen cones defined in Definition \ref{defn:cone_zoo} is, in fact, a $(\gamma, \cD, R)$-cone with $\gamma \in [0, +\infty)$, $R \in (0, +\infty]$ and $\cD \subset \R^n$ either a proper directional cone or all of $\R^n$. 
\end{rem}

\subsection{The fundamental nature of the family of monotonicity cones}\label{sec:fundamental}

Our family of monotonicity cones $\cM(\gamma, \cD, R)$ is ``fundamental'' in the following sense.

\begin{thm}[The Fundamental Family Theorem]\label{thm:fundamental}
	If $\F$ is a subequation which is $\cM$-monotone for some monotonicity cone subequation $\cM$, then $\F$ is $\cM(\gamma, \cD, R)$-monotone for some $\gamma, R \in (0, +\infty)$ and some directional cone $\cD$.
	\end{thm} 

\begin{proof}
	It suffices to find $\cM(\gamma, \cD, R) \subset \cM$. This follows from the next two lemmas.
	\end{proof}

\begin{lem}\label{lem:Fund1}
	Given a monotonicity cone subequation $\cM$. If there exist $\veps > 0$ and a directional cone $\cD \subset \R^n$ such that $\{-1\} \times (\cD \cap B_{\veps}(0)) \times \{I\} \subset \cM$, then the $(\gamma, \cD, R)$-monotonicity cone with $R < \veps$ and $\gamma:= 1/R$  satisfies
	$$
	\cM(\gamma, \cD, R) \subset \cM.
	$$	
	\end{lem}

\begin{proof}
	Suppose that $(r,p,A) \in \cM(\gamma, \cD, R)$. If $p = 0$, then $r \leq - \gamma|p| = 0$ and $A \geq \frac{|p|}{R} I = 0$ so that $(r,p,A) \in \cM_0 \subset \cM$.
	
	Suppose now that $p \neq 0$. It suffices to show that 
\begin{equation}\label{FL1}
	(r,p,A) = \frac{|p|}{R} \left[ \left( -1,  \frac{R}{|p|} p, I \right) + (-s, 0, P) \right] 
\end{equation}
with $s \geq 0, P \geq 0$ and $ \frac{R}{|p|} p \in \cD \cap B_{\veps}(0)$, because then the facts
\begin{itemize}
	\item[(i)] $ \left(-1, \frac{R}{|p|}p, I \right) \in \cM$ by hypothesis;
	\item[ii)] $\cM$ is $\cN \times \{0\} \times \cP$-monotone;
	\item[iii)] $\cM$ is a cone,
\end{itemize}
combined with \eqref{FL1} proves that $(r,p,A) \in \cM$. To see that \eqref{FL1} is true, note that this equality defines
$$
	s = - \frac{rR}{|p|} \quad \text{and} \quad P = A - \frac{|p|}{R}I.
$$
Now, by the definition of $(r,p,A) \in \cM(1/R, \cD, R)$, one has $r \leq - \frac{|p|}{R}$ and $A \geq \frac{|p|}{R}I$. Hence $P \geq 0$ and $\frac{(-r)R}{|p|} \geq 1$ so that $s \geq 0$ and $\frac{R}{|p|}p \in \cD \cap B_{\veps}(0)$ as $\cD$ is a cone and $R < \veps$.
\end{proof}

\begin{lem}\label{lem:Fund2}
	Given a monotonicity cone subequation $\cM$.  There exist $\veps > 0$ and a directional cone $\cD \subset \R^n$ such that 
	$$
		\{-1\} \times (\cD \cap B_{\veps}(0)) \times \{I\} \subset \cM.
	$$
\end{lem}

\begin{proof}
	By the topological condition (T), $\Int \, \cM \neq \emptyset$. Pick $(r,p,A) \in \Int \, \cM$. By perturbing $p$ we can assume that $p \neq 0$, Pick $\delta \in (0, |p|)$ small so that $\{r\} \times B_{\delta}(p) \times A \subset \Int \, \cM$. For $t \geq t_0 > 0$ large, one has $P:= tI - A \geq 0$ and $s:= r + t \geq 0$. Hence
$$
(-t, q, tI) = (r, q, A) + (-s,0, P) \in \Int \, \cM, \  \ \ \forall \ q \in B_{\delta}(p), \ \forall \ t \geq t_0 > 0.
$$
Since $\Int \, \cM$ is a cone, 
$$
	\{-1\} \times \frac{1}{t} B_{\delta}(p) \times \{I\} \subset \Int \, \cM \ \ \text{for each} \ t \geq t_0 > 0.
$$
Take $\veps:= \delta/t_0$ and $\cD$ the cone on $B_{\delta}(p)$. Then
$$
	\cD \cap B_{\veps}(0) \subset \bigcup _{t \geq t_0} \frac{1}{t} B_{\delta}(0) \cup \{0\} ,
$$
which proves that $\{-1\} \times (\cD \cap B_{\veps}(0)) \times \{I\} \subset \cM$.
\end{proof}

It is important to note that the fundamental nature of the family will ensure the validity of the comparison principle locally (see Theorem \ref{thm:LC}).

\section{The Zero Maximum Principle for Dual Monotonicity Cones}\label{sec:ZMP} 

In this section, we examine the validity of the zero maximum principle for $\cMt$-subharmonic functions if $\cM$ is a monotonicity cone subequation. Its validity can be  to the existence of a global, regular and strictly $\cM$-subharmonic function. This function generates an approximation from above of the $\cMt$-subharmonic function zero and has the advantage that the definitional comparison of Lemma \ref{lem:DCP} (see formula \eqref{ZMP5} below) applies since it is regular and strict (unlike zero). 

\begin{defn}[Strict approximator]\label{defn:strict_approx} Suppose that $\cM$ is a convex cone subequation; that is, a convex cone with vertex at the origin which satisfies the subequation constraint conditions (P), (N) and (T). Given a domain $\Omega \subset \subset \R^n$, we say $\cM$ {\em admits a strict approximator on $\Omega$} if there exists $\psi$ with
	\begin{equation}
	\psi \in C(\overline{\Omega}) \cap C^2(\Omega)  \ \text{and} \ \ J^2_x \psi \in  \Int \, \cM \ \ \text{for each} \ x \in \Omega.
	\end{equation}
	\end{defn}

\vspace{2ex}

\begin{thm}[The Zero Maximum Principle]\label{thm:ZMP} Suppose that $\cM$ is a convex cone subequation that admits a strict approximator on $\Omega$. Then the zero maximum principle (ZMP) holds for $\wt{\cM}$ on $\overline{\Omega}$; that is,
	\begin{equation}\label{ZMP}
	 z\leq 0\ \ {\rm on} \ \partial \Omega \ \ \Rightarrow\ \ z\leq 0 \ \ {\rm on} \ \Omega
	\end{equation}
	for all $z\in \USC(\overline{\Omega})$ which are $\cMt$-subharmonic on $\Omega$.
\end{thm}

Notice that the (ZMP) is the comparison principle for $\cMt$ in the case where $u = z$ and $w \equiv 0$, because by assumption  $u :=z$ is  $\cMt$-subharmonic and $w :=0$ is  $\cMt$-superharmonic since  $J_x^2 (-w) \equiv (0,0,0) \in  \cMt$. The proof is an elementary consequence of the definitions (as is the proof of definitional comparison).

\begin{proof} The dual $\cMt$ also satisfies (P), (N) and (T) by property (7) of Proposition \ref{prop:duality}. Since $\cMt$ has property (N), $z -  m \in \cMt(\Omega)$ for each $m \in (0, +\infty)$, as noted in Remark \ref{rem:N-translates}. 
	
	Since $z - m < 0$ on $\partial \Omega$ which is compact, one has 
	\begin{equation}\label{ZMP4}
	z - m + \veps \psi \leq 0 \ \text{on} \ \partial \Omega \ \ \ \text{for each} \ \veps > 0 \ \text{sufficiently small}.
	\end{equation}
	Now, since $z - m \in \cMt(\Omega)$ and since $\veps \psi \in C(\overline{\Omega}) \cap C^2(\Omega)$ is strictly $\cM$-subharmonic on $\Omega$ (by coherence and $\cM$ being a cone), one has
	\begin{equation}\label{ZMP5}
	z - m + \veps \psi \leq 0 \ \text{on} \ \Omega
	\end{equation}
	by the Definitional Comparison of Lemma \ref{lem:DCP} with $\F = \cMt$ and $\wt{\F} = \wt{\cMt} = \cM$. Taking the limit in \eqref{ZMP5} as $m \searrow 0$ and $\veps \searrow 0$ gives $z \leq 0$ on $\Omega$. 
\end{proof}

As a corollary to this general theorem, we obtain the Zero Maximum Principle for each $(\gamma, \cD, R)$-monotonicity cone as in Definition \ref{defn:cone_zoo} (see also Remark \ref{rem:cone_zoo_compress}), with a restriction on the size of the domain if $R$ is finite. The following result was originally given in Theorem B.2 of \cite{HL13b}.

\begin{thm}\label{thm:ZMP_for_M}
	Let $\cM$ be a $(\gamma, \cD, R)$-monotonicity cone subequation. Given  $z \in \USC(\overline{\Omega})$ which is $\cMt$-subharmonic on $\Omega$, one has 
	$$
	z \leq 0 \ \text{on} \ \partial \Omega \ \ \Longrightarrow \ \ z \leq 0 \ \text{on} \ \Omega \leqno{\rm{(ZMP)}}
	$$
	as follows:\\
	\underline{Case $R = +\infty$:} For arbitrary
	 $\Omega \subset \subset \R^n$. This case includes $\cM(\gamma, \cP) = \cM^{\prime}(\gamma) \times  \cP$ with $\gamma \in [0, +\infty)$, where the case $\gamma = 0$ is $\cN \times \R^n \times \cP = \cM(\cN, \cP)$ and the case 
	$$ \cM(\gamma, \cD, \cP) := \cM(\gamma) \cap \cM(\cD) \cap \cM(\cP) := \cM^{\prime}(\gamma, \cD) \times \cP, $$
	and hence any of the larger monotonicity cone subequations, namely 
	$$\cM(\gamma), \cM(\cD), \cM(\cP), \cM(\gamma, \cD) \ \text{and} \ \cM(\cD,\cP).
	$$
	\underline{Case $R$ finite:} For domains
	 $\Omega$ which are contained in a translate of the truncated cone $\cD_R := \cD \cap B_R(0)$; that is,
	\begin{equation}\label{ZMP_Domain1}
	\Omega \subset (y +  \cD) \cap B_R(y) \ \text{for some} \ y \in \R^n.
	\end{equation} 
	This case includes $\cM(\gamma, \cD, R) := \cM(\gamma) \cap \cM(\cD) \cap \cM(R)$ with $R$ finite and hence any of the larger monotonicity cone subequations, namely
	$$\cM(R), \cM(\gamma, R) \ \text{and} \  \cM(\cD, R) \ \text{with} \ R \ \text{finite}.
	$$
\end{thm}

\begin{proof} Since $\cM$ is a monotonicity cone subequation, by Theorem \ref{thm:ZMP} it suffices to show that $\Omega$ admits a strict approximator $\psi$. It can be constructed as a quadratic polynomial
	$$
	\psi(x) := -c  + \frac{1}{2} |x - y|^2
	$$
	with $c > 0$ and $y \in \R^n$ chosen to ensure that 
	\begin{equation}\label{jn0}
	J^2_x \psi = \left( -c  + \frac{1}{2} |x - y|^2, (x - y), I \right) \in  \Int \, \cM \ \ \text{for every} \ x \in \Omega.
	\end{equation}
	Using the definition of the interior of $ \cM$ as given in \eqref{Int_GDRCone}, the condition \eqref{jn0} requires that:
	\begin{equation}\label{jn1}
		R I > |x-y| I  \ \ \text{for every} \ x \in \Omega; \ \ \text{that is,} \ \Omega \subset B_R(y)
	\end{equation}
	\begin{equation}\label{jn2}
	(x-y) \in \Int \, \cD \ \ \text{for every} \ x \in \Omega; \ \ \text{that is,} \ \Omega \subset y + \Int \, \cD;
	\end{equation}
	and finally that
	\begin{equation}\label{jn3}
	- c  + \frac{1}{2} |x - y|^2 < - \gamma |x - y| \ \ \text{for every} \ x \in \Omega.
	\end{equation}
	
	The hypothesis \eqref{ZMP_Domain1} is equivalent to \eqref{jn1} and \eqref{jn2}. By \eqref{jn1}, choosing $c > \frac{1}{2} R^2 + \gamma R$ ensures \eqref{jn3}, which completes the proof in the case $R < +\infty$. 
	
	 In the case $R = + \infty$, the condition \eqref{jn1} is automatic for each $y \in \R^n$ and since $\Omega$ is bounded, one can always pick $y \in \R^n$ such that \eqref{jn2} holds. Finally, choose any $R < +\infty$ with $\Omega \subset B_R(y)$ and choose $c > \frac{1}{2} R^2 + \gamma R$ to ensure \eqref{jn3}.
\end{proof}

We remark that Theorem \ref{thm:ZMP_for_M} applies to all of the cones $\cM$ in Definition \ref{defn:cone_zoo}. The cones in Parts I and II correspond to the special case $\cD = \R^n$. It is important to note that there are a priori restrictions on the domain when $R$ is finite which can be essential, as will be shown in Proposition \ref{prop:ZMP_failure} below, whose proof applies the following important fact concerning reduced  subequations to the dual $\wt{\cM}$ of a reduced monotonicity cone subequation $\cM$. As is standard in differential topology, {\em reduced} means that the jet variable $r \in \R$ is silent. As will be discussed in section \ref{sec:reductions}, this is equivalent to the following monotonicity property that strenghtens property (N):
\begin{equation}\label{M_reduced}
(r,p,A) \in \cM \ \ \Rightarrow \ \ (r + s, p, A) \in \cM \ \text{for every} \ s \in \R.
\end{equation}
Note that $\cM$ is reduced if and only if its dual $\wt{\cM}$ is reduced. 

\begin{lem}\label{lem:ReducedM_MP} Suppose that $\cG$ is a reduced subequation in the sense \eqref{M_reduced}. Then the zero maximum principle (ZMP) holds for $\cG$ on $\overline{\Omega}$ if and only if the maximum principle (MP) holds for $\cG$ on $\overline{\Omega}$; that is,
$$
\sup_{\Omega} u \leq \sup_{\partial \Omega} u \leqno{\rm{(MP)}}
$$
for each $u \in \USC(\overline{\Omega})$ which is $\cG$-subharmonic on $\Omega$.	
	\end{lem}

\begin{proof} The proof that (MP) implies (ZMP) on $\overline{\Omega}$ is immediate since if $z \in \USC(\overline{\Omega})$ is $\cG$-subharmonic in $\Omega$ with $z \leq 0$ on $\partial \Omega$, the (MP) applied to $u = z$ gives $z \leq 0$ on $\Omega$, as desired. Conversely, suppose that the (ZMP) holds and take any $u \in \USC(\overline{\Omega})$ which is $\cG$-subharmonic in $\Omega$. The function $z \in \USC(\overline{\Omega})$ defined by $z:= u - \sup_{\partial \Omega}u$ satisfies $z \leq 0$ on $\partial \Omega$ and is $\cG$-subharmonic in $\Omega$ (since $\cG$ is a reduced subequation). Hence $z \leq 0$ on $\Omega$ by the (ZMP). Note  that $u$ is $\cG$-subharmonic if and only if $u - c$ is $\cG$-subharmonic for any constant $c \in \R$. Now the result follows easily.	
\end{proof}

\begin{prop}[Failure of (ZMP) for $\wt{\cM}(R)$ on large balls]\label{prop:ZMP_failure} In $\R^n$ with $n \geq 2$, consider the reduced (convex) monotonicity cone subequation
\begin{equation}\label{MR_Cone}
	\cM(R) = \left\{ (r,p,A) \in \J^2: \ A \geq \frac{|p|}{R}I \right\} \ \ \text{with} \ R \in (0, + \infty).
\end{equation}
Then, the (ZMP) for $\wt{M}(R)$ \underline{fails} on $\overline{\Omega}$ with $\Omega = B_{R'}(0)$ for each $R' > R$.
\end{prop}

\begin{proof} Since $\wt{\cM}(R)$ is a reduced subequation, the (ZMP) for $\wt{\cM}(R)$ holds if and only if the (MP) holds for $\wt{\cM}(R)$ (see Lemma \ref{lem:ReducedM_MP}).  We exhibit a radial counterexample to the (MP) for $\wt{\cM}(R)$ on $B_{R'}(0)$ with $R' > R$ using the radial calculations as recorded in Remark \ref{rem:radial_calculus}. Consider 
\begin{equation}\label{ZMP_CE}
u(x) := \psi(|x|) \ \ \text{with} \ \ \psi(t):= t - \frac{t^2}{2R}.
\end{equation}		
Note that 
\begin{equation}\label{psi_derivs}
\psi'(t) = \frac{R - t}{R} \quad \text{and} \quad \psi^{\prime \prime}(t) = - \frac{1}{R}
\end{equation}	
and $\psi$ has its only critical point in $t=R$ with global maximum value $R/2$. Hence
\begin{equation}\label{max_u}
u(x) := |x| - \frac{1}{R} \frac{|x|^2}{2} \ \ \text{has its global maximum value on the sphere} \ |x| = R,
\end{equation}
and hence \underline{fails} to satisfy the (MP) on $\overline{\Omega}$ for any ball $\Omega = B_{R'}(0)$ with radius $R' > R$.

It remains only to show that $u$ is $\wt{\cM}(R)$-subharmonic on $\R^n$. It is easy to see that (use \eqref{GDRCD} with $\gamma = 0$ and $\cD = \R^n$):	
\begin{equation}\label{MR_Dual_Cone}
	\wt{\cM}(R) = \left\{ (r,p,A) \in \J^2: \ \lambda_{\rm max}(A) + \frac{|p|}{R} \geq 0 \right\}.
\end{equation}
The function $|x|$ does not have any upper test functions at $x = 0$, so neither does $|x|$ minus the quadratic $\frac{|x|^2}{2R}$. For $x \neq 0$, where $u$ is smooth, we show that its $2$-jet satisfies $J^2_xu \in \wt{\cM}(R)$ by using the radial calculus. For $t = |x| \neq 0$, using the radial formula \eqref{radial_calc_3} together with \eqref{psi_derivs}, we have
\begin{equation}\label{RC1}
p := Du(x) = \psi^{\prime}(|x|) \frac{x}{|x|}  = \frac{R - |x|}{R}\frac{x}{|x|} 
\end{equation}
and
\begin{equation}\label{RC2}
A := D^2 u(x) = \frac{\psi^{\prime}(|x|)}{|x|} P_{x^{\perp}} + \psi^{\prime \prime}(|x|) P_{x} = \left( - \frac{1}{R} + \frac{1}{|x|} \right) P_{x^{\perp}} - \frac{1}{R} P_{x}.
\end{equation}
Hence, for $n \geq 2$ we have $\lambda_{\rm max}(A) = - \frac{1}{R} + \frac{1}{|x|}$. In particular, if $0 < |x| \leq R$, $\lambda_{\rm max}(A) > 0$ and hence $\lambda_{\rm max}(A) + \frac{|p|}{R}> 0$. On the other hand, if $R < |x|$ then by \eqref{RC1} and \eqref{RC2} we have
$$
	\frac{|p|}{R} = \frac{|x| - R}{R} > 0\ \ \text{and} \ \ \lambda_{\rm max}(A) = - \frac{(|x| - R)}{R|x|} 
$$
so that 
	$$
	\lambda_{\rm max}(A) + \frac{|p|}{R} = \frac{|x| - R}{R} \left[ \frac{1}{R} - \frac{1}{|x|}  \right] > 0.
$$
\end{proof}

\section{The Comparison Principle for $\cM$-monotone Subequations}\label{sec:CP}

In this section, we examine the central question of the paper which is the validity of comparison (C) for $\F$ on $\overline{\Omega}$:  
\begin{equation}\label{CPV1}
u \leq w \ \text{on} \ \partial \Omega \ \ \Longrightarrow \ \ u \leq w \ \text{on} \ \Omega,
\end{equation}
or equivalently, the zero maximum principle for the {\em comparison differences}
\begin{equation}\label{CPD}
u - w \leq 0 \ \text{on} \ \partial \Omega \ \ \Longrightarrow \ \ u - w \leq 0 \ \text{on} \ \Omega,
\end{equation}
{\em if $u \in \USC(\overline{\Omega})$ and $w \in \LSC(\overline{\Omega})$ are $\F$-subharmonic and $\F$-superharmonic respectively on $\Omega$}. As noted in the discussion of \eqref{CP1_intro}-\eqref{CP2_intro}, if one uses Dirichlet duality and defines $v:= -w$, the comparison (C) is equivalent to the zero maximum principle for sums (ZMP for Sums)  on $\overline{\Omega}$:
\begin{equation}\label{CPV2}
u + v \leq 0 \ \text{on} \ \partial \Omega \ \ \Longrightarrow \ \ u + v \leq 0 \ \text{on} \ \Omega
\end{equation}
{\em  if $u$ and $v \in \USC(\overline{\Omega})$ are $\F$-subharmonic and $\wt{\F}$-subharmonic respectively on $\Omega$}.  This second form \eqref{CPV2} is the one which will be proved. Moreover, since $\wt{\wt{\F}} = \F$, the version \eqref{CPV2} of comparison immediately implies the following symmetry
\begin{equation}\label{CP_Duality}
\mbox{ comparison for $\F$ on $\overline{\Omega} \ \ \Leftrightarrow $\ \ comparison for $\wt{\F}$ on $\overline{\Omega}$.}
\end{equation} 

Our method is dependent on being able to find a subequation  $\cH$ with two properties:
\begin{equation}\label{CP_ansatz1}
\F(X) + \wt{\F}(X) \subset \cH(X), \ \ \text{for every open set} \ X \subset \R^n
\end{equation}
and $\cH(X)$ satisfying the zero mazimum principle (ZMP); that is,
\begin{equation}\label{CP_ansatz2}
h\leq 0 \ \text{on} \ \partial \Omega \ \ \Longrightarrow \ \ h \leq 0 \ \text{on} \ \Omega, \ \ \forall \, \Omega \subset \subset X, \ h \in \USC(\overline{\Omega}) \cap \cH(\Omega). 
\end{equation}
The first step is to find $\cH$ which satisfies \eqref{CP_ansatz1} and the second step is to show that \eqref{CP_ansatz2} holds. We will discover $\cH$ infinitesimally, which reduces to montotonicity by using duality. At the infinitesimal (2-jet) level, $\cH$ must be the dual $\wt{\cM}$ of a monotonicity set $\cM$ for $\F$. This is done in Lemma \ref{lem:CP_Jets} below,  but first we prove that a {\em subharmonic addition} such as \eqref{CP_ansatz1} is implied by its infinitesimal version, {\em jet addition}. 

\begin{thm}[The Subharmonic Addition Theorem]\label{thm:SAT}
For arbitrary subequation constraint sets $\F, \cG$ and $\cH$ of $\J^2$, 
\begin{equation}\label{jet_addition}
\text{(Jet Addition)} \quad \F + \cG \subset \cH
\end{equation}
implies
\begin{equation}\label{subharmonic_addition}
\text{(Subharmonic Addition)} \quad \F(X) + \cG(X) \subset \cH(X)
\end{equation}
for the subharmonics on each open set $X \subset \R^n$.
	\end{thm}
We include the complete proof of this constant coefficient result for the reader's convenience.

\begin{proof} 
	Given $u \in \F(X)$ and $v \in \G(X)$, it suffices to show that about each $x_0 \in X$ there is an open ball $B = B_{\rho}(x_0) \subset \subset \Omega$ such that $u + v \in \cH(B)$. Since $u,v \in \USC(X)$, they are bounded from above on any compact subset of $X$. Both $u$ and $v$ can be written as a decreasing limit of quasi-convex {\em sup convolution approximations}, if $u$ and $v$ are also locally bounded from below.

	To this end, by shrinking $B$ if necessary, since $\F$ and $\G$ satisfy conditions (T) and (P), about each $x_0 \in X$ one can find quadratic functions $\varphi, \psi$ which are $\F, \G$-subharmonic and bounded on a common $B= B_{\rho}(x_0)$ (see Remark \ref{rem:LBS}). The sequences of functions
	\begin{equation}\label{AT1}
	u_m:= \max\{u, \varphi - m\} \in \F(B) \ \ \text{and} \ \ v_m:= \max\{u, \psi - m\} \in \G(B),  \ \ m \in \N
	\end{equation}
	are bounded from above and below on $B$. The $\F,\G$-subharmonicity claims use the Maximum Property (B) of Proposition \ref{prop:B} and the Negativity Property (N) of $\F, \G$ applied to $\varphi, \psi$ which are $C^2$. 
	
	Using these truncating approximations, in the proof of $u + v \in \cH(B)$ we may assume that $u,v$ are bounded on $B$; that is, there exists $N > 0$ such that
	\begin{equation}\label{AT2}
	|u(x)|, |v(x)| \leq N,  \ \ \forall \ x \in B.
	\end{equation}
	Indeed, if the sum of the truncations in \eqref{AT1} satisfies $u_m + v_m \in \cH(B)$ for each $m \in \N$, the Decreasing Sequence Property (E) of Proposition \ref{prop:B} shows that the limit satisfies $u + v \in \cH(B)$, as desired.
	
	Now, assuming \eqref{AT2}, one passes to the sup convolutions
	\begin{equation}\label{AT3}
	u^{\veps}(x) := \sup_{y \in B} \left\{ u(y) - \frac{1}{\veps} |x-y|^2 \right\}, \ \ x \in B, \veps > 0 ,
	\end{equation}
	and similarly for $v^{\veps}$. One has $u^{\veps}, v^{\veps}$ are $2/\veps$-quasiconvex and decrease to $u,v$ (where one uses that $u,v$ are bounded below for the limit statement and hence the need for the truncation \eqref{AT1}). Moreover one has that
	\begin{equation}\label{AT4}
	u^{\veps} \in \F(B_{\delta}) \quad \text{and} \quad v^{\veps} \in \G(B_{\delta})
	\end{equation}
	where $B_{\delta}:= \{ x \in B: \ {\rm dist}(x, \partial B) > \delta \}$ and $\delta = \sqrt{2 \veps N}$. One uses the Translation Property (D) and the Families Locally Bounded Above Property (F) of Proposition \ref{prop:B}.
	
	By Alexandroff's Theorem, \eqref{AT4} and the jet addition hypothesis ($\F + \G \subset \cH$), one has that the quasi-convex $u^{\veps}, v^{\veps}$ satisfy
	\begin{equation}\label{AT5}
	J^2_x(u^{\veps} + v^{\veps}) \in \cH \ \ \text{for almost every} \ x \in B_{\delta}.
	\end{equation}
	For quasi-convex functions, the statement \eqref{AT5} yields $u^{\veps} + v^{\veps} \in \cH(B_{\delta})$ by the Almost Everywhere Theorem of Lemma \ref{lem:AET}. The desired conclusion follows from the Decreasing Sequence Property (E) of Proposition \ref{prop:B} by considering the limit along a sequence corresponding to $\veps = \veps_j \to 0^+$.
\end{proof}

\begin{rem}[Subharmonic Addition for locally quasi-convex functions]\label{rem:SAT_LQC} For locally quasi-convex functions, Theorem \ref{thm:SAT} extends from constant coefficient subequations to arbitary subequations, and hence from any open set $X$ in Euclidian space to a manifold $X$. Namely
$$
	\F + \G \subset \cH \ \Rightarrow \ u + v \in \cH(X), \ \forall\, u \in \F(X), v \in \G(X) \ \text{locally quasi-convex}. 
$$
This is immediate from the proof of Theorem \ref{thm:SAT} above, since the sup convolution step is unneccessary if $u$ and $v$ are assumed to be locally quasi-convex, and the other steps do not use translation invariance.	
	\end{rem}

In the special case  $\G:= \wt{\F}$, the Subharmonic Addition Theorem concludes the desired subharmonic addition \eqref{CP_ansatz1} stating that:
\begin{equation}\label{JAH}
\text{if} \ \F + \wt{\F} \subset \cH \ \ \text{then} \ \ \F(X) + \wt{\F}(X) \subset \cH(X) .
\end{equation}
Next, using duality, we reduce the jet addition hypothesis $\F + \wt{\F} \subset \cH$ to a monotonicity hypotheis $\F + \wt{\cH} \subset \F$ on $\F$. This is a key step in the basic method of this paper.

\begin{lem}[Jet addition, duality and monotonicity]\label{lem:CP_Jets} For any two subequation constraint sets $\F, \cH \subset \J^2$, one has
	\begin{equation}\label{AlgAT}
	\F + \wt{\F} \subset \cH \ \iff \ \F + \wt{\cH} \subset \F.
	\end{equation}
\end{lem}

\begin{proof} One sees that for $J = (r,p,A) \in \J^2$ one has
	\begin{equation}\label{CPJ1}
	J + \wt{\F} \subset \cH \ \iff \ J + \wt{\cH} \subset \F
	\end{equation}
	since $\wt{\F - J} = \wt{\F} + J \subset \cH  \iff \wt{\cH} \subset \F-J \iff  J + \wt{\cH} \subset \F$, by the the  elementary properties (1) and (3) of the Dirichlet dual in  Proposition \ref{prop:duality}. Taking all $J \in \F$ in \eqref{CPJ1} yields the lemma.
\end{proof}

Consequently, finding a subequation $\cH$ with the desired jet addition property  $\F + \wt{\F} \subset \cH$ requires that $\cH$ equals the dual $\wt{\cM}$ of a monotonicity subequation $\cM$ for $\F$; that is, satisfying $\F + \cM \subset \F$. We can summarize as follows.

\begin{thm}[Suharmonic addition, duality and monotonicity]\label{thm:SAT_MD} Suppose that $\cM \subset \J^2$ is a monotonicity cone subequation and that $\F \subset \J^2$ is an $\cM$-monotone subequation constraint set. Then, for every open set $X \subset \R^n$, one has 
\begin{equation}\label{SAT_MD}
\F(X) + \wt{\F}(X) \subset \cMt(X).
\end{equation}
\end{thm}

Combining Theorem \ref{thm:SAT_MD} with Theorem \ref{thm:ZMP} yields our general method for proving comparison.

\begin{thm}[The General Comparison Theorem]\label{thm:CP_general}
Suppose that a subequation $\F \subset \J^2$ is $\cM$-monotone for some convex cone subequation $\cM$. If $\cM$ admits a strict approximator $\psi$ on $\Omega$ (in the sense of Definition \ref{defn:strict_approx}), then comparison (C) holds for $\F$ on $\overline{\Omega}$.
	\end{thm}

\begin{proof}
	Suppose that $u$ and $v \in \USC(\overline{\Omega})$ are $\F$-subharmonic and $\wt{\F}$-subharmonic respectively on $\Omega$. Taking $X = \Omega$ in Theorem \ref{thm:SAT_MD}, we have $z:= u + v \in \wt{\cM}(\Omega)$, and hence $z \in \USC(\overline{\Omega}) \cap \wt{\cM}(\Omega)$. By Theorem \ref{thm:ZMP}, since $\cM$ has a strict approximator on $\Omega$, such a $z$ satisfies the (ZMP); that is,
	\begin{equation}
	u + v \leq 0 \ \ \text{on} \ \ \partial \Omega \ \ \Rightarrow \ \ u + v \leq 0 \ \ \text{on} \ \ \Omega.
	\end{equation}
	This is precisely \eqref{CPV2}, which as noted above is one way of formulating the comparison principle (C). 
\end{proof}

We are now ready for the main result. 

\begin{thm}[The Fundamental Family Comparison Theorem]\label{thm:comparison} Suppose that $\F \subset \J^2$ is an $\cM$-monotone subequation constraint set where $\cM \subset \J^2$ is a monotonicity cone subequation. Given $u, v \in \USC(\overline{\Omega})$ which are $\F, \wt{\F}$-subharmonic on $\Omega$ one has
	\begin{equation*}\label{CP3}
	{\rm (C)} \ \ \ \ u + v \leq 0 \ \text{on} \ \partial \Omega \ \ \Longrightarrow \ \ u + v \leq 0 \ \text{on} \ \Omega
	\end{equation*}

1) for each $\Omega$ contained in a translate of the truncated cone $\cD_R := \cD \cap B_R(0)$ if $\cM$ contains one of the cones  
\begin{equation}\label{MGDR_Cone}
\cM(\gamma, \cD, R) := \cM(\gamma) \cap \cM(\cD) \cap \cM(R) \ \ \text{with $R$ finite}
\end{equation}

\noindent and

2) for arbitrary $\Omega \subset \subset \R^n$ if $\cM$ contains one of the cones
\begin{equation}\label{MGDP_Cone}
\cM(\gamma, \cD, \cP) := \cM(\gamma) \cap \cM(\cD) \cap \cM(\cP) := \cM^{\prime}(\gamma, \cD) \times \cP.
\end{equation}

Moreover, by the Fundamental Family Theorem  \ref{thm:fundamental}, every monotonicity cone subequation $\cM$ contains a cone of the type \eqref{MGDR_Cone} so case 1) always holds. Finally, any cone of the type \eqref{MGDP_Cone} satisfes
$\cM(\gamma, \cD, \cP) \supset \cM(\gamma, \cD, R)$ for each $R$ finite, and hence case 1) implies case 2).
\end{thm}

\begin{proof} 	Suppose that $u$ and $v \in \USC(\overline{\Omega})$ are $\F$-subharmonic and $\wt{\F}$-subharmonic respectively on $\Omega$. Again, by taking $X = \Omega$ in Theorem \ref{thm:SAT_MD}, we have $z:=u + v \in \wt{\cM}(\Omega)$ and hence  $z \in \USC(\overline{\Omega}) \cap \cM(\Omega)$. Therefore it suffices to verify the (ZMP) for such $z$. In the cases 1) and 2), this is exactly what Theorem \ref{thm:ZMP_for_M} states in the cases $R$ finite and $R = +\infty$, respectively.
\end{proof}

The size of the domain $\Omega$ in Theorem \ref {thm:comparison} 1) is sharp for the subequation $\F = \cM(R)$ when $R$ is finite. 

\begin{exe}\label{CP_failure} With $n \geq 2$ and $R$ finite, comparison fails for $\F= \cM(R)$ on $\overline{\Omega}$ with $\Omega = B_{R'}(0)$ for each $R' > R$. Indeed, as noted in \eqref{CP_Duality}, one has
\begin{equation*}
\mbox{ comparison for $\F = \cM(R)$ on $\overline{\Omega} \ \ \Leftrightarrow $\ \ comparison for $\wt{\F} = \wt{\cM}(R)$ on $\overline{\Omega}$.}
\end{equation*} 	
By Proposition \ref{prop:ZMP_failure} we know that the (ZMP) fails for $\wt{\cM}(R)$ on $\overline{\Omega}$ with $\Omega = B_{R'}(0)$ for each $R' > R$, which completes the claim.
\end{exe}

A larger family of subequations with maximal monotonicity $\cM(R)$ and failure of comparison on balls of radius $R' > R$ will be presented in Proposition \ref{exe:CE1_CP}.

On the other hand, the fundamental nature of the family of $\cM(\gamma, \cD, R)$ cones gives rise to the local validity of the comparison principle for subequations with this minimal monotonicty.

\begin{thm}[Local Comparison]\label{thm:LC} 	If $\F$ is a subequation which is $\cM$-monotone for some monotonicity cone subequation $\cM$, then the comparison principle holds locally on $\R^n$; in particular, there exists $\rho > 0$ which depends on $\cM$ such that for all domains $\Omega \subset B_{\rho}(x_0)$ with  $x_0 \in \R^n$ arbitrary
	\begin{equation*}\label{CP3}
	{\rm (C)} \ \ \ \ u + v \leq 0 \ \text{on} \ \partial \Omega \ \ \Longrightarrow \ \ u + v \leq 0 \ \text{on} \ \Omega 
	\end{equation*}
for each pair $u \in \USC(\overline{\Omega}) \cap \F(\Omega)$ and $v\in \USC(\overline{\Omega}) \cap \wt{\F}(\Omega)$.
\end{thm}
\begin{proof} Since $\F$ is $\cM$-monotone for a monotonicity cone subequation, by Theorem \ref{thm:fundamental} there exists a cone $\cM(\gamma, \cD, R)$ in the fundamental family with $\cM(\gamma, \cD, R) \subset \cM$ and hence $\F$ is $\cM(\gamma, \cD, R)$-monotone. Then by Theorem \ref{thm:CP_general}, comparison holds on all domains $\Omega$ contained in a translate of the truncated cone $\cD_R := \cD \cap B_R(0)$. Clearly there exists $\rho > 0$ such that $B_{\rho}(y_0) \subset \cD_R$ for some $y_0 \in \cD_R$ and hence $\Omega \subset B_{\rho}(x_0)$ is contained in a translation of $\cD_R \supset B_{\rho}(y_0)$.
\end{proof}

We give one final comment in this section. In the proof of Lemma \ref{lem:AT_Cones} (and hence for the proof of Theorem \ref{thm:comparison}), one needed the local existence of bounded $\F, \wt{\F}$ subharmonics on potentially small balls. However, one can also find subharmonics on potentially large balls in various ways. We record this observation for future reference.

\begin{rem}\label{rem:BS} If one knows the existence of particular 2-jets $(r_1, p_1, A_1) \in \F$, the construction of explicit bounded and smooth subharmonics simplifies considerably. For example, if there exists $(r_1,0,0) \in \F$, then any constant function $\varphi \equiv r_0$ with $r_0 \leq r_1$ will do by the negativity property (N). Moreover, if $(r_1, p_1, 0) \in \F$, then any affine function 
	$\varphi(x):= r_0 +] \langle p_0, x-x_0 \rangle$ will be $\F$ subharmonic on $B_{\rho}(x_0)$ if
	$$
	p_0:= p_1 \quad \text{and} \quad r_0 - r_1 + \rho |p_0| \leq 0.
	$$
	If one has neither of these two possibilities, about each $x_0 \in \Omega$ one can use the $\cM$-monotonicity of $\F$ to construct quadratic polynomials
	\begin{equation}\label{BS1}
	\varphi(x) := r_0 + \langle p_0, x - x_0 \rangle + \frac{\lambda_0}{2} |x - x_0|^2
	\end{equation}
	with $r_0 < 0 < \lambda_0, p_0 \in \R^n$ chosen to ensure that $J^2_x \varphi \in \F$ for each $x \in B_{\rho}(x_0) \subset \Omega$, for some $\rho > 0$. Starting from any $(r_1, p_1, A_1) \in \F$, one uses property (P) to show that $(r_1, p_1, \lambda_1 I) \in \F$ for $\lambda_1$ large enough. Using the $\cM$-monotonicity it suffices to exhibit  $(r_0, p_0, \lambda_0) \in \R \times \R^n \times \R$ and $\rho > 0$ such that for $x \in B_{\rho}(x_0)$ and 
	\begin{equation}\label{BS2}
	J^2_x \varphi = \left( \varphi(x), p_0 + \lambda_0 (x - x_0), \lambda_0 I \right), := (r_1, p_1, \lambda_1 I) + (r(x), p(x), (\lambda_0 - \lambda_1)I)
	\end{equation}
	one has $(r(x), p(x), (\lambda_0 - \lambda_1)I) \in \cM$, which requires:
	\begin{equation}\label{BS3}
	r(x) \leq -\gamma |p(x)|; \ \ p(x) \in \cD; \ \ R(\lambda_0 - \lambda_1) I \geq |p(x)| I, 
	\end{equation}
	where the last condition in \eqref{BS3} holds for every $\lambda_0 \geq \lambda_1$ in the case $R= + \infty$. The reader can verify easily that for a suitable radius $\rho$ one can find $(r_0, p_0, \lambda_0)$ for which \eqref{BS3} holds.
\end{rem}

\section{Comparison on Arbitrary Domains by Additional Monotonicity}\label{sec:improvemnts}

By Theorem \ref{thm:fundamental}, any subequation constraint set $\F$ which is $\cM$-monotone for some monotonicity cone subequation $\cM$ must have at least the monotonicity of one of the monotoncity cone subequations $\cM(\gamma, \cD, R) \subset \cM$ belonging to our fundamental family. If $R = +\infty$, then comparison holds for $\F$ on arbitrary bounded domains by Theorem \ref{thm:comparison}. If $R < + \infty$, then (again by Theorem \ref{thm:comparison}) comparison holds for $\F$ on domains $\Omega$ for which a translate of $\Omega$ is contained in the truncated cone $\cD \cap B_R(0)$. This result is sharp if the maximal monotonicity cone subequation $\cM_{\F}$ for $\F$ is $\cM(\gamma, \cD, R)$ with $R$ finite (see Proposition \ref{prop:CE1_CP} for an example).  However, this leaves room for improvement if $\cM_{\F}$ is large enough, and this is the subject of the present section. 

Comparison may still hold for all domains $\Omega \subset \subset \R^n$. We explore this possibility here, continuing with our monotonicity technique, looking for larger, not smaller, monotonicity cone subequations, and highlight two examples.
These two examples contain
\begin{equation}\label{def:M(R)}
	\cM(R) = \left \{ (r,p,A) \in \J^2: \ \lambda_{\rm min} (A) \geq \frac{|p|}{R} \right\} .
	\end{equation}

\begin{defn}\label{defn:newcones}
	Fix $R \in (0, +\infty)$. Define 
\begin{equation}\label{MsubR}
 	\cM_R := \left\{ (r,p,A) \in \J^2: \ A \geq 0 \ \text{and} \ \left( \lambda_1(A) \cdots \lambda_n(A) \right)^{1/n} \geq \frac{|p|}{R} \right\}
 \end{equation}
 and
 \begin{equation}\label{MsuperR}
 \cM^R := \left\{ (r,p,A) \in \J^2: \ A \geq 0 \ \text{and} \ \langle A e, e \rangle \geq \frac{|\langle p, e \rangle|}{R}, \ \forall \,   e \in \R^n \ \text{with} \ |e| = 1 \right\}.
 \end{equation}
	\end{defn}

These variants of $\cM(R)$ are indeed convex cone subequations and are all larger than $\cM(R)$ in a precise sense.

\begin{prop}\label{prop:newcones}
For $R \in (0, +\infty)$ fixed, 
\begin{equation}\label{newcones1}
\mbox{ $\cM_R$ and $\cM^R$ are convex cone subequations}
\end{equation}
and 
\begin{equation}\label{newcones2}
\mbox{ $\cM_R$ and $\cM^R$ contain $\cM(R^{\prime}) \ \ \Longleftrightarrow \ \ R^{\prime} \leq R$.}
\end{equation}
\end{prop}

\begin{proof}
$\cM_R$ and $\cM^R$ are convex cones since each is defined by an inequality of the form $h(p,A) \geq 0$ where $h$ is a concave function on its domain. Property (N) is automatic as the variable $r$ is silent in both cases, property (P) follows since each $h(p,A)$ is increasing in $A$ on its domain. Property (T) is satisfied since each $\cM$ is a convex cone with  $\Int \, \cM$ non-empty. 
\end{proof}

Comparison always holds for all bounded domains for subequations $\F$ which are $\cM$-monotone if $\cM$ contains either $\cM_R$ or $\cM^R$. 

\begin{thm}\label{thm:CP_improvents}
	Suppose that $\F \subset \J^2$ is a subsequation which is $\cM$-monotone. If $\cM$ contains either $\cM_R$ or $\cM^R$ for some $R$, then comparison holds for $\F$ on all bounded domains $\Omega \subset \R^n$. 
	\end{thm}

\begin{proof}
	Since $\Omega$ is bounded and (by translation) may be assumed to satisfy $0 \notin \Omega$, by Theorem \ref{thm:CP_general}, it suffices to establish the following lemma. 
\end{proof}

\begin{lem}[Radial polynomial approximators]\label{lem:strict_improvement} Given $\rho > 0$ there exists $m \in \N$ such that
	\begin{equation}\label{radial_poly}
	\psi(x) := \frac{|x|^{m+1}}{m+1} 
	\end{equation}
	defines a strict approximator on $B_{\rho}(0) \setminus \{0\}$ for $\cM_R$ and for $\cM^R$.
	\end{lem}

\begin{proof} Making use of the radial calculation of Example 2 of Remark \ref{rem:radial_calculus},  $\psi$ has reduced 2-jet 
\begin{eqnarray}\label{radial1}
(p,A):= (D\psi(x), D_x^2 \psi)  &=& |x|^{m-1} \left( x, I + (m-1)P_x \right) \\ &=& |x|^{m-1} \left( x, P_{x^{\perp}} + mP_x \right) ,
\end{eqnarray}
where we recall that $P_x$ is the othogonal projection onto the line $[x]$ through $x \in \R^n \setminus \{0\}$ and $P_{x^{\perp}} = I - P_x$ is the orthogonal projection on $[x]^{\perp}$. Hence the claim that $\psi$ is a strict $\cM$-approximator on $B_{\rho}(0) \setminus \{0\}$ is equivalent to the claim
\begin{equation}\label{radial2}
\left(x,  I + (m-1)P_{x} \right) \in \Int \, \cM \ \ \text{for every} \ x \ \text{with} \ 0 < |x| < \rho.
\end{equation}
Now we verify \eqref{radial2} for $\cM_R$ and $\cM^R$ Note that $I + (m-1)P_{x}= P_{x^{\perp}} + mP_x$ has $n-1$ eigenvalues equal to $1$ and one eigenvalue equal to $m$. 

For $\cM = \cM_R$, the claim \eqref{radial2} becomes (since $ I + (m-1) P_{x} > 0$)
$$
 m^{1/n} - \frac{|x|}{R} > 0,
$$
which holds if and only if $m > (\rho/R)^n$.

For $\cM = \cM^R$, with arbitrary $e \in \R^n$ satisfying $|e| = 1$, the claim \eqref{radial2} becomes
\begin{equation}\label{radial3}
\langle \left( I + (m-1) P_{x} \right) e, e \rangle - \frac{|\langle x, e \rangle|}{R} > 0.
\end{equation}
A simple calculation gives
$$
\langle \left( I + (m-1) P_{x} \right) e, e \rangle = 1 + (m-1) \langle x/|x|, e \rangle^2.
$$
and hence the needed \eqref{radial3} can be written as
\begin{equation}\label{radial4}
g(t):= 1 - 
|x| t + (m-1)t^2 > 0 \ \ \text{for} \ t:= | \langle x/|x|, e \rangle | \geq 0.
\end{equation}
The quadratic polynomial $g$ takes on its mimumum at $t_0:= \frac{|x|}{2R(m-1)}$ with minimum value
$$
	g(t_0) = 1 - \frac{|x|^2}{4R^2 (m-1)} > 1 - \frac{\rho^2}{4R^2 (k-1)}, \ \ \forall \, x \ \text{with} \ |x| < \rho.  
$$
Taking $m$ sufficiently large gives $g(t_0) > 0$ and hence \eqref{radial4} 
\end{proof}

We remark that in the case of $\cM = \cM(R)$ one has
$$
\lambda_{\rm min} \left(I + (m-1)P_{x}\right) = 1 > \frac{|x|}{R}
$$
only if $\rho \leq R$ which leads to the restriction on domain size for this case.

Theorem \ref{thm:CP_improvents} easily extends to non-reduced subequations (where the variable $r$ is not silent) as follows. Recall that for $0 \leq \gamma < \infty$

\begin{equation}\label{Mgamma_recall}
	\cM(\gamma) := \{ (r,p,A) \in \J^2: \ r \leq -\gamma|p| \}.
\end{equation}

\begin{thm}\label{thm:CP_improvents2}
	Suppose that $\F \subset \J^2$ is a subsequation which is $\cM$-monotone. If $\cM$ contains either $\cM(\gamma) \cap \cM_R$ or $\cM(\gamma) \cap \cM^R$ for some $\gamma, R$, then comparison holds for $\F$ on all bounded domains $\Omega \subset \subset \R^n$. 
\end{thm}

\begin{proof} First choose $m \in \N$ as in Lemma \ref{lem:strict_improvement}. Replace $\psi$ in \eqref{radial_poly} by the radial polynomial
	\begin{equation}\label{radial_poly2}
	\psi(x) := \frac{|x|^{m+1}}{m+1} - C
	\end{equation}
with $C > 0$ large to be determined. The reduced 2-jet remains unchanged but $r:=  \frac{|x|^{m+1}}{m+1} - c$. Since $|p| = |x|^m$, we have that 
$$ 
r \leq -\gamma |p| \ \ \iff \ \  \frac{|x|^{m+1}}{m+1} + \gamma |x!^m \leq C.
$$
This holds on the the set where $|x| \leq \rho$ if $C \geq \frac{|\rho|^{m+1}}{m+1} + \gamma |\rho!^m$.
\end{proof}

For $\cM =  \cM^-_{\lambda, \Lambda, R}$, the claim \eqref{radial2} becomes
\begin{equation}\label{radial5}
\lambda (n -1 + m) - \frac{\lambda n |x|}{R} > 0.
\end{equation}
Since $|x| < \rho$, one has \eqref{radial5} if and only if $m >1 +  n \left( \frac{\rho}{R} - 1 \right)$.

For $\cM = \cM(R)_{\delta}$, the claim \eqref{radial2} becomes
\begin{equation}\label{radial6} 
1 + \delta (n - 1 + m) - \frac{|x|}{R} > 0.
\end{equation}
Since $|x| < \rho$, one has \eqref{radial6} if and only if $m > 1 - n + \frac{\rho - R}{\delta R}$.

For $\lambda, \Lambda \in \R$ with $0 \leq \lambda \leq \Lambda < +\infty$, define 
\begin{equation}\label{MpucciR}
\cM^-_{\lambda, \Lambda, R} := \left\{ (r,p,A) \in \J^2: \ \lambda \, {\rm tr} \, A^+ + \Lambda \, {\rm tr} \, A^- \geq \frac{\lambda n |p|}{R} \right\}.
\end{equation}
For $\delta \in \R$ with $\delta > 0$, define
\begin{equation}\label{MdeltaR}
\cM(R)_{\delta} := \left\{ (r,p,A) \in \J^2: \lambda_{\rm min} (A) + \delta \, {\rm tr} \, A \geq \frac{|p|}{R}  \right\}.
\end{equation}  

\section{Failure of Comparison with Insufficient Maximal Monotonicity}\label{sec:limitations}

In this section, we give some examples of subequation constraint sets $\F$ for which comparison fails to hold on a family of bounded domains $\Omega$. Necessarily, this failure requires that the maximal monotonicity cone $\F_{\cM}$ for $\F$ does contain one of the elements of our fundamental family with $R = +\infty$ nor any of the cones discussed in the previous section on additional monotonicity. We focus on two such examples. The first shows (as claimed in the introduction of section \ref{sec:improvemnts}) that Theorem \ref{thm:comparison} is sharp in the case when $R$ is finite; that is, $R$ gives an upper bound on the diameter of $\Omega$ for which comparison holds. The second, shows just how bad the situation can be. Comparison can fail on arbitrarily small balls. 

\subsection{Finite $R$ and failure of comparison on large domains}

We begin with a simple family of examples which illustrates the sharpness of Theorem \ref{thm:comparison} on the comparison principle in the case when $R$ is finite. 
 
\begin{exe}\label{exe:CE1_CP}
	In dimension $n \geq 2$, with $k \in \{ 1, \ldots, n \}$ and $R \in (0, +\infty)$ consider the subequation constraint sets
	\begin{equation}\label{E4A}
	\F^{\pm}_{k, R} := \{ (r,p,A) \in \J^2: \lambda_k(A) \pm \frac{|p|}{R} \geq 0 \},
	\end{equation}
	where $\lambda_1(A) \leq \cdots \leq \lambda_k(A)$ are the ordered eigenvalues of $A \in \cS(n)$. defining the subharmonics for the operators $F^{\pm}_{k, R}(Du,D^2u) = \lambda_k(D^2u) \pm \frac{1}{R}|Du|$. 
One easily checks that each $\F^{\pm}_{k, R}$ is a subequation constraint set and that the following duality relations hold:
	\begin{equation}\label{FkR_duality}
	\wt{\F}^{+}_{k, R} = \F^{-}_{n + k -1, R} \quad \text{and} \quad  	\wt{\F}^{-}_{k, R} = \F^{+}_{n + k -1, R}. 
	\end{equation}
Notice that two members of the family are cone subequation cones that we have seen before, namely
	\begin{equation}\label{FkR_M}
\F^{-}_{1, R} = \cM(R) \quad \text{and} \quad  	\F^{+}_{n, R} = \wt{\cM}(R), 
\end{equation}
where we note that only the first cone $\cM(R)$ is convex.
		\end{exe}
	
\begin{prop}\label{prop:CE1_CP} For the family of subequations $\F^{\pm}_{k, R}$ in Example \ref{exe:CE1_CP}, one has the following statements. 
	\begin{itemize}
		\item[(a)]  For each $k$ and $R$, the maximal monotonicity cone $\cM_{\F^{\pm}_{k, R}}$ of $\F^{\pm}_{k, R}$ equals
			$$
		\cM(R) := \left\{ (s,q,B) \in \J^2: \ B \geq \frac{|q|}{R}I \right\} = \left\{ (s,q,B) \in \J^2: \lambda_1(B) - \frac{|q|}{R} \geq 0 \right\}.
		$$
		Consequently, comparison \underline{holds} for each $\F^{\pm}_{k, R}$ on $\overline{\Omega}$ for every domain $\Omega$ contained in a ball $B_{R}$ of radius $R$.
		\item[(b)] For each $k = 2, \ldots n$, comparison \underline{fails} for $\F^{\pm}_{k,R}$ on any ball $B_{R'}$ with radius $R' >R$. 
	\end{itemize}
	\end{prop}

\begin{proof} It suffices to consider the family of subequations $\F^{+}_{k,R}$ since each $\F^{-}_{k,R}$ is dual to $\F^{+}_{n -k + 1,R}$. This is because dual suubequations have the same maximal monotonicity cone (see Proposition \ref{prop:MMC2} (b)) and the comparison principle holds for a subequation $\F$ if and only if it holds for its dual subequation $\wt{\F}$. 
	
We begin by showing that each $\F^{+}_{k,R}$ is $\cM(R)$-monotone.
Given any $(r,p,A) \in \F^{+}_{k, R}$ and any $(s,q,B) \in \cM(R)$, making use of the {\em dual Weyl inequality}
$$
	\lambda_k(A + B) \geq \lambda_k(A) + \lambda_1(B), \ \ \text{for each} \ A, B \in \Symn, k = 1, \ldots, n,
$$
the triangle inequality on $\R^n$, and using $\lambda_1(B) \geq |q|/R$ one has
$$
	\lambda_k(A + B) + \frac{|p + q|}{R}  \geq \lambda_k(A) + \lambda_1(B) + \frac{|p|}{R} - \frac{|q|}{R} \geq 0.
$$
Hence we have $\cM(R) \subset \cM_{\F^{+}_{k,R}}$, the maximal monotonicity cone (as defined in Definition \ref{defn:MMC}). It remains to check the reverse inclusion; that is,
\begin{equation}\label{MMC1}
(r+s, p+q,A+B) \in \F^{+}_{k,R}, \ \ \forall \, (r,p,A) \in \F^{+}_{k,R} \ \ \Rightarrow \ \ (s,q,B) \in \cM(R).
\end{equation}
Since both $\F^{+}_{k,R}$ and $\cM(R)$ are reduced subequations, the condition \eqref{MMC1} can be written as
\begin{equation}\label{MMC2}
\lambda_k \left(A + B +  \frac{|p + q|}{R} I \right) \geq 0, \  \ \forall \, (p,A): \ \lambda_k \left( A + \frac{|p|}{R} I \right) \geq 0 \ \ \Rightarrow \ \ B \geq \frac{|q|}{R}I.
\end{equation}
We will use the fact that
\begin{equation}\label{MMC3}
\lambda_k \left(A + P \right) \geq 0, \ \ \forall \, A \in \cS(n): \ \lambda_k \left( A \right) \geq 0  \ \ \Rightarrow \ \ P \geq 0;
\end{equation}
that is, the maximal monotonicity cone for $\{ A \in \cS(n): \ \lambda_k(A) \geq 0 \}$ is $\cP$. Let $(q,B) \in \R^n \times \cS(n)$ satisfy the hypothesis in \eqref{MMC2} which is equivalent to 
\begin{equation}\label{MMC4}
\lambda_k \left(D + B + \frac{|p + q| - |p|}{R} I \right) \geq 0, \ \ \forall \, (p,D): \ \lambda_k \left( D \right) \geq 0  \ \ \Rightarrow \ \ B \geq \frac{|q|}{R}I.
\end{equation}
By applying \eqref{MMC3} to the pair $(q, B)$ satisfying \eqref{MMC4}, one finds
$$
B + \frac{|p + q| - |p|}{R} I \geq 0 \ \ \text{for every} \ p \in \R^n,
$$
which yields $B - \frac{|q|}{R} I \geq 0$ with $p = -2q$. This completes the proof that $\cM(R)$ is the maximal monotonicity cone for each $\F^{+}_{k,R}$.

It then follows from Theorem \ref{thm:comparison} that comparison holds for all domains $\Omega$ contained in a ball  $B_{R}$ of radius $R$.
	
Next we note that the same radial counterexample $u(x):= |x| - \frac{|x|^2}{2R}$ to the (MP) for $\wt{\cM}(R) = \F_{n,R}^+$ on balls $B_{R'}$ with radius $R'  > R$ (see Proposition \ref{prop:ZMP_failure}) is a counterexample to the (MP) for the reduced subequation $\F^{+}_{k,R}$ on $B_{R'}$ if $k \geq 2$. This is because $D^2u(x)$ has $n-1$ eigenvalues equal to $-\frac{1}{R} + \frac{1}{|x|}$ which are all greater than the remaining eigenvalue $-\frac{1}{R}$ (see formula \eqref{RC2}). Hence with $(p,A):= (Du(x), D^2u(x))$, for each $x \neq 0$:
\begin{equation}\label{kn_subeq}
	\lambda_k(A) + \frac{|p|}{R} = \lambda_n(A) + \frac{|p|}{R} \ \ \text{if} \ k \geq 2,
\end{equation}
which shows that the $\wt{\cM}(R) = \F_{n,R}^+$-subharmonic function $u$ is also $\F^{+}_{k,R}$-subharmonic on $\R^n \setminus \{0\}$ for $k \geq 2$. Recall that $u$ is trivially $\wt{\cM}(R)$-subhamonic at the origin because there are no upper test jets at the origin, and hence the same claim for $\F^{+}_{k,R}$.

Finally, since $\F^{+}_{k,R}$ is a reduced subequation cone, the constant function defined by $w \equiv \sup_{\Omega}u$ is $\F^{+}_{k,R}$-harmonic (superharmonic)
and hence the failure of the (MP) implies the failure of comparison.

\end{proof}

We remark that Example \ref{exe:CE1_CP} is a special case of a larger family of counterexamples to the validity of comparison principles and Alexandroff estimates. See section 4 of \cite{GV17} for operators involving truncated Laplacians and truncated Pucci maximal and minimal operators. Moreover, Propsoition \ref{prop:CE1_CP} easily generalizes with the standard eigenvalues $\lambda_k(A)$ replaced by the G{\aa}rding eigenvalues $\lambda_k^{\pol}(A)$ of a  {\em G{\aa}rding-Dirichlet polynomial} $\pol$ (see the discussion of subsection \ref{subsec:garding} for the relevant notions).  

\begin{exe}\label{exe:CE2_CP} Let  $\pol$ be a G{\aa}rding-Dirichlet polynomial of degree $m$ on $\cS(n)$ in the sense of Definition \ref{defn:garding_op} whose {\em (ordered) G\aa rding $I$-eigenvalues of $A$} are denoted by
	\begin{equation}
	\lambda_1^{\pol}(A) \leq \lambda_2^{\pol}(A) \leq \ldots \leq \lambda_m^{\pol}(A).
	\end{equation}
and $\pol$ is normalized to have $\pol(I) = 1$. Again assuming that $n \geq 2$, the same conclusions of Proposition \ref{prop:CE1_CP} hold for the subsequation constraint set
	\begin{equation}\label{CE2_SE}
\F_{k, \beta}^{\pol} := \{ (r,p,A) \in \J^2: \lambda_k^{\pol}(A) + \beta |p| \geq 0 \} \ \ \text{where} \ k \in \{ 2, \ldots, n \}.
\end{equation}

\end{exe}

\subsection{Failure of comparison on arbitrarily small domains}

We now give a family of examples for which comparison fails on arbitrarily small balls. In fact, we will exhibit subequations for which existence of the Dirichlet problem will hold on arbitrary balls, but the comparison principle, the maximum principle and uniqueness for the Dirichlet problem will all fail (see Theorem \ref{thm:failure} below). The argument will make use of the considerations of Appendix \ref{sec:esiuni} on maximal and minimal solutions, and hence the proof Theorem \ref{thm:failure} will be given in Appendix \ref{sec:failure}. The examples we present will involve subequations $\F$ whose maximal monotonicity cone $\cM_{\F}$ has empty interior, and, as such, cannot admit strict approximators on any domain, no matter how small. The examples are {\em reduced} in the sense that no constraint is placed on the jet variable $r \in \R$ and hence the subsequations $\F$ will be considered as subsets of $\R^n \times \cS$ (see section \ref{sec:reductions} for more on reductions).

We begin by defining the subequations and making some preliminary observations. For $p \neq 0$ in $\R^n$, we recall that the orthogonal projection onto the subapces $[p]$ and $[p]^{\perp}$ are (respectively)
\begin{equation}\label{proj_p_p_perp}
P_{p} = \frac{1}{|p|^2} p \otimes p \quad \text{and} \quad P_{p^{\perp}} = I - P_{p}.
\end{equation}

\begin{exe}\label{exe:upper_lower}
	For $\alpha \in (1, + \infty)$, define 
	\begin{equation}\label{Bmap}
	B(p,A) := A +  |p|^{\frac{\alpha - 1}{\alpha}} \left( P_{p^{\perp}} + \alpha P_{p} \right)) \ \ \text{if} \ \ p \neq 0 \ \ \text{and} \ \ B(0,A) := A.
	\end{equation}
	Notice that the map $B: \R^n \times \Symn \to \Symn$ is continuous.  Consider the operators $F, G \in C(\R^n \times \cS(n))$ defined by
	\begin{equation}\label{define_FG}
	F(p,A) := \lambda_{\rm min}(B(p,A)) \quad \text{and} \quad  G(p,A) := \lambda_{\rm max}(B(p,A))
	\end{equation}
	along with the (reduced) subequations defined by 
	\begin{equation}\label{exeF1}
	\F := \{ (p,A) \in \R^n \times \Symn:  \lambda_{\rm min} (B(p, A)) \geq 0 \} 
	\end{equation}
	and
	\begin{equation}\label{exeG1}
	\cG := \{ (p,A) \in \R^n \times \Symn:  \lambda_{\rm max} (B(p, A)) \geq 0 \}.
	\end{equation}
	When $\alpha = 2$, the subequation $\F$ was introduced in \cite{HL11} as an example where existence holds, but uniqueness fails. Hence, we are considering generalizations of that example.
	
	The fact that the closed sets $\F$ and $\G$ are (reduced) subequations can be seen as follows. Property (N) is automatic since $\F$ and $\G$ are independent of the jet variable $r \in \R$. Property (P) holds for $\F$ and $\cG$ since the operators $F$ and $G$ are increasing in $A \in \cS(n)$. 
	
	To prove the topological property (T) and to show compatibility between $F$ and $\F$ and between $G$ and $\G$, we use a general lemma which we state in the reduced case. By {\em compatibility} we mean the relation \eqref{CF2} below (see Definition \ref{defn:compatible_pair}). This notion will play an important role in our treatment of comparison for operators $F$.
	
	\begin{lem}\label{lem:compatibility} Suppose that $F \in C(\R^n \times \cS(n))$ is a degenerate elliptic operator; that is, $F = F(p,A)$ is increasing in $A$ on all of $\cS(n)$. If $F$ is linear on lines through $I \in \cS(n)$ in the sense that
	$$
	F(p, A + tI) = F(p,A) + t \ \ \text{for each} \ t \in \R,
	$$
	then the set $\F:= \{ (p, A) \in \R^n \times \cS(n): \ F(p,A) \geq 0 \}$ satisfies:
	\begin{equation}\label{CF1}
	\Int \, \F = \{ (p, A) \in \R^n \times \cS(n): \ F(p,A) > 0 \};
	\end{equation}
		\begin{equation}\label{CF2}
	\partial \F = \{ (p, A) \in \R^n \times \cS(n): \ F(p,A) = 0 \};
	\end{equation}
		\begin{equation}\label{CF3}
	\sim \F = \{ (p, A) \in \R^n \times \cS(n): \ F(p,A) < 0 \}.
	\end{equation}
		
		\end{lem}
	
	\begin{proof} Since $\{ (p, A): \ F(p,A) > 0 \}$ is an open subset of $\F$, it is contained in $\Int \, \F$. If $F(p,A) = 0$, then $(p,A)$ is approximated by $(p, A + \veps I) \in \Int \, \F$ for each $\veps > 0$ since $F(p, A + \veps I) = F(p,A) + \veps = \veps > 0$. Such a $(p,A)$ is also approximated by $(p, A - \veps I) \not\in \F$ for each $\veps > 0$ since $F(p, A - \veps I) = F(p,A) - \veps = -\veps < 0$. This proves \eqref{CF1}, \eqref{CF2} and \eqref{CF3} as well as property (T): $\F = \overline{\Int \, \F}$. 
		\end{proof}  
	
	This result applies to both $F$ and $G$ defined by \eqref{define_FG} above since $B(p, A + tI) = B(p,A) + tI$ and $\lambda_j (B(p,A) + tI) = \lambda_j (B(p,A)) + t$ for each $j = 1, \ldots n$.

	The Dirichlet duals of $\F$ and $\G$ defined in \eqref{exeF1} and \eqref{exeG1} are given by
	\begin{equation}\label{exeF1D}
	\wt{\F} := \left\{ (p,A) \in \R^n \times \Symn: \lambda_{\rm max} \left( A - |p|^{\frac{\alpha - 1}{\alpha}} \left( P_{p^{\perp}} + \alpha P_{p} \right) \right) \geq 0 \right\}
	\end{equation}
	and
	\begin{equation}\label{exeG1D}
	\wt{\cG} := \left\{ (p,A) \in \R^n \times \Symn: \lambda_{\rm min} \left( A - |p|^{\frac{\alpha - 1}{\alpha}} \left( P_{p^{\perp}} + \alpha P_{p} \right) \right) \geq 0 \right\}.
	\end{equation}
	Note that one has the inclusions
	\begin{equation}\label{FandG}
	\F \subset \cG \quad  \text{and} \quad \wt{\cG} \subset \wt{\F}.
	\end{equation}
	and that 
	\begin{equation}\label{FG_0}
	(p,A) := (0,0) \in \F \cap \wt{\F} \cap \cG \cap   \wt{\cG}.
	\end{equation}
\end{exe}

\begin{rem}\label{rem:exeFG} If $\alpha = 1$ were to be considered, then 
	$$
	\F = \R^n \times (\cP - I), \wt{\F} = \R^n \times (\cPt + I), \cG = \R^n \times (\cPt - I) \ \text{and} \ \wt{\cG} = \R^n \times (P + I)
	$$ are all pure second order, so comparison holds for all of them and $\Int \, \cP$ governs boundary convexity, as it is the asymptotic interior of each of them (see the discussion in Appendix \ref{sec:esiuni}).
\end{rem}

Next we describe the maximal monotonicty cones for the two subequations in this Example \ref{exe:upper_lower}. 

\begin{lem}\label{lem:MMC} For each $\alpha \in (1, +\infty)$ the (reduced) subequations 
	\begin{equation*}\label{exeF1}
	\F := \{ (p,A) \in \R^n \times \Symn:  \lambda_{\rm min} (B(p, A)) \geq 0 \} 
	\end{equation*}
	and
	\begin{equation*}\label{exeG1}
	\cG := \{ (p,A) \in \R^n \times \Symn:  \lambda_{\rm max} (B(p, A)) \geq 0 \}.
	\end{equation*}
where 
	$$B(p,A):=  A +  |p|^{\frac{\alpha - 1}{\alpha}} \left( P_{p^{\perp}} + \alpha P_{p} \right)) \ \ \text{if} \ \ p \neq 0 \ \ \text{and} \ \ B(0,A) := A, $$
	have (reduced) maximal monotonicity cone $\cM = \{ 0 \} \times \cP \subset \R^n \times \Symn$.
\end{lem}

\begin{proof}	By Proposition \ref{prop:MMC2} (b), $\F$ and $\wt{\F}$ have the same maximal monotonicity cones. We will show that $\cM_{\wt{\F}} = \{0\} \times \cP$. It suffices to show that if $(q,B) \in \cM_{\wt{\F}}$ then $q$ must be equal $0$ because the fiber of $\wt{\F}$ over $\{0\}$ is $\wt{\cP}$, which has maximal monotonicity $\cP$. Suppose that $(q,B) \in \cM_{\wt{\F}}$ then since $\cM_{\wt{\F}}$ satisfies property (P), one has
	$$ (q,\lambda I) \in \cM_{\wt{\F}} \ \ \text{if} \ \  \lambda I \geq B. $$
	Since $\cM_{\wt{\F}}$ is a monotonicity cone for $\wt{\F}$ with $(0,0) \in \wt{\F}$, one has for any $\lambda \in \R$ such that $\lambda I \geq B$
	$$
	(tq, t \lambda I) = (0,0) + (tq, t \lambda I) \in \wt{\F} + \cM_{\wt{\F}} \subset \wt{\F} \ \ \text{for all} \ t > 0;
	$$
	that is
	\begin{equation}\label{MMCFD}
	\lambda_{\rm max} \left( t \lambda I - t^{\frac{\alpha - 1}{\alpha}} |q|^{\frac{\alpha - 1}{\alpha}} \left( P_{q^{\perp}} + \alpha P_{q} \right) \right) \geq 0  \ \ \text{for all} \ t > 0.
	\end{equation}
	Pick $\lambda > 0$ large enough to ensure that $\lambda I \geq B$. Since the eigenvalues of $P_{q^{\perp}} + \alpha P_{q}$ are $1$ and $\alpha > 1$, and since $(\alpha -1)/\alpha = 1 - 1/\alpha$, the inequality \eqref{MMCFD} is equivalent to 
	$$
	\lambda \geq \frac{|q|^{\frac{\alpha - 1}{\alpha}}}{t^{1/\alpha}} \ \ \text{for all} \ t > 0,
	$$
	which implies $q = 0$, as desired.
	
	An analogous argument shows that $\cM_{\cG} = \cM_{\wt{\cG}} = \{0\} \times \cP$.
\end{proof}
	
	Notice that the interiors of the maximal monotonicity cones are empty, and hence strict approximators cannot be found, which suggests that comparison may fail. Indeed, comparison does fail. 

\begin{thm}\label{thm:failure}
	Let $R \in (0, +\infty)$ and let $B_R \subset \R^n$ be the open $R$-ball about $0$. The functions $z \in C^{\infty}(\R^n)$ and $h \in C^{2, \alpha - 1}(\R^n)$ defined by 
	$$ z(x) := 0 \ \ \text{and} \ \  h(x) := -\frac{|x|^{1 + \alpha}}{1 + \alpha} + \frac{R^{1 + \alpha}}{1 + \alpha} , \ \ x \in \R^n
	$$
	are both $\F$ and $\cG$ harmonic on all of $\R^n$. They both have boundary values $\varphi = 0$ on $\partial B_R$. Thus comparison, uniqueness and the maximum principle all fail for $\F$ and $\cG$ on $B_R$, which can be an arbitrarily small ball.
\end{thm}

The direct proof is provided, for the convenience of the reader, in  Appendix \ref{sec:failure} where we compute the one variable radial subequations associated to $\F$ and to $\G$.

\section{Special Cases: Reduced Constraint Sets}\label{sec:reductions}

We have given in Theorem \ref{thm:CP_general} a general comparison principle on domains $\Omega \subset \subset \R^n$ for subequation constraint sets  $\F \subset \J^2 = \R \times \R^n \times \Symn$ which satisfy the assumptions: 
\begin{equation}\label{CP1}
	\mbox{ $\F$ is $\cM$-monotone with $\cM$ a monotonicity cone subequation;}
	\end{equation}
\begin{equation}\label{CP2}
	\mbox{ $\cM$ admits a strict approximator on $\Omega$.}
\end{equation}
Moreover, in Definition \ref{defn:cone_zoo}, we have introduced the family of $(\gamma, \cD, R)$ monotonicity cone subequations and we have described for which domains $\Omega$ a given cone $\cM(\gamma, \cD, R)$ admits a strict approximator $\psi$ on $\Omega$.  

In this section, we discuss the special cases when at least one of the constraint factors of $\F \subset \R \times \R^n \times  \Symn$ is ``silent''  in the sense that no restriction is placed on the jet variable corresponding to that factor\footnote{In terms of nonlinear differential operators this means that $F(u, Du, D^2u)$ is independent of at least one of the variables $u, Du$ or $D^2u$.}. In these cases, the subset of the remaining factors will be called the {\em reduced constraint set for $\F$} and will be denoted by $\F^{\prime}$. The silent factors can be included in any associated monotonicity cone subequation $\cM$ for $\F$, making $\cM$  ``large'' and hence increasing the likelihood of \eqref{CP2} being true.

We start with the {\em pure second order case}. Although it has been treated in some detail in \cite{HL09}, a discussion is in order here as a prelude to the main results of this section which concern the {\em gradient free} case, where we prove the analogue of the {\em subaffine theorem}. We apply the Convention \ref{conv:reduction} throughout this section.

At this point, some readers may wish to turn to the Summary Remark \ref{rem:special_cases} for an overview.

\subsection{Pure second order}

By definition, a {\em pure second order constraint set} is a subset $\F \subset \J^2 = \R \times \R^n \times \Symn$ of the form 
$$
	\F = \R \times \R^n \times \F^{\prime}.
$$
That is, the factor $\R \times \R^n$ is silent and the reduced constraint set $\F^{\prime}$ is a subset of $\Symn$. An equivalent definition in terms of monotonicity is that $\F$ is $\R \times \R^n \times \{0\}$-monotone; that is,
\begin{equation}\label{PSO1}
(r,p,A) \in \F \ \Rightarrow (r + s, p + q, A) \in \F \ \ \text{for each} \ s \in \R, q \in \R^n.
\end{equation}
In particular, $\F$ is $(-\infty,0] \times \R^n \times \{0\}$-monotone so that $\F$ automatically satisfies both the negativity property (N) and the directionality property (D) with respect to the cone $\cD$ which is all of $\R^n$. The positivity property (P) holds for $\F$ if and only if the reduced constraint set $\F^{\prime}$ is $\cP$-monotone; that is
\begin{equation}\label{PSO2}
	\F^{\prime} + \cP \subset \F^{\prime}.
\end{equation}
The remaining subequation property, namely the topological property (T) which asks that $\F$ is the closure of its interior, is equivalent to this being true for the reduced constraint set $\F^{\prime}$, that is,
\begin{equation}\label{PSO3}
	\F^{\prime} = \overline{ \Int \, \F^{\prime}}.
\end{equation}

This topological property \eqref{PSO3} follows from the positivity property \eqref{PSO2} as long as $\F$ (or equivalently $\F^{\prime}$) is closed. To see this, first note that the positivity property \eqref{PSO2} implies that $\F^{\prime} + \Int \, \cP \subset \Int \, \F^{\prime}$. In particular, if $A \in \F^{\prime}$, then $A + \veps I \in \Int \, \F^{\prime}$ for every $\veps > 0$. Consequently,  Summarizing then, $\F \subset \J^2$ is a {\em pure second order subequation} if 
\begin{equation}\label{PSO4}
\mbox{ $\F$ is closed and $\cM(\cP) = \R \times \R^n \times \cP$-monotone.}
\end{equation}
In terms of the reduced constraint set $\F^{\prime}$ with  $\F = \R \times \R^n \times \F^{\prime}$ the definition takes the simpler equivalent form
\begin{equation}\label{PSO5}
\mbox{ $\F^{\prime} \subset \Symn$  is closed and $\cP$-monotone.}
\end{equation}

Employing Convention \ref{conv:reduction}, we will refer to such an $\F^{\prime}$ as a {\em pure second order subequation} \footnote{In \cite{Kv95}, open sets $\Theta \subsetneq \Symn$ which correspond to $\Int \, \F^{\prime}$ were called {\em elliptic sets}. Pure second order subequations  $\F^{\prime}$ were introduced as closed subsets of $\Symn$ with the positivity property in \cite{HL09}, where they were denoted by $F$ and called {\em Dirichlet sets}. Such $\F^{\prime}$ were called {\em elliptic sets} and denoted by $\Theta$ in \cite{CP17}.} Taking $\F^{\prime} \subsetneq \Symn$ as a proper subset ensures that $\F(X)$ is a proper subset of $\USC(X)$. 

Comparison for arbitrary bounded domains $\Omega$ follows easily from Theorem \ref{thm:comparison} since the monotonicity cone $\cM(\cP) = \R \times \R^n \times \cP$ for pure second order subequations  contains the monotonicity cone $\cM(\cN) \cap \cM(\cP) = \cN \times \R^n \times \cP$. Note that this is a special case of a $(\gamma, \cD, R)$-monotonicity cone with $\gamma = 0, \cD = \R^n$ and $R = +\infty$. 

The differential inclusion \eqref{FSH-DI} defining $\F$-subharmonicity at $x_0$ is:
\begin{equation}\label{SHE1}
J^2_{x_0} \varphi = (\varphi(x_0), D\varphi(x_0), D^2 \varphi(x_0)) \in \F 
\end{equation}
for each  $C^2$ upper test function $\varphi$ for $u$ at $x_0$. This reduces to the simpler statement that $u \in \F(X)$ if and only if $u \in \USC(X)$ and, for each $x_0 \in X$, one has 
\begin{equation}\label{SHE3}
 \mbox{$D^2 \varphi(x_0) \in \F^{\prime}$ \ \ for each $C^2$ upper test function $\varphi$ for $u$ at $x_0$.}
\end{equation}

\begin{exe}\label{rem:PSO} What might be considered the most basic example in all of viscosity theory is the example $\F^{\prime} = \cP$. Its dual is the {\em subaffine subequation}
	\begin{equation}\label{P_dual}
	\cPt := \{ A \in \Symn: \ \lambda_{\rm max}(A) \geq 0  \}.
	\end{equation}
\end{exe}

One has the elementary facts:
\begin{equation}\label{SAS1}
	\mbox{ $\F$ is a pure second order subequation if and only if $\FD$ is} 
\end{equation}
and
\begin{equation}\label{SAS2}
	\F^{\prime} + \wt{\F}^{\prime} \subset \cPt \ \ \text{for each pure second order subequation} \ \F = \R \times \R^n \times \F^{\prime}.
\end{equation}
This pure second order jet addition theorem \eqref{SAS2} implies the following subharmonic addition theorem 
	\begin{equation}\label{SAS3}
	\F^{\prime} (X) + \wt{\F}^{\prime} (X) \subset \cPt(X),
	\end{equation}
as discussed in Theorem \ref{thm:SAT} for general sums $\F + \G \subset \cH$ of subequations in $\J^2$. 
Let $\A$ denote the space of affine functions on $\R^n$. As defined in Example \ref{exe:SA}, for each open set $X \subset \R^n$, let $\SA(X)$ denote the set of all $u \in \USC(X)$ with the {\em subaffine property}
\begin{equation}\label{SAProp}
\mbox{$u \leq a$ on $\partial \Omega \ \Rightarrow \	u \leq a$ on $\Omega$ \quad for every $\Omega \subset \subset X$ and $a \in \A$.}  
\end{equation}
One of the key results of \cite{HL09} is that
	\begin{equation}\label{SAChar}
\cPt(X) = \SA(X)
	\end{equation}
The proof of \eqref{SAChar} is also obtained here as a special easier case of the proof of Theorem \ref{thm:SAPChar} as indicated in its proof. Combining \eqref{SAChar} with \eqref{SAS3}, one has the {\em subaffine theorem} of \cite{HL09}:
	\begin{equation}\label{SAS4}
\F(X) + \wt{\F}(X) \subset \SA(X).
\end{equation}
This gives rise to an alternate proof of comparison principle for pure second order subequations because subaffine functions clearly satisfy the Zero Maximum Principle, as the function zero is affine.

\subsection{Gradient free}\label{subsec:GFSE}

By definition, a {\em gradient free constraint set} is a subset $\F \subset \J^2 = \R \times \R^n \times \Symn$ of the form 
\begin{equation}\label{GF1}
\F = \{ (r,p,A) \in \J^2: \ p \in \R^n \ \text{and} \ (r,A) \in \F^{\prime} \} \ \text{where} \ \F^{\prime} \subset \R \times \Symn. 
\end{equation}
That is, the factor $\R^n$ is silent and the reduced constraint set $\F^{\prime}$ is a subset of $\R \times \Symn$. When it is convenient, we will reorder the jet variables as $(p,r,A)$ so that the definition of $\F$ being gradient free can be restated as
\begin{equation}\label{GF_reordered}
	\F = \R^n \times \F^{\prime} \ \ \text{with} \  \F^{\prime} \subset \R \times \Symn.
\end{equation}
An equivalent definition in terms of monotonicity is that $\F$ is $\{0\} \times \R^n \times \{0\}$-monotone; that is,
\begin{equation}\label{GF2}
(r,p,A) \in \F \ \Rightarrow (r, p + q, A) \in \F \ \ \text{for each} \ q \in \R^n.
\end{equation}
This is the directionality property (D) with the convex cone $\cD \subset \R^n$ taken to be all of $\R^n$. Incorporating the subequation properties, we arrive at the definition of the concept we seek.  

\begin{defn}\label{defn:GFSE1}  
	A {\em gradient free subequation} is any closed, non-empty subset $\F \subsetneq \J^2 = \R \times \R^n \times \Symn$  which is $\cM$-monotone with respect to the monotonicity cone subequation $\cM(\cN) \cap \cM(\cP) = \cN \times \R^n \times \cP$; that is,
	\begin{equation}\label{GFSE1}
	\F +  (\cN \times \R^n \times \cP) \subset \F
	\end{equation}
	\end{defn}

It is obvious that this monotonicity property is equivalent to the combined monotonicity properties (P), (N) and (D) with $\cD = \R^n$. Therefore, in the gradient free case, one always has comparison on arbitrary bounded domains by applying Theorem \ref{thm:comparison}.

\begin{thm}\label{thm:CP_GF} Let $\F \subset \J^2$ be a gradient free subequation and let $\Omega \subset \subset \R^n$ be arbitrary.  Given $u, v \in \USC(\overline{\Omega})$ which are $\F, \wt{\F}$-subharmonic on $\Omega$ one has
	\begin{equation*}\label{CP_GF}
	u + v \leq 0 \ \text{on} \ \partial \Omega \ \ \Longrightarrow \ \ u + v \leq 0 \ \text{on} \ \Omega.
	\end{equation*}
	\end{thm} 

In analogy with the pure second order case, all of the above discussion can be reformulated in simpler terms by starting with the reduced constraint set $\F^{\prime} \subset \R \times \Symn$ and proceeding as follows.

\begin{defn}\label{defn:GFSE2}
	A closed subset $\F^{\prime} \subset \R \times \Symn$ is called a {\em gradient free subequation} if 
	\begin{equation}\label{GFSE2}
	\F^{\prime} +  \cQ \subset \F^{\prime} \quad \text{where} \ \ \cQ:= \cN \times \cP.
	\end{equation}
\end{defn}
We note that the topological property (T) is again automatic since each $(r,A) \in \F^{\prime}$ can be approximated by $(r,A) + \veps (-1, I)$ which lies in
$\F^{\prime} + \Int \, \cQ \subset \Int \, \F^{\prime}$.

\begin{exes}\label{exe:GFSE} The most basic example is $\F^{\prime} = \cQ = \cN \times \cP$. Another example is its dual
	\begin{equation}\label{Q_dual}
	\cQt := \{ (r,A) \in \R \times \Symn: \ r \leq 0 \ \ \text{or} \ \ A \in \cPt \}
	\end{equation}
	with $\cPt$ as in \eqref{P_dual}, the subaffine subequation.
\end{exes}

 One has the elementary facts:
 \begin{equation}\label{GFS1}
 \mbox{ $\F$ is a gradient free subequation if and only if $\FD$ is,} 
 \end{equation}
 and
 \begin{equation}\label{GFS2}
 \mbox{ $\F^{\prime} + \wt{\F}^{\prime} \subset \cQt$ for each gradient free subequation $\F$ determined by $\F^{\prime}$.}
 \end{equation}
 This jet addition result extends to the {\em subharmonic addition theorem}
	\begin{equation}\label{GFS3}
	\mbox{$\F(X) + \wt{\F}(X) \subset \cQt(X)$ \quad for any open subset $X \subset \R^n$.}
	\end{equation}
	Here, in the spirit of Convention \ref{conv:reduction}, one has $u \in \F(X)$   if and only if $u \in \USC(X)$ and, for each $x_0 \in X$ one has
	$$
	\mbox{$(\varphi(x_0), D^2 \varphi(x_0)) \in \F^{\prime}$ 	for each $C^2$ upper test function $\varphi$ for $u$ at $x_0$.}
		$$

The claims up to \eqref{GFS2} are purely algebraic statements and easily verified. The claim \eqref{GFS3} follows \eqref{GFS2} and from the general results of section \ref{sec:CP}. In particular, see Theorem \ref{thm:SAT}, Lemma \ref{lem:CP_Jets} and Theorem \ref{thm:SAT_MD}. 

We will finish this discussion by providing alternate characterizations of $\cQt(X)$ including the one claimed in Example \ref{exe:SAP}, culminating in the {\em subaffine plus theorem} (see Theorem \ref{thm:SAPT} below). The key notion  is that of {\em subaffine plus functions}, which we proceed to describe. 

For any domain $\Omega \subset \subset \R^n$ and with $\A$ the space of affine functions on $\R^n$, denote by 
\begin{equation}\label{Ap}
\Ap(\overline{\Omega}) := \{ \alpha:= a_{|\overline{\Omega}} : a \in \A \ \text{and} \ a \geq 0 \ \text{on} \ \overline{\Omega} \},
\end{equation}
the space of {\em affine plus functions on $\overline{\Omega}$}. 

\begin{defn}\label{defn:SAp}
	For $X \subseteq \R^n$ open, a function $u \in \USC(X)$ is said to be {\em subaffine plus on $X$} if for every open subset $\Omega \subset \subset X$, it has the {\em subaffine plus property}: 
	\begin{equation}\label{SAp1}
	u \leq \alpha \ \text{on} \ \partial \Omega \ \Rightarrow \ u \leq \alpha \ \text{on} \ \Omega \quad \text{for every} \ \alpha \in \Ap(\overline{\Omega}).
	\end{equation}
	Denote by $\SAp(X)$ the space of all such functions. 
\end{defn}

The $\cQt$-subharmonic functions $u$ on $X$ can be characterized by being subaffine plus on $X$, and/or by the positive part $u^+ := \max \{u,0\}$ being subaffine on $X$. 

\begin{thm}[Subaffine plus characterizations]\label{thm:SAPChar} If $X \subseteq \R^n$ is open, then 
	\begin{equation}\label{SApC1}
	\cQt(X) = \SAp(X) = \{ u \in \USC(X): \ u^+  \in \SA(X) = \cPt(X) \}.
	\end{equation}
\end{thm}

\begin{proof} The proof requires verifying three inclusions, where the first inclusion provides motivation for the concept of {\em subaffine plus} introduced in Definition \ref{defn:SAp}. For completeness, we include a proof of the characterization $ \cPt(X) = \SA(X)$ which was stated in \eqref{SAChar} and used in the statement of Theorem \ref{thm:SAPChar}.
	
\noindent{\bf Step 1:} $\cQt(X) \subset \SAp(X)$.

Suppose that $u \in \cQt(X)$. Since $\cQt$ is a closed set satisfying (P) and (N) one has {\em definitional comparison} (see Lemma \ref{lem:DCP}); that is, 
	\begin{equation}\label{DefComp1}
	u \leq w \ \text{on} \ \partial \Omega \ \Rightarrow \ u \leq w \ \text{on} \ \Omega.
	\end{equation}
	for functions $w$ that satisfy
	\begin{equation}\label{DefComp2}
	w \in C(\overline{\Omega}) \cap C^2(\Omega) \ \text{with} \ (w(x), D^2w(x)) \not\in \cQt, \ \forall \ x \in \Omega.
	\end{equation}
	If $w$ is a degree two polynomial then $D^2_x w (x):= A$ is independent of $x$, so the requirement $(w(x), D^2w(x)) \not\in \cQt, \ \forall \ x \in \Omega$ in \eqref{DefComp2} becomes
	\begin{equation}\label{DefComp2bis}
	\mbox{ $w(x) > 0$ \quad and \quad  $D^2w(x) :=A < 0$ \quad for each $x \in \Omega$}.
	\end{equation}
Such functions $w$ can be constructed by starting from any $\alpha \in \Ap(\overline{\Omega})$ and then defining	
	\begin{equation}\label{DefComp3}
	w_{\veps}(x) := \alpha(x) + \veps \left( \frac{R^2 - |x|^2}{2} \right), \ \ x \in \overline{\Omega},
	\end{equation}
	with $\veps > 0$ and $R > 0$ chosen large enough so that $\overline{\Omega} \subset B_R(0)$; that is, $R^2 - |x|^2 > 0$ on $\overline{\Omega}$. One computes that
	$$
	w_{\veps}(x) > 0 \ \ \text{for every} \ x \in \overline{\Omega} \quad \text{and that} \quad D^2w_{\veps}(x) = - \veps I < 0 \ \ \text{for every} \ x,
	$$
	so that \eqref{DefComp2bis} is satisfied. If, in addition, $u \leq \alpha$ on $\partial \Omega$, then  $u \leq w_{\veps}$ on $\partial \Omega$. Applying \eqref{DefComp1} yields
	$$
	u \leq w_{\veps} \ \text{on} \ \Omega
	$$
	with $\veps > 0$ arbitrary. Taking the limit as $\veps \to 0^+$ yields $u \leq \alpha$ on $\Omega$ as desired and hence $u \in \SAp(X)$.
		
	\noindent {\bf Step 2:} $\SAp(X) \subset \{ u \in \USC(X): \ u^+  \in \SA(X) \}$.
	
	Suppose that $u \in \SAp(X)$. Given $\Omega \subset \subset X$ and $a \in \A$ with $u^+ \leq a$ on $\partial \Omega$, since $0 \leq u^+$ it follows that $0 \leq a$ on $\partial \Omega$. Now $-a$ is also affine, so it satisfies the maximum principle, and hence $0 \leq a$ on $\overline{\Omega}$. This proves that $\alpha:= a_{|\overline{\Omega}} \in \Ap(\overline{\Omega})$. Since $u \in \SAp(X)$ and $u \leq \alpha$ on $\partial \Omega$, it follows that $u \leq \alpha$ on $\Omega$. Also $0 \leq \alpha = a$ on $\overline{\Omega}$. Therefore $u^+ \leq \alpha = a$ on $\overline{\Omega}$. This proves that $u^+ \in \SA(X)$.

	\noindent {\bf Step 3:} $\{ u \in \USC(X): \ u^+ \in \cPt(X) \} \subset \cQt(X)$.
	
	We argue by contradiction to show that $ u \not\in \cQt(X) \ \Rightarrow \ u^+ \not\in \cPt(X)$. If $u \in \USC(X) \setminus \cQt(X)$  then by the Bad Test Jet Lemma \ref{lem:nonFSH} there exist $x_0 \in X$, $\veps, \rho > 0$, $p \in \R^n$ and $(r,A) \not\in \cQt$ such that	
	\begin{equation}\label{NQt1}
	u(x) \leq   r + \langle p, x - x_0 \rangle + \frac{1}{2} \langle A(x - x_0), x - x_0 \rangle - \veps |x - x_0|^2, \ \forall \ x \in B_{\rho}(x_0)
	\end{equation}
	and
	\begin{equation}\label{NQt2}
	u(x_0) = r.
	\end{equation}
	Recall that $(r,A) \not\in \cQt$ means that
	\begin{equation}\label{NQt3}
	r > 0 \quad \text{and} \quad A < 0.
	\end{equation}
	Now $r > 0$ means that the right hand side of \eqref{NQt1} is positive for $\rho$ sufficiently small. Therefore \eqref{NQt1} holds on $B_{\rho}(x_0)$ with $u(x)$ replaced by $u^+(x)$ (and we have equality at $x = x_0$). Thus $u^+$ has an upper test jet $(r,p,A)$ at $x = x_0$ with $A < 0$, which proves that $u^+ \not\in \cPt(X)$.
	
	Finally, for the sake of completeness, we include the following argument.
	
	\noindent {\bf Step 4:} $\cPt(X) = \SA(X)$.
	
	To see that $\cPt(X) \subset \SA(X)$, consider the simpler version $w_{\veps}(x):= a(x) - \frac{\veps}{2}|x|^2$ of \eqref{DefComp3}, but with $a \in \A$ arbitrary. Since $D^2w_{\veps}(x) = -\veps I \not\in \cPt$ for every $x$, the subaffine property follows as in the argument after \eqref{DefComp3}, but for any $u \in \cPt(X)$, which proves that $u \in \SA(X)$.
	
	For the inclusion $\SA(X) \subset \cPt(X)$, see \eqref{NQt1} with $A \not\in \cPt$; that is, $A < 0$ (negative definite). It follows that $u \not\in \SA(X)$.  
	\end{proof}

Since $\cQt(X) = \SAp(X)$, we can restate the {\em subharmonic addition theorem} \eqref{GFS3} in the following way by making use of Convention \ref{conv:reduction}.

\begin{thm}[The Subaffine Plus Theorem]\label{thm:SAPT} If $\F \subset \J^2$ is a gradient free subequation, then for any open set $X \subset \R^n$
	\begin{equation}\label{SAPT}
	\F(X) + \wt{\F}(X) \subset \SAp(X).
	\end{equation}
	\end{thm}

We conlude this subsection with a few remarks concerning Theorems \ref{thm:SAPChar} and \ref{thm:SAPT}.

\begin{rem}\label{rem:SAPT1}
	It is easy to see that the set inclusion \eqref{GFS3} is actually an equality. Thus \eqref{SAPT} is also an equality.
\end{rem}

\begin{rem}\label{rem:SAPT2}
	The subaffine plus theorem (Theorem \ref{thm:SAPT}) yields another proof of  comparison (Theorem \ref{thm:CP_GF}) for gradient free subequations because subaffine plus functions clearly satisfy the Zero Maximum Principle, as the function zero is subaffine plus.
\end{rem}

\begin{rem}\label{rem:SAPT3}
	While $\cQt(X) = \SAp(X) = \{ u \in \USC(X): \ u^+  \in \SA(X) \}$, one might wonder if 
	\begin{multline}\label{SAPCF}
	\cQt(X) = \{ u \in \USC(X): \ \forall \, \Omega \subset \subset X \ \text{and} \ a \in \A, \ \text{one has} \\  u \leq a^+ \ \text{on} \ \partial \Omega \ \Rightarrow \ u \leq a^+ \ \text{on} \ \Omega \}?
	\end{multline}
	
	We leave it to the reader to show that the right hand side of \eqref{SAPCF} is contained in $\cQt(X)=\SAp(X)$. Making use of an affine $u$, we now give an example that shows that \eqref{SAPCF} is \underline{not} an equality. First note that $\cPt(X) \subset \cQt(X)$. Using the test \eqref{C2FSH} for $C^2$ functions, any $u \in \A(\R^n)$ belongs to $\cPt(\R^n)$ since $D^2u(x) = 0 \in \cPt$ for all $x$. In dimension $n=1$, consider
	$$
	\mbox{ $u(x) = x, \ a(x) = 2(x-1)$ \ and \ $\Omega = (0,2)$.} 
	$$
	One has $u = a^+$ on $\partial \Omega$ but $u(1) = 1 > 0 = a^+(1)$.
\end{rem}

\subsection{First order and pure first order} 

By definition, a {\em first order constraint set} is a subset $\F \subset \J^2 = \R \times \R^n \times \Symn$ of the form
$$
\F = \F^{\prime} \times \Symn.
$$
That is, the second order factor is silent and the reduced constraint set $\F^{\prime}$ is a subset of $\R \times \R^n$. Such a set $\F$ automatically satisfies the positivity condition (P). Hence $\F = \F^{\prime} \times \Symn$ is a {\em first order subequation} if $\F$ also satisfies the the properties (N) and (T), which in terms of $\F^{\prime}$ means 
\begin{equation}\label{NFO}
	(r,p) \in \F^{\prime} \ \Rightarrow \ (r + s, p) \in \F^{\prime} \ \text{for each} \ s \leq 0
\end{equation}
and
\begin{equation}\label{TFO}
\F^{\prime} = \overline{ \Int \, \F^{\prime}}. 
\end{equation}

Any monotonicity cone $\cM$ for $\F$ can always be enlarged to include the silent factor $\Symn$. Hence we can replace our family of $(\gamma, \cD, R)$-cones by the family of {\em $(\gamma, \cD)$-convex cones} whose elements are defined by
\begin{equation}\label{GDCC}
\cM^{\prime}(\gamma, \cD) := \{(r,p) \in \R \times \R^n: \ r \leq - \gamma|p| \ \text{and} \ p \in \cD \},
\end{equation}
where $ \gamma \in [0, +\infty)$ and $\cD \subset \R^n$ is a directional cone as in Definition \ref{defn:property_D}. In particular, 
$$
\cM(\gamma, \cD, +\infty) = \cM^{\prime}(\gamma, \cD) \times \cP \subset \cM^{\prime}(\gamma, \cD) \times \Symn,	
$$
so the strict approximators constructed in Theorem \ref{thm:ZMP_for_M} for $\cM(\gamma, \cD, +\infty)$ are valid for $\cM = \cM^{\prime}(\gamma, \cD) \times \Symn$ and for any domain $\Omega \subset \subset \R^n$. This proves that the (ZMP) holds for $\wt{\cM}(\overline{\Omega})$ functions and hence we always have comparison in this case. The following notational device will be used for convenience: given a subequation constraint set $\F$ and a bounded domain $\Omega$, denote by
\begin{equation}\label{FSH_on_closure}
\F(\overline{\Omega}) := \{ u \in \USC(\overline{\Omega}): \ u \ \text{is $\F$-subharmonic on} \ \Omega\} = \USC(\overline{\Omega}) \cap \F(\Omega).
\end{equation}

\begin{thm}\label{thm:CP_FO}  Let $\F = \F^{\prime} \times \Symn \subset \J^2$ be a first order subequation such that $\F^{\prime}$ is $\cM^{\prime}(\gamma, \cD)$-monotone with $\gamma \in [0, +\infty)$ and $\cD \subset \R^n$ a directional cone which can be all of $R^n$. Then comparison always holds; that is, for any domain $\Omega \subset \subset \R^n$, given $u \in \F(\overline{\Omega})$ and $v \in \wt{\F}(\overline{\Omega})$ one has
	\begin{equation*}\label{CP_GF}
	u + v \leq 0 \ \text{on} \ \partial \Omega \ \ \Longrightarrow \ \ u + v \leq 0 \ \text{on} \ \Omega.
	\end{equation*}
\end{thm} 

Finally, a {\em pure first order subequation} $\F \subset \R \times \R^n \times \Symn$ has both the first and third factors silent. That is, 
$$
	\F = \R \times \F^{\prime} \times \Symn
$$ 
with reduced constraint set $\F^{\prime} \subset \R^n$ satisfying the single condition (T):
\begin{equation}\label{TPFO}
\F^{\prime} = \overline{ \Int \, \F^{\prime}}. 
\end{equation}
The subequation conditions (N) and (P) are automatic in this case. However, for comparison $\F^{\prime}$ must also satisfy the directionality property (D)
\begin{equation}\label{DPFO}
\mbox{$\F^{\prime} + \cD \subset \F^{\prime}$; that is, \ $p \in \F^{\prime} \ \Rightarrow (p+q) \in \F^{\prime}, \forall \, q \in \R^n$}
\end{equation}
for some non-empty closed convex cone $\cD \subset \R^n$ with vertex at the origin. In this case, the comparison principle of Theorem \ref{thm:CP_FO} applies to $\F^{\prime} \subset \R^n$ and arbitrary domains $\Omega \subset \subset \R^n$ by taking $\gamma = 0$.

\subsection{Zero order free} \label{subsec:ZOFSE}

By definition, a {\em zero order free constraint set} is a subset $\F \subset \J^2 = \R \times \R^n \times \Symn$ of the form 
$$
\F = \R \times \F^{\prime}.
$$
That is, the zeroth order factor is silent and the reduced constraint set $\F^{\prime}$ is a subset of $\R^n \times \Symn$. Such a set $\F$ automatically satisfies the negativity condition (N), but not the positivity condition (P) nor the topological condition (T). Here $\F = \R \times\F^{\prime}$ is a {\em zero  order free subequation} if the reduced constraint set $\F^{\prime} \subset \R^n \times \Symn$ satisfies
\begin{equation}\label{PRSO}
(p,A) \in \F^{\prime} \ \Rightarrow \ (p, A + P) \in \F^{\prime} \ \text{for each} \ P \geq 0
\end{equation}
and
\begin{equation}\label{TFO}
\F^{\prime} = \overline{ \Int \, \F^{\prime}}. 
\end{equation}

One might  as well replace our family of $(\gamma, \cD, R)$-cones by the family of cones whose elements are $\R \times \cM^{\prime}(\cD, R)$ where 
\begin{equation}\label{MC_RSO}
\cM^{\prime}(\cD, R) := \left\{ (p,A) \in \R^n \times \Symn: \ p \in \cD \ \text{and} \ A \geq \frac{|p|}{R} I \right\},
\end{equation}
where $\cD \subset \R^n$ is a non-empty closed convex cone with vertex at the origin. Notice that 
$$
	\cM(0, \cD, R) \subset \R \times \cM^{\prime}(\cD, R).
$$
Since this inclusion is reversed by duality, our previous results apply, yielding the following comparison principle.

\begin{thm}\label{thm:CP_RSO}  Let $\F = \R \times \F^{\prime} \subset \J^2$ be a zero order free subequation such that $\F^{\prime}$ is $\cM^{\prime}(\cD, R)$-monotone. Then comparison holds; that is, given given $u \in \F(\overline{\Omega})$ and $v \in \wt{\F}(\overline{\Omega})$ one has
	\begin{equation*}\label{CP_GF}
	u + v \leq 0 \ \text{on} \ \partial \Omega \ \ \Longrightarrow \ \ u + v \leq 0 \ \text{on} \ \Omega,
	\end{equation*}
	where $\Omega$ is any domain $\Omega \subset \subset \R^n$ contained in a translate of the truncated cone $\cD \cap B_R(0)$ if $R < + \infty$ and $\Omega$ is an arbitrary bounded domain in the case $R = +\infty$.
\end{thm} 

\subsection{Summary}\label{summary_special_cases}

In this subsection, we give a brief summary of the cases discussed above.

\begin{sumrem}\label{rem:special_cases}
	For a subequation $\F \subset \J^2 = \R \times \R^n \times \Symn$, there are six cases where a reduced constraint set $\F^{\prime}$ can replace $\F$. In three of these cases, two factors of $\F$ are silent.
	
	\noindent 1. \underline{Pure second order.} $\R \times \R^n$ is silent and $\F = \R \times \R^n \times \F^{\prime}$ with $\F^{\prime} \subsetneq \Symn$ a closed set satisfying (P): $\F^{\prime} + \cP \subset \F^{\prime}$. 
	
	\noindent 2. \underline{Pure first order.} $\R \times \Symn$ is silent and $\F = \R  \times \F^{\prime} \times \Symn$ with $\F^{\prime} \subsetneq \R^n$ satisfying (T): $\F^{\prime} = \overline{ \Int \, \F^{\prime}}$. 
	
	\noindent 3. \underline{Zeroth order.} $\R^n \times \Symn$ is silent and $\F = (-\infty, r_0]  \times \R^n \times \Symn$ with $r_0 \in \R$.
	
	The remaining three cases have just one silent factor. 
	
	\noindent 4. \underline{Zero order free.} $\R$ is silent and $\F = \R  \times \F^{\prime}$ with $\F^{\prime} \subsetneq \R^n \times \Symn$ satisfying (P): $\F^{\prime} + (\{0\} \times \cP) \subset \F^{\prime}$ and satisfying (T): $\F^{\prime} = \overline{ \Int \, \F^{\prime}}$. 
	
	\noindent 5. \underline{Gradient free.} $\R^n$ is silent and $\F = \{ (r,p,A) \in \J^2: \ p \in \R^n \ \text{and} \ (r,A) \in \F^{\prime} \}$ with $\F^{\prime} \subsetneq \R \times \Symn$ a closed set satisfying (P) and (N): $\F^{\prime} + (\cN \times \cP) \subset \F^{\prime}$. 
	
	\noindent 6. \underline{First order.} $\Symn$ is silent and $\F = \F^{\prime} \times \Symn$ with $\F^{\prime} \subsetneq \R \times \R^n$ satisfying (N): $\F^{\prime} + (\cN \times \{0\}) \subset \F^{\prime}$ and satisfying (T):  $\F^{\prime} = \overline{ \Int \, \F^{\prime}}$. 
	
	Our main result Theorem \ref{thm:comparison} applies to case 5 (gradient free) and hence to case 1 (pure second order) and to the trivial case 3 (zeroth order) with no further restrictions on the subequation and yields comparison for arbitrary domains $\Omega \subset \subset \R^n$. A further condition on the reduced constraint set $\F^{\prime}$ is required in each remaining case, in which the gradient constraint is \underline{not} silent. 
	
	In case 2 (pure first order), in order to apply Theorem \ref{thm:comparison} we must assume directionality (D):  $\F^{\prime} + \cD \subset  \F^{\prime}$ for a proper cone $\cD \subset \R^n$. Then Theorem \ref{thm:comparison} applies and comparison holds on  arbitrary domains $\Omega \subset \subset \R^n$. 
	
	In case 4 (zero order free), in order to apply Theorem \ref{thm:comparison} we must assume that $\F^{\prime} \subset \R^n \times \Symn$ is $\cM^{\prime}$-monotone for 
	$$
	\cM^{\prime}(\cD,R) := \left\{ (p,A) \in \R^n \times \Symn: \ p \in \cD \ \text{and} \ A \geq \frac{|p|}{R} I \right\}
	$$
	for some convex cone $\cD \subsetneq \R^n$ and some $R$ with $0 < R \leq +\infty$. Theorem \ref{thm:comparison} applies and comparison holds on  arbitrary domains $\Omega \subset \subset \R^n$ if $R = +\infty$. Otherwise, comparison holds for domains $\Omega$ contained in a translate of the truncated cone $\cD_R:= \cD \cap B_R(0)$.
	
	In case 6 (first order), in order  to apply Theorem \ref{thm:comparison} we must assume that $\F^{\prime} \subset \R \times \R^n$ is $\cM$-monotone for the monotonicity cone 
	$$
	\cM^{\prime}(\gamma,\cD) := \left\{ (r,p) \in \R \times \R^n : p \in \cD \ \text{and} \ r \leq - \gamma |p| \right\} 
	$$
	for some $\gamma  \in [0, +\infty)$ and directional conse $\cD$. In which case, Theorem \ref{thm:comparison} applies and comparison holds on  arbitrary domains $\Omega \subset \subset \R^n$.
	
	Although, in case 3 (zeroth order) our results apply and comparison holds for all such subequations on arbitrary domains $\Omega \subset \subset \R^n$, no explicit discussion was presented since no constraint is placed on derivatives. This is the case when $\F$ is $\cM(\gamma, \cD)$-monotone with $\gamma = 0$ and $\cD = \R^n$ (so that $\cM(\gamma, \cD) = \cM(\cN)$) which is included in Theorem \ref{thm:CP_general}. However, the proof of comparison is trivial, as $\wt{\F} = (-\infty, -r_0] \times \R^n \times \Symn$.
\end{sumrem}

\section{Subequation Constraint Sets and Nonlinear Operators}\label{sec:sets_operators}

In this section, we will discuss some key issues concerning applications of the potential theoretic comparison principles to comparison principles for constant coefficient nonlinear operators. 

Attention is restricted to the constant coefficient case. This discussion goes beyond what is already given in \cite{HL18b}. It is helpful to be guided by two types of examples from the pure second order case. One, where the operator $F$ is defined and ``elliptic'' on the full jet space $\J^2$ and one where $F$ must be restricted (constrained) to a proper subset of $\J^2$ in order to be ``elliptic''. This dichotomy is illustrated by the minimal eigenvalue operator
\begin{equation}\label{MEO}
\mbox{$F(r,p,A) := \lambda_{\rm min}(A)$ \ \ which is increasing in $A$ on all of $\Symn$,}
\end{equation}
and the Monge-Amp\`{e}re operator
\begin{equation}\label{MAO}
\mbox{$F(r,p,A) := {\rm det} \, A$ \ \ which is increasing in $A$ only on $\cP \subsetneq \Symn$.}
\end{equation}

 Ideally, one would like to start from an equation 
\begin{equation}\label{PDE}
F(u, Du, D^2u) = 0
\end{equation}
defined by a function $F: {\rm dom}(F) \subseteq \J^2 \to \R$ and determine when there exists a subequation constraint set $\F \subset \J^2$ so that $\F$-subharmonics and $\F$-superharmonics ($-\wt{\F}$-subharmonics) correspond to viscosity subsolutions and supersolutions (with admissibility constraints) to the PDE \eqref{PDE}. A natural attempt would be to realize $\F$ in the form  
\begin{equation}\label{superlevel}
\F := \{ J = (r,p,A) \in {\rm dom}(F): \ \ F(J) \geq 0\},
\end{equation}
where one would also need
\begin{equation}\label{level}
\partial \F := \{ J = (r,p,A) \in {\rm dom}(F): \ \ F(J) = 0\}.
\end{equation}
The minimal monotonicity properties (P) and (N) for $\F$ can be deduced from the familiar monotonicity in $(r,A)$ for $F = F(r,p,A)$, which, in general, will not hold on all of ${\rm dom}(F) \subseteq \J^2$. This imposes an a priori {\em constraint} on admissible values of the 2-jets and one will first need to restrict $F$  and take ${\rm dom}(F) \subsetneq \J^2$ to be the {\em effective domain} on which $F$ is suitably monotone in $(r,A)$. In this {\em constrained case}, one might as well redefine ${\rm dom}(F) = \F$ so that the effective domain is a subequation constraint set. Ensuring that $\F$ has the needed topological property (T) is a more delicate matter and will be discussed below. Finally, to complete the applications, one would want to try to establish the needed structural conditions on $F$ which ensure that $\F$ is suitably $\cM$-monotone to yield comparison. 

Classes of examples and illustrations will be given in the following subsections. For example, unconstrained case examples include {\em canonical operators} as discussed in subsection \ref{subsection:canonical} and constrained case examples include {\em Dirichlet-G\aa rding polynomials} as discussed in subsection \ref{subsec:garding}. These classes are representative but, of course, not exhaustive for the dichotomy noted above. In particular, the minimal eigenvalue operator \eqref{MEO} is a canoncial operator for the pure second order subequation $\cP \subset \cS(n)$ and the Monge-Amp\`{e}re operator \eqref{MAO} is one of the most basic and important Dirichlet-G\aa rding polynomials.

\subsection{Compatible operator-subequation pairs and topological tameness}

We proceed with a precise discussion of the relationship between operators $F$ and subequation constraint sets $\F$, beginning with the following definition of {\em compatibility} related to the desire of realizing \eqref{level} in the constrained case. 

\begin{defn}\label{defn:compatible_pair} A {\em  compatible operator-subequation pair} $(F, \F)$ consists of either 
	\begin{equation}\label{case1}
	\mbox{$\F = \J^2$ and $F \in C(\J^2)$} \quad \text{(the {\em unconstrained case})}
	\end{equation}
	or
	\begin{equation}\label{case2}
	\mbox{a subequation $\F \subsetneq  \J^2$ and $F \in C(\F)$  \quad \text{(the {\em constrained case})}.}
	\end{equation}
	In this case \eqref{case2}, one requires that $F$ and $\F$ are {\em compatible} in the following sense: 
	\begin{equation}\label{compatible1}
	\mbox{ $\displaystyle{ c_0:= \inf_{\F} F}$ \ \  is finite}
	\end{equation}
	and
	\begin{equation}\label{compatible2}
	\partial \F = \{ J \in \F: \ F(J) = c_0 \}.
	\end{equation}
\end{defn}
Note that one can replace $F$ by $F - c_0$ and reduce to the situation in which $c_0 = 0$. 

Perhaps the simplest examples of compatibile pairs come from the protype operators noted above; namely, 
$\left( \lambda_{\rm min}(A), \J^2 \right)$ and $\left( {\rm det} \, A,  \R \times \R^n \times \cP \right)$
are compatible operator-subequation pairs in the unconstrained and constrained cases respectively. Notice that in this pure second order case, we will also refer to
\begin{equation}\label{CP_PSO}
	\left( \lambda_{\rm min}(A), \cS(n) \right) \quad \text{and} \quad \left( {\rm det} \, A,  \cP \right)
\end{equation}
as {\em compatible (pure second order) pairs}, where one makes the obvious modification of Definition \ref{defn:compatible_pair} in this and other reduced cases.

Compatibilty of a pair $(F, \F)$ will be used to define {\em $\F$-admissible viscosity subsolutions, supersolutions and solutions} of the equation $F(u,Du,D^2u) = c$ for each {\em admissible level} $c \in F(\F)$ (see Definition \ref{defn:AVSolns}). Our treatment of comparison will be based on the {\em corrspondence principle} of Theorem \ref{cor:AVSolns}. For this principle, in addition to compatibility of the pair $(F, \F)$, we also require that the pair has the minimal monotonicity of being {\em proper elliptic} ($\cM_0 = \cN \times \{0\} \times \cP$-monotonicity) and a non-degeneracy property of {\em topological tameness} for $F$ on $\F$. We proceed to discuss these two additional ingredients.

For a pair $\F$ and $F \in C(\F, \R)$, the monotonicity of the operator $F$ is related to the monotonicity of the set $\F$ as follows. 

\begin{defn}\label{defn:PEP} Suppose that $(F, \F)$ is an operator-subequation pair. For any subset $\cM \subset \J^2$, we will say that $(F, \F)$ is {\em $\cM$-monotone} if
\begin{equation}\label{Mmo1}
\mbox{ $\F$ is $\cM$-monotone; that is, $\F + \cM \subset \F$}
\end{equation}
and, in addition, $F$ is $\cM$-monotone on $\F$; that is,
\begin{equation}\label{Mmo2}
\mbox{ $F(J + J') \geq F(J)$ \  \ $\forall \, J \in \F, \ \forall \, J' \in \cM$.}
\end{equation}
	\end{defn}
Of course, if the pair $(F, \F)$ is $\cM$-monotone, then each particular upper level set such as $\F_0 := \{ J \in \F: F(J) \geq 0\}$ will be $\cM$-monotone, but the converse can fail to be true. However, by considering upper level sets for every $c \in \R$, the converse is trivially true.

\begin{lem}\label{lem:upper_levels} Given an operator-subequation pair $(F, \F)$ and a subset $\cM \subset \J^2$, the pair $(F, \F)$ is $\cM$-monotone if and only if the upper level sets $\F_c := \{ J \in \F: F(J) \geq c\}$ are $\cM$-monotone for all $c \in \R$.
	\end{lem} 

\begin{proof} The forward implication is a direct consequence of the defining conditions \eqref{Mmo1} and \eqref{Mmo2} of $\cM$-monotonicty for $(F, \F)$. For the converse, note that each $J \in \F$ belongs to $\F_c$ for each $c \leq F(J)$. Using the $\cM$-monotonicity of every sublevel set $\F_c$, for each $J \in \F$ and each $J' \in \cM$ one has 
$$ J + J' \in \F_c \subset \F \ \ \text{for each} \ c \leq F(J), $$ 
which proves \eqref{Mmo1}. Moreover, for each $J \in \F$ one has 
$$
	F(J + J') \geq c \ \ \text{for every} \ J' \in \cM \ \text{and every} \ c \leq F(J),
$$
which for $c = F(J)$ proves \eqref{Mmo2}.
\end{proof}

\begin{rem}[Admissible levels of $F$ for compatible pairs]\label{rem:level_sets} For a compatible pair $(F, \F)$, when considering  upper level sets $\F_c := \{ J \in \F: \ F(J) \geq c \}$, lower level sets $\F^c := \{ J \in \F: \ F(J) \leq c \}$ or level sets $\F(c):= \{ J \in \F: F(J) = c\}$ one should, of course, restrict attention to values $c \in F(\F)$ (those values which lie in the range of $F$). Otherwise, many statements become either redundant or empty, such as considering the $\cM$-monotonicity of $\F_c$ in Lemma \ref{lem:upper_levels}) if $c$ is not in the range of $F$. In particular, in the constrained case, we will consider only those $c \in \R$ with $c \geq c_0 := \inf_{\F} F$.
	\end{rem}

\begin{defn}\label{defn:admissible_levels}
Given a compatible operator-subequation pair $(F, \F)$, a number $c \in \R$ is called an {\em admissible level} (for $(F, \F)$) if $c \in F(\F)$.	
	\end{defn}

We can interpret operator {\em ellipticity} and {\em properness} using this lemma.

\begin{defn}\label{defn:PE_pairs} An operator-subequation pair $(F, \F)$ is said to be {\em proper elliptic} if the pair $(F, \F)$ is $\cM_0 = \cN \times \{0\} \times \cP$-monotone in the sense of Definition \ref{defn:PEP}.	
\end{defn}

\begin{rem}\label{rem:ED} For certain cases when $F$ and $\F$ depend on the jet variable $r$, it is occasionally interesting to drop the requirement of negativity (N). Such is the case for questions concerning {\em generalized principle eigenvalues} (see \cite{BGI18} and \cite{BP21} and the references therein). In such cases, one simply  requires that $\F$ satisfy properties (P) and (T) and that $F$ be $\cM = \{0\} \times \{0\} \times \cP$-monotone on $\F$. In this case, we say that $(F, \F)$ is {\em (degenerate) elliptic}. Note that in the reduced zero order free case, elliptic is the same as proper elliptic since the negativity (N) is automatic. 
\end{rem}

\begin{cor}\label{cor:PEP}
	An operator-subequation pair $(F,\F)$ is proper elliptic if and only if the upper level sets $\F_c := \{ J \in \F: \ F(J) \geq 0 \}$ are $\cM_0 = \cN \times \{0\} \times \cP$-monotone for all admissible levels $c \in F(\F)$. 
	\end{cor}

In the rest of this section, we will consider only compatible proper elliptic operator-subequation pairs in the sense of Definitions \ref{defn:compatible_pair} and \ref{defn:PEP}. In light of Corollary \ref{cor:PEP}, for every admissible level $c \in F(\F)$, the upper level set 
\begin{equation}\label{PEULS}
\mbox{	$\F_c:= \{J \in \F: F(J) \geq c \}$ is closed, non-empty and satisfies (P), (N).} 
\end{equation}
This means that each $\F_c$ is almost a subequation. One needs only the topological property (T). For this purpose,  we place an additional structural condition on $F$, which is easy to verify; for example, it is obviously satisfied if the operator $F$ is real analytic.

\begin{defn}\label{defn:tameness} A proper elliptic operator $F \in C(\F,\R)$ is said to be {\em topologically tame} if the level set
	\begin{equation}\label{levelset}
	 \F(c):= \{ J \in \F: F(J) = c\}
	 \end{equation} 
	 has empty interior for every admissible level $c \in F(\F)$. 
	\end{defn}
This condition rules out obvious pathologies. For example, if $v$ is a local $C^2$ solution near $x_0 \in \R^n$ to $F(v,Dv,D^2v) = c$ with $J^2_{x_0}v \in \Int \, \F(c) \neq \emptyset$, then for all $C^2$ functions $\varphi$ with sufficiently small $C^2$-norm, $u:= v + \varphi$ is also a local solution to $F(v,Dv,D^2v) = c$. Moreover, such a $v$ always exists in this pathological case. For example, pick a 2-jet $J \in \Int \, \F(c)$ and let $\varphi$ be the quadratic polynomial with 2-jet $J$ at $x_0$. Then picking $\varphi$ with compact support near $x_0$ and small $C^2$ norm, one has lots of counterexamples to comparison. 

Some strict monotonicity for the operator $F$ provides a convenient way to rule out such pathologies, which we now discuss. As in \eqref{PEULS} and in \eqref{levelset}, we will denote by $\F_c:= \{J \in \F: F(J) \geq c \}$ and $\F(c):= \{ J \in \F: F(J) = c\}$.

\begin{thm}[Topological tameness equivalences]\label{thm:tameness}
	Suppose that $(F, \F)$ is a compatible proper elliptic operator-subequation pair which is $\cM$-monotone for some convex cone subequation $\cM$. Then the following are equivalent:
	\begin{itemize}
		\item[1)] $F$ is topologically tame; that is, for each admissible level $c \in F(\F)$, the level set $\F(c)$ has no interior;
		\item[2)] $F(J + J_0) > F(J)$ for each $J \in \F$ and each $J_0 \in \Int \, \cM$;
		\item[3)] For some $J_0 \in \Int \, \cM$, $F(J + tJ_0) > F(J)$  for each $J \in \F$ and for each $t>0$;
		\item[4)]$\{ J \in \F: F(J) > c\} = \Int \, \F_c$ for each admissible level $c \in F(\F)$;
		\item[5)] $\F(c) = \F_c \cap \left( - \wt{\F}_c \right)$ for each admissible level $c \in F(\F)$.
	\end{itemize}
	\end{thm}

Before presenting the proof, some remarks are in order. The condition 2) is a strict version of the hypothesis that the operator $F$ is $\cM$-monotone. Condition 3) says that it is enough to have this strictness along the rays determined by a single $J_0 \in \Int \, \cM$. Condition 4), by making use of the definitions, is the statement that: {\em for every $c \in F(\F)$ and every lower-semicontinuous function $w$} 
$$
\mbox{ $4)^{\prime}$ \ \ \ \  {\em $w$ is a supersolution of $F = c \ \ \Longleftrightarrow \ \ -w$ is $\wt{\F}_c$-subharmonic,}}
$$ 
as will be made precise in Definition \ref{defn:AVSolns} and Theorem \ref{cor:AVSolns} below.

\begin{proof}[Proof of Theorem \ref{thm:tameness}] 
There are seven implications to check. The proof that 1), 2) and 3) are equivalent does not use the assumption that $F$ and $\F$ are compatible.
	
\noindent \underline{$2) \Rightarrow 3)$}: This is obvious.

\noindent \underline{$3) \Rightarrow 1)$}: Assume that 1) is false; that is, for some $c \in \R$ there exists $J \in \Int \, \F(c)$. Then given $J_0 \in \Int \, \cM$ with $J_0 \neq 0$, one has $J + t J_0 \in \Int \, \F(c)$ for each $t > 0$ sufficiently small, which contradicts 3).

\noindent \underline{$1) \Rightarrow 2)$}: Assume that 2) is false; that is, there exist $J_1 \in \F$ and $J_0 \in \Int \, \cM$ such that $F(J_1 + J_0) \leq F(J_1)$.
Since $F$ is $\cM$-monotone, one has $F(J_1 + J_0) \geq F(J_1)$ and hence
\begin{equation}\label{tame1}
	F(J_1 + J_0) = F(J_1) := c \ \ \text{for some} \ J_1 \in \F \ \text{and for some} \  J_0 \in \Int \, \cM.
\end{equation}
We will show that for $\veps > 0$ sufficiently small, $J_1 + (1 - \veps)J_0 \in \Int \, \F(c)$, which contradicts 1).

First note that 
\begin{equation}\label{tame2}
\mbox{ if $J \in \cU_1 := J_1 + \Int \, \cM$ then $J \in \F$ and $c:= F(J_1) \leq F(J)$}
\end{equation}
by the $\cM$-monotonicity of the operator $F$. 

Second, since $J_0 \in \Int \, \cM$, $J_0 - \mathcal{B} \subset \cM$ if $\mathcal{B}$ is a small ball about the origin in $\J^2$. Next we prove that
 \begin{equation}\label{tame3}
 \mbox{ if $J \in \cU_2 := J_1 + J_0 - \left( \mathcal{B} \cap \Int \, \cM \right)$ then $J \in \F$ and $F(J) = c$.}
 \end{equation}
 Suppose that $J:= J_1 + J_0 - J'$ with $J' \in \left( \mathcal{B} \cap \Int \, \cM \right)$. Then $J \in \F$ since $J_1 \in \F$ and $J_0 - J' \in \cM$. Moreover, since $J' \in \cM$ one has
 $$
 	F(J) \leq F(J + J') = F(J_1 + J_0) = c,
 $$
 again using that $F$ is $\cM$-monotone. Hence the open set
  \begin{equation}\label{tame4}
 \cU_1 \cap  \cU_2 \subset \Int \, \F(c). 
 \end{equation}
 Finally, $J_1 + (1 - \veps)J_0 \in \cU_1$ if $0 < \veps < 1$ and $J_1 + (1 - \veps)J_0 \in \cU_2$ if $\veps$ is small enough to ensure $\veps J_0 \in \mathcal{B}$, proving that $ \cU_1 \cap  \cU_2 \neq \emptyset$.
	
\noindent \underline{$2) \Rightarrow 4)$}: the compatibility assumption \eqref{compatible2} of Definition \ref{defn:compatible_pair} can be rephrased, since $\F = \F_{c_0}$, as $\partial \F:= \F \setminus \Int \, \F = \{ J \in \F: \ F(J) = c_0 \}$; that is,
\begin{equation}\label{COSP}
	\Int \, \F = \{ J \in \F: \ F(J) > c_0 \}.
\end{equation}
Now, by \eqref{COSP}, for $c \geq c_0$,
\begin{equation}\label{FSLS}
	\{ F > c \} := \{ J \in \F: \ F(J) > c \} = \{ J \in \Int \, \F: \ F(J) > c \}
\end{equation}
is an open subset of $\J^2$. Since $\{ F > c\}$ is contained in $\F_c$, it is part of the interior of $\F_c$. 

Conversely, suppose that $J \in \Int \, \F_c$. Pick $J_0 \in \Int \, \cM$. Then for $\veps > 0$ sufficiently small $J - \veps J_0 \in \F_c$. By 2), one has
$$
	F(J) = F((J - \veps J_0) + \veps J_0) > F(J - \veps J_0) \geq c,
$$
and hence $J \in \{ F > c\}$.

\noindent \underline{$4) \Rightarrow 1)$}: If 1) were false, then for some $c \in \R$ there would exist an open set $\cU \subset \F(c) \subset \F_c$. Thus $\cU \subset \Int \, \F_c$ but $\cU \not\subset \{ J \in \F: \ F(J) > c\}$, so that 4) would be false.

\noindent \underline{$4) \Rightarrow 5)$}: One has
$$
	\F_c \cap \left( - \wt{\F}_c \right) = \{ J \in \F: \ F(J) \geq c\} \cap \left( \sim \Int \, \F_c \right),
$$
which, by 4), equals $\{J \in \F: F(J) \geq c \ \text{and} \  F(J) \leq c \} := \F(c)$.

\noindent \underline{$5) \Rightarrow 1)$}: Suppose that 1) is false. Then for some $c \in \R$ there is an open set $\cU \subset \F(c)$. Hence $\cU \subset \Int \, \F_c$ so that $\F(c) \subset \, \sim \left( \Int \, \F_c \right)$ is false, which contradicts 5).
\end{proof}	

\begin{rem}\label{rem:tameness} Under the hypotheses of the theorem, if $F$ is also topologically tame, then it follows easily from 4) that for each $c \in F(\F)$,  the upper level set $\F_c$ is a subequation.
	\end{rem}

\subsection{The correspondence principle for compatible pairs}

 We now discuss an important consequence of Theorem \ref{thm:tameness} which will be essential for our treatment of comparison for classes nonlinear operators in the next section. Recall that $\varphi$ is a {\em $C^2$  (upper/lower) test function for $u$ at $x_0$} if
$$
\mbox{$u - \varphi \gtreqless 0$ \ near \ $x_0$ \ and \ $u - \varphi = 0$ \ at \ $x_0$.}
$$
We will denote by $J^{2, \pm}_{x_0}u \subset \J^2$ the spaces of {\em (upper/lower) test jets for $u$ at $x_0$}; that is, the set of all $J = J^2_{x_0} \varphi$ where $\varphi$ is a $C^2$ (upper/lower) test function for $u$ at $x_0$.

\begin{defn}\label{defn:AVSolns} Let  $(F, \F)$ be a compatible operator-subequation pair as in Definition \ref{defn:compatible_pair}. Let $\Omega$ be a domain in $\R^n$ and let $c \in F(\F)$ be an admissbile level.
	\begin{itemize}
		\item[(a)] A function $u \in \USC(\Omega)$ is said to be an {\em $\F$-admissible viscosity subsolution of $F(u,Du,D^2u) = c$ in $\Omega$} if for every $x_0 \in \Omega$ one has
		\begin{equation}\label{AVSub}
		\mbox{$J \in J^{2, +}_{x_0}u \ \ \Rightarrow \ \   J \in \F$ \ \ \text{and} \ \ $F(J) \geq c$.}
		\end{equation}
			\item[(b)] A function $u \in \LSC(\Omega)$ is said to be an {\em $\F$-admissible viscosity supersolution  of $F(u,Du,D^2u) = c$ in $\Omega$} if
		\begin{equation}\label{AVSuper}
			\mbox{$J \in J^{2, -}_{x_0}u  \ \ \Rightarrow$ \ \ either [ $J \in \F$ and \ $F(J) \leq c$\, ] \ or \ $J \not\in \F$.}
		\end{equation}
	\end{itemize}
	We say that $u \in C(\Omega)$ is an {\em $\F$-admissible viscosity solution of $F(u,Du,D^2u) = c$ in $\Omega$} if both (a) and (b) hold.
	\end{defn}
Note that in the unconstrained case ($\F = \J^2$) these definitions of $\J^2$-admissible viscosity (sub, super) solutions are standard and called merely {\em viscosity (sub, super) solutions}, respectively. In the constrained case ($\F \subsetneq \J^2$), we are taking a systematic approach to what is often done ad-hoc for particular examples. 

The following result formalizes the previous considerations in order to illustrate a general situation in which the potential theoretic approach in terms of subequation constraint sets $\F_c = \{ J \in \F: \ F(J) \geq c \}$ corresponds to the PDE approach of $\F$-admissible viscosity solutions to $F = c$. 

\begin{thm}[The Correspondence Principle for Compatible Pairs]\label{cor:AVSolns} 	Suppose that $(F, \F)$ is a compatible proper elliptic operator-subequation pair which is $\cM$-monotone for some convex cone subequation $\cM$. Suppose also that $F$ is topologically tame. Then for every admissible level  $c \in F(\F)$ and for every domain $\Omega \subset \R^n$ one has:
\begin{itemize}
	\item[(a)] $u \in \USC(\Omega)$ is an $\F$-admissible viscosity subsolution of $F(u,Du,D^2u) = c$ in $\Omega$ if and only if $u$ is $\F_c$-subharmonic on $\Omega$;
\item[(b)] $u \in \LSC(\Omega)$ is an $\F$-admissible viscosity supersolution of $F(u,Du,D^2u) = c$ in $\Omega$ if and only if $u$ is $\F_c$-superharmonic on $\Omega$, which is, by Definition \ref{defn:FH},  saying that $-u$ is $\wt{\F}_c$-subharmonic on $\Omega$.
\end{itemize}
In particular, for every admissible level  $c \in F(\F)$, one has comparison for the subequation $\F_c$ on a domain $\Omega$ if and only if one has comparison for the equation $F(u, Du, D^2 u) = c$ on $\Omega$
\end{thm}

We recall that by comparison we mean the validity of the comparison principle 
\begin{equation}\label{AVCP}
u \leq w \ \ \text{on} \ \ \partial \Omega \ \ \Rightarrow \ \ u \leq w \ \ \text{on} \ \ \Omega
\end{equation}
for all pairs $u \in \USC(\overline{\Omega})$ and $w \in \LSC(\overline{\Omega})$ which satisfy respectively (a) and (b). 

\begin{proof}[Proof of Theorem \ref{cor:AVSolns}]
For part (a),  the definition \eqref{AVSuper} of $u \in \USC(\Omega)$ being an $\F$-admissible subsolution in $x_0$ is equivalent to the statement that
\begin{equation}\label{AVS0}
J^{2,+}_{x_0}u \subset  \{ J \in \F: F(J) \geq c \} := \F_c,
\end{equation}
where $J^{2,+}_{x_0}u \subset \F_c$ defines $\F_c$-suharmonicity in $x_0$.

For part (b), the definition \eqref{AVSuper} of $u \in \LSC(\Omega)$ being an $\F$-admissible supersolution in $x_0$  is equivalent to the statement that
\begin{equation}\label{AVS1}
J^{2,-}_{x_0}u \subset \{ J \in \F: F(J) \leq c \} \cup (\sim \F).
\end{equation}
Since $F$ is topologically tame and $\cM$-monotone for some convex cone subequation $\cM$, point 4) of Theorem \ref{thm:tameness} yields
\begin{equation}\label{AVS2}
\Int \, \F_c = \{ J \in \F: F(J) > c \}
\end{equation}
and hence
\begin{equation}\label{AVS3}
\sim \Int \, \F_c = \{ J \in \F: F(J) \leq c \} \cup (\sim \F).
\end{equation}
Combining \eqref{AVS1} and \eqref{AVS3} yields
$$
J^{2,-}_{x_0}u \subset \sim \Int \, \F_c,
$$
which by Remark \ref{Fsuper} is one way to define that $u \in \LSC(\Omega)$ is $\F_c$-superharmonic.
\end{proof}

In the following subsections \ref{subsec:structure}, \ref{subsec:canonical} and \ref{subsec:lipschitz}, we will discuss how the monotonicity property $\F + \cM \subset \F$ provides additional structure with important consequences for any subequation $\F$ which admits a monotonicity subequation cone $\cM$. This {\em Structure Theorem} (see Theorem \ref{thm:structure}) provides the existence of a {\em canonical operator} $F \in C(\J^2)$ associated to any such $\cM$-monotone subequation $\F$. Such operators give rise to a rich class of examples for which the Correspondence Principle of Theorem \ref{cor:AVSolns} applies. Moreover, the Structure Theorem yields a uniqueness result for  {\em subequation branches} of a given equation $\cH \subset \J^2$ (see Corollary \ref{cor:structure} in subsection \ref{subsec:structure} and Proposition \ref{prop:branches} in subsection \ref{subsec:branches}). The beautiful class of {\em G{\aa}rding polynomial operators} is discussed in subsection \ref{subsec:garding}. This will further expand (in section \ref{sec:CP_operators}) applications of the Correspondence Principle  for obtaining additional comparison principles for compatible operator-subequation pairs (in both constrained and unconstrained cases).

\subsection{A structure theorem derived from subequation monotonicity}\label{subsec:structure}

The family of lines in $\J^2$ in any fixed direction $J_0 \in \Int \, \cM$ provides structure to a subequation $\F$ which admits the monotoncity cone subequation $\cM$. We recall that part of the definition of $\F \subset \J^2$ being a subequation is that $\F \neq \emptyset, \J^2$ and part of the definition of $\cM$ being a monotonicity cone subequation is that $\cM$ has non empty interior. The following fundamental result is contained in part (2) of Lemma 9.9 in \cite{HL11}, which was stated for manifolds but not proven there. See also Theorem 3.2 of \cite{HL10} for the construction in the pure second order case.

\begin{thm}[The Structure Theorem]\label{thm:structure}
	Suppose that $\F \subset \J^2$ is a subequation constraint set which admits a monotonicity cone subequation $\cM$. Fix $J_0 \in \Int \, \cM$. Given $J \in \J^2$ arbitrary, the set
	\begin{equation}\label{Interval_J}
	I_J := \{ t \in \R: \ J + tJ_0 \in \F \} 
	\end{equation}
	is a closed interval of the form $[t_J, +\infty)$ with $t_J \in \R$ (finite). Moreover
	\begin{itemize}
		\item[(a)] $J + t J_0 \not\in \F \ \iff \ t < t_J$;
		\item[(b)] $J + t_J J_0 \in \partial \F$;
		\item[(c)] $ J + t J_0 \in \Int \, \F \ \iff \ t > t_J$;
		\end{itemize}
	and any one of the relations (a), (b) or (c) uniquely determines $t_J \in \R$ from $J \in \J^2$ and $J_0 \in \Int \, \cM$ in the sense that
	\begin{equation}\label{tJ1}
	\mbox{$t_J$ is the unique element of $\R$ for which  $J + t_J J_0 \in \partial \F$}
	\end{equation}
	and 
	\begin{equation}\label{tJ2}
		t_J = \inf \{ t \in \R: J + t_J J_0 \in \Int \, \cM \} = \sup  \{ t \in \R: J + t_J J_0 \not\in \F \}.
	\end{equation}
	\end{thm}

\begin{proof}
With $J \in \J^2$ arbitrary, we first show that $I_J = [t_J, +\infty)$ for a unique value $t_J \in \R$. The proof involves four steps.

\noindent {\em Step 1: One has $J + t J_0 \in \cM$ for all $t$ sufficiently large.} \\
Indeed, since $\cM$ is a cone, for $t > 0$  this is equivalent to having
$$
	\frac{1}{t} J + J_0 \in \cM \ \ \text{for all} \ t \ \text{sufficiently large},
$$
which holds since $J_0 \in \Int \, \cM$.

\vspace{1ex}
\noindent {\em Step 2: $I_J := \{ t \in \R: \ J + tJ_0 \in \F \}$ is non-empty.} \\
Indeed, since $\F \neq \emptyset$ we can pick any $J_1 \in \F$ and then notice that
$$
	J + tJ_0 = J_1 + (J - J_1) + tJ_0 \in \F + \cM \subset \F \ \ \text{for all} \ t \ \text{sufficiently large},
$$
since $\F$ is $\cM$-monotone and $(J - J_1) + tJ_0 \in \cM$ for all large $t$ by Step 1.

\vspace{1ex}
\noindent {\em Step 3: If $t \in I_J$ then $t+s \in I_J$ for each $s \geq 0$.}\\
It is enough to notice that
$$
J + (t + s)J_0 = (J + tJ_0) + sJ_0 \in \F + \cM \subset \F, 
$$
since $\F$ is $\cM$-monotone. 

By Step 2 and Step 3, one has either $I_J = [t_J, +\infty)$ or $I = \R$.

\vspace{1ex}
\noindent {\em Step 4: One has $I_J   \neq \R$ and hence $I_J = [t_J, +\infty)$ for some $t_J \in \R$.}\\
Suppose not. Then $J + tJ_0 \in \F$ for all $t \in \R$. Let $J' \in \J^2$ be arbitrary. By Step 1, there exists $s \geq 0$ such that
$$
	(J' - J) + sJ_0 \in \cM
$$ 
and hence
\begin{equation}\label{I_not_R}
	(J + tJ_0) + (J' - J + sJ_0) \in \F + \cM \subset \F \ \ \text{for all} \ t \in \R.
\end{equation}
Taking $t = -s$ in \eqref{I_not_R} yields $J' \in \F$ for arbitrary $J' \in \F$, which contradicts $\F \neq \J^2$. 

It remains to verify the claims (a), (b) and (c). Claim (a) follows from the fact that by construction
\begin{equation}\label{Min_Interval_J}
t_J = \min \{ t \in \R: \ J + tJ_0 \in \F\}.
\end{equation}

For claim (b), notice that $J + (t_J - \veps)J_0 \not\in \F$ for each $\veps > 0$ and hence $J + t_J J_0 \not\in \Int \, \F$. Therefore
$$
	J + t_J J_0 \in \F \setminus \Int \, \F = \partial \F.
$$

For claim (c), if $s > 0$ then
$$
J + (t_J + s)J_0 = (J + t_J J_0) + sJ_0 \in \partial \F + \Int \, \cM \subset \F + \Int \, \cM \subset \Int \, \F,
$$ 
by the set identity \eqref{SI1} of Proposition \ref{prop:subequation_cones}. This proves the implication $( \Leftarrow)$ of claim (c). However, by claims (a) and (b), one has
$$
 t \leq t_J \ \Rightarrow  J + tJ_0 \in \partial \F \cup (\sim \F) = \, \sim \Int \, \F,
$$
which is contrapositive to the implication $(\Rightarrow)$ of claim (c).

Finally, the formulas \eqref{tJ1} and \eqref{tJ2} follow from (a), (b) and (c).
\end{proof}

An important corollary of the Structure Theorem \ref{thm:structure} is that $\cM$-monotone subequations $\F$
are uniquely determined by their boundaries $\partial \F$.

\begin{cor}\label{cor:structure}
		Suppose that $\F \subset \J^2$ is a subequation constraint set which admits a monotonicity cone subequation $\cM$. Then
		\begin{equation}\label{F_from_boundary}
		\F = \partial \F + \cM.
		\end{equation}
	\end{cor}

\begin{proof}
	Since $\F + \cM \subset \F$, we have $\partial \F + \cM \subset \F$. For the reverse inclusion, fix any $J_0 \in \Int \, \cM$. Given $J \in \F$, by the $\cM$-monotonicity of $\F$ with respect to the cone $\cM$, one has
	\begin{equation}\label{cor_mono}
	J + t J_0 \in \F \ \ \text{for every}  \ t \geq 0.
	\end{equation}
	Since $t_J$ is the minimal $t \in \R$ for which $J + t J_0 \in \F$, one has $t_J \leq 0$ and hence
	$$
		J = (J + t_J J_0) -t_J J_0 \in \partial \F + \cM. 
	$$
\end{proof}

The Structure Theorem \ref{thm:structure} also provides a tool for showing the existence of {\em canonical operators} as well as {\em graphing functions for boundaries}  $\partial \F$ under the monotonicity assumption of this structure theorem. 

\subsection{Canonical operators for subequations with monotonicity}\label{subsec:canonical} The Structure Theorem \ref{thm:structure} provides a canonical procedure for constructing an operator $F \in C(\J^2)$ with ``nice'' properties associated to a given subequation $\F$ as long as $\F$ admits a monotonicity cone subequation $\cM$ (and one fixes an element $J_0 \in \Int \, \cM$). First, $(F, \F)$ is a compatible operator-subequation pair with monotonicity $\cM$, providing $\F$ with at least one compatible operator which is natural. Second, $F$ is defined on all of $\J^2$ and $(F, \J^2)$ will be shown to be a compatible (unconstrained case) proper elliptic operator-subequation pair in the sense of Definition \ref{defn:compatible_pair} where the operator $F$ topologically tame on $\J^2$. This gives a rich family (including $(F, \F)$) of pairs for which the Correspondence Principle of Theorem \ref{cor:AVSolns} holds. The canonical operator $F$ is closely related to the potential theory equation $\partial \F$. For \underline{any} hyperplane (co-dimension one subspace) $W_0 \subset \J^2$ transverse to the line (one dimensional subspace) $[J_0]$ through $J_0$, the restriction $g:= F \, \vline _{W_0} :W_0 \to \R$ is the unique graphing function for $\partial \F$ over $W_0$; that is
$$
	\partial \F : \{ J + g(J)J_0: \ \ J \in W_0 \}.
$$
A judicious choice of the hyperplane $W_0$ results in the Lipschitz regularity of $g$ and $F$ with respect to a natural seminorm, as will be discussed in the next subsection.

The canonical operator is defined in terms of the Structure Theorem as follows.

\begin{defn}\label{defn:canonical_op} Suppose that $\F \subset \J^2$ is a subequation constraint set which admits a monotonicity cone subequation $\cM$. For fixed $J_0 \in \Int \, \cM$, the {\em canonical operator for $\F$ (determined by $J_0$)} $F: \J^2 \to \R$ is defined by 
\begin{equation}\label{canonical_op}
\mbox{$F(J) := - t_J$ \ \  where $t_J$ is defined by \eqref{tJ1}} 
	\end{equation}
(or by either of the two formulas in \eqref{tJ2}).
	\end{defn}

We proceed to analyze the properties of the canonical operator outlined above, beginning with some structural properties.

\begin{prop}[Structural properties of the canonical operator]\label{prop:canonical_structure}
	Suppose that $F$ is the canonical operator for $\F$ (determined by a fixed $J_0 \in \Int \, \cM$). Then
	\begin{itemize}
		\item[(a)] $F$ decomposes $\J^2$ into three disjoint pieces:
		\begin{equation}\label{canonical_op1}
		 \partial \F = \{F(J) = 0\}, \ \  \Int \, \F = \{  F(J) > 0\} \ \	 \text{and} \ \	\J^2 \setminus \F = \{ F(J) < 0\};
		 \end{equation}
		\item[(b)] $F$ is strictly increasing in the direction $J_0$, in fact
		\begin{equation}\label{canonical_op2}
		F(J + t J_0) = F(J) + t, \ \ \text{for each} \ J \in \J^2 \ \text{and for each} \ t \in \R;
		\end{equation} 
		\item[(c)] $F$ is proper elliptic on $\J^2$, in fact, $F$ is $\cM$-monotone on $\J^2$; that is,
			\begin{equation}\label{canonical_op3}
		F(J + J') \geq F(J), \ \ \text{for each} \ J \in \J^2 \ \text{and for each} \ J' \in \cM.
		\end{equation}
	
	\end{itemize}
\end{prop}

\begin{proof} The decomposition \eqref{canonical_op1} of $\J^2$ is a restatement of (a), (b) and (c) of the Structure Theorem \ref{thm:structure}, as follows. First consider those $J \in \J^2$ such that $F(J) := -t_J = 0$. By Theorem \ref{thm:structure}, we have
	$$
	J \in \partial \F, \ \ J + tJ_0 \not\in \F \ \text{for all} \ t < 0 \ \ \text{and} \ \ J + tJ_0 \in \Int \, \F \ \text{for all} \ \ t > 0
	$$
	and hence $\partial \F = \{F(J) = 0\}$, as desired. For arbitrary $J \in \J^2$, use the definition of $F$ and the relations (a) and (c) of Theorem \ref{thm:structure} to find
		$$
	F(J) < 0 \ \iff \ t_J > 0 \ \iff \ J \not\in \F
	$$
	and
	$$
		F(J) > 0 \ \iff \ t_J < 0 \ \iff \ J \in \Int \, \F.
	$$
	
	Next, using the definition of $F$, the property \eqref{canonical_op2} requires showing that for every $J \in \J^2$ 
	$$
		t_{(J + tJ_0)} = t_J - t \ \ \text{for every} \ t \in \R.
	$$
	By construction $t_J$ and $t_{(J + tJ_0)}$ are the unique real numbers such that
	$$
	J + t_J J_0 \in \partial \F \ \ \text{and} \ \ (J + tJ_0) + t_{(J + t J_0)} J_0 \in \partial \F
	$$
	and hence $t_J = t + t_{(J + tJ_0)}$, proving \eqref{canonical_op2}.
	
	Finally, we note that the property \eqref{canonical_op2} of part (b) yields
	$$
	F(J + t J_0) > F(J) \ \ \text{for each} \ J \in \J^2 \ \text{and for each} \ t > 0,
	$$ 
	which is condition 3) of Theorem \ref{thm:tameness} concerning topological tameness. Hence all of the other equivalent forms 1), 2), 4) and 5) hold, where the strict monotonicity condition 2) is stronger than the condition \eqref{canonical_op3}. 
	\end{proof}

The following result shows that property \eqref{canonical_op2} plus any one of the relations in \eqref{canonical_op1} uniquely  determine $F$ (for $J_0 \in \Int \, \cM$ fixed) and hence they could be taken as defining properties for the canonical $F$ determined by $J_0$.

\begin{prop}\label{prop:unique_F} Suppose that for some $J_0 \in \J^2$, an operator $F: \J^2 \to \R$ satisfies the affine property \eqref{canonical_op2}: that is,
		\begin{equation}\label{affine_property} 
		F(J + t J_0) = F(J) + t, \ \ \text{for each} \ J \in \J^2 \ \text{and for each} \ t \in \R.
		\end{equation}
		Suppose that $\F$ is a subequation which admits a monotonicity cone subequation $\cM$ with $J_0 \in \Int \, \cM$. If any one of the following relations holds
		$$
		\text{a)} \  \partial \F = \{F(J) = 0\}, \ \ \text{b)} \  \Int \, \F = \{  F(J) > 0\} \ \	 \text{or} \ \	\text{c)} \ \J^2 \setminus \F = \{ F(J) < 0\},
		$$
		then $F$ is the canonical operator (determined by $J_0$) for $\F$.
	\end{prop}

\begin{proof}
With $J_0 \in \Int \, \cM$ fixed, the affine property \eqref{affine_property} shows that for each $J \in \J^2$, the restriction of $F$ to the line $\ell_J = \{ J + t J_0: \ t \in \R\}$ is continuous, strictly increasing and has range equal to $\R$. Hence there is a unique value $t^* \in \R$ such that
\begin{equation}\label{CO1}
0 = F(J + t^* J_0) =  F(J) + t^*.
\end{equation}
By Definition \ref{defn:canonical_op}, we only need to show that $t^* = t_J$ where $t_J \in \R$ is the critical parameter in Theorem \ref{thm:structure} which divides $\ell_J$ into the three pieces ( a unique point on $\partial \F$, an open ray in $\Int \, \cM$ and an open ray in $\J^2 \setminus \F$). The three relations a), b) and c) imply that $t^* = t_J$ as defined by \eqref{tJ1} and the first and second fomulas of \eqref{tJ2} respectively.
\end{proof}

Next, we observe that an immediate consequence of the Structure Theorem \ref{thm:structure} and the definition of the canonical operator $F$ for $\F$ is that the {\em equation} $\partial \F = \F \cap (-\wt{\F})$ can be graphed and the canonical operator can be recovered from the graphing function $g$.

\begin{prop}[Canonical operators and graphing the equation $\partial \F$]\label{prop:graphing_function} Suppose that $\F \subset \J^2$ is a subequation constraint set which admits a monotonicity cone subequation $\cM$. Let $F$ be the canonical operator for $\F$ determined by $J_0 \in \Int \, \cM$.  Fix $W_0 \subset \J^2$ a hyperplane in $\J^2$ transversal to $[J_0]$. Then one has:
	\begin{itemize}
		\item[(a)] The equation $\partial \F \subset  W_0 \oplus [J_0]$ is the graph of $g: W_0 \to \R$ defined by
\begin{equation}\label{graphing_function}
	g(J') := - F(J'), \ \ J' \in W_0,
\end{equation}
which is to say that $g:= -F_{|W_0}$, or, equivalently 
	\begin{equation}\label{graph_boundary}
\partial \F  = \{\ J \in \J^2: \ J = J' + g(J') J_0  \ \ \text{where} \ J' \in W_0\};
\end{equation}
\item[(b)] The epigraph of $g$ satisfies:
	
		\begin{equation}\label{epigraph}
	\F  = \{\ J \in \J^2: \ J = J' + t J_0  \ \ \text{where} \ \ t \geq g(J') \ \ \text{and} \ J' \in W_0\},
	\end{equation}
	\item[(c)] The canonical operator $F$ is recovered from the graphing function $g$ by
	\begin{equation}\label{F_from_g}
	F(J) = F(J' + t J_0) = t - g(J'), \ \ \text{for each} \ J \in \J^2.
	\end{equation}

\end{itemize}
	\end{prop}

\begin{proof} Consider the splitting $\J^2 = W_0 \oplus [J_0]$. Each $J \in \J^2$ can be decomposed uniquely into 
	\begin{equation}\label{splitting}
	J = J' + tJ_0 \ \ \text{with} \ J' \in W_0 \ \text{and} \ t \in \R.
	\end{equation}
	With respect to this decomposition \eqref{splitting}, the Structure Theorem \ref{thm:structure} with $J' \in W_0 \subset \J^2$ arbitrary says that 
	$$
	 J' + t J_0 \in \partial \F \ \ \iff \ \ t = t_{J'} = - F(J') = g(J') 
	$$
	and
	$$
	 J' + t J_0 \in \partial \F \ \ \iff \ \ t \geq t_{J'} = - F(J') = g(J'),
	$$
	which gives the claims in parts (a) and (b). The part (c) is immediate. 
\end{proof}

We summarize the above considerations by noting that canonical operators $F$ associated to subequations $\F$ give a natural way to form  (unconstrained) compatible pairs where $F$ is also topologically tame and hence the Correspondence Principle of Theorem \ref{cor:AVSolns} applies at every level $c \in \R$.

\begin{thm}[Canonical operators and compatible pairs]\label{thm:CanOp_ComPair} 
	Suppose that a subequation $\F$ admits a monotonicity cone subequation $\cM$. Let $F \in C(\J^2)$ be the canonical operator for $\F$ determined by any fixed $J_0 \in \Int \, \cM$. Then: 
\begin{itemize}
	\item[(a)] $(F, \J^2)$ is a compatible proper elliptic operator-subequation pair;
	\item[(b)] $F(\J^2) = \R$ and the operator $F$ is topologically tame; 
	\item[(c)] for each $c \in \R$, the set $\F_c:= \{ J \in \J^2: \ F(J) \geq c \}$ is a subequation constraint set with $\F_0 = \F$ and the pair $(F, \F_c)$ satisfies the compatibility conditions 
$$
	\inf_{\F_c} F = c \quad \text{and} \quad \partial \F_c = \{J \in \F_c: \ F(J) = c \}.
	$$
\end{itemize} 
In addition, the canonical operator (determined by $J_0 \in \Int \, \cM$) for the dual subequation $\wt{\F}$ is given by 
	\begin{equation}\label{dual_operators}
	\wt{F}(J) := - F(-J) \ \ \text{for all} \ J \in \J^2, \ \ \text{where also} \ \wt{\F}_c = \F_{-c},
	\end{equation}
and the analogous statements of (a), (b) and (c) hold for $(\wt{F}, \J^2)$ and $(\wt{F}, \wt{\F}_c)$.
\end{thm}

\begin{proof} The proper ellipticity of $F$ on $\J^2$ is a consequence of Proposition \ref{prop:canonical_structure} (c). By Definition \ref{defn:compatible_pair}, the pair $(F, \J^2)$ is compatible if $F \in C(\J^2)$ and this follows from the fact that $F$ is actually Lipschitz continuous, as will be proven in the following subsection. Next, as noted in the proof of part (c) of Proposition \ref{prop:canonical_structure}, $F$ is topologically tame in the sense of Definition \ref{defn:tameness} since the canonical operator satisfies the affine property \eqref{canonical_op2}. The affine property also shows that $F(\J^2) = \R$. 
	
For part (c), each $\F_c$ is closed by the continuity of $F$ and $\F_c$ is non-empty and not all of $\J^2$ since $F(\J^2) = \R$.
Since the pair $(F, \J^2)$ is $\cM$-montone, each $\F_c$ is $\cM$-monotone by Lemma \ref{lem:upper_levels} for the monotonicity cone subequation $\cM$. Hence $\F_c$ satisfies properties (P), (N) and (T) (see Proposition \ref{prop:subequation_cones}). The compatibility claim of part (c) follows from statement 4) of Theorem \ref{thm:tameness} and the continuity of $F$.
	
Finally, if $\F$ is $\cM$ monotone then so is $\wt{\F}$ (using \eqref{J_invariance}) and since
$$
	-F(-J) := \min \{ t \in \R: \ -J + tJ_0 \in \F \} := t_J
$$
if $t < t_J$ one has 
$$
	-J + t J_0 \in \, \sim \Int \, \F \ \ \text{and hence} \ J - t J_0 \in \wt{\F},
$$
which shows that $-F(-J) = - \min \{ t \in \R: \ J + tJ_0 \in \wt{\F} \}$, as needed. The remaining claims for $\wt{F}_c, (\wt{F}, \J^2)$ and $(\wt{F}, \wt{\F}_c)$ then follow.
\end{proof}

Combining the compatibility result of Theorem \ref{thm:CanOp_ComPair} with the correspondence principle and the general comparison result for subequations gives the following comparison principle for canonical operators.

\begin{thm}[Comparison for canonical operators]\label{thm:CP_canonical_ops}	
Let $\F \subset \J^2$ be a subequation constraint set with monotonicity cone subequation $\cM$. Suppose that $\cM$ admits a strict approximator $\psi$ on a bounded domain $\Omega$; that is, $\psi \in C(\overline{\Omega}) \cap C^2(\Omega)$ such that $J^2_x \psi \in \Int \, \cM$ for each $x \in \Omega$. Then, for each $J_0 \in \Int \, \cM$ fixed, the canonical operator $F$ for $\F$ determined by $J_0$ satisfies the comparison principle at every level $c \in \R$; that is,
\begin{equation}\label{CP_CO}
\mbox{$u \leq w$ on $\partial \Omega \ \ \Rightarrow \ \  u \leq w$ on $\Omega$}
\end{equation}
for $u \in \USC(\overline{\Omega})$ and $w \in \LSC(\overline{\Omega})$ which are respectively viscosity subsolutions and supersolutions to $F(u,Du,D^2u) = c$ on $\Omega$.
\end{thm}

\begin{proof} By Theorem \ref{thm:CanOp_ComPair}, we have that $(F, \J^2)$ is compatible proper elliptic operator-subequation pair, every $c \in \R$ is an admissible level of $F$ and $F$ is topologically tame. Hence, by the correspondence principle of Theorem \ref{cor:AVSolns}, at every level $c \in \R$ the comparison principle \eqref{CP_CO} holds for viscosity subsolutions and supersolutions of the equation $F = c$ if and only if it holds for $\F_c$-subharmonics and $\F_c$-superharmonics. Since $\cM$ is a monotonicty cone for each subsequation $\F_c$ and since $\cM$ admits a strict approximator $\psi$ on $\Omega$, comparison for $\F_c$ follows from the genrale comparison  principle of Theorem \ref{thm:CP_general}. 
\end{proof}

We conclude this subsection with a few relevant examples and constructions.

\begin{exe}\label{exe:canonical_PSO} The pure second order convexity subequation $\F = \R \times \R^n \times \cP = \cM(\cP)$ is $\cM$-monotone for $\cM  = \cM(\cP)$ and the operator
	\begin{equation}\label{can_op_convexity}
	F(J) = F(r,p,A) = \lambda_{\rm min}(A)
	\end{equation}
	is easily seen to be the canonical operator for $\F$ determined by $J_0 = (0,0,I)$ (or any $J_0 = (r_0, p_0,I)$). Similarly
	\begin{equation}\label{can_op_subaffine}
	\wt{F}(J) = \wt{F}(r,p,A) = \lambda_{\rm max}(A)
	\end{equation}
	is the canonical operator (with the same $J_0$) for the subaffine subequation $\wt{\F} = \R \times \R^n \times \wt{\cP}$,
	which is also $\cM(\cP)$-monotone.
\end{exe}

Families of subequations which are $\cM$-monotone for a fixed convex cone subequation $\cM$ have particularly nice properties.
For example, closed sets $\F$ which are $\cM$-monotone automatically satisfy the topological property and are subequations by Proposition \ref{prop:subequation_cones} (see the discussion after the proposition for counterexamples where monotonicity is lacking). Moreover, given an arbitrary family of such subequations, by considering the dual family as well, intersections and unions lead to four associated subequations with computable canonical operators, as follows. 

\begin{thm}[Canonical operators, duality, intersections and unions]\label{thm:canonical_inf}
	Suppose that $\{ \F_{\sigma} : \sigma \in \Sigma \}$ is an arbitrary family of subequations which are all $\cM$-monotone for a given monotonicity subequation cone $\cM$. Let $J_0 \in \Int \, \cM$ be fixed, but arbitrary. Let $F_{\sigma} \in C(\J^2)$ denote the canonical operator (determined by $J_0$) associated to the subequation $\F_{\sigma}$ and consider the dual operator $\wt{F}_{\sigma} \in C(\J^2)$ defined by $\wt{F}_{\sigma}(J):= - F_{\sigma}(-J)$, which is the canonical operator (determined by $J_0$) for the dual ($\cM$-monotone) subequation $\wt{\F}_{\sigma}$ by \eqref{dual_operators}. 
	\begin{itemize}
		\item[(a)] The intersection $\F  := \bigcap_{\sigma \in \Sigma} \F_{\sigma}$ (if non empty) is an $\cM$-monotone subequation with canonical operator $F \in C(\J^2)$ (determined by $J_0$) given by the infimum
			\begin{equation}\label{canonical_F}
		F(J) := \inf_{\sigma \in \Sigma} F_{\sigma}(J), \ \ J \in \J^2.
		\end{equation}
	\end{itemize}
		Applying this to the dual family $\{\wt{F}_{\sigma}\}_{\sigma \in \Sigma}$, we have that:
		\begin{itemize}
		\item[(b)] The intersection $\G:= \bigcap_{\sigma \in \Sigma} \wt{\F}_{\sigma}$ (if non empty) is an $\cM$-monotone subequation with canonical operator $G \in C(\J^2)$ (determined by $J_0$) given by the infimum
		\begin{equation}\label{canonical_G}
		G(J) := \inf_{\sigma \in \Sigma} \wt{F}_{\sigma}(J), \ \ J \in \J^2.
		\end{equation}
		\item[(c)] The closure of the union $ \overline{\bigcup_{\sigma \in \Sigma} \F_{\sigma}}$ (if not equal to all of $\J^2$) is an $\cM$-monotone subequation $\cH$ with canonical operator $H \in C(\J^2)$ (determined by $J_0$) given by the supremum
\begin{equation}\label{canonical_H}
H(J):=\sup_{\sigma \in \Sigma} F_{\sigma}(J), \ \ J \in \J^2.
\end{equation}
		In fact,
		\item[(d)] The intersection $\G:= \bigcap_{\sigma \in \Sigma} \wt{\F}_{\sigma}$ and the closure of the union $\overline{\bigcup_{\sigma \in \Sigma} \F_{\sigma}}$ (with $\G \neq \emptyset$ and $\cH:= \wt{\G} \neq \J^2$) are dual subequations with dual canonical operators (determined by $J_0$) $G$ and $H$.
	\end{itemize} 
Applying this to $\F:= \bigcap_{\sigma \in \Sigma} \F_{\sigma}$, we have that:
\begin{itemize}
		\item[(e)] The intersection $\F:= \bigcap_{\sigma \in \Sigma} \F_{\sigma}$ and the closure of the union $\overline{\bigcup_{\sigma \in \Sigma} \wt{\F}_{\sigma}}$ (with $\F \neq \emptyset$ and $\cE:= \wt{\F} \neq \J^2$) are dual subequations with dual canonical operators (determined by $J_0$) $F$ and $E$, where
			\begin{equation}\label{canonical_K}
		E(J) := \sup_{\sigma \in \Sigma} \wt{F}_{\sigma}(J), \ \ J \in \J^2.
		\end{equation}  
	\end{itemize}
\end{thm}

\begin{proof}
	We begin by noting that $\F, \G, \cH$ and $\cE$ are all $\cM$-monotone subequations (if $\F$ and $\G$ are nom empty and $\cH$ and $\cE$ are not all of $\J^2$) by Proposition \ref{prop:algebra}. Next, we recall the following consequence of Propositions \ref{prop:canonical_structure} and \ref{prop:unique_F}: given an $\cM$-monotone subequation $\F$ and given $J_0 \in \Int \, \cM$ an operator $F$ is the canonical operator for $\F$ (determined by $J_0$) if and only if one has both the affine property
	\begin{equation}\label{recall_affine}
	F(J + tJ_0) = F(J) + t, \ \ \forall \, J \in \J^2, t \in \R
	\end{equation}
	and the structural relations
		\begin{equation}\label{recall_structural}
	\text{a)} \  \partial \F = \{F(J) = 0\}, \ \ \text{b)} \  \Int \, \F = \{  F(J) > 0\}, \	\ 	\text{c)} \ \J^2 \setminus \F = \{ F(J) < 0\},
\end{equation}
where it is sufficient to have \eqref{recall_affine} and only one of the conditions in \eqref{recall_structural}.

For part (a), it remains to show that the inf operator $F$ of \eqref{canonical_F} is the canonical operator of $\F$ (determined by $J_0 \in \Int \, \cM$). Since each $F_{\sigma}$ is canonical for $\F_{\sigma}$, we have \eqref{recall_affine} for $F_{\sigma}$, which then implies the validity of \eqref{recall_affine} for the  infimum $F := \inf_{\sigma \in \Sigma} F_{\sigma}$. We will show that relation c) of \eqref{recall_structural} holds. We have
\begin{eqnarray*}
 \sim \, \F & = & \sim \, \left( \bigcap_{\sigma \in \Sigma} \F_{\sigma} \right) = \left\{ J \in \J^2: \ F_{\sigma} < 0, \ \text{for some} \ \sigma \in \Sigma \right\} \\
 & = & \{ J \in \J^2: \ F(J) < 0 \}.
\end{eqnarray*}
This completes the proof of part (a), and also of part (b) for the dual family.

The proofs of parts (c) and (d) are intertwined. Start from part (b) with the subequation $\G:= \bigcap_{\sigma \in \Sigma} \wt{\F}_{\sigma}$ and canonical operator $G:= \inf_{\sigma \in \Sigma} \wt{F}_{\sigma}$ for $\G$. Next, we prove that the canonical operator $\wt{G}$ for the dual subequation $\wt{\cG}$ is the operator $H$ defined by \eqref{canonical_H}. That is, $\wt{G} = H$.  Using the definitions, we have
\begin{eqnarray}\label{GT1}
	\wt{G}(J) &:=& -G(-J) := - \inf_{\sigma \in \Sigma} \wt{F}_{\sigma}(-J) := - \inf_{\sigma \in \Sigma} (-F_{\sigma}(J)) \nonumber \\
	& = & \sup_{\sigma \in \Sigma} F_{\sigma}(J) := H(J), \ \ \forall \, J \in \J^2.
\end{eqnarray}
Since $H = \wt{G}$ is the canonical operator for the dual subequation $\cH:= \wt{\G}$, to complete the proof of (c) and (d) we must show that
\begin{equation}\label{GT2}
	\cH = \overline{ \bigcup_{\sigma \in \Sigma} \F_{\sigma}}.
	\end{equation}
Since $H$ is the canonical operator for the $\cM$-monotone subequation $\cH$, the conditions a) - c) of \eqref{recall_structural} applied to $H$ yield \begin{equation}\label{GD1}
	\cH = \{ J \in \J^2: \ H(J) := \sup_{\sigma \in \Sigma}F_{\sigma}(J) \geq 0 \}.
\end{equation}
Hence $\F_{\sigma}:= \{J \in J^2: \ F_{\sigma}(J) \geq 0 \} \subset \cH$ for each $\sigma \in \Sigma$, which shows that 
$\overline{ \bigcup_{\sigma \in \Sigma} \F_{\sigma}} \subset \cH$, since $\cH = \wt{G}$ is closed. For the containment $\cH \subset \overline{ \bigcup_{\sigma \in \Sigma} \F_{\sigma}}$, it suffices to show that $\Int \, \cH \subset \overline{ \bigcup_{\sigma \in \Sigma} \F_{\sigma}}$ since the subequation $\cH$ satisfies the topological property $\cH = \overline{ \Int \, \cH}$. Now, by the relation b) of \eqref{recall_structural} applied to $H$, we have
$$
J \in \Int \, \cH \ \Leftrightarrow \ H(J) := \sup_{\sigma \in \Sigma} F_{\sigma}(J) > 0 \ \Leftrightarrow \ F_{\sigma'}(J) > 0\ \ \text{for some} \ \sigma' \in \Sigma,
$$
in which case $J \in \Int \, \F_{\sigma'}$, by relation b) of \eqref{recall_structural} applied to $F_{\sigma'}$. In fact, this proves that $\Int \, \cH \subset \bigcup_{\sigma \in \Sigma} \Int \, \F_{\sigma}$ so that $\Int \, \cH = \bigcup_{\sigma \in \Sigma} \Int \, \F_{\sigma}$. This completes parts (c) and (d).

Finally, part (d) applied to $\F:= \bigcap_{\sigma \in \Sigma} \F_{\sigma}$ immediately gives part (e).
\end{proof}

Note that when $\Sigma$ is a finite index set, the $\inf$ in \eqref{canonical_F} becomes a minimum. Interesting examples come from the gradient free-case, where we note that the zero order negativity subequation $\cM(\cN) = \cN \times \R^n \times \Symn$ has canonical operator (with $J_0 + (-1, p_0, A_0)$) given by\
\begin{equation}\label{canonical_op_N}
F(r,p,A) = -r. 
\end{equation}

\begin{exe}\label{exe:canonical_GF} The gradient free negative-convex subequation $\F = \cN \times \R^n \times \cP = \cM(\cN) \cap \cM(\cP)$ is $\cM$-monotone for $\cM  = \cM(\cN) \cap \cM(\cP)$ and the operator
	\begin{equation}\label{can_op_negative_convex}
	F(J) = F(r,p,A) = \min \{ -r, \lambda_{\rm min}(A) \}
	\end{equation}
	is easily seen to be the canonical operator for $\F$ determined by $J_0 = (-1,0,I)$  (or any $J_0 = (-1, p_0, I)$). By duality \eqref{dual_operators}, one can compute the canonical operator for the dual subequation to find
	$$
	\wt{F}(r,p,A) = \max \, \{ -r, \lambda_{\rm max} (A) \}
	$$
	as the canonical operator for the subaffine-plus subequation 
	$$
	\wt{\F} = \{ (r,p,A) \in \J^2: \ r \leq 0 \ \ \text{or} \ \ A \in \wt{\cP} \}.
	$$
\end{exe}

\subsection{Lipschitz regularity of subequation boundaries}\label{subsec:lipschitz}
 
Monotonicity $\cM$ for a subequation $\F$ forces the associated ``equation'' $\partial \F$ to have Lipschitz regularity.  
The relationship \eqref{F_from_g} shows that the canonical operator $F$ for $\F$ (determined by $J_0$) will be Lipschitz continuous on $\J^2$ if and only if one the graphing functions $g$ is Lipschitz. These functions actually are 1-Lipschitz with respect to natural seminorms built from the graphing function of the monotonicity cone subequation $\cM$ over $W_0$, if $W_0$ is chosen carefully. The geometric reason for this regularity is that $\cM$-monotonicity for $\F$ means that the translates of the cones $\cM$ and $-cM$ with vertices on $\partial \F$ must lie in the epigraph of $F$ and its complement, respectively. This cone pinching is the Lipschitz property. 

 To start, denote by $|| \cdot ||^+ : W_0 \to \R$ the graphing function for $\partial \cM$ over $W_0$ given by Proposition \ref{prop:graphing_function}.  That is,
\begin{equation}\label{graph_bdy_M}
\partial \cM = \{ J'  + ||J'||^+ J_0 : \ \  J' \in W_0\}.
\end{equation}
Using \eqref{epigraph}, one also has that $\cM$ is the epigraph of $|| \cdot ||^+$; that is, 
\begin{equation}\label{epigraph_M}
\cM = \{ J'  + t J_0: \ \ \text{where} \ t \geq ||J'||^+ \ \text{and} \ J' \in W_0\}.
\end{equation}
Note that, since $\cM$ is a cone, the function $||\cdot||^+$ is {\em positively homogeneous of degree one}; that is,
\begin{equation}\label{seminorm1}
||tJ||^+ = t||J||^+ \ \ \text{for each} \ t \geq 0 \ \text{and for each} \ J' \in W_0. 
\end{equation}
Since $\cM$ is a convex cone, the function $||\cdot||^+$ is also  {\em subadditive}; that is,
\begin{equation}\label{seminorm2}
||J + J'||^+ \leq ||J||^+ + || J'||^+ \ \ \text{for each pair} \ J, J' \in W_0.
\end{equation}
Such {\em sublinear functions} always come in pairs. More precisely, given $||\cdot||^+$ satisfying \eqref{seminorm1} and \eqref{seminorm2}, by defining
\begin{equation}\label{neg_seminorm}
||J||^- := ||-J||^+ \ \ \text{for each} \ J \in W_0,
\end{equation}
the function $||\cdot||^-$ also satisfies \eqref{seminorm1} and \eqref{seminorm2}. For certain choices of the transverse hyperplane $W_0$ to $J_0$, the functional $|| \cdot||^{+}$ (as well as $||\cdot||^-$) is a {\em seminorm} on $W_0$; that is, in addition to the sublinearity  \eqref{seminorm1} and \eqref{seminorm2} for $||\cdot||^{+}$, one also has
\begin{equation}\label{seminorm3}
||J||^{+} \geq 0 \ \text{for all} \ J \in W_0.
\end{equation}

This will be proven in Proposition \ref{prop:Lischitz_seminorm} below, but first we show that $\cM$-monotonicity of a subequation $\F$ is equivalent to a weak 1-Lipschitz condition on the graphing function $g$ for $\partial \F$.

\begin{prop}[Lipschitz regularity of $\partial \F$] \label{prop:Lipschitz_g}
	Suppose that $\F$ is a subequation and that $\cM$ is a convex cone subequation. Then $\F + \cM \subset \F$ if and only if the graphing function $g$ for $\partial \F$ (defined in Proposition \ref{prop:graphing_function}) satisfies
	\begin{equation}\label{Lipschitz_g}
	-||J'||^- \leq g(J + J') - g(J) \leq ||J'||^+ \ \ \text{for each pair} \ J,J' \in W_0.
	\end{equation}
\end{prop}

\begin{proof}
	Given $J, J' \in W_0$, by \eqref{graph_boundary} and \eqref{graph_bdy_M} we have
	$$
	J + g(J) J_0 \in \partial \F \ \ \text{and} \ \ J' + ||J'||^+ J_0 \in \partial \cM. 
	$$
	Assume that $\F + \cM \subset \F$. Then the sum 
	$$
	J + J' + (g(J) + ||J'||^+) J_0  \ \text{belongs to} \ \F. 
	$$ 
	By \eqref{epigraph} one has
	$$
	g(J) + ||J'||^+ \geq g(J + J'),
	$$
	which is the right-hand inequality in the Lipschitz bound \eqref{Lipschitz_g}. This right-hand inequality implies the left-hand inequality. Replace $J'$ by $-J'$, the right-hand inequality in \eqref{Lipschitz_g} can be restated as
	$$
	g(J - J') - g(J) \leq  ||-J'||^+; 
	$$
	that is, 
	$$
	-	||J'||^-	\leq g(J) - g(J - J'),
	$$
	which by relabeling is the left-hand inequality.
	
	For the converse, assume that the Lipschitz bound \eqref{Lipschitz_g} holds. By \eqref{epigraph}, the elements of $\F$ are all of the form $J + tJ_0$ with $J \in W_0$ and $t \geq g(J)$ and by \eqref{epigraph_M} the elements of $\cM$ are all of the form $J' + sJ_0$ with $J' \in W_0$ and $s \geq ||J'||^+$. Therefore, the elements of $\F + \cM$ are all of the form
	\begin{equation}\label{epigraph_FM}
	J + J' +(t+s)J_0 \ \ \text{with} \ J,J' \in W_0, \ t \geq g(J) \ \text{and} \ s \geq ||J'||^+.
	\end{equation}
	This jet belongs to $\F$ (again by \eqref{epigraph}) if and only if
	\begin{equation}\label{epigraph_F}
	t + s \geq g(J + J'),
	\end{equation}
	but by  the second inequality in the Lipschitz estimate \eqref{Lipschitz_g} we have
	$$
	t + s \geq g(J) + ||J'||^+ \geq g(J + J'),
	$$
	which gives \eqref{epigraph_F}, as needed to conclude 	$\F + \cM \subset \F$.
\end{proof}	

Assuming for the moment that $||\cdot||^+$ is a seminorm (by choosing a suitable hyperplane $W_0$), if one considers the seminorm $||\cdot||$ on $W_0$ defined by the sum $||\cdot||:= || \cdot||^+ + ||\cdot||^-$,  then the estimate \eqref{Lipschitz_g} yields the Lipschitz estimate
\begin{equation}\label{Lipschitz_g2}
	| g(J + J') - g(J)| \leq ||J'|| \ \ \text{for each pair} \ J,J' \in W_0,
\end{equation}
which completes the claim that $g$ and hence $F$ are continuous for Theorem \ref{thm:CanOp_ComPair}

Finally, we show how to choose $W_0$ so that $||\cdot||^+ \geq 0$ on $W_0$, which is  the remaining seminorm property \eqref{seminorm3}. This is a general fact about finite dimensional inner product spaces $(V, \langle \cdot , \cdot \rangle)$ of which the 2-jet space $\J^2$ is an example. Consider a closed convex cone $\cM \subset V$ (with vertex at the origin). The {\em edge} $E$ of $\cM$ is defined to be 
\begin{equation}\label{edge}
E := \cM \cap (-\cM)
\end{equation}
and one can show that the vector subspace $E$ contains all other vector subspaces of $\cM$ (where $E = \{ 0 \}$ is possible). The {\em (convex cone) polar} $\cM^{\circ}$ of $\cM$ is defined by
\begin{equation}\label{cone_polar}
\cM^{\circ} := \{ J \in V: \ \langle J, J' \rangle \geq 0 \ \ \text{for each} \ J' \in \cM \}.
\end{equation} 
Recall that the {\em Bipolar Theorem} says that the polar of $\cM^{\circ}$ is $\cM$. Let $S$ denote the {\em span of $\cM^{\circ}$ in $V$}. If $W$ is any linear subspace its polar $W^{\circ} = W^{\perp}$ is just its orthogonal complement. This notion of polar will also be used in subsection \ref{subsec_linear} when we discuss Hamilton-Jacobi-Bellman operators and we record a few observations now. 

\begin{rem}\label{rem:polars} In a finite dimensional inner product space $(V, \langle \cdot , \cdot \rangle)$, given any  subset $T \subset V$ one can define its (convex cone) polar as in \eqref{cone_polar}; that is,
\begin{equation}\label{polar}
T^{\circ} := \{ J \in V: \ \langle J, J' \rangle \geq 0 \ \ \text{for each} \ J' \in T \}.
\end{equation}
One knows that $T^{\circ}$ is always a closed convex cone and that
\begin{equation}
T^{\circ} = C(T)^{\circ} \ \ \text{where} \ C(T) \ \text{is the closed convex hull of} \ T.
\end{equation} 
	
	\end{rem}

Returning to the judicious choice of the needed hyperplane $W_0$, we will need the following fact. 

\begin{lem}\label{lem:decompose_V}
	With $V, E$ and $S$ as above, one has that
	\begin{equation}\label{decompose_V}
	V = E \oplus S \ \ \text{is an orthogonal decomposition}
	\end{equation}
	\end{lem}
\begin{proof}
	If $e \in E$ and $J \in \cM_+$, then $\pm e \in \cM$ and hence $\pm \langle e, J \rangle \geq 0$; that is, $\langle e, J \rangle = 0$ and hence $E \perp S$. Now $\cM_+ \subset S$ implies that the polars satisfy $S^{\perp} \subset (\cM_+)_+ = \cM$. Hence the linear subspace $S$ satisfies $S^{\perp} \subset E$ so that $V = S^{\perp} \oplus S \subset E \oplus S$, which forces \eqref{decompose_V}.
\end{proof}
	
We are now ready to describe the judicious choice of the hyperplane $W_0$. Pick $J_0 \in \Int \, \cM \neq \emptyset$ as in the Structure Theorem. Now pick the hyperplane $W_0$ to have normal $J_0^{\prime} \in \Int_{\rm rel} \, \cM_+$, the interior of the convex cone $\cM_+$ relative to the vector space $S = {\rm span} \, \cM_+$. Of course, since $J_0^{\prime} \in \cM_+$, the original convex cone $\cM$ satisfies
\begin{equation}\label{M_in_H}
	\cM \subset H
	\end{equation}
where $H$ is the closed half-space 
\begin{equation}\label{half-space1}
	H:= \{ J \in V: \ \langle J, J_0^{\prime} \rangle \geq 0 \}, \ 
\end{equation}
whose boundary satisfies
\begin{equation}\label{half-space2}
 \partial H = W_0.
\end{equation}

\begin{lem}\label{lem:Edge} Suppose that $W_0$ is the hyperplane transversal to $J_0$  with normal $J_0^{\prime} \in \Int_{\rm rel} \cM_+$ as decribed above. Then one has 
	\begin{equation}\label{W01}
	\cM \cap W_0 = E
	\end{equation}
	and
	\begin{equation}\label{W02}
 \langle J_0^{\prime}, J_0 \rangle > 0.
	\end{equation}
	which implies the transversality $W_0 \transv J_0$.
	\end{lem}

\begin{proof}
	For the relation \eqref{W01}, we begin by noting that the relation $E \subset \cM \cap S^{\perp} \subset \cM \cap W_0$ was established above. Now suppose that $J \in \cM \cap W_0$ and using \eqref{decompose_V} decompose $J := J_E + J_S$ with $J_E \in E$ and $J_S \in S$. Since $J_0^{\prime} \in \Int_{\rm rel} \cM_+$ and $J_S \in S$, if $\veps > 0$ is sufficiently small then $J_0^{\prime} - \veps J_S \in \cM_+$. Finally, since $J \in \cM \cap W_0$ one has
	$$
	0 \leq \langle J, J_0^{\prime} - \veps J_S \rangle = \langle J, J_S \rangle - \veps \langle J, J_S \rangle = - \veps \langle J_S, J_S \rangle,
	$$
	proving $J_S = 0$ and so $J = J_E \in E$ as desired.
	
	For the transversality \eqref{W02}, with $J_0^{\prime} \in \cM_+, J_0 \in \cM$ one has $\langle J_0^{\prime} , J_0 \rangle$, where $J_0 \in \Int \, \cM$ easily implies that $\langle J_0^{\prime}, J_0 \rangle > 0$, as above.
\end{proof}

The needed seminorm property \eqref{seminorm3} follows from the above considerations. 

\begin{prop}\label{prop:Lischitz_seminorm} Let $\cM$ be a convex cone subequation with $J_0 \in \Int \, \cM$ fixed. Suppose that $W_0$ is the hyperplane transversal to $J_0$  with normal $J_0^{\prime} \in \Int_{\rm rel} \cM_+$ as in Lemma \ref{lem:Edge} and let $|| \cdot ||^{+}$ be the graphing function of $\partial \cM$ over $W_0$. Then, one has
	\begin{equation}\label{semi1}
	||J||^+ \geq 0 \ \ \text{for each} \ J \in W_0
	\end{equation}
	and for each $J \in W_0$ we have
	\begin{equation}\label{semi2}
	||J||^+ = 0 \ \ \iff \ \ J \in E = \cM \cap (-\cM).
	\end{equation}
	\end{prop}

\begin{proof}
	It suffices to note that the statements \eqref{semi1} and \eqref{semi2} are equivalent to the statements
	\begin{equation}\label{semi3}
\cM \subset H := \{J \in V: \ \langle J, J_0^{\prime} \rangle \geq  0 \}
	\end{equation}
	and
	\begin{equation}\label{semi4}
	\cM \cap \partial H = E \ \text{(where $W_0 = \partial H$)},
	\end{equation}
	which were noted in \eqref{M_in_H} -- \eqref{half-space2}.
\end{proof}

\subsection{Dirichlet-G{\aa}rding operators} \label{subsec:garding}

In our dichotomy between constrained and unconstrained operator-subequation pairs, the constrained case is best illustrated by examples involving {\em Dirichlet-G\aa rding  operators} $\pol$, and provides many interesting examples of compatible operator-subequation pairs illustrating the correspondence principle of Theorem \ref{cor:AVSolns}. The most basic example is the classical Monge-Amp\`{e}re operator $F(A) := {\rm det} \, A = \lambda_1(A) \cdots \lambda_n(A)$ on $\cS(n)$.
With the standard restriction of $F$ to the convexity subequation $\cP \subset \cS(n)$,  the pair $({\rm det}, \cP)$ is a pure secord order compatible pair. This pair can be thought of as the  ``universal example'' by applying the following procedure starting from any G\aa rding polynomial $\pol$ of degree $m$.  Simply substituting the so-called {\em G\aa rding eigenvalues} $\lambda_k^{\pol}(A)$ for the standard eigenvalues yields the {\em G\aa rding-Dirichlet} operator $\pol(A) := \lambda_1^{\pol}(A) \cdots \lambda_m^{\pol}(A)$, and restricting $\pol$ to the {\em closed G\aa rding cone} $\overline{\Gamma}_{\pol}$ (the pull-back of $[0, +\infty)^m$ under the eigegenvalue map $\lambda^{\pol}(A):= (\lambda_1^{\pol}(A), \ldots, \lambda_m^{\pol}(A))$ yields a multitude of interesting examples (see  Examples \ref{exes:DG_polys} and Examples \ref{exes:more_polys} below). In addition, this provides a unified approach to studying many of the most important pure second order subequations and we refer the reader to \cite{HL13a} and \cite{HL10} for a modern, self-contained and detailed treatment. In this subsection, we will focus mainly on pure second order operators $F(r,p,A) := \pol(A)$ with $\pol$ a Dirichlet-G\aa rding polynomial (see Definition \ref{defn:garding_op}), but it is important to note they give important building blocks for operators $F=F(r,p,A)$ which contain some $\pol(A)$ as a factor. This will play a key role in section \ref{sec:CP_operators} on comparison principles for operators $F$ in constrained cases. Moreover, we present a new construction in Lemma \ref{lem:build_DG} below which produces gradient-free compatible Dirichlet-G\aa rding pairs from pure second order compatible Dirichlet-G\aa rding pairs.

In what follows, let $\pol$ be a homogeneous polynomial of degree $m$ on $\Symn$. Suppose that $\pol$ is {\em $I$-hyperbolic}, that is, $\pol(I) > 0$ and, for any given $A \in \Symn$ the one-variable polynomial $t \mapsto \pol(tI + A)$ has exactly $m$ real roots,  $t_k^{\pol}(A)$ for $k = 1,\ldots, m$.  whose negatives  $\lambda_k^{\pol}(A):= - t_k^{\pol}(A)$ are called the {\em G\aa rding eigenvalues, or the $I$-eigenvalues,  of $A$}. Up to permutation, we can order the G\aa rding eigenvalues
\begin{equation}\label{G_evals}
\lambda_1^{\pol}(A) \leq \lambda_2^{\pol}(A) \leq \ldots \leq \lambda_m^{\pol}(A),
\end{equation}
where we will often denote $\lambda_1^{\pol}(A)$ by $\lambda_{\rm min}^{\pol}(A)$ and $\lambda_m^{\pol}(A)$ by $\lambda_{\rm max}^{\pol}(A)$. If we normalize $\pol$ to have $\pol(I) = 1$ then it factors as
\begin{equation}\label{g_poly_factor1}
\pol(tI + A) =  \prod_{k=1}^m (t + \lambda_k^{\pol}(A)),
\end{equation}
so that 
\begin{equation}\label{g_poly_factor2}
\pol(A) =  \prod_{k=1}^m \lambda_k^{\pol}(A) \ \ \text{and} \ \ \lambda_k^{\pol}(A + sI) = \lambda_k^{\pol}(A) + s, \ \ k = 1, \ldots, m.
\end{equation}
Note that by the product formula \eqref{g_poly_factor2} every operator G\aa rding operator $\pol$ is a {\em generalized Monge-Amp\`{e}re operator}, where the standard eigenvalues $\lambda_k(A)$ for the special case $\pol = {\rm det}$ are replaced by the G\aa rding $I$-eigenvalues $\lambda_k^{\pol}(A)$ for a general $I$-hyperbolic polynomial $\pol$.

The {\em (open) G\aa rding cone} $\Gamma$ can be defined by
\begin{equation}\label{open_Garding_cone}
 \Gamma := \{ A \in \Symn : \, \lambda_{\rm min}^{\pol}(A) > 0\}.
\end{equation}
The conditions \eqref{g_poly_factor2} easily imply that the closed G\aa rding cone satisfies
\begin{equation}\label{Garding_cone}
\overline \Gamma := \{ A \in \Symn : \, \lambda_{\rm min}^{\pol}(A) \geq 0\} \ \ \text{and} \ \ \partial \Gamma =  \{ A \in \overline \Gamma : \, \pol(A) = 0\}
\end{equation}

G\aa rding's theory includes two important results; namely the convexity of the G\aa rding cone $\Gamma$ and the strict $\Gamma$-monotonicity of the G\aa rding eigenvalues.

\begin{thm}[G\aa rding, \cite{Ga59} ]\label{thm:DirGarPoly} Suppose that $\pol$ is an $I$-hyperbolic polynomial of degree $m$ on $\Symn$. Then the G\aa rding cone $\Gamma$ is an open convex cone with vertex at the origin and the ordered $I$-eigenvalues of $A$ are strictly $\Gamma$-monotone; that is, for each $k = 1, \ldots m$
	\begin{equation}\label{strict_Gamma_monotone}
	\lambda_k^{\pol}(A + B) > \lambda_k^{\pol}(A) \ \ \text{for each} \ A \in \Symn, B \in \Gamma.
	\end{equation}
\end{thm}

The statement of Theorem \ref{thm:DirGarPoly} combines Theorem 5.1 of \cite{HL13a} on convexity (which is also shown to be equivalent to the convexity of $\lambda_{\rm max}^{\pol}(A) = -\lambda_{\rm min}^{\pol}(-A)$, or of the concavity of $\lambda_{\rm min}^{\pol}$) and Theorem 6.2 of \cite{HL13a} on monotonicity. The reader is refered to \cite{HL13a} for the proofs.

The closed G\aa rding cone $\overline \Gamma$ is a closed convex cone with non-empty interior $\Gamma$,  but it must also satisfy the positivity condition
\begin{equation}\label{Gamma_P}
\overline \Gamma + \cP \subset \overline \Gamma \qquad \text{(equivalently, $\cP \subset \overline \Gamma$),}
\end{equation}
in order to be a subequation. (Perhaps it is worth mentioning the fact that the requirement \eqref{Gamma_P} implies that $\pol$ is hyperbolic in every positive definite direction $P > 0$ is $\Symn$). 

\begin{defn}\label{defn:garding_op}
	A homogeneous polynomial $\pol$ on $\Symn$ which is $I$-hyperbolic and for which the closed G\aa rding cone satisfies positivity $\overline{\Gamma} + \cP \subset \overline \Gamma$ will be called a {\em Dirichlet-G\aa rding polynomial}.
	\end{defn}

Dirichlet-G\aa rding polynomials yield a rich class of important examples of pure second order operators and subequations which are amenable to the theory developed in \cite{HL09}. As noted above, they will be exploited in section \ref{sec:CP_operators} to illustrate comparison principles for second order operators $F$ which have some $\pol(A)$ as a factor. In the pure second order case, we record the following elementary properties. 

\begin{prop}\label{prop:DG_pairs} Suppose that $\pol$ is a Dirichlet-G\aa rding polynomial on $\cS(n)$ with closed G\aa rding cone $\overline \Gamma$. By restricting $\pol$ to $\overline \Gamma$, one has: 
	\begin{itemize}
		\item[(a)] $(\pol, \overline \Gamma)$ is a compatible constrained operator-subequation pair;
		\item[(b)] The operator $\pol$ is tame on $\overline \Gamma$; in fact, 
		\begin{equation}\label{recall_tameness}
		\exists \, C > 0 \ \ \text{such that} \ \pol(A + tB) > \pol(A) + Ct^{1/m}, \ \ \forall \, A \in \overline{\Gamma}, B \in \Gamma;
		\end{equation}
		\item[(c)] The operator $\pol$ is topologically tame on $\overline{\Gamma}$.
		
	\end{itemize} 
\end{prop}

\begin{proof} The compatibility of part (a) is immediate from \eqref{Garding_cone}. For the proof of tameness (b), we refer the reader to Propositon 6.11 of \cite{HL18b}, which makes use \eqref{g_poly_factor2}. Tame implies topologically tame. But also note that topological tameness is immediate since $\pol$ is real analytic.
\end{proof}

The properties (a) and (c) in Proposition \ref{prop:DG_pairs} imply that 
\begin{equation}\label{CP_gGamma}
\mbox{ comparison for the operator $\pol \, \vline_{ \, \overline{\Gamma}} \ \ \Leftrightarrow \ \ $  comparison for the subequation $\overline\Gamma$,}
	\end{equation}
by Theorem \ref{cor:AVSolns}. Property (b) plays a key role in comparison on domains $\Omega$ for inhomogeneous equations $\pol(D^2u) = \psi(x)$ with $\psi$ a continuous function on $\Omega$ (see \cite{HL18b}).

Now we turn to listing some of these interesting and important examples.  

\begin{exes}[Pure second order Dirichlet-G\aa rding operator-subequation pairs]\label{exes:DG_polys}
Let  $\lambda_1(A) \leq \cdots \leq \lambda_n(A)$ denote  the ordered eigenvalues of $A \in \cS(n)$. As noted above, the most basic example is 
	\begin{enumerate}
		\item[(1)] (The elementary Monge-Amp\`{e}re pair): $\pol(A) := {\rm det} \, A$ \ and \ $\overline  \Gamma = \cP$.
		\end{enumerate}
There are many others, including the following examples, where

\begin{enumerate}
	\item[(2)] (The $k$-Hessian pair): For $k = 1, 2, \ldots n$, denote by
	$$
	\sigma_k(\lambda):= \sum_{1 \leq i_1 < \cdots < i_k \leq n} \left( \lambda_{i_1}  \cdots \lambda_{i_k} \right), \ \ \ \forall \, \lambda = (\lambda_1,\ldots, \lambda_n) \in \R^n,
	$$
	the kth elementary symmetric function of $\lambda \in \R^n$. The $k$-Hessian pair is
	$$
	\pol(A) := \sigma_k(\lambda(A)) \quad \text{and} \quad \overline{\Gamma} = \{ \sigma_j(\lambda(A)) \geq 0,  j = 1, \ldots , k \}.
	$$
	Here the G\aa rding $I$-eigenvalues of $\pol$ have no explicit formula in terms of the standard eigenvalues for $1 < k < n$, but they are real, since the roots of $\sigma_k(\lambda(A + tI))$ are critical points of $\sigma_{k+1}(\lambda(A + tI))$. Of course, when $k=n$ one recovers the Monge-Amp\`{e}re operator (1). The notion of a principal eigenvalue for this pair was recently studied in \cite{BP21}.
\end{enumerate}	

	\begin{enumerate}
	\item[(3)] (The geometric $k$-convexity pair): This example was introduced and studied in \cite{HL09}. It is geometrically significant because the plurisubharmonics (or simply subharmonics or subsolutions) are precisely the upper semicontinuous functions that restrict to all affine $k$-planes to be classically (Laplacian) subharmonic.  For $k = 1, 2, \ldots n$, consider the symmetric polynomial
	$$
	\tau_k(\lambda):= \prod_{1 \leq i_1 < \cdots < i_k \leq n} \left( \lambda_{i_1} + \cdots + \lambda_{i_k} \right)  \ \ \ \forall \, \lambda = (\lambda_1,\ldots, \lambda_n) \in \R^n. 
	$$
	The geometric $k$-convexity pair is 
	$$ 
		\pol(A) := \tau_k(\lambda(A)) \quad \text{and} \quad\overline{\Gamma} = \{ \lambda_1(A) + \cdots + \lambda_k(A) \geq 0 \}.
	$$
	The G\aa rding $I$-eigenvalues of $\pol$ are the pull-backs to $\cS(n)$ (under the eigenvalue map $\lambda: \cS(n) \to \R$) of the factors in the above product. In particular, the minimum G\aa rding $I$-eigenvalue is $\lambda_1(A) + \cdots + \lambda_k(A)$, which is the canonical operator for $\overline \Gamma$. This canonical operator was recently studied in \cite{BGI18}, where  the name {\em truncated Laplacian} was introduced. 
\end{enumerate}	

\begin{enumerate}
	\item[(4)] (The Lagrangian plurisubharmonic Monge-Amp\`{e}re pair): This is a new pair, introduced in \cite{HL09} and the main object of study in \cite{HL17}. Its subharmonics are those upper semicontinuous functions whose restrictions to arbitrary Lagrangian $n$-planes in $\CF^n$ are classically (Laplacian) subharmonic. The closed G\aa rding cone can be defined by 
	$$
	\overline{\Gamma}:= \{ A \in \cS(2n): \ {\rm tr}(A \, \vline_{ \, L}) \geq 0, \  \text{for all Lagrangian $n$-planes in} \ \CF^n \}.
	$$
	 However, a description of the Dirichlet G\aa rding operator $\pol$ is somewhat involved. We encourage the reader to consult \cite{HL17} for details, including the Lagrangian pluripotential theory. 
\end{enumerate}	
Also, there are versions of (1), (2) and (3) over $\CF$ or $\HF$ instead of $\R$. See \cite{HL13a} and \cite{HL10} for a detailed discussion.
	\end{exes}

\begin{exes}[Constructing more examples]\label{exes:more_polys}
We describe three standard methods for constructing Dirichlet-G\aa rding  polynomials from a given  Dirichlet-G\aa rding polynomial $\pol$. Suppose that $\pol$ is $I$-hyperbolic of degree $m$ with ordered G\aa rding $I$-eigenvalues $\lambda_1^{\pol}(A) \leq \cdots \leq \lambda_m^{\pol}(A)$ and G\aa rding cone $\Gamma_{\pol}$. The first two methods generalize the constructions given in examples (2) and (3) above.
\begin{itemize}
\item[(I)] (Partial derivatives in the direction $I$/elementary symmetric functions): For each $k = 0, 1, \ldots, m$, the degree $m - k$ polynomial
$$
\pol_k(A):= \frac{d^k}{dt^k}\pol(A + tI)|_{t = 0} = \sigma_{m-k}(\lambda^{\pol}(A)) \ \ \text{(modulo a positve rescaling)}
$$ 
is also a Dirichlet-G\aa rding polynomial which is $I$-hyperbolic whoose open G\aa rding cones are nested
\begin{equation}\label{nested_cones}
\Gamma_{\pol} \subset \cdots \Gamma_{\pol_k} \ \ \text{with} \ \cP \subset \overline{\Gamma}_{\pol_k}
\end{equation}
When $\pol:= {\rm det}$, this procedure produces (2) above.
\end{itemize}

\begin{itemize}
	\item[(II)] ($k$-fold sums of G\aa rding eigenvalues): For each $k = 1, \ldots, m$, the degree $\displaystyle{ \left( \begin{array}{c} m \\ k \end{array} \right)}$ polynomial
	$$ 
	\pol_k(A):= \prod_{1 \leq i_1 < \cdots < i_k \leq n} \left( \lambda^{\pol}_{i_1}(A) + \cdots + \lambda^{\pol}_{i_k}(A) \right)
	$$
	is also a Dirichlet-G\aa rding polynomial which is $I$-hyperbolic with G\aa rding cone $\Gamma_{\pol_k}$ that satisfy \eqref{nested_cones}. The G\aa rding $I$-eigenvalues of $\pol_k$ are the factors in the above product. When $\pol:= {\rm det}$, one has example (3) above.
\end{itemize}
Note that method I decreases the  degree $m$, while method II increases the degree $m$. There is a third method, which has its origins in the work of Krylov (see Definition 2.13 of \cite{Kv95}) and leaves the degree of $\pol$ fixed.
\begin{itemize}
	\item[(III)] ($\veps$-(uniformly) elliptic regularization): For $\veps > 0$, the degree $m$ polynomial 
	$$
	\pol_{\veps}:= \prod_{k = 1}^m \left( \lambda_k^{\pol}(A) + \veps {\rm tr}^{\pol} \, (A) \right),
	$$
	where $ {\rm tr}^{\pol} \, (A):= \sum_{k=1}^m \lambda_k^{\pol}(A)$ is a Dirichlet-G\aa rding polynomial. 
	\end{itemize}

The reader is refered to section 5 of \cite{HL10} for additional details on these three methods.
	\end{exes}

Next we indroduce a new consruction, which produces a Dirichlet-G\aa rdng polynomial $\polh$ of degree $m$ on $\R \times \cS(n)$  from
a  Dirichlet-G\aa rdng polynomial $\pol$ of degree $m$ on $\cS(n)$ by cleverly 
``adding a real variable $r \in \R$''. This leads to new gradient-free compatible proper elliptic pairs $(\polh, \overline{\Gamma}_{\polh})$.

\begin{lem}\label{lem:build_DG} Suppose that $\pol$ is a Dirichlet-G\aa rding polynomial on $\cS(n)$ which is hyperbolic in the direction $I$ of degree $m$ and with G\aa rding eigenvalues $\lambda_k^{\pol}(A)$ ($k = 1, \ldots, m$) and G\aa rding cone $\Gamma_{\pol}$ (normalized to have $\pol(I) = 1$). Define the degree $m$ polynomial $\polh$ on $\R \times \cS(n)$ by
	\begin{equation}\label{def_h}
	\polh(r, A) := \pol(A - rI), \ \ (r,A) \in \R \times \cS(n).
	\end{equation}
Then $\polh$ is $\left(-\frac{1}{2},\frac{1}{2}I\right)$-hyperbolic with G\aa rding eigenvalues
\begin{equation}\label{h_evals}
\lambda_k^{\polh}(r,A) := \lambda_k^{\pol}(A) - r \ \ (k = 1, \ldots, m)
\end{equation}
and G\aa rding cone
\begin{equation}\label{h_cone}
\Gamma_{\polh}:= \{ (r,A) \in \R \times \cS(n): \ A - rI \in \Gamma_{\pol} \} = 
\{ (r,A) : \ r \leq \lambda_1^{\pol}(A) \} 
\end{equation}
Moreover, 
\begin{equation}\label{h_monotonicity}
\mbox{$\overline{\Gamma}_{\pol}$ is $\cP$-monotone  \ \ $\Leftrightarrow$ \ \   $\overline{\Gamma}_{\polh}$ is $(\cN \times \cP)$-monotone,}
\end{equation}
and hence $(\polh, \overline{\Gamma}_{\polh})$ is a compatible gradient-free subequation pair.
\end{lem}

\begin{proof}
First notice that for each $t \in \R$ and each $(r,A) \in \R \times \cS(n)$ one can easily show that

\begin{equation*}
	\polh \left( t \left( -\frac12, \frac12 I \right) + (r,A) \right) = \prod_{k=1}^m \left(t + \lambda_k^{\pol}(A-rI) \right) = \prod_{k=1}^m \left(t + (\lambda_k^{\pol}(A) - r) \right),
\end{equation*}	
so that $\polh$ is $\left(-\frac{1}{2}, \frac{1}{2} I \right)$-hyperbolic with G\aa rdning eigenvalues as claimed in \eqref{h_evals} and G\aa rding cone as claimed in \eqref{h_cone}.

Now, since the closed G\aa rding cones are convex,  
$$
\mbox{$\overline{\Gamma}_{\pol}$ is $\cP$-monotone  \ \ $\Leftrightarrow$ \ \  $\cP \subset \overline{\Gamma}_{\pol}$}
$$
and
$$
\mbox{$\overline{\Gamma}_{\polh}$ is $(\cN \times \cP)$-monotone  \ \ $\Leftrightarrow$ \ \  $\cN \times \cP \subset \overline{\Gamma}_{\polh}$.}
$$
To compelete the proof of \eqref{h_monotonicity}, note that
$$
\lambda_1^{\pol}(A) - r \geq 0, \ \forall \, (r,A) \in \cN \times \cP \ \ \Leftrightarrow \ \ \lambda_1^{\pol}(A) \geq 0, \ \forall \, A \in \cP. 
$$

	\end{proof}

It is also of interest that, just as in the pure second order case, 
$$
\cH:= \{ (r,A) \in \R \times \cS(n): \ \polh(r,A) = 0 \}.
$$
has $m$ {\em subequation branches}
$$
\{ (r,A) \in \R \times \cS(n): r \leq \lambda_k^{\pol}(A) \}, \ \ k = 1, \ldots, m,
$$
with principal (smallest) branch
$$
\Lambda_1^{\pol}:= \{ (r,A) \in \R \times \cS(n): \ r \leq \lambda_1^{\pol}(A) \} = \overline{\Gamma}.
$$
This important notion of branches will be further developed in the next subsection. To get started, consider the pure second order equation
\begin{equation}\label{garding_equation}
	\cH := \{ A \in \Symn: \ \pol(A) = 0 \}
\end{equation}
and its {\em principal branch} (see \eqref{Garding_cone}):
\begin{equation}\label{principal_branch}
\Lambda_1^{\pol} := \{ A \in \Symn: \ \lambda_1^{\pol}(A) \geq 0\} = \overline{\Gamma}.
\end{equation}
Since $\Lambda_1^{\pol}$ is a convex cone, it is automatically $\overline{\Gamma} = \Lambda_1^{\pol}$-monotone. Now, $\cP \subset \overline{\Gamma}$ implies that this principle branch $\Lambda_1^{\pol}$ is a pure second order subequation which is topologically tame and hence comparison holds on arbitrary bounded domains. Since $\cP \subset \overline \Gamma$, by the monotonicity in G\aa rding's Theorem \ref{thm:DirGarPoly}, each of the other branches
\begin{equation}\label{higher_branches}
\Lambda_k^{\pol} := \{ A \in \Symn: \ \lambda_k^{\pol}(A) \geq 0 \}  \ \ k = 2, \ldots , m,
\end{equation}
are subequations. Note that $\Lambda_1^{\pol} \subset \Lambda_2^{\pol} \subset \cdots \Lambda_m^{\pol}$ and since $-\lambda_k^{\pol}(-A) = \lambda_{m-k+1}^{\pol}(A)$ the dual subequation is $\wt{\Lambda}_k^{\pol} = \Lambda_{m-k+1}$. The {\em branch condition} means that $\partial \Lambda_k^{\pol} \subset \cH$. The second part of \eqref{g_poly_factor2} says that $\lambda_k^{\pol}$ is the canonical operator for $\Lambda_k^{\pol}$.

The most basic example is $\pol(A) = {\rm det}(A)$ in which case $\lambda^{\pol}_k(A) = \lambda_k(A)$ for $k = 1, \ldots n$ and $\overline{\Gamma} = \cP$.

\subsection{Subequation branches}\label{subsec:branches}

We conclude this section with the notion of subequation branches. This notion makes sense for any closed set $\cH$ contained in the jet space $\J^2$ and is independent of the existence of an operator $F \in C(\F, \R)$ whose zero level set satisfies $\{ F = 0 \} = \cH$; that is, the case where $\cH$ can be thought of as an {\em equation constraint set} for the partial differential equation  $F(u,Du,D^2u) = 0$. 

\begin{defn}\label{defn:branch}
	Suppose that $\cH \subset \J^2$ is any closed set. A subequation $\F_0 \subset \J^2$ is called a {\em subequation branch of $\cH$} if  $\partial \F_0 \subset \cH$.
\end{defn}

In general, an equation constraint set  $\cH$ may admit more than one subequation branch. We have seen several examples in the last subsection. A standard pure first order example is the Eikonel equation $\cH = \{ p \in \R^n: \ |p| = 1\}$. It admits two subequation branches
$$
\F^- =  \{ p \in \R^n: \ |p| \leq 1\} \ \ \text{and} \ \ \F^
+ =  \{ p \in \R^n: \ |p| \geq 1\}
$$
whose boundaries are $\cH$. However, if one subequation branch $\F$ of $\cH$ has monotonicity $\cM$ which is a subequation and if $\partial \F$ is all of $cH$, then the equation $\cH$ uniquely determines the subequation branch $\F$ as another consequence of the Structure Theorem \ref{thm:structure}. The precise statement is as follows.

\begin{prop}\label{prop:branches}
	If a subequation $\F \subsetneq \J^2$ admits a monotonicity cone subequation $\cM$, then one knows by Corollary \ref{cor:structure} that $\F = \partial \F + \cM$. Consequently, if $\cH = \partial \F$ and $\cH = \partial \F^{\prime}$ where $\F^{\prime}$ is also an $\cM$-monotone subequation, then $\F = \F^{\prime}$. 
\end{prop}

\section{Comparison Principles for Nonlinear Operators}\label{sec:CP_operators}

 In this section, we will illustrate the use of Theorem \ref{cor:AVSolns} (Corrspondence Principle) which represents the part of the theory in which there is an equivalence between comparison at the potential theoretic (subequation $\F$) level and the PDE (operator $F$) level. We will illustrate cases in which Theorem \ref{cor:AVSolns} is applicable and cases in which it is not, and give some indication as to how one might proceed in cases {\bf not} covered herein. 
 

More precisely, we present comparison principles on bounded domains $\Omega \subset \R^n$ for constant coefficient proper elliptic nonlinear partial differential equations
\begin{equation}\label{PDEc}
	F(u, Du, D^2u) = c \ \ \text{in} \ \Omega,
	\end{equation}
where the nonlinear operator defined by $F$ has an associated {\em compatible subequation constraint set $\F$} in the sense of Definition \ref{defn:compatible_pair}.  Comparison will be formulated for {\em $\F$-admissible viscosity subsolutions and supersolutions} of \eqref{PDEc} in the sense of Definition \ref{defn:AVSolns}, which means that for each $x_0 \in \Omega$
	\begin{equation}\label{AVSub11}
\mbox{$J \in J^{2, +}_{x_0}u \ \ \Rightarrow \ \   J \in \F$ \ \ \text{and} \ \ $F(J) \geq c$.}
\end{equation}
and
\begin{equation}\label{AVSuper11}
\mbox{$J \in J^{2, -}_{x_0}u  \ \ \Rightarrow$ \ \ either [ $J \in \F$ and \ $F(J) \leq c$\, ] \ or \ $J \not\in \F$.}
\end{equation}
respectively, where $J^{2, \pm}_{x_0}$ are the spaces of upper (lower) 2-jets for $u$ at $x_0$.
 
We will require that $F$ is {\em topologically tame} in the sense of Definition \ref{defn:tameness} and will examine structural conditions on $F$ for which the pair $(F, \F)$ is $\cM$-monotone for some monotonicity cone subequation in the sense of Definition \ref{defn:PEP}. Under these assumptions, Theorem \ref{cor:AVSolns} states that for each $c \in \R$ and each domain $\Omega$, comparison holds at the PDE level for $F = c$ on if and only if comparison holds in the potential theoretic sense for the subequation constraint set 
\begin{equation}\label{SCS11}
	\F_c := \{ J \in \F: \ F(J) \geq c \}.
\end{equation}

\subsection{Proper elliptic gradient-free operators}
We begin with a class of compatible pairs with monotonicity cone $\cM(\cN,\cP):= \cN \times \R^n \times \cP$.

\begin{defn}\label{defn:GFCP} A pair $(F,\F)$ is called a {\em compatible proper elliptic gradient-free operator-subequation pair} if 
\begin{equation}\label{GFCP1}
F(r, p, A) := G(r,A) \ \ \text{and} \ \ 	\F := \{ (r,p,A) \in \J^2: \ (r,A) \in \cG \}
\end{equation}
where $G: \cG \subseteq \R \times \Symn \to \R$ is continuous and such that the pair $(G, \cG)$ satisfies the following conditions: $\cG$ is closed, non empty, the pair is $\cQ$-monotone; that is, 
		\begin{equation}\label{E0A}
	\cG + \cQ \subset \cG \ \ \ \text{where} \ \ \cQ = \cN \times \cP = \{(s,P) \in \R \times \Symn: \ s \leq 0 \ \text{and} \ P \geq 0\},
	\end{equation}
	and 
	\begin{equation}\label{E0B}
	G(r + s,A + P) \geq G(r, A) \ \ \text{for each} \ (r,A) \in \cG \ \ \text{and each} \ (s,P) \in \cQ,
	\end{equation}
	and finally either $\cG = \R \times \Symn$ (the unconstrained case) or 
	\begin{equation}\label{COSP_GF}
	c_0:= \inf_{\cG} G \ \ \text{is finite and} \ \ \partial \G = \{ (r,A) \in \cG: G(r,A) = c_0\}
	\end{equation} 
	(the constrained case).
	\end{defn}
	
Under the hypotheses \eqref{E0A} - \eqref{COSP_GF}, it is clear that the pair $(F,\F)$ is indeed a compatible proper elliptic operator-subequation pair in the sense of Definitions \ref{defn:compatible_pair} and \ref{defn:PE_pairs}. One could, of course, suppress the variable $p \in \R^n$ and merely consider the reduced pair $(G, \cG)$ with the reduced monotonicity cone $\cM^{\prime}(\cN, \cP) = \cN \times \cP$. 
	
	Note that \eqref{E0A} - \eqref{E0B} say that $(F, \F)$ is $\cM = \cM(\cN, \cP)$-monotone in the sense of Definition \ref{defn:PEP}. Lemma \ref{lem:upper_levels} then applies so that this hypotheses is equivalent to the statement that for every $c \in \R$ the upper level set
		\begin{equation}\label{E0C}
	\F_c := \{ (r, p, A) \in \F: \ F(r,p,A) = G(r,A) \geq c \}
	\end{equation}
	is a subequation constraint set which is $\cM(\cN, \cP)$-monotone.
	
	 The comparison principle in this situation is now a restatement of the previously developed theory.

\begin{thm}\label{thm:E0}
	Suppose that $(F, \F)$ is a compatible proper elliptic gradient-free pair as in Definition \ref{defn:GFCP}. Then for every bounded domain $\Omega$ and every $c \in F(\F)$, one has the comparison principle 
	\begin{equation}\label{E0E}
	\mbox{$u \leq w$ on $\partial \Omega \ \ \Rightarrow \ \  u \leq w$ on $\Omega$}
	\end{equation}
	for each pair  $u \in \USC(\overline{\Omega})$ and $w \in \LSC(\overline{\Omega})$ with
	\begin{itemize}
		\item[(a)]  $u$ is $\F_c$-subharmonic and $w$ is $\F_c$-superharmonic  (i.e.\ $-w$ is $\wt{\F}_c$-subharmonic).
	\end{itemize}
	If one also requires the additional hypothesis that the operator $F$ is topologically tame; that is,
\begin{equation}\label{TT_GF}
\mbox{ $\{(r,p, A) \in \F: \ F(r,p,A) = G(r,A) = c \}$ has empty interior $\forall \ c \in \R$,}
\end{equation}
then one can replace (a) by the equivalent hypothesis that $u \in \USC(\overline{\Omega})$ and $w \in \LSC(\overline{\Omega})$ satisfy
\begin{itemize}
		\item[(b)] $u$ is an $\F$-admissible viscosity subsolution and $w$ is an $\F$-admissible supersolution to $F(u,Du,D^2u) := G(u,D^2u) = c$ on $\Omega$.
	\end{itemize}
Here, as always,  $c \in F(\F)$.
\end{thm}

\begin{proof} The comparison principle \eqref{E0E} for $\F_c$-subharmonics $u$ and superharmonics $w$, on every bounded domain $\Omega$, follows from Theorem \ref{thm:CP_GF} since $\F_c$ is a gradient-free subequation. Finally, by Theorem \ref{cor:AVSolns}, since $F$ is topologically tame and since $(F, \F)$ is a compatible proper elliptic operator-subequation pair which is $\cM$-monotone for a convex cone subequation, for each $c \in F(\F)$, the  comparison principle holds for $\F$-admissible viscosity subsolutions $u$ and supersolutions $w$ of $F(u,Du,D^2u) = G(u, D^2u) = c$.  
\end{proof} 

We remark that Theorem \ref{thm:E0} also includes the pure second order case where $F(r,p,A) := G(A)$ with $G$ increasing on a closed non-empty $\cP$-monotone subset $\cG$ of $\Symn$ such that $(G, \G)$ is a compatible pair (with the obvious reduction of suppressing also the variable $r \in \R$ and using the reduced monotonicity cone $\cM^{\prime}(\cP) = \cP$). 

We conclude this subsection  with some examples to illustrate the applicability of Theorem \ref{thm:E0} as well as some ``bad'' examples for which Theorem \ref{thm:E0} does not apply in this topologically tame proper elliptic gradient-free pairs case. The first examples make use of {\em canonical operators} in {\em unconstrained cases}.

\begin{exe}[Gradient-free canonical operators]\label{exe:gf1}
	Consider any gradient-free subequation $\F \subsetneq \J^2$ (see  Definition \ref{defn:GFSE1}); that is, $\F$ is closed, non-empty and 
	$\cM(\cN,\cP):= \cN \times \R^n \times \cP$-monotone. Fix $J_0 \in \Int \, \cM(\cN, \cP)$ (which the reader may wish to standardize as $J_0 = (-1,0,I)$) and let $F \in C(\J^2)$  be the canonical operator for $\F$ (see Definition \ref{defn:canonical_op}); that is, for each $J \in \J^2$
	$$
	\mbox{$F(J) := - t_J$ \ \ where $t_J$ is the unique element of $\R$ such that $J + t_J J_0 \in \partial \F$.} 
	$$
	Since $\cM = \cM(\cN, \cP)$ is a monotonicity cone subequation, by Theorem \ref{thm:CanOp_ComPair}, both $(F, \J^2)$ and the dual pair $(\wt{F}, \J^2)$ are both unconstrained compatible proper elliptic operator-subequation pairs with $F$ and $\wt{F}$ topologically tame.
	Hence the comparison principle of Theorem \ref{thm:E0} applies to both pairs.
\end{exe}

It is worth stressing that each gradient-free subequation $\F$ gives rise to a family of admissible operators (parameterized by $J_0 \in \Int \, \cM(\cN, \cP)$). With the standard choice $J_0 = (-1, 0, I)$, the canonical operator is $F(J) = F(r,p,A) = \min \{ -r, \lambda_{\rm min}(A) \}$ by \eqref{can_op_negative_convex} and with dual $\wt{F}(J):= - F(-J) = \max \, \{ -r, \lambda_{\rm max} (A) \}$ by \eqref{dual_operators}. 

\begin{exe}[Perturbations of pure second order canonical operators]\label{exe:gf2}
	Consider any pure second order subequation $\cH \subset \Symn$; that is, $\cH$ is closed, proper, non-empty and $\cP$-monotone. 
	Define a gradient-free operator $F: \J^2 \to \R$ by
	\begin{equation}\label{gf2_operator}
	F(r,p,A) := H(A) + h(r),
	\end{equation}
	where $h \in C(\R)$ be non-increasing and $H \in C(\Symn)$ is the canonical operator for $\cH$ (determined by  $A_0 \in \Int \, \cP$, where the standard choice is $A_0:=I$). This operator is given by the pure second order version of Definition \ref{defn:canonical_op}; that is, for each $A \in \Symn$
	$$
	\mbox{$H(A) := - t_A$ \ \ where $t_A$ is the unique element of $\R$ such that $A + t_A A_0 \in \partial \cH$,} 
	$$
	and one has 
	\begin{equation}\label{gf2_property}
	H(A + tA_0) = H(A) + t \ \  \text{for each} \ A \in \Symn, t \in \R.
	\end{equation}
	One has that $(F, \J^2)$ is an unconstrained operator-subequation pair which is $\cM(\cN, \cP)$-monotone since $F(r,p,A)$ is increasing in $A$ by \eqref{gf2_property} and non-increasing in $r$ by the monontonicity hypothesis on $h$. Moreover, for each $t > 0$ (in the case that $J_0$ is the standard choice for simplicity)
	$$
		F((r,p,A) + t(-1, 0, I)) = H(A) + t + h(r - t) \geq H(A) + t + h(r) = F(r,p,A) + t,
	$$
	so that $F$ is topologically tame by condition 3) of Theorem \ref{thm:tameness}. Hence the comparison principle of Theorem \ref{thm:E0} applies to $(F, \J^2)$.
	
		One could also use the dual operator $\wt{H}$ for the dual subequation $\wt{\cH}$. More generally, one can take finite sums 
	$$
	F(r,p,A) = \sum_{k = 1}^N H_k(A) + h(r),
	$$
	if each $\cH_k \subsetneq \Symn$ is a pure second order subequation with canonical operator $H_k$ (determined by the same $A_0 \in \Int \, \cP$).
	
	A simple instructive example is $F(r,p, A) := \lambda_{\rm min}(A) - r$  which is the sum of the canonical operators for $\cM(\cP)$ and $\cM(\cN)$. 
\end{exe}

Next we turn to the constrained gradient-free case, where {\em Dirichlet-G{\aa}rding polynomials} generate many interesting examples. the reader might want to look back at subsection \ref{subsec:garding} (especially Examples \ref{exes:DG_polys} and \ref{exes:more_polys}) to review just how rich the class of examples is. 

\begin{exe}[Operators involving Dirichlet-G{\aa}rding polynomials]\label{exe:gf3}
	Consider $G(r,A) := h(r) \pol(A)$ where $\pol$ is a  Dirichlet-G{\aa}rding polynomial of degree $m$ in the sense of Definition \ref{defn:garding_op} and $h \in C((-\infty, 0])$ is continuous with
	\begin{equation}\label{hyp_h}
	\mbox{$h$ is non-increasing, \ $h \geq 0$  \ \ \text{and} \ \ $h(0) = 0 \ \iff \ r = 0$.}
	\end{equation}
	Take as the domain for $G$ the set $\cG:= \cN \times \overline{\Gamma} \subset \R \times \Symn$ where $\cN = (-\infty, 0]$ and $\overline{\Gamma} = \{ A \in \Symn: \ \lambda_k^{\pol}(A) \geq 0, k = 1, \ldots n \}$ is the closed G{\aa}rding cone. $\overline{\Gamma}$ is assumed to satisfy $\overline{\Gamma} + \cP \subset \overline{\Gamma}$, or equivalently $\cP \subset \overline{\Gamma}$. Recall that by normalizing $\pol(A)$ (which is $I$-hyperbolic) to have $\pol(I) = 1$, one has
	\begin{equation}\label{garding1}
	\pol(A + tI) = \prod_{k=1}^m \left( \lambda_k^{\pol}(A) + t \right) \ \ \text{and} \ \ \pol(A) = \prod_{k=1}^m \lambda_k^{\pol}(A).
	\end{equation}
	Since $h \geq 0$ on $\cN$ and since $ \lambda_k^{\pol}(A) \geq 0$ for each $k$ defines $\overline{\Gamma}$, one has that
	$$
	G(r,A) = h(r) \pol(A) \geq 0 \ \text{for each} \ (r,A) \in \cG:= \cN \times \overline{\Gamma}
	$$
	and since $G(0,A) = 0$, one has the first compatibility condition 
	$$
	\inf_{\cG} G = 0 \ \text{(finite)}.
	$$
	Moreover, using the third property in \eqref{hyp_h} one easily verifies the second compatibility condition
	$$
		\partial \cG = \{ (r,A) \in \cG: G(r,A) = h(r) \pol(A) = 0 \}.
	$$
	Thus $(G, \cG)$ is a constrained case compatible gradient-free operator-subequation pair.  
	$G$ is proper elliptic and topologically tame since for each $(r,A) \in \cG$ and each $(s,P) \in \cQ = \cN \times \cP$, one has
	$$
	G(r + s, A + P) = h(r+s) \pol(A+P) \geq h(r) \pol(A+P) \geq h(r) \pol(A),
	$$
	and with  $J_0^{\prime} = (-1, I) \in \Int \, \cQ$ and $t > 0$ one has
	\begin{eqnarray*}
	G( (r,A) + t(-1, I)) &=& h(r - t)  \pol(A + tI) = h(r-t) \prod_{k=1}^m \left( \lambda_k^{\pol}(A) + t \right) \\
	& \geq & h(r) \prod_{k=1}^m \left( \lambda_k^{\pol}(A) + t \right) \geq h(r)  \prod_{k=1}^m  \lambda_k^{\pol}(A) = G(r,A).
\end{eqnarray*}
Hence the comparison Theorem \ref{thm:E0} applied to the constrained case pair $(G, \cG)$.  
\end{exe}

Obviously one could replace $\pol(A)$ and $\overline{\Gamma}$ by any of its factors $\lambda_k^{\pol}(A)$ and the branch $\Lambda_k^{\pol} := \{ A \in \Symn: \ \lambda_k^{\pol}(A) \geq 0 \}$. Moreover, the same holds for $G(r,A) = h(r) H(A)$ if $h$ is above and $H(A)$ is the canonical operator (determined by $J_0 \in \Int \, \cP$) for $\cH \subset \Symn$ a $\cP$-invariant pure second order subequation.

\begin{exe}[The hypebolic affine shere equation]\label{exe:cheng_yau} The partial differential equation 
	\begin{equation}\label{cheng_yau1}
	{\rm det}(D^2 u) = \left( \frac{L}{u} \right)^{n+2}, \ \ L \leq 0
	\end{equation} 
	arises in the study of {\em hyperbolic affine spheres} with mean curvature $L$ where $u < 0$ is convex and vanishes on the boundary of $\Omega \subset \R^n$ convex (see Cheng-Yau \cite{CY86} and the references therein). The equation \eqref{cheng_yau1} is covered by Example \ref{exe:gf3} if one takes $\pol(A) = {\rm det}(A)$ and $h(r) = (-r)^{n+2}$ and $-L \geq 0$ corresponds to the admissible levels. 
	
\end{exe}

\begin{rem}\label{rem:generalized_equations}
	Products like $G(r,A) = -r \, {\rm det}(A)$ in Example \ref{exe:gf3} (where $\pol(A) = {\rm det}(A)$ and $h(r) = -r$) are good examples which are admissible for the Correspondence Principle (Theorem \ref{cor:AVSolns}) in the constrained case. However, sums like
	\begin{equation}\label{malefico}
	G(r,A) := {\rm det}(A) - r
	\end{equation}
	are bad examples. For this example,
	$$
	G(r+s,A+P) \geq G(r,A) \ \text{for all}\ (r,A) \in \R \times \cP, (s,P) \in \cN \times \cP, 
	$$
	so that, with domain $\cG = \R \times \cP$, the pair $(G, \cG)$ is a $\cQ = \cN \times \cP$- monotone gradient-free pair. However with this maximal domain $\R \times \cP$ of $\cQ$-monotonicity of the operator $G$
\begin{equation}\label{not_cp1}
		\inf_{\cG} G = - \infty \ \text{(is not finite)}
	\end{equation}
	and hence the pair fails to satisfy the first condition in \eqref{COSP_GF} for compatibility. If one cuts down the domain to $\cG:= (-\infty, 0] \times \cP$ so that the inf in \eqref{not_cp1} is zero (finite), then the jet subset
	 \begin{equation}\label{not_cp2}
	 \F \equiv \{ (r,A) \in (-\infty, 0] \times \cP:  \ {\rm det}(A) - r \geq 0 \}
	 \end{equation}
	 has boundary $\partial \F$ including $(-\infty,0] \times \{ 0\}$, so that all negative $C^2$ affine functions are $\F$-harmonic but the operator $G$ is not zero. Since $\partial \F$ is much larger than the zero set of the operator $\{ (r,A) \in \cG: \ G(r,A) = 0 \}$, the second condition in \eqref{COSP_GF} for compatibility fails.
	 
	 The problem for these examples is that the subharmonics of $- \wt{\cG}$ do not corespond to $\cG$-admissible supersolutions of the equation $\cH$ defined by $G(r, A) = 0$. However, for such examples, one can make use of the notion of a {\em generalized equation} in which one looks for another subsequation constraint set $\cE \subset \R \times \Symn$ (different form $\cG$) such that
	 $$
	 	\cH = \cG \cap (- \wt{\cE}).
	 $$
	 We will not pursue this further here, but refer to
	 \cite{HL19} for details and where the pure second order case is discussed at length.
	\end{rem}

In order to treat situations with gradient dependence, as is commonly known, a  Lipschitz condition in $p$ is helpful. These ideas will be explored next, The next three subsections treat unconstrained compatible pairs $(F, \J^2)$, which also indicate how many classical results can be recovered by our monotoncity method in the presence of a suitable monotonicity cone $\cM$ for the pair.

\subsection{Degenerate elliptic operators with strict monotonicity in $r$}\label{subsec:strict_r} 
Our next result concerns the unconstrained case of a well known example class of operators.  It also shows how the $\cM(\gamma, \R^n, \cP)$ cones with $\gamma \geq  0$ arise naturally in an important example class with gradient dependence. 

\begin{thm}\label{thm:CP_DECr}
	Suppose that $F : \J^2 \to \R$ is continuous and satisfies the following structural condition: there exist $\alpha > 0$ and $\beta \geq 0$ for each $(r,p,A) \in \J^2$ and each $(s,q,P) \in \cN \times \R^n \times \cP$ one has
	\begin{equation}\label{DECr1}
	F(r+s, p+q, A+P) - F(r,p,A) \geq - \alpha s - \beta |q|. 
	\end{equation}
	Then $(F, \J^2)$ is an unconstrained compatible proper ellipitc operator-subequation pair which is $\cM$-monotone for the monotonicity cone subequation
	\begin{equation}\label{DECr2}
	\cM = \cM(\gamma, \R^n, \cP) := \{ (s,q,P) \in \J^2: \ \ s \leq -\gamma |q|, \ q \in \R^n, \ P \in \cP\}
	\end{equation}
	if $\gamma:= \beta / \alpha$. Consequently, for each admissible level $c \in F(\J^2)$, the set
	\begin{equation}\label{DECr3}
		\F_c := \{ (r,p,A) \in \J^2: F(r.p.A) \geq c \}
	\end{equation}
	is $\cM(\gamma, \R^n, \cP)$ monotone. Moreover, the operator $F$ is topologically tame and hence for every bounded domain $\Omega$ one has the comparison principle: 
	\begin{equation}\label{DECr4}
	\mbox{$u \leq w$ on $\partial \Omega \ \ \Rightarrow \ \  u \leq w$ on $\Omega$}
	\end{equation}
	for $u \in \USC(\overline{\Omega})$ and $w \in \LSC(\overline{\Omega})$ which are respectively $\F_c$-subharmonic and $\F_c$-superharmonic in $\Omega$, or equivalently if $u,w$ are viscosity subsolutions and supersolutions to $F(u,Du,D^2u) = c$ on $\Omega$.
\end{thm}

\begin{proof} By the definition \eqref{case1} and the hypothesis $F \in C(\J^2)$, with $\F := \J^2$ one has that $(F,\F)$ is an unconstrained compatible pair. $\F = J^2$ is trivially $\cM$-monotone for every $\cM$. Using the structural condition \eqref{DECr1}, one has for each  $(r,p,A) \in \J^2$ and each $(s,q,P) \in \cM(\gamma, \R^n, \cP) \subset \cN \times \R^n \times \cP$
\begin{equation}\label{DECr5}
  F(r + s, p + q, A + P)  \geq  F(r,p, A) - \alpha s - \beta |q|  \geq F(r,p,A),
  \end{equation}
  where $- \alpha s - \beta |q| \geq 0$ since $s \leq - \frac{\beta}{\alpha} |q|$. Hence $F$ is $\cM(\gamma, \R^n, \cP)$-monotone and since $\cM(\gamma, \R^n, \cP) \supset \cM_0 = \cN \times \{0\} \times \cP$, the pair is proper elliptic in accordance with Definition \ref{defn:PE_pairs}. By Lemma \ref{lem:upper_levels}, every upper level set $\F_c$ is $\cM(\gamma, \R^n, \cP)$-monotone and the comparison \eqref{DECr4} on each bounded domain $\Omega$ for $\F_c$-subharmonic, superharmonic pairs follows from Theorem \ref{thm:comparison} by the $\cM(\gamma, \R^n, \cP)$-monotonicity of each $\F_c$. 
  
  Finally, $F$ is topologically tame by part (2) of Theorem \ref{thm:tameness} since $F(J + J_0 ) > F(J)$ for each $J \in \J^2$ and $J_0 \in \Int \, \cM(\gamma, \R^n, \cP)$ where $- \alpha s - \beta |q| > 0$ in \eqref{DECr5}. Hence the comparison principle \eqref{DECr4} for (unconstrained) viscosity subsolution, supersolution pairs $(u,v)$ of $F(u,Du,D^2u) = c$ is equivalent to the comparison for $\F_c$-subharmonic, superharmonic pairs by Theorem \ref{cor:AVSolns}.  
 	\end{proof} 
 
 Concerning Theorem \ref{thm:CP_DECr}, a few examples and remarks are worth noting. 
 
 \begin{exe}\label{exe:GFCase} The case $\beta = 0$ of Theorem \ref{thm:CP_DECr} yields comparison for proper elliptic gradient-free operators $F(r,p,A) = G(r,A)$ with strict monotonicity in the $r$ variable; that is, $G \in C(\R \times \cS(N))$ such that for some $\alpha > 0$ one has 
 	\begin{equation}\label{exe:GFG}
 	G(r + s, A + P) - G(r,A) \geq -\alpha s, \ \ \forall \, (r,A) \in \R \times \cS(N), s \leq 0, P \geq 0.
 	\end{equation}
With respect to the unconstrained case of Theorem \ref{thm:E0}, the strict monotonicity in $r$ ensures the topological tameness, which was a hypothesis in the previous theorem. 
 	
 	The case $\beta > 0$ of Theorem \ref{thm:CP_DECr} yields comparison for operators of the form 
 \begin{equation}\label{LipGrad}
 	F(r,p,A) = G(r,A) + \langle b, p \rangle
 	\end{equation}
 	with $G$ as above and $b \in \R^n \setminus \{0\}$. One can then choose $\beta = |b|$. Notice that if the strict monotonicity constant $\alpha = 1$ in \eqref{exe:GFG}, then $\gamma = \beta$; that is, the monotonicity cone parameter equals $|b|$.
 \end{exe}

\begin{rem}\label{rem:Jensen1} Theorem \ref{thm:CP_DECr} is closely related to one of the two cases in the grounbreaking paper of Jensen \cite{Je88} on the maximum principle for viscosity solutions to constant coefficient equations. In the situation of degenerate ellipticity and strict monotonicity in $r$, the first case of Jensen's Theorem 3.1 in \cite{Je88} gives a maximum principle (which then implies the comparison principle) for a subsolution/supersolution pair with more regularity that we require. He assumes that the pair belongs to $C(\overline{\Omega}) \cap W^{1, \infty}(\Omega)$. On the other hand,  Jensen does not require the Lipschitz in $p$ condition which we need in order to have a montonicity cone with non empty interior. The second case of Jensen's theorem (for uniformly elliptic operators which are   Lipschitz in $p$), will be discussed in the next subsection (see Theorem \ref{thm:CP_UE}).
	\end{rem}

\subsection{Proper operators with some degree of strict ellipticity}\label{subsec:strict_ell}

In this subsection, we examine classes of proper operators with gradient dependence in which the weak monotonicity assumption of degenerate ellipticity is strengthened to include some measure of strict monotonicity in the Hessian variable but there may be no strict monotonicity in the $r$ variable as in the previous subsection.   Our first result is stated in the unconstrained case, as was done in Theorem \ref{thm:CP_DECr}. It makes use of $\cM(\cN, R)$-monotoncity with finite $R$.  
 
 \begin{thm}\label{thm:CP_SE}
 	Suppose that $F : \J^2 \to \R$ is continuous 
 	 and is proper elliptic; that is, it safisfies for each $(r,p,A) \in \J^2$ and each $(s,P) \in \cN \times \cP$:
 	 	\begin{equation}\label{SE1}
 	F(r+s, p, A+P) \geq F(r,p,A).
 	\end{equation}
 	In addition, suppose that $F$ satisfies the following structural condition: there exist $\alpha, \beta > 0$ such that for each $(r,p,A) \in \J^2$ and each $\mu \geq 0,q \in \R^n$ one has
 	\begin{equation}\label{SE2}
 	F(r, p+q, A+ \mu I) - F(r,p,A) \geq  \alpha \mu - \beta |q|.
 	\end{equation}  
 	Then $(F, \J^2)$ is an unconstrained proper elliptic compatible operator-subequation pair which is $\cM$-monotone for the monotonicity cone subequation
 	\begin{equation}\label{SE3}
 	\cM = \cM(\cN, R):= \left\{ (s,q,P) \in \J^2: \ s \leq 0, \  q \in \R^n, \text{and} \ P \geq \frac{|q|}{R}I     \right\}
 	\end{equation}
 	if $R \leq  \alpha / \beta$. Consequently, if $R \leq  \alpha / \beta$ then for each $c \in F(\J^2)$ the set
 	\begin{equation}\label{SE4}
 	\F_c := \{ (r,p,A) \in \J^2: F(r.p.A) \geq c \}
 	\end{equation}
 	is $\cM(\cN, R)$-monotone. Moreover, the operator $F$ is topologically tame and hence for every bounded domain $\Omega$ which is contained in a translate of $B_R(0)$, one has the comparison principle: 
 	\begin{equation}\label{SE5}
 	\mbox{$u \leq w$ on $\partial \Omega \ \ \Rightarrow \ \  u \leq w$ on $\Omega$}
 	\end{equation}
 	for $u \in \USC(\overline{\Omega})$ and $w \in \LSC(\overline{\Omega})$ which are respectively $\F_c$-subharmonic and $\F_c$-superharmonic in $\Omega$, or equivalently if $u,w$ are viscosity subsolutions and supersolutions to $F(u,Du,D^2u) = c$ on $\Omega$.
\end{thm}
 	
\begin{proof} One follows the same argument of the proof of Theorem \ref{thm:CP_DECr}. It is clear that $(F, \J^2)$ is an unconstrained proper elliptic operator-subequation pair and that $\J^2$ is trivially $\cM$-monotone. For the $\cM$-monotonicity of $F$, if $(r,p,A) \in \J^2$ and $(s,q,P) \in \cM(\cN, R)$ then by \eqref{SE1} - \eqref{SE3} one has
 	\begin{eqnarray*}
 		F(r + s, p + q, A + P) & \geq &   F\left(r, p + q, A + \frac{|q|}{R}I\right) \\
 		& \geq & F(r, p, A) + |q| \left(\frac{\alpha}{R} - \beta \right) \ge 0,
 	\end{eqnarray*}
 where the condition $R \leq  \alpha / \beta$ is needed. Again, by Lemma \ref{lem:upper_levels}, every upper level set $\F_c$ is $\cM(\cN, R)$-monotone and the comparison \eqref{SE5} on each domain $\Omega$ contained in  a translate of $B_R(0)$ for $\F_c$-subharmonic, superharmonic pairs follows from Theorem \ref{thm:comparison} by the $\cM(\cN, R)$-monotonicity of each $\F_c$. The topological tameness of $F$ again follows from  part (2) of Theorem \ref{thm:tameness} by using the strict monotonicity in \eqref{SE2}. Hence, by Theorem \ref{cor:AVSolns}, the comparison \eqref{SE5} also holds for viscosity subsolution, supersolution pairs $(u,w)$ of $F(u,Du,D^2u) = c$ for each $c \in \R$ and each  $\Omega$ contained in  a translate of $B_R(0)$.
\end{proof}

Before proceeding, a remark on the terminology of notions of ellipticity.

\begin{rem}\label{rem:NTD} The monotonicity property in $A$ of \eqref{SE2}; that is, with $\alpha > 0$
	\begin{equation}\label{SEI}
	F(r, p, A + \mu I) -  F(r, p, A) \ge \alpha \mu  \ \ \ \text{for each} \ \mu \geq 0,
	\end{equation}
	might well be called {\em strict (uniform) ellipticity in the direction $I \in \Symn$.} In Bardi-Mannucci \cite{BM06} this partial (strict) uniform ellipticity was called {\em non-totally degenerate ellipticity} since in the quasi-linear case it corresponds to what Bony called non-totally degenerate in \cite{Bo69}. One should also note that in the language of \cite{HL19} the condition \eqref{SEI} would be called {\em (linear) tameness in $A$ of $F$ on $\J^2$}. 
\end{rem}

It is worth noting the limit case of $\beta = 0$ in Theorem \ref{thm:CP_SE}, which gives a gradient-free situation with $\cM(\cN, \cP)$-monotonicity and comparison on arbritary bounded domains.

\begin{thm}\label{thm:CP_SEGF}
	Suppose that $F : \J^2 \to \R$ is continuous 
	and is proper elliptic; that is, it safisfies for each $(r,p,A) \in \J^2$ and each $(s,P) \in \cN \times \cP$:
	\begin{equation}\label{SEGF1}
	F(r+s, p, A+P) \geq F(r,p,A).
	\end{equation}
	In addition, suppose that $F$ satisfies the following structural condition: there exist $\alpha > 0$ such that for each $(r,p,A) \in \J^2$ and each $\mu \geq 0,q \in \R^n$ one has
	\begin{equation}\label{SEGF2}
	F(r, p+q, A+ \mu I) - F(r,p,A) \geq  \alpha \mu.
	\end{equation}  
	Then $(F, \J^2)$ is an unconstrained proper elliptic compatible operator-subequation pair which is gradient-free and $\cM$-monotone for the monotonicity cone subequation
	\begin{equation}\label{SEGF3}
	\cM = \cM(\cN, \cP):= \left\{ (s,q,P) \in \J^2: \ s \leq 0, \  q \in \R^n, \text{and} \ P \geq 0     \right\}.
	\end{equation}
	Consequently, for each $c \in F(\J^2)$ the set
	\begin{equation}\label{SEGF4}
	\F_c := \{ (r,p,A) \in \J^2: F(r.p.A) \geq c \}
	\end{equation}
	is $\cM(\cN, \cP)$-monotone. Moreover, the operator $F$ is topologically tame and hence for every bounded domain $\Omega$, one has the comparison principle: 
	\begin{equation}\label{SEGF5}
	\mbox{$u \leq w$ on $\partial \Omega \ \ \Rightarrow \ \  u \leq w$ on $\Omega$}
	\end{equation}
	for $u \in \USC(\overline{\Omega})$ and $w \in \LSC(\overline{\Omega})$ which are respectively $\F_c$-subharmonic and $\F_c$-superharmonic in $\Omega$, or equivalently if $u,w$ are viscosity subsolutions and supersolutions to $F(u,Du,D^2u) = c$ on $\Omega$.
\end{thm}

\begin{proof}
	The proof is almost identical to that of Theorem \ref{thm:CP_SE}. We limit ourselves to showing that the pair $(F, \J^2)$ is gradient-free which is equivalent to the $\cM(\cN, \cP)$-monotonicity of $F$ on all of $\J^2$. By combining \eqref{SEGF1} with \eqref{SEGF2} with $\mu = 0$, for each $(r,p,A) \in \J^2$ and each $(s,q,P) \in \cM(\cN, \cP)$ one has 
	$$
	F(r + s, p + q, A + P) \geq F(r, p + q, A) = F(r,p +q, A + 0I) \geq F(r,p,A) + 0,  
	$$
	as needed.
\end{proof}

Next we give a few examples covered by Theorems \ref{thm:CP_SE} and \ref{thm:CP_SEGF}.

\begin{exe}\label{exe:SE} The case $\beta = 0$ of Theorem \ref{thm:CP_SEGF} yields comparison on arbitrary bounded domains for proper elliptic gradient-free operators $F(r,p,A) = G(r,A)$ which are strictly ellipitc in the direction $I \in \cS(N)$; that is, $G \in C(\R \times \cS(N))$ such that  
	\begin{equation}\label{exe:SEG1}
	G(r + s, A + P) \geq G(r,A), \ \ \forall \, (r,A) \in \R \times \cS(N), s \leq 0, P \geq 0.
	\end{equation}
	and for some $\alpha > 0$ 
	\begin{equation}\label{exe:SEG2}
	G(r, A + \mu I) - G(r,A) \geq  \alpha \mu, \ \ \forall \, (r,A) \in \R \times \cS(N), \mu \geq 0.
	\end{equation}
	With respect to the unconstrained case of Theorem \ref{thm:E0}, the strict monotonicity in \eqref{exe:SEG2} ensures the topological tameness, which was a hypothesis in the previous theorem. 
	
	The case $\beta > 0$ of Theorem \ref{thm:CP_SE} yields comparison on domains contained in translates of $B_R(0)$ with $R \leq \frac{\alpha}{\beta}$ for operators of the form 
	\begin{equation}\label{SE_LipGrad}
	F(r,p,A) = G(r,A) - \langle b, p \rangle
	\end{equation}
	with $G$ as above and $b \in \R^n \setminus \{0\}$. One can then choose $\beta = |b|$. 
\end{exe}

 The next remark concerns the result obtained in Theorem \ref{thm:CP_SE}.
 
 \begin{rem}\label{rem:thm:CP_SE}
 The restriction that $R \leq \frac{\alpha}{\beta}$  can be viewed in two ways. First, one can say that comparison holds on all bounded domains $\Omega$ provided that the Lipschitz constant $\beta$ is small relative to the diameter $2R$ of $\Omega$ and the partial ellipticity constant $\alpha$; that is, if
 	\begin{equation}
 		\beta \leq \frac{\alpha}{R}.  
 	\end{equation} 
 	On the other hand, for fixed Lipschitz constant $\beta$ and partial ellipticity constant $\alpha$, comparison is ensured only for domains with diameter $2R$ satisfying 
 	\begin{equation}
 		R \leq \frac{\alpha}{\beta}.
 		\end{equation}
 \end{rem}
 
  	This remark raises two interesting questions. The first question is what minimal further strengthening of the notion of ellipticity gives comparison on arbitrarily large domains independent of the Lipschitz constant $\beta$? One classical answer will be given below in the constrained context by stengthening the notion of ellipticity. The second question is whether the monotonicity cone $\cM(\cN,\cD,R)$ is the maximal monotonicity cone $\cM_{\F}$ for $\F$ defined by \eqref{SE4}? Perhaps, at least in some important special cases, additional structure in $F$ can lead to a larger monotonicity cone with strict approximators (and hence comparison) on arbitrary domains with large Lipschitz constant. This second question was considered at the potential theoretic level in section \ref{sec:improvemnts}. 

The following strengthening of the ellipticity recovers Jensen's uniformly elliptic result\footnote{ We will use the term {\em strictly elliptic} since we are asking only for a one-sided ellipticity bound, reserving {\em uniformly elliptic} to a two-sided bound.} in \cite{Je88} by using our monotonicity method and adds an unconstrained potential theoretic version of the result at all admissible levels $c \in F(\J^2)$. 

\begin{thm}\label{thm:CP_UE}
	
	Suppose that $F : \J^2 \to \R$ is continuous and that $F$ is proper and strictly elliptic; that is, there exists $\lambda > 0$ such that for each $(r, p, A) \in \J^2$ one has
	\begin{equation}\label{ue1}
	F(r + s, p, A + P) - F(r,p,A) \geq \lambda \, {\rm tr} \, P \ \ \text{for each}  \ s \in \cN \ \text{and each} \ P \in \cP.
	\end{equation}
	In addition, suppose that $F$ satisfies the following structural condition: there exists $\beta > 0$ such that for each $(r,p,A) \in \J^2$ one has
		\begin{equation}\label{ue2}
	F(r, p + q, A) -  F(r, p, A) \ge  - \beta|q| \ \ \ \text{for each}  \ q \in \R^n.
	\end{equation}
	Then $(F, \J^2)$ is an unconstrained proper elliptic compatible operator-subequation pair which is $\cM$-monotone for the monotonicity cone subequation
		\begin{equation}\label{ue3}
	\cM = \cM_{\lambda, \beta}(\cN, \cD, \cP) := \left\{ (s,q,P) \in \cN \times \cD \times \cP: \ \lambda \, {\rm tr} \, P \geq \beta |q|    \right\}
	\end{equation}
	with directional cone $\cD =\R^n$. Consequently, for each $c \in F(\J^2)$ the set
	\begin{equation}\label{ue4}
	\F_c := \{ (r,p,A) \in \J^2: F(r.p.A) \geq c \}
	\end{equation}
	is $\cM_{\lambda, \beta}(\cN, \R^n, \cP)$-monotone. Moreover, the operator $F$ is topologically tame and hence for every bounded domain $\Omega$ one has the comparison principle: 
	\begin{equation}\label{ue5}
	\mbox{$u \leq w$ on $\partial \Omega \ \ \Rightarrow \ \  u \leq w$ on $\Omega$}
	\end{equation}
	for $u \in \USC(\overline{\Omega})$ and $w \in \LSC(\overline{\Omega})$ which are respectively $\F_c$-subharmonic and $\F_c$-superharmonic in $\Omega$, or equivalently if $u,w$ are viscosity subsolutions and supersolutions to $F(u,Du,D^2u) = c$ on $\Omega$.
	\end{thm}

\begin{proof}
	For each directional cone $\cD \subseteq \R^n$ (in the sense of Definition \ref{defn:property_D}), one clearly has that $\cM_{\lambda, \beta}(\cN, \cD, \cP)$ is a monotonicity cone subequation which contains $\cM_0 = \cN \times \{0\} \times \cP$ and that $(F, \J^2)$ is an unconstrained compatible proper ellpitic operator-subequation pair, by imitiating the argument of Theorem \ref{thm:CP_SE} where the structural conditions \eqref{ue1}-\eqref{ue2} play the same role as \eqref{SE1}-\eqref{SE2}. The conclusion that each subequation $\F_c$ is $\cM_{\lambda, \beta}(\cN, \R^n, \cP)$-monotone again follows from Lemma \ref{lem:upper_levels}. The topological tameness of $F$ follows from part (2) of Theorem \ref{thm:tameness} by using the strict monotonicity in \eqref{ue2}. 
	
	Hence, by the general comparison result of Theorem \ref{thm:CP_general}, the comparison principles for each $c \in F(\J^2)$ and each bounded domain $\Omega$ reduce to the question of whether the monotonicity cone subequation $\cM_{\lambda, \beta}(\cN, \R^n, \cP)$ admits a strict approximator in the sense of Definition 
	\ref{defn:strict_approx} on a given bounded domain $\Omega$. We will give the argument for general $\cD$, eventhough here we need only the special case $\cD = \R^n$. Since $\Omega$ is bounded and since the convex cone $\cD$ has interior, by translating $\Omega$, we can assume that $\overline{\Omega} \subset \cD_R = \cD \cap B_R(0)$ for some $R > 0$ and and that $0 \notin \Omega$.
	We will look for $\psi \in C^{\infty}(\R^n \setminus \{0\})$ of the form 
	\begin{equation}\label{ue_sa}
	\psi(x) = \frac{1}{\mu} e^{\mu |x|} - m \ \ \text{with} \ \ \mu, m > 0, 
	\end{equation}
	to be determined so that for each $x \in \Omega$, one has $(s,q,P) := (\psi(x), D \psi(x), D^2 \psi(x)) \in \Int \, \cM_{\lambda, \beta}(\cN, \cD)$ where
	 \begin{equation}\label{UE_Requests}
	 \Int \, \cM_{\lambda, \beta}(\cN, \cD) = \{ s < 0, q \in \Int \, D, P > 0 \ \text{and} \ \lambda {\rm tr} P > \beta |q| \}.
	\end{equation}
	Clearly
	\begin{equation}\label{UESA1}
	 s := \psi(x)  = e^{\mu|x|} - m < 0 \ \ \text{if} \ m < e^{\mu R}
	\end{equation}
	and 
		\begin{equation}\label{UESA2}
	q := D \psi(x)  =  e^{\mu|x|}\frac{x}{|x|} \in \Int \, \cD
	\end{equation}
	since $\overline{\Omega} \subset \cD \cap (B_R(0) \setminus \{0\})$. Moreover, as computed in Example 4 of Remark \ref{rem:radial_calculus}, the Hessian of the radial function $\psi$  is $P:= D^2 \psi(x) = \frac{1}{|x|} e^{\mu |x|} P_{x^{\perp}} + \mu e^{\mu |x|}P_x$ and hence the eigenvalues of $P$ are
	\begin{equation}\label{UESA3}
	 \frac{e^{\mu |x|}}{|x|} \ \text{(with multiplicity $n-1$)} \ \ \text{and} \ \ \mu e^{\mu |x|} \ \text{(with multiplicity $1$)}.
	\end{equation}
Hence $P > 0$ and one has
	$$
	\lambda \, {\rm tr} \, P - \beta |q| =  e^{\mu |x| }\left[ \frac{\lambda (n - 1)}{|x|} + \lambda \mu - \beta \right] > 0, 
	$$
	by choosing $\mu > \beta/\lambda$. 
\end{proof}

\subsection{From linear operators to Hamilton-Jacobi-Bellman operators}\label{subsec_linear} 

In this subsection, we will discuss some unconstrained situations which illustrate the use of $\cM = \cN \times \cD \times \cP$-monotonicity for linear equations and certain Hamilton-Jacobi-Bellman equations. Admittedly, in the linear case, there is nothing new here, but perhaps it is still useful to illustrate how this classical case fit into this part of the theory.

Consider the class {\em proper elliptic linear operators} $F: \J^2 \to \R$; that is, $F$ is linear and $\cM_0 = \cN \times \{0\} \times \cP$-monotone on all of $\J^2$. Each such operator is determined by the choice of a non zero coefficient vector $J' :=(a,b,E) \in \cN \times \R^n \times \cP$; that is, 
\begin{equation}\label{L1}
	F(r,p,A) := \langle J', J \rangle = {\rm tr}(EA) + \langle b, p \rangle + ar, \ \ \text{for each} \ J:=(r,p,A) \in \J^2.
\end{equation}
Since $E \geq 0$ in $\cS(n)$ and $a \leq 0$ in $\R$, one has that $(F, \J^2)$ is an unconstrained case proper elliptic pair.  Moreover, since $F$ is linear and $(a,b,E)$ is non zero, $F$ is topologically tame and the range $F(\J^2)$ ia all of $\R$.

\begin{thm}[Linear equations]\label{thm:linear} Suppose that $F$ is a proper elliptic linear operator with non zero coefficient vector $(a,b,E) \in \cN \times \R^n \times \cP$ as defined in \eqref{L1}. Then for every $c \in \R$ the affine half-space
	\begin{equation}\label{AHSc}
	\F_c := \{ (r,p,A) \in \J^2: \ F(r,p,A):=  {\rm tr}(EA) + \langle b, p \rangle + ar \geq c \}
	\end{equation}
is a subequation constraint set.  

For each bounded domain $\Omega$ in $\R^n$ one has the comparison principle 
\begin{equation}\label{CP_L}
\mbox{$u \leq w$ on $\partial \Omega \ \ \Rightarrow \ \  u \leq w$ on $\Omega$}
\end{equation}
for $u \in \USC(\overline{\Omega})$ and $w \in \LSC(\overline{\Omega})$ which are respectively $\F_c$-subharmonic and $\F_c$-superharmonic in $\Omega$, or equivalently,  if $u$ and $w$ are respectively  a viscosity subsolution and a viscosity supersolution to $F(u,Du,D^2u) = c$ on $\Omega$. Or if one prefers, since the dual $\wt{\F}_c$ of $\F_c$ is $\F_{-c}$
\begin{equation}\label{CP_L2}
\mbox{$u + v \leq 0$ on $\partial \Omega \ \ \Rightarrow \ \  u + v \leq 0$ on $\Omega$}
\end{equation}
for each pair $u \in \F_c(\overline{\Omega})$ and $v \in \F_{-c}(\overline{\Omega})$.
\end{thm}

\begin{proof} It is easy to verify that each $\F_c$ is a subequation using the definition \eqref{AHSc} with $E \in \cP, a \in \cN$ and $(a,b,E)$ non zero. The pair $(F, \J^2)$ is $\cM$-mononotone for the fundamental product monotoncity subequation
\begin{equation}\label{ML}
	\cM := \cN \times \cD_b \times \cP = \{ (s,q,P) \in \J^2: s \leq 0, q \in \cD, P \in \cP \}
\end{equation}
where the directional cone $\cD_b \subset \R^n$ is $\cD_0 := \R^n$ in the gradient-free case of $b = 0$ and is the half-space
\begin{equation}\label{DL}
	\cD_b:= \{ q \in \R^n: \ \langle b, q \rangle\geq 0 \}
\end{equation}
in the remaining case $b \neq 0$. Indeed, for each $(r,p,A) \in \J^2$ one has
\begin{eqnarray*}
F(r + s, p + q, A + P) & = & {\rm tr}(E(A + P)) + \langle b, p + q \rangle + a(r + s) \\
	& = & F(r,p,A) + {\rm tr}(EP) + \langle b, q \rangle + as \geq F(r,p,A)
\end{eqnarray*}
since ${\rm tr}(EP), \langle b, q \rangle$ and $as$ are all non-negative. Hence each $\F_c$ is also $\cM$-monotone and the comparison principle \eqref{CP_L} for its subharmonics and superharmonics follows from Theorem \ref{thm:comparison}. 
Since $F$ is topologically tame on $\J^2$, by Theorem \ref{cor:AVSolns} one will have the correspondence between $\F_c$-subharmonics/superharmonics and viscosity subslutions/supersolutions to $F = c$ since $(F, \J^2)$ is $\cM$-monotone.
\end{proof}

We now examine in more detail the linear case by showing that proper elliptic linear operators are canonical operators for the relevant half-space subequation. This will allow us to also represent certain special Hamilton-Jacobi-Bellman operators in terms of canonical operators.

\begin{lem}[Linear operators are canonical operators]\label{lem:LCO}
	Let $F$ be a proper elliptic linear operator with non zero coefficient vector $J' = (a,b,E) \in \cN \times \R^n \times \cP$
\begin{equation}\label{F_linear}
F(J) = F(r,p,A) := \langle J', J \rangle = {\rm tr}(EA) + \langle b, p \rangle + ar, \ \ \forall \, J = (r,p,A) \in \J^2.
\end{equation}
Then one has the following statements. 
	\begin{itemize}
	\item[(a)] The closed linear half-space $\F:= \{ J \in \J^2: F(J) := \langle J', J \rangle  \geq 0 \}$ with boundary orthogonal to $J' \in \Int \, \F$ is a monotonicity cone subequation which is seld-dual; that is, $\wt{\F} = \F$.
	\item[(b)] The maximal monotonicity cone $\cM_{\F}$ of $\F$ as defined in Definition \ref{defn:MMC} is just $\F$ iteself; that is, $\cM_{\F} = \F$.
		\item[(c)] Choosing any $J_0 \in \Int \, \F$, the rescaled proper elliptic linear operator with coefficient vector $J'/ \langle J', J_0 \rangle$,
	\begin{equation}\label{rescaled_operator}
		\overline{F}(J) := \frac{1}{\langle J', J_0 \rangle} F(J) = \frac{\langle J', J \rangle}{ \langle J', J_0 \rangle }, \ \ J \in \J^2
		\end{equation}
 is the canonical operator (determined by $J_0$) for $\F$ in the sense of Definition \ref{defn:canonical_op}. 
	\end{itemize}
	\end{lem}

\begin{proof} As noted in \eqref{CP_L}, with $c=0$ the closed half-space
		\begin{equation}\label{F_linear3}
	\F = \{ J \in \J^2: \  \langle J', J \rangle \geq 0 \}.
	\end{equation}
is a subequation, since it contains the minimal monotonicity set $\cM_0 = \cN \times \{0\} \times \cP$. The duality claim is obvious, since $\F(J) = - F(-J) = F(J)$, which completes part (a). For part (b), since $\F$ is a convex cone subequation, one has $\cM_{\F} = \F$ by Proposition \ref{prop:MMC2}. Finally, for part (c), notice that the normalization of \eqref{rescaled_operator} yields the affine property
\begin{equation}\label{F_linear4}
		\overline{F}(J + tJ_0) = \frac{ \langle J', J + t J_0 \rangle}{ \langle J', J_0 \rangle } = \overline{F}(J) + t \ \ \text{for each} \ J \in \J^2, t \in \R.
\end{equation}
The boundary of the half-space $\F$ is the hyperplane
\begin{equation}\label{F_linear5}
	\partial \F = \{ J \in \J^2: \ \overline{F}(J) = 0 \}.
\end{equation}
	Since \eqref{F_linear4} - \eqref{F_linear5} hold, $\overline{F}$ is the canonical operator for $\F$ determined by $J_0$ by Proposition \ref{prop:unique_F}.	
\end{proof}	  

\begin{rem}\label{rem:CO} In the formulation of canonical operators $F$ for a given subequation constraint set $\F$ which is $\cM$-monotone, we have made an inessential normalization with respect to the affine property \eqref{F_linear4}, which could be generalized to ask that for some $k > 0$ one has
\begin{equation}\label{AP}
		F(J + tJ_0) = F(J) + t k \ \ \text{for each} \ J \in \J^2, t \in \R.
\end{equation}
We have fixed this normalizing constant to be one, but general $k > 0$ has been used Proposition 6.19 of \cite{HL19} in the pure second order case. This normalization does not affect the validity of the relation \eqref{F_linear5} nor the other aspects of how $\partial \F$ decomposes $\J^2$, namely (see Proposition \ref{prop:unique_F}):
\begin{equation}\label{decomp}
	\Int \, \F = \{ J \in \J^2: \ F(J) > 0 \} \ \ \text{and} \ \ \J^2 \setminus \F = \{ J \in \J^2: \ F(J) < 0 \}.
\end{equation}	
That is, the formula for a canonical operator $F$ can be made to depend on both $J_0 \in \Int \, \cM$ and the normalizing constant $k$, which merely reparameterizes the distance to the boundary $\F$ of an $\cM$-monotone subequation $\F$. 
	\end{rem}
	
We conclude this subsection with a discussion of the comparison principle for a special class of Hamilton-Jacobi-Bellman equations. Admittedly the application is perhaps a bit contrived, but it does show that certain Hamilton-Jacobi-Bellman operators are canonical for the relevant convex cone subequation. We begin by introducing the class we will treat. Let $\Sigma$ be any index set and consider an arbitrary family of proper elliptic linear operators $\mathfrak{F}:= \{ F_{\sigma}: \J^2 \to \R\}_{\sigma \in \Sigma}$ where the linear operator 
\begin{equation}\label{HJB1}
	F_{\sigma}(J) = F_{\sigma}(r,p,A) := {\rm tr}(E_{\sigma} A) + \langle b_{\sigma}, p \rangle + a_{\sigma} r = \langle J_{\sigma} , J \rangle
\end{equation}
has non zero coefficient vector $J_{\sigma} = (a_{\sigma}, b_{\sigma}, E_{\sigma})$. By {\em proper ellipticity}
\begin{equation}\label{HJB2}
a_{\sigma} \leq 0 \ \text{in} \ \R \ \ \text{and} \ \ E_{\sigma} \geq 0 \ \text{in} \ \cS(n); \ \text{that is}, \ J_{\sigma} \in \cN \times \R^n \times \cP, \ \ \forall \, \sigma \in \Sigma.
\end{equation}
The associated linear subequations $\F_{\sigma} := \{ J \in \J^2: F_{\sigma}(J) \geq 0 \}$ are the halfspaces with boundary orthogonal
to $J_{\sigma}$.

We must also require a condition on the set of coefficient vectors in order for the interesection $\F := \bigcap_{\sigma \in \Sigma} \F_{\sigma}$ to be non empty. There are several equivalent ways of formulating this condition. We start with the geometric property of being directed or pointed in the following sense.

\begin{defn}\label{defn:pointed_set} A subset of non zero vectors $S:= \{J_{\sigma}\}_{\sigma \in \Sigma}$ in a finite dimensional inner product space $(V, \langle \cdot , \cdot \rangle)$ is said to be {\em pointed} if for some $J_0 \in V \setminus \{0\}$ (called the {\em axis}) 
	\begin{equation}\label{pointed1}
	\mbox{  $\exists \, \veps > 0$ such that $\langle J_{\sigma}, J_0 \rangle \geq \veps ||J_{\sigma}||, \  \forall \, \sigma \in \Sigma$,}
	\end{equation}
	or equivalently (with $R:= 1/\veps$),
	\begin{equation}\label{pointed2}
	\exists \,  R> 0 \ \text{such that} \ \langle J_{\sigma}, J_0 \rangle > 0 \ \text{and} \  \frac{||J_{\sigma}||}{\langle J_{\sigma}, J_0 \rangle} \leq R, \ \forall \, \sigma \in \Sigma.
	\end{equation}
\end{defn}

Now we examine the case of taking the infimum over such a family of proper ellipitc linear operators.

\begin{thm}[Infimum of a family of linear operators]\label{thm:HJB_inf}
	Suppose that $\mathfrak{F}:= \{F_{\sigma}\}_{\sigma \in \Sigma}$ and $\{ \F_{\sigma}\}_{\sigma \in \Sigma}$ are as above. 
	\begin{itemize}
		\item[(a)] The intersection 
		\begin{equation}\label{HJB_constraint1}
		\F_:= \bigcap_{\sigma \in \Sigma} \F_{\sigma} \subset \J^2.
		\end{equation}
		is a (convex cone) subequation if and only if 
		\begin{equation}\label{QC}
		\mbox{ the subset $S = \{J_{\sigma}\}_{\sigma \in \Sigma}$ of coefficients is pointed for some axis $J_0 \neq 0$.}
		\end{equation}
		\item[(b)] Assume that the intersection $\F$ is a subequation with $J_0$ as in \eqref{QC}. Renormalize each linear operator $F_{\sigma}$, as in \eqref{rescaled_operator},
		\begin{equation}\label{F_RN}
		\overline{F}_{\sigma}(J):= \frac{\langle J_{\sigma}, J \rangle}{\langle J_{\sigma}, J_0 \rangle} = \frac{1}{\langle J_{\sigma}, J_0 \rangle} F_{\sigma}(J).
		\end{equation}
		to be the canonical operator for $\F_{\sigma}$ with respect $J_0$. Then the infimum operator
		\begin{equation}\label{CO_inf}
		F(J) := \inf_{\sigma \in \Sigma} \overline{F}_{\sigma}(J), \ \ J \in \J^2
		\end{equation}
		is the canonical operator for the intersection subequation $\F$ with respect to $J_0$. Moreover, $F$ is a concave function on $\J^2$ (and hence continuous).
	\end{itemize}	
\end{thm}

Before giving the proof, a few remarks are in order.

\begin{rem}[Geometric interpretations]\label{rem:pointed}
	Both of the conditions \eqref{pointed1} and \eqref{pointed2} for $S$ to be pointed with axis $J_0 \neq 0$ have a geometric interpretation. First, with $\veps = ||J_0|| \cos{\theta}$ defining $\theta \in (0, \pi/2)$, condition \eqref{pointed1} says that the angle $\sphericalangle(J_0, J_\sigma) < \theta$ for each $J_{\sigma} \in S$; that is, $\langle \frac{J_0}{||J_0||} , \frac{J_{\sigma}}{||J_{\sigma}||} \rangle \geq \cos{\theta}$, or said differently, 
	\begin{equation}\label{pointed3}
	S \ \text{is contained in} \ C_{J_0, \theta} := \{ J \in \J^2: \langle J_0, J \rangle \geq \cos{\theta} ||J_0|| \, ||J|| \},
	\end{equation}
	where $C_{J_0, \theta}$ is called the {\em circular cone with axis $J_0$ and angle $\theta$}. Notice that if $R \in (0, + \infty)$ is related to $\theta \in (0, 2 \pi)$ by $\cos{\theta} = (1 + R^2)^{-1/2}$, then $C_{J_0, \theta} = C_{J_0}(R)$, the cone over $\overline{B}_R(J_0/||J_0||)$; that is,
	$$
	C_{J_0}(R) := \{ J \in J^2: \ J = tJ_0 + J' \ \text{where} \ J' \perp J_0 \ \text{and} \ ||J'|| \leq t||J_0||R  \}.
	$$ 
	
	Second, with the normalization $\bar{J}_{\sigma}:= J_{\sigma}/ \langle J_{\sigma}, J_0 \rangle$, so that each $\bar{J}_{\sigma}$ belongs to the affine hyperplane $\{ J \in \J^2: \ \langle J, J_0 \rangle = 1 \}$, condition \eqref{pointed2} says that
	\begin{equation}\label{pointed4}
	\{ \bar{J}_{\sigma} \}_{\sigma \in \Sigma} \subset B_R(0), \ \text{the ball of radius $R$ and center $0$ in $\J^2$}.
	\end{equation}
	
	There are many equivalent formulations of $S$ being pointed involving the closed convex cone hull $C(S)$ of $S$. For example, $S$ is pointed if and only if $C(S)$ contains no nontrivial subspaces; that is, $C(S)$ has no {\em edge} (as defined in \eqref{edge}).
\end{rem}

\begin{proof}[Proof of Theorem \ref{thm:HJB_inf}] For part (a), first note that the five properties of closedness, positivity (P), negativity (N), being a cone and being convex are all preserved under arbitrary intersections. Thus the interesection $\F$ is a closed convex cone satisfying (P) and (N), so that $\F$ is a (convex cone) subequation if and only if the topological property (T) $\F = \overline{ \Int \, \F}$ holds. Since $\F$ is closed and convex,  $\F = \overline{ \Int \, \F}$ if and only if $\Int \, \F \neq \emptyset$. Therefore
	\begin{equation}\label{interior_nonempty}
	\F := \bigcap_{\sigma \in \Sigma}\F_{\sigma} \ \text{is a subequation} \ \ \Leftrightarrow \ \ \Int \, \F \neq \emptyset.
	\end{equation}
	By the definitions, $\F:= \bigcap_{\sigma \in \Sigma} \F_{\sigma}:= \{ J \in \J^2: \ \langle J, J_{\sigma} \rangle \geq 0, \ \forall \, \sigma \in \Sigma \} := S^{\circ}$, the {\em (convex cone) polar} (as defined in \eqref{polar}) of the set of coefficient vectors $S := \{J_{\sigma}\}_{\sigma \in \Sigma}$. Recall that the polar of any set in an inner product space is always a closed convex cone (see Remark \ref{rem:polars}). The next lemma completes the proof of part (a), since by hypothesis $S$ is pointed with axis $J_0$, or equivalently, $S \subset C_{J_0, \theta}$ with some $\theta \in (0,\pi/2)$ as noted in \eqref{pointed3}.
	
	\begin{lem}\label{lem:interior_cone} Let $S = \{J_{\sigma}\}_{\sigma \in \Sigma}$ be a collection of non zero vectors in a finite dimensional inner product space. Then
		\begin{equation}\label{pointed_lemma}
		J_0 \in \Int \, S^{\circ} \ \ \Leftrightarrow \ \ S \subset C_{J_0, \theta} \ \text{for some} \ \theta \in (0, \pi/2).
		\end{equation}	
	\end{lem}
	\begin{proof}
		It is straightforward to compute the polar of a circular cone
		\begin{equation}\label{polar_CC}
		C_{J_0, \theta}^{\circ} = C_{J_0, \theta'} \ \ \text{with} \ \theta' = \pi/2 - \theta \in (0, \pi/2).
		\end{equation}
		Hence $S \subset C_{J_0, \theta}$ with $\theta \in (0, \pi/2)$ if and only if 
		\begin{equation}\label{common_cone}
		C_{J_0, \theta'} \subset S^{\circ} = \F \ \ \text{with} \ \theta' = \pi/2 - \theta \in (0, \pi/2).
		\end{equation}
		Finally, $J_0 \in \Int \, \F \ \Leftrightarrow \ C_{J_0, \theta'} \subset \F$ for some $\theta' \in (0, \pi/2)$.
	\end{proof}
	
For part (b), first notice that since $J_0 \in \Int \left( \bigcap_{\sigma \in \Sigma} \F_{\sigma} \right) \subset \bigcap_{\sigma \in \Sigma} \Int \, \F_{\sigma}$, the axis $J_0 \in \Int \, \F_{\sigma}$ for each $\sigma \in \Sigma$, so that $\langle J_{\sigma}, J_0 \rangle > 0$ By Lemma \ref{lem:LCO} (c), each renormalized operator $\overline{F}_{\sigma}$ is the canonical operator with respect to $J_0$ for $\F_{\sigma}$. Since  the intersection $\F$ is a convex cone subequation, by Proposition \ref{prop:MMC2} we have that $\F$ is its own maximal monotonicity cone; that is $\F = \cM_{\F}$ with $J_0 \in \Int \, \cM_{\F}$. Since $\cM_{\F}$ is a monotonicity cone for the intersection $\F$, $\cM_{\F}$ is a monotonicity cone for each $\F_{\sigma}$ and $\overline{F}_{\sigma}$ is the canonical operator for $\F_{\sigma}$ with respect to $J_0 \in \Int \, \cM_{\F}$. Since the intersection is non empty, $F := \inf_{\sigma \in \Sigma} \overline{F}_{\sigma}$ is the canonical operator for the interesection $\F$ with respect to $J_0$ by Theorem \ref{thm:canonical_inf}(a).
\end{proof}

\begin{rem}\label{rem:pointed_cones} Any convex cone subequation $\F \subset \J^2$ can be written as the intersection of a family of half-space subsequations.  As noted in the proof, the polar $\F^{\circ}$ of $\F$ is pointed. Choose any generating set $S:= \{ J_{\sigma} \}_{\sigma \in \Sigma}$ of non zero vectors in $\F^{\circ}$ so that $C(S) = \F^{\circ}$. Then since $\F^{\circ}$ must be pointed by Lemma \ref{lem:interior_cone}, $S$ is a pointed set in the sense of Definition \ref{defn:pointed_set}. 
	\end{rem}

Armed with Theorem \ref{thm:HJB_inf}, we briefly discuss the comparison principle  for concave Hamilton-Jacobi-Bellaman operators $F$ which are the infimum over a renormaized family of proper elliptic linear operators whose coefficients are a pointed set  $S:= \{ J_{\sigma} \}_{\sigma \in \Sigma} \subset \cN \times \R^n \times \cP$. Comparison will always hold for such operators, with possbily some restriction on the diameter of the domain $\Omega$. The main point is contained in the following remark.

\begin{rem}\label{rem:CP_HJB} Suppose that $S:= \{ J_{\sigma} \}_{\sigma \in \Sigma} \subset \cN \times \R^n \times \cP$ is a set of non zero coefficient vectors which is pointed with axis $J_0 \neq 0$. By Theorem \ref{thm:HJB_inf}, we know that the concave Hamilton-Jacobi-Bellam operator 
\begin{equation}\label{CO_inf_rem}
F(J) := \inf_{\sigma \in \Sigma} \overline{F}_{\sigma}(J) :=   \inf_{\sigma \in \Sigma} \frac{\langle J_{\sigma}, J \rangle}{\langle J_{\sigma}, J_0 \rangle}, \ \ J \in \J^2
\end{equation}
is the canonical operator (detemined by $J_0$) for the monotonicity cone subequation
\begin{equation}\label{intersection_subequation}
	\cM_{\F} = \F := \bigcap_{\sigma \in \Sigma} \{ J \in \J^2: \ \overline{F}_{\sigma}(J) \geq 0\}.
\end{equation}
Hence $(F, \J^2)$ is an (unconstrained case) compatible operator-subequation pair with $F$ topologically tame and the pair is $\cM_{\F}$ monotone. Comparison on a domain $\Omega$ for viscosity subsolution/supersolution pairs of $F$ (or for pairs of $\cM_{\F}=\F$ subharmonics/superharmonics) reduces to the validity of the (ZMP) for the dual subequation $\wt{\cM}_{\F}$ on $\Omega$, which holds  if one has the existence of a $C^2$ strict $\cM_{\F}$-subharmonic $\psi$ on $\Omega$. 

Since $\cM_{\F}$ is a mononotonicity cone subequation, by Theorem \ref{thm:fundamental} it contains some element $\cM(\gamma, \cD, R)$ of our fundamental family with $\gamma,  R \in (0, +\infty]$ and $\cD \subseteq \R^n$ a directional cone. Hence $\cM_{\F}$ does indeed admit a strict approximator (with the restriction on the diameter of $\Omega$ in the case that the maximal cone $\cM(\gamma, \cD, R) \subset \cM_{\F}$ has $R$ finite). Hence, comparison in some form will always hold (with possible restrictions of domain diameter).

This leads to the following important question: under what assumption on the coefficients $S:= \{ J_{\sigma} \}_{\sigma \in \Sigma}$ will we have a given inclusion 
\begin{equation}\label{cone_inclusion}
\mbox{$\cM \subset \cM_{\F}$ \ \ for a given monotoncicity cone subequation $\cM$?}
\end{equation}
The needed inclusion \eqref{cone_inclusion} is equivalent to the reverse inclusion for the polars
\begin{equation}\label{polar_condition}
\cM_{\F}^{\circ} \subset \cM^{\circ}
\end{equation}
and recalling that $S^{\circ} = \cM_{\F}$, a necessary and sufficient condition on $S$ in order to have \eqref{cone_inclusion} is
\begin{equation}\label{S_condition}
S \subset \cM^{\circ}.
\end{equation}
As noted before, since $\cM_0:= \cN \times \{0\} \times \cP \subset \cM$ for every monotonicity cone subequation $\cM$, one must have the set $S$ of coefficient vectors contained in
$$
\cM^{\circ} \subset \cM_0^{\circ} = \cN \times \R^n \times \cP;
$$
that is, a set $S$ of proper elliptic coefficient vectors in $\J^2$.
	\end{rem}

Combining Theorem \ref{thm:HJB_inf} with the considerations of Remark \ref{rem:CP_HJB} yields the following result. The reader might wish to consult Definition \ref{defn:cone_zoo} and Remark \ref{rem:cone_zoo_compress} to review the family of monotonicity cones as well as Theorem \ref{thm:ZMP_for_M} on the validity of (ZMP) for the duals of our family of monotonicity cone subequations $\cM$.

\begin{thm}[Comparison for the inf of a pointed family of linear operators]\label{thm:CP_HJB}
	Suppose that $\{J_{\sigma}\}_{\sigma \in \Sigma}$ be a pointed set (with axis $J_0 \neq 0$) of non zero vectors in $\cN \times \R^n \times \cP \subset \J^2$. Consider the associated (normalized) proper elliptic linear operators 
	\begin{equation}\label{normalized_operators}
	\overline{F}_{\sigma}(J):= \frac{\langle J_{\sigma}, J \rangle}{\langle J_{\sigma}, J_0 \rangle} = \frac{1}{\langle J_{\sigma}, J_0 \rangle} F_{\sigma}(J).
	\end{equation}
	and half-space subequations
	$$
	\F_{\sigma} := \{ J \in \J^2: \overline{F}_{\sigma}(J) \geq 0 \},
	$$
	whose intersection $\cM_{\F} = \F := \bigcap_{\sigma \in \Sigma} \F_{\sigma}$ is a convex cone subequation for which
	$$
	F(J) := \inf_{\sigma \in \Sigma} \overline{F}_{\sigma}(J)
	$$
	is the canonical operator (determined by $J_0$) for $\cM_{\F}=\F$.
	\begin{itemize}
		\item[(a)] Suppose that the coefficent vectors satisfy $S \subset \cM^{\circ}$ with $\cM$ being one of the monotonicity cone subsequations $\cM(\gamma), \cM(\cP), \cM(\cD), \cM(\gamma, \cD), \cM(\cD, \cP)$ or $\cM(\gamma, \cD, \cP)$ (the case $R = +\infty$ of Theorem \ref{thm:ZMP_for_M}). Then for every $c \in \R$ one has the comparison principle
		\begin{equation}\label{CP_HJB1}
		\mbox{$u \leq w$ on $\partial \Omega \ \ \Rightarrow \ \  u \leq w$ on $\Omega$}
		\end{equation}
		for $u \in \USC(\overline{\Omega})$ and $w \in \LSC(\overline{\Omega})$ which are a viscosity subsolution/supersolution pair for Hamiton-Jacobi-Bellman equation $F(u,Du,D^2u)=c$. In this case, $\Omega$ is an arbitrary bounded domain.
		\item[(b)] If, instead, $S \subset \cM^{\circ}$ with $\cM$ being one of the monotonicty cone subequations $\cM(R), \cM(\gamma, R), \cM(\cD, R)$ or $\cM(\gamma, \cD, R)$ with $R$ finite (the case $R$ finite of Theorem \ref{thm:ZMP_for_M}), then the comparison principle for $F$ holds on domains contained in a translate of the truncated cone $\cD \cap B_R(0)$, which is a ball of radius $R$ in the case $\cD = \R^n$.
	\end{itemize}
\end{thm}

In order to implement Theorem \ref{thm:CP_HJB}, which requires the condition \eqref{S_condition} on the coefficients; that is, $S \subset \cM^{\circ}$, we list some of the polars of our family of monotoncity cone subequations $\cM$. The proof is left to the reader.

\begin{prop}[Polars of some monotonicity cone subequations]\label{prop:polar_cones}
One has the following polar formulas.
	\begin{itemize}
		\item[(a)] For $\cM(\cP) := \{ (r,p,A) \in \J^2: A \in \cP\}$, one has:
	\begin{equation}\label{polar1}
	\cM(\cP) ^{\circ} =  \{0\} \times \{0\} \times \cP .
	\end{equation}
			\item[(b)] For $\cM(R) := \{ (r,p,A) \in \J^2: A \geq \frac{|p|}{R}I \}$ with $R\in (0, +\infty)$, one has:
	\begin{equation}\label{polar2}
	\cM(R)^{\circ} = \left\{ (s,q,B) \in \J^2: s = 0, B \geq 0 \ \text{and} \ \ {\rm tr}B \geq R|q| \right\}. 
	\end{equation}	
	\item[(c)] For $\cM(\gamma) := \{ (r,p,A) \in \J^2: r \leq -\gamma |p| \}$ with $\gamma \in [0, +\infty)$, one has:
	\begin{equation}\label{polar3}
	\cM(\gamma)^{\circ} = M'(1/\gamma) \times \{0\} = \left\{ (s,q,B) \in \J^2: B = 0, s \leq - \frac{1}{\gamma}|q| \right\},  
	\end{equation}
	which includes $\cM(\cN)^{\circ} = \cN \times \{0\} \times \{0\}$ in the case $\gamma = 0$.		
\item[(d)] For $\cM(\gamma, R) := \{ (r,p,A) \in \J^2: r \leq -\gamma |p| \ \ \text{and} \ \ A \geq \frac{|p|}{R}I \}$ with $\gamma \in [0, +\infty)$ and $R \in (0, +\infty]$, one has:
	\begin{equation}\label{polar4}
	\cM(\gamma, R)^{\circ} = \left\{ (s,q,B) \in \J^2: B \geq 0 \ \text{and} \  {\rm tr}B \geq R(|q| + \gamma s)  \right\},
	\end{equation}
	which includes $\cM(\cN,R)^{\circ} = \left\{ (s,q,B) \in \J^2: s \leq 0, B \geq 0 \ \text{and} \  {\rm tr}B \geq R|q|   \right\}$ in the case $\gamma = 0$.
	\item[(e)] For $\cM(\gamma, \cP) := \{ (r,p,A) \in \J^2: r \leq -\gamma |p| \ \ \text{and} \ \ A \geq 0 \}$ with $\gamma \in [0, +\infty)$, one has:
	\begin{equation}\label{polar5}
	\cM(\gamma, \cP)^{\circ} = \left\{ (s,q,B) \in \J^2: B \geq 0 \ \text{and} \ r \leq - \frac{1}{\gamma} |q| \right\}.
	\end{equation}
		\item[(f)] For $\cM(\cN, \cP) := \cN \times \R^n \times \cP = \{ (r,p,A) \in \J^2: r \leq 0 \ \ \text{and} \ \ A \geq 0 \}$, one has:
	\begin{equation}\label{polar6}
	\cM(\cN, \cP)^{\circ} = \cN \times \{0\} \times \cP = \left\{ (s,q,B) \in \J^2: s \leq 0\ \ \text{and} \ \  B \geq 0 \right\}.
	\end{equation}
\item[(g)] For $\cM(\cN, \cD, \cP) := \cN \times \cD \times \cP$ with $\cD \subsetneq \R^n$, one has:
	\begin{equation}\label{polar7}
\cM(\cN, \cD, \cP) ^{\circ} =  \cN \times \cD^{\circ} \times \cP.
	\end{equation}
	Finally, for the interesection of any of the cones $\cM$ in the cases (a) - (f) with $\cM(\cD)$ (with $\cD \subsetneq \R^n$), the polar $(\cM \cap \cM(\cD))^{\circ}$ can be expressed as $\overline{(\cM^{\circ} + \cM(\cD)^{\circ})}$, but an explicit description is more complicated and is left to the interested reader. 
	\end{itemize}
\end{prop}
We now present two representative examples of pointed families $S = \{J_{\sigma}\}_{\sigma \in \Sigma}$ which give rise to comparison for the associated Hamilton-Jacobi-Bellman operators, with and without restrictions on the size of the domain $\Omega$. First, we give a simple example where there is an a priori  restriction on the size of the domain.

\begin{exe}\label{exe:HJB1} For  $R \in (0, +\infty)$, consider the convex cone subequation defined in \eqref{RCone} and \eqref{exe:CE1_CP} 
\begin{equation}\label{F+1R}
	\cM(R) = \F^{-}_{1, R} := \left\{ (r,p,A) \in \J^2 : \lambda_1(A) - \frac{|p|}{R} \ge 0 \right\},
\end{equation}
where $\lambda_1(A)$ is the smallest eigenvalue of $A \in \Symn$.
Consider the index set
$$
	\Sigma := S^{n-1} \times S^{n-1} = \{\sigma = (\xi, \eta) \in \R^n \times \R^n: \ \ |\xi| = 1 = |\eta| \}. 
$$
Since for all $A \in \Symn$ and $p \in \R^n$,
	\[
	\lambda_1(A) = \inf_{\xi \in S^{n-1}} {\rm tr}((\xi \otimes \xi))
	A), \qquad \text{and} \quad |p| = - \inf_{\eta \in S^{n-1}} \langle \eta, p
	\rangle,
	\]
	where $ {\rm tr}((\xi \otimes \xi)A) = \langle A \xi, \xi \rangle = \langle A, P_{\xi} \rangle$, one has that
	\[
	\F^{-}_{1, R} = \bigcap_{\sigma \in \Sigma} \F_{\sigma}, 
	\]
	where
	$$ \quad \F_{\sigma} :=
	\{ (r,p,A) \in \J^2 : \langle J_{\sigma}, J \rangle \ge 0 \} \ \ \text{with} \ \ J_{\sigma} := \left(0, \frac{\eta}{R}, \xi \otimes \xi \right).
	$$
	One can easily check that the set of coefficient vectors $S := \{J_{\sigma} \}_{\sigma \in \Sigma}$ is pointed
	with axis $J_0 = (0, 0, I)$. Furthermore, since $\langle J_{\sigma}, J_0
	\rangle = 1$ for all $\sigma \in \Sigma$, the infimum operator defined in
	\eqref{CO_inf} is given by
	\[
	F(J) = \inf_{\sigma \in \Sigma} \frac{\langle J_{\sigma}, J \rangle}{\langle
		J_{\sigma}, J_0 \rangle} = \lambda_1(A) - \frac{|p|}{R}.
	\]
	Finally, we know that the polar $S^{\circ}$ of the set $S$ of coefficient vectors is the maximal monotonicity cone $\cM_{\F^{+}_{1,R}}$ of $\F^{+}_{1,R}$, which was shown in Proposition \ref{prop:CE1_CP} (a) to satisfy $\cM_{\F^+_{1,R}} = \cM(R)$ and hence $S^{\circ} = C(S)^{\circ} = \cM(R)$, where $C(S)$ is the closed convex hull of $S$. By the bipolar theorem, one has $S \subset C(S)= \cM(R)^{\circ}$. Hence, by Theorem \ref{thm:HJB_inf} (b), the comparison principle for the the equation $F(u, Du, D^2u) = c $ holds on domains $\Omega$ contained in a ball of radius $R$. 
\end{exe}

We now turn to a class of Hamilton-Jacobi-Bellman
equations for which comparison holds in arbitrary bounded domains.
These examples are motivatedby a question of Optimal Control. A good general reference for this subject is the monograph of Bardi-Capuzzo Dolcetta \cite{BCD97}.

\begin{exe}\label{exe:HJB2}  One important problem in Optimal Control concerns an agent who seeks to minimize  an
	infinite-horizon discounted cost functional by acting on its drift and
	volatility parameters. We consider infima of linear operators of the form
	\begin{equation}\label{OCO}
	F_\sigma(J) = F_\sigma (r,p,A) = {\rm tr}(E_\sigma A) + \langle
	b_\sigma, p \rangle + cr = \langle J_{\sigma}, J \rangle, \sigma \in \Sigma
	\end{equation}
	where $\delta:= -c > 0$ is the {\em discount factor}, $b_\sigma$ is the {\em drift term}
	and $E_\sigma$ is the {\em (squared) volatility}. In this example, $E_\sigma$
	is allowed to vary in bounded sets; that is, for some $m > 0$
	\begin{equation}\label{volatility_bound}
	\|E_\sigma\| \le m \qquad \forall \sigma \in \Sigma,
	\end{equation}
	The set of drifts $S_d:= \{b_\sigma\}_{\sigma \in \Sigma}$ will be taken to be pointed with axis $b_0 \in \R^n \setminus \{0\}$; that is,
	\begin{equation}\label{pointed_drift}
	\exists \, \varepsilon' > 0 \text{ such that } \langle b_0, b_\sigma
	\rangle \ge \varepsilon' |b_\sigma|, \ \forall \sigma \in \Sigma,
	\end{equation}
	which means that all possible drifts share a ``preferred'' direction $b_0$. We will denote by
	\begin{equation}\label{drift_cone}
		\cD := C(S_d) \ \ \text{the closed convex hull of} \ S_d,
	\end{equation}
	whose polar $\cD^{\circ}$ agrees with the polar $S_d^{\circ}$ of $S_d$.
	
	The set of coefficient vectors $S = \{J_\sigma\}_{\sigma \in \Sigma}$ is pointed with axis $J_0 = (-1,b_0, I)$ in the sense of\eqref{pointed1}; that is, there exists  $\veps > 0$ such that
\begin{equation}\label{pointed_OCO}
	\langle J_{\sigma}, J_0 \rangle \geq \veps ||J_{\sigma}||, \  \forall \, \sigma \in \Sigma.
\end{equation}
In fact, using \eqref{pointed_drift}, $c \leq 0$, and $E_{\sigma} \geq 0$ is $\cS(n)$, we find 
\begin{equation}\label{est_below}
	\langle J_{\sigma}, J_0 \rangle = |c| + \langle b_{\sigma}, b_0 \rangle + {\rm tr} E_{\sigma} \geq |c| + \veps' |b_{\sigma}| +  C_n ||E_{\sigma}|| > 0,
\end{equation}	
for some constant $C_n > 0$. Since
\begin{equation}\label{est_above}
	\veps^2||J_{\sigma}||^2 = \veps^2 (c^2 +  |b_{\sigma}|^2 + ||E_m||^2).
	\end{equation}
Squaring the inequality in \eqref{est_below} and comparing with \eqref{est_above} shows that \eqref{pointed_OCO} holds for $\veps = \veps'$.
	
Since the set of coefficient vectors $S  = \{J_\sigma\}_{\sigma \in \Sigma}$ is pointed, the intersection 
$$
	\F = \bigcap_{\sigma \in \Sigma}  \{ (r,p,A) \in
	\J^2 : \langle J_\sigma, J \rangle \ge 0 \} 
$$ 
is a convex cone subequation with maximal monotonicty cone $\cM_{\F} = \F$ and the infimum operator defined by
\[
F(J) = \inf_{\sigma \in \Sigma} \frac{\langle J_{\sigma}, J \rangle}{\langle
	J_{\sigma}, J_0 \rangle}
\]
is the canonical operator for $\F$ by Theorem \ref{thm:HJB_inf}.
Finally, since $S_d \subset C(S_d):= \cD$, one has 
$$
	S \subset C(S) \subset \cN \times \cD \times \cP = (\cN \times \cD^{\circ} \times \cP)^{\circ} = \cM^{\circ}
$$
for the monotonicty cone subequation $\cM := \cM(\cN, \cD, \cP)= \cM(\gamma, \cD, \cP)$ with $\gamma = 0$. Hence one has the comparison principle for the Hamilton-Jacobi-Bellman equation $F(u, Du, D^2 u) = c$ for each $c \in \R$ by Theorem \ref{thm:CP_HJB}.
\end{exe}

We conclude this subsection with the following observations. 

\begin{rem}[More general convex subequations as intersections]\label{rem:convexity} Choosing a pointed family $\{F_{\sigma}\}_{\sigma \in \Sigma}$ of linear operators as above, and then choosing constants $\{c_{\sigma}\}_{\sigma \in \Sigma}$, the affine half-spaces $\F_{\sigma}:= \{ J \in \J^2: \ F_{\sigma}(J) \geq c_{\sigma} \}$ have intersection $\F:= \bigcap_{\sigma \in \Sigma} \F_{\sigma}$ which is a convex subequation, and all convex subsequations arise in this manner.	Moreover, Theorem \ref{thm:CP_HJB} carries to these more general convex (but not cone) subequations $\F$. The details are left to the reader.
	\end{rem}

\begin{rem}[Duality for pointed familes of linear operators]\label{rem:HJB_sup} Linear operators $F$ and their half-space subequations $\F:= \{F(J) \geq 0 \}$ are self-dual, while the dual of $\F_c:= \{F(J) \geq c \}$ is $\F_{-c}:= \{F(J) \geq -c \}$. Therefore, given a family of such operators $\mathfrak{F}:=  \{F_{\sigma}\}_{\sigma \in \Sigma}$ as in Theorem \ref{thm:HJB_inf}, the list of four subequations in the general Theorem \ref{thm:canonical_inf} (on duality, unions and interesections) reduces to two
\begin{equation}\label{union_intersection_linear}
	\F:= \bigcap_{\sigma \in \Sigma} \F_{\sigma} \quad \text{and} \quad \cE:= \overline{\bigcup_{\sigma \in \Sigma} \F_{\sigma}}.
	\end{equation}
If the operators $F_{\sigma}$ are normalized to be canonical with respect to $J_0$ as in \eqref{normalized_operators}, then
\begin{equation}\label{sup_inf_linear}
F:= \inf_{\sigma \in \Sigma} F_{\sigma} \quad \text{and} \quad E:= \sup_{\sigma \in \Sigma} F_{\sigma}.
\end{equation}
are the corresponding canonical operators for $\F$ and $\cE$.

A comparison principle analogous to Theorem \ref{thm:CP_HJB} for $F, \F$ holds for $E, \cE$. This is left to the interested reader. Also, as in Remark \ref{rem:convexity}, constants $\{c_{\sigma}\}_{\sigma \in \Sigma}$ can be employed yielding comparison for appropriate unions $\cE$ of affine half-spaces and a supremum operator $E$. This is, of course, no surprise as $\cE$ and $E$ are dual to $\F$ and $F$. 
	\end{rem}

\subsection{Proper elliptic operators with  directionality in the gradient} 

Next we consider proper elliptic operators whose monotonicity also includes  directionality in the gradient with respect to a directional cone $\cD \subsetneq \R^n$ when $\cD$ is a proper subset. More precisely, we will consider proper elliptic pairs $(F, \F)$  when the subequation constraint set $\F$ satisfies the directionality condition (D) 
\begin{equation}\label{D_recall}
 (r,p,A) \in \F \ \ \Rightarrow \ \ F(r, p + q, A) \in \F \ \text{for each} \ q \in \cD
\end{equation}
and the operator $F$ has the corresponding monotonicity in the gradient variable
\begin{equation}\label{DforF}
F(r,p,A) \leq F(r, p + q, A) \ \ \text{for each} \ (r,p,A) \in \F, q \in \cD.
\end{equation}
For simplicity, we will focus on the fundamental product monotonicity cone given by $\cM := \cN \times \cD \times \cP$ and prove comparison on arbitrary bounded domains 
in both constrained and unconstrained cases with $\cM$-monotonicity. When the operator $F$ needs to be constrained in order to be proper elliptic, there is a dichotomy similar to what we have seen in the gradient-free case. More precisely, operators such as
\begin{equation}\label{Dexe1}
	F(r,p,A) = -rp_n \, {\rm det}(A) \ \ \text{with} \ \cD = \{ p = (p', p_n) \in \R^n: \ p_n \geq 0 \}
	\end{equation}
are compatible with $\F = \cM = \cN \times \cD \times \cP$ while operators such as
\begin{equation}\label{Dexe2}
F(r,p,A) =  -r \, {\rm det}(A) -p_n 
\end{equation}
do not satisfy compatibility with respect to the constraint $\F = \cM = \cN \times \cD \times \cP$ (or any other constraint for that matter). We leave it to the reader to verify this last claim, making use of the discussion in Remark \ref{rem:generalized_equations} on the gradient-free case. Although the example \eqref{Dexe2} behaves badly, in Theorem \ref{thm:CP_SMp} below we describe a general class of operators $F(r,p,A):= G(r,p',A) - p_n$ which are good unconstrained case examples (where comparison holds).
 
Examples that we will treat include equations of Monge-Amp\`{e}re type that arise in optimal transport (see Example \ref{exe:OTE} below). All of the examples we treat in this subsection (with $\cD \subsetneq \R^n$) are all ``very weakly'' parabolic. This will lead to the treatment of genuinely parabolic equations in the next subsection (see Example \ref{exe:parabolic_eqs} below).

We begin with a general class of operators that includes \eqref{Dexe1} before discussing good versions of \eqref{Dexe2}. Recall that $\cD \subsetneq \R^n$ is a {\em directional cone} if it is a closed convex cone with vertex at the origin (see Definition \ref{defn:property_D}).

\begin{thm}\label{thm:directionality} Consider the operator defined by  
	\begin{equation}\label{F_D1}
F(r,p,A) = d(p) \, G(r,A).
\end{equation}	
Suppose that $(G, \G)$ is a (constrained case) compatible proper elliptic gradient-free pair in the sense of Definition \ref{defn:GFCP} with the normalization  
 \begin{equation}\label{G_D1}
 \inf_{\G} G = 0 \ \ \text{so that} \ \ \partial \G = \{ (r,A) \in \G: \ G(r,A) = 0 \}. 
 \end{equation}
Given $\cD \subsetneq \R^n$ a directional cone (which is then a $\cD$-monotone pure first order subequation) and given a continuous function $d: \cD \to \R$. Suppose that $d$ is $\cD$-monotone
\begin{equation}\label{g_D2}
d(p + q) \geq d(p) \ \ \text{for each} \ p, q \in \cD
\end{equation}
and suppose that $(d, \cD)$ is a compatible pair; that is,
\begin{equation}\label{g_D1}
d(p) \geq 0 \ \text{for all} \ p \in \cD \quad \text{and} \quad d(p) = 0 \ \text{if and only if} \ p \in \partial \cD;
\end{equation}
Then $(F, \F)$ is a compatible proper elliptic pair for the subequation $\F$ defined by
\begin{equation}\label{pair_D1}
	\F := \{ (r,p,A) \in \J^2: \ p \in \cD, (r,A) \in \G \ \text{and} \ F(r,p,A) \geq 0 \}.
\end{equation}
and the pair is $\cM$-monotone for the monotonicity cone subequation
\begin{equation}\label{M_D1}
\cM = \cN \times \cD \times \cP := \{ (s,q,P): \ s \leq 0, p \in \cD, P \in \cP \}.
\end{equation}
Consequently, for every $c \in F(\F)$ and for every bounded domain $\Omega \subset \R^n$ one has the comparison principle 
\begin{equation}\label{CP_D1}
\mbox{$u \leq w$ on $\partial \Omega \ \ \Rightarrow \ \  u \leq w$ on $\Omega$}
\end{equation}
for $u \in \USC(\overline{\Omega})$ and $w \in \LSC(\overline{\Omega})$ which are respectively $\F_c$-subharmonic and $\F_c$-superharmonic in $\Omega$ where $
\F_c := \{ (r,p,A) \in \F: \ F(r,p,A) \geq c \}$. 

If one also requires that $F$ is topologically tame on $\F$, then the comparison principle \eqref{CP_D1} equivalently holds if $u$ and $w$ are respectively  an $\F$-admissible subsolution and an $\F$-admissible supersolution to $F(u,Du,D^2u) = c$ on $\Omega$.
\end{thm}

\begin{proof} Being proper elliptic, the pair $(G, \G)$ is $\cQ$-monotone which together with the sign conditions $d \geq 0$ on $\cD$ and $G \geq 0$ on $\G$ easily leads to $(F, \F)$ being a proper elliptic pair. Indeed, if $(r,p,A) \in \F$ then $p \in \cD$, $(r,A) \in \G$ and $d(p) \, G(r,A) \geq 0$ so that for each $s \leq 0$ and $P \in \cP$ one has  $(r+s,A+P) \in \G$ and 
$$
F(r + s, p, A + P) = d(p)G(r+s,A+P) \geq d(p) \, G(r,A) = F(r,p,A) \geq 0.
$$
The pair is compatible with
$$
\inf_{\F} F = 0 \ \ \text{and} \ \ \partial \F = \{ (r,p,A) \in \F: \ F(r,p,A) = 0 \},
$$
where one uses \eqref{G_D1} and \eqref{g_D1}.

The pair $(F, \F)$ is $\cM$-monotone for the monotonicity cone subequation \eqref{M_D1}.
	Indeed, for each $(r,p,A) \in \F$ and each $(s,q,P) \in \cM$ one has
\begin{equation}\label{ULS_D1}
	F(r+s,p+q,A+P) = d(p+q) \, G(r + s, A + P) \geq d(p) \, G(r,A) = F(r,p,A),
\end{equation}
by the proper ellipticity of $G$ on $\G$ and the directionality condition \eqref{g_D2}.
Hence, by Lemma \ref{lem:upper_levels}, for each upper level set $\F_c$ in \eqref{ULS_D1}
is $\cM$-monotone. 

Finally, if the operator $F$ is topologically tame (see Definition \ref{defn:tameness}); that is, for every admissible level $c \in F(\F)$ the interior of
$$
	\F(c):= \{ (r,p,A) \in \F: \ F(r,p,A) = c \}
$$
is empty, then one has the correspondence principle of Theorem \ref{cor:AVSolns}
and hence the comparison principle for $\F$-admissible subsolutions, supersolutions $u, w$.
\end{proof}

Now we give some explicit examples where this Theorem \ref{thm:directionality} applies and make a few observations.

\begin{exe}\label{exe:directionalty1} Start with one of the gradient-free pairs $(G, \G)$
\begin{equation}\label{exeD1}
G(r,A)= h(r) \pol(A) \ \ \text{and} \ \ \G= \N \times \overline{\Gamma} \subset \R \times \cS(n)
\end{equation}	
of Example \ref{exe:gf3}, where $\pol$ a Dirichlet-G{\"a}rding polynomial and $h \in C((-\infty,0]))$ is non-negative, non-decreasing and satisfies $h(r) = 0$ if and only if $r = 0$. The prototype is $G(r,A):= -r \, {\rm det}(A)$ as in \eqref{Dexe1}. As for the pair $(d, \cD)$ we mention
	\begin{equation}\label{exe_g1}
	d(p)= p_n \ \ \text{and} \ \ \cD = \{ (p', p_n) \in \R^n: \ p_n \geq 0 \} \ \ \text{(a half-space)},
	\end{equation}
		\begin{equation}\label{exe_g2}
	d(p)= \prod_{j = 1}^n p_j  \ \ \text{and} \ \ \cD = \{ (p_1, \ldots , p_n) \in \R^n: \ p_j \geq 0  \ \text{for each} \ k = j , \ldots n \}
	\end{equation}
	and more generally for some $k \in \{1, \ldots, n \}$
		\begin{equation}\label{exe_g3}
	d(p)= \prod_{j = 1}^k p_j \ \ \text{and} \ \ \cD = \{ (p_1, \ldots , p_n) \in \R^n: \ p_j \geq 0 \ \text{for each} \ j= 1, \ldots  k \}.
	\end{equation}
Now set $F(r,r,A) := d(p) \, G(r,A)$. In all of these examples, $F$ is topologically tame since $F$ is real analytic.
	\end{exe}

An interesting special case comes from a very special form of {\em optimal transport}, an important subject which has received much recent attention. Excellent general references include the monograph of Villani \cite{V} and the survey paper of De Fillipis and Figalli \cite{DF14}.

\begin{exe}[Optimal transport with uniform source density]\label{exe:OTE} A partial differential equation of the form
	\begin{equation}\label{OTE}
	g(Du) \, {\rm det}(D^2u) = c, \ \ c \geq 0
	\end{equation}
arises in the theory of optimal transport, under some restrictive assumptions. In general, one has a function $f = f(x)$ in place of the constant $c$, where $f$ represents the mass density in the source configuration and $d$ represents the mass density of the target configuration (with the mass balance $||f||_{L^1} = ||d||_{L^1}$). One seeks to transport the mass with density $f$ onto the mass with density $d$ at minimal transportation cost (which is quadratic respect to transport distance). The solution of this minimization problem is given by the gradient of a convex function $u$, which turns out to be a generalized solution of the equation \eqref{OTE}. In the special case of uniform source density $f \equiv c$ and with target density $d$ having some directionality, comparison principles can be obtained as a special case of Example \ref{exe:directionalty1} with $h(r) :\equiv 1$ and $\pol(A) := {\rm det A}$.  
\end{exe}

\begin{exe}\label{exe:parabolic_eqs} In the case where the gradient factor $d$ is defined by \eqref{exe_g1} and $G(r,A)$ depends only on $A' \in \cS(n-1)$ (second order derivatives only in the {\em spatial variables} $x' \in \R^{n-1}$), one has a fully nonlinear {\em parabolic} equation of the kind considered by Krylov in his extension of Alexandroff's methods to parabolic equations in \cite{Kv76}. Such genuinely parabolic situations will be discussed in the next subsection. For the example here, it can be treated by Theorem \ref{thm:PCP_CC} below 
	\end{exe}

We now treat the unconstrained case which takes into account the difficulty posed by operators such as the one defined in \eqref{Dexe2}. The operators are proper elliptic with strict monotonicity in a gradient variable.

\begin{thm}\label{thm:CP_SMp} Suppose that $F: \J^2 \to \R$ is continuous and of the form
\begin{equation}\label{PE1}
F(r,p,A) := G(r,p',A) - p_n
\end{equation}	
where $G: \R \times \R^{n-1} \times \Symn \to \R$ is continuous and satisfies the two conditions of minimal $(\cN \times \{0\} \times \cP$)-monotonicity
\begin{equation}\label{PE2}
G(r,p',A) \leq  G(r + s, p', A + P) \ \ \text{for each} \  s \leq 0, P \in \cP \ \ 
\end{equation}
and directional monotonicity for some $\beta > 0$ fixed
\begin{equation}\label{PE3}
G(r,p' + q',A) -  G(r, p', A) \geq -\beta |q'| \ \ \text{for each} \ q' \in \R^{n-1},
\end{equation}	
where the directional cone $\cD \subsetneq \R^n = \R^{n-1} \times \R$ is the circular cone defined by
\begin{equation}\label{parabolic_cone}
	\cD := \{ q = (q', q_n) \in \R^n: \ \ - q_n \geq \beta |q'| \}.
\end{equation}
Then $(F, \J^2)$ is an unconstrained case of a proper elliptic pair which is $\cM$-monotone for the monotonicity cone
	\begin{equation}\label{PE5}
	\cM := \left\{ (s,q,P) \in \J^2: \ s \leq 0, \  q \in \cD \ \text{and} \ P \in \cP     \right\} = \cN \times \cD \times \cP.
	\end{equation}
	Consequently, for every $c \in F(\J^2)$ and for every bounded domain $\Omega \subset \R^n$ one has the comparison principle 
	\begin{equation}\label{CP_D2}
	\mbox{$u \leq w$ on $\partial \Omega \ \ \Rightarrow \ \  u \leq w$ on $\Omega$}
	\end{equation}	
	for $u \in \USC(\overline{\Omega})$ and $w \in \LSC(\overline{\Omega})$ which are respectively $\F_c$-subharmonic and $\F_c$-superharmonic in $\Omega$ with
	$\F_c := \{ (r,p,A) \in \J^2: \ F(r,p,A)  \geq c \}$. 
	
	If one also requires that $F$ is topologically tame on $\F$, then the comparison principle \eqref{CP_D2} equivalently holds if $u$ and $w$ are respectively  a viscosity subsolution and a viscosity supersolution to $F(u,Du,D^2u) = c$ on $\Omega$.
\end{thm}

\begin{proof} One follows the same argument of the proofs of Theorems \ref{thm:CP_DECr} and \ref{thm:CP_SE}. It is sufficient to observe that if $(r,p,A) \in \F$ and $(s,q,P) \in \cM$ then by \eqref{PE2} and \eqref{PE3} one has
	\begin{eqnarray*}
		F(r + s, p + q, A + P) & =& G(r + s, p' + q', A + P) - (p_n + q_n) \\
		 & \geq &   G(r, p' + q', A) - (p_n + q_n) \\
		& \geq & G(r, p', A) - p_n - q_n - \beta|q'| \geq 0,
	\end{eqnarray*}
since $(r,p,A) \in \F$ and $q = (q',q_n) \in \cD$. Therefore $(F, \J^2)$ is proper elliptic and $\cM$-monotone. By Lemma \ref{lem:upper_levels}, each $\F_c$ is $\cM$-monotone and the comparison principle for each $c \in F(\J^2)$ follows from Theorem \ref{thm:comparison} since $\cM$ is a basic product monotonicity cone. Finally, if $F$ is topologically tame the correspondence principle of Theorem \ref{cor:AVSolns} applies to yield comparison for (standard) viscosity subsolutions and supersolutions of the equation $F(u,Du,D^2u) = c$ for each  $c \in F(\J^2)$.
 \end{proof}

This subsection highlights one interesting feature of the approach to comparison principles by monotonicity cones and duality. Namely, that ``very weakly'' parabolic equations are included naturally into the general theory. We will continue to develop this idea in the next subsection that includes genuinely parabolic equations with a parabolic form of the comparison principle.

\subsection{Parabolic operators}\label{subsec:parabolic}

If $F$ is genuinely parabolic in the sense that $G$ as in \eqref{F_D1} or \eqref{PE1} depends only on derivatives with respect to the {\em spatial variables} $x' \in \R^{n-1}$, we establish a genuinely parabolic version of the comparison principle involving the parabolic boundary. First we give some notation. In $\R^n = \R^{n-1} \times \R$, in which both points of $\Omega$ and the gradient variables live, we will drop the primes and replace $(x'x_n)$ and $(p', p_n)$ by $(x,t)$ and $(p, \tau)$ respectively. In the matrix variable $\Symn$, we will use the primes for elements and subsets of $\mathcal{S}(n-1)$ as follows. For $A \in \Symn$ we will denote by $A'$ the $(n-1)\times(n-1)$ minor in the $x$-variables; that is,
$$
	A = \left( \begin{array}{cc} A' & \ast \\ \ast & \ast \end{array} \right).
$$
We will also need the matrix sets
\begin{equation}\label{Pprime}
	\cP' := \{ P' \in \mathcal{S}(n-1): P' \geq 0 \}
\end{equation}
and 
\begin{equation}\label{Pstar}
	\cP^{\ast} := \{ A \in \Symn: \ A' \in \cP' \}.
\end{equation}
Finally, one defines the {\em parabolic boundary of $\Omega \subset \R^n$} by 
\begin{equation}\label{PB}
\partial^- \Omega := \{ (x,t) \in \partial \Omega: \  t < T \} \ \ \text{where} \ \
T:= \sup \{t \in \R: \ (x,t) \in \Omega \}
\end{equation}

We begin with the parabolic version of the unconstrained case of Theorem \ref{thm:CP_SMp}.

\begin{thm}\label{thm:CP_P1} Suppose that $F: \J^2 \to \R$ is continuous and of the form
\begin{equation}\label{P1}
	F(r,(p, \tau), A) := G(r,p,A') - \tau
\end{equation}
where $G: \R \times \R^{n-1} \times \mathcal{S}(n-1) \to \R$ is continuous and satisfies for each $(r,p,A') \in \R \times \R^{n-1} \times \cS(n-1)$ the two conditions, first of $\cN \times \{0\} \times \cP'$-monotonicity 
\begin{equation}\label{P2}
	G(r,p,A') \leq  G(r + s, p, A' + P') \ \ \text{for each} \  s \leq 0, P' \in \cP
\end{equation}
and second for some $\gamma > 0$ fixed
\begin{equation}\label{P3}
	G(r,p + q,A') -  G(r, p, A') \geq -\gamma |q| \ \ \text{for each} \ q \in \R^{n-1},
\end{equation}
where $\cD \subset \R^n$ is the circular directional cone of \eqref{parabolic_cone}; that is, with the current notation
\begin{equation}\label{ParCone}
	\cD := \{(q, \sigma) \in \R^{n-1} \times \R: \ \ - \sigma \geq \gamma |q| \}. 
\end{equation}
Then $(F, \J^2)$ is an unconstrained proper elliptic pair which is $\cM^{\ast}$-monotone for the monotonicity cone
\begin{equation}\label{Mstar_cone}
	\cM^{\ast} = \cN \times \cD \times \cP^{\ast} = \{ (s,(q, \sigma),P) \in \J^2: s \leq 0, (q, \sigma) \in \cD \ \text{and} \ P \in \cP^{\ast} \} 
\end{equation} 
where $\cP^{\ast}$ is defined by \eqref{Pstar}. Consequently, for every $c \in F(\J^2)$ and for every bounded domain $\Omega \subset \R^n$ one has the parabolic comparison principle 
\begin{equation}\label{PCP}
u \leq w \ \ \text{on} \ \partial^{-} \Omega \ \ \Rightarrow \ \  u  \leq w \ \ \text{on} \ \Omega.
\end{equation}
for $u \in \USC(\overline{\Omega})$ and $w \in \LSC(\overline{\Omega})$ which are respectively $\F_c$-subharmonic and $\F_c$-superharmonic in $\Omega$ with respect to the subequation constraint set
\begin{equation}\label{ParFc}
	\F_c := \{ (r,p,A) \in \J^2: \ F(r,p,A)  \geq c \}.
\end{equation} 
	
If one also requires that $F$ is topologically tame on $\F$, then the parabolic comparison principle \eqref{PCP} (equivalently) holds if $u$ and $w$ are respectively  a viscosity subsolution and a viscosity supersolution to $F(u,Du,D^2u) = c$ on $\Omega$.
	\end{thm}

We note that the parabolic comparison principle \eqref{PCP} for $\F_c$ is equivalent to the statement
\begin{equation}\label{PCP2}
u + v \leq 0 \ \ \text{on} \ \partial^{-} \Omega \ \ \Rightarrow \ \  u + v \leq 0 \ \ \text{on} \ \Omega.
\end{equation}
for each pairs $u,v \in \USC(\overline{\Omega})$ which are respectively $\F_c$ and $\wt{\F}_c$-subharmonic on $\Omega$ .

\begin{proof} One has that $F$ is proper elliptic on all of $\J^2$ by \eqref{P2} since for each $(r, (p,\tau), A) \in \J^2$ and each $s \leq 0$ and $P \in \cP$ one has
$$
	F(r + s, (p, \tau), A + P) = G(r + s, p, A' + P') - \tau \geq G(r, p, A') - \tau = F(r, (p, \tau), A).
$$
For each $c \in F(\J^2)$ the set $\F_c$ defined by \eqref{ParFc} is a subequation constraint set. It is clear that $\F_c$  is non-empty for $c \in F(\J^2)$ and closed by the continuity of $F$. The topological property (T) will follow from Proposition \ref{prop:subequation_cones} once one establishes the $\cM^{\ast}$-monotonicity of each $\F_c$ with respect to the monotonicity cone subequation $\cM^{\ast}$. Moreover, since the cone $\cM^{\ast}$ contains the minimal monotonicity set $\cM_0 = \cN \times \{0 \} \times \cP$, the negativity and positivity properties (N) and (P) for $\F_c$ follow from the $\cM^{\ast}$-monotonicity, which does hold. Indeed, for $(r,(p, \tau), A) \in \F_c$ and $(s, (q, \sigma), P) \in \cM^{\ast}$ then by \eqref{P2} and \eqref{P3} one has
\begin{eqnarray*}
		F(r + s, (p, \tau) + (q, \sigma), A + P) & =& G(r + s, p + q, A' + P') - (\tau + \sigma) \\
		& \geq &   G(r, p + q, A') - (\tau + \sigma) \\
		& \geq & G(r, p, A') - \tau - \sigma - \gamma|q| \\
		& \geq & G(r, p, A') - \tau \geq c,
\end{eqnarray*}
since  $(r,(p, \tau), A) \in \F_c$  and $(q,\sigma) \in \cD$.
	
In preparation for the proof of the parabolic comparison principle in the form \eqref{PCP2}, notice that since
$$
	\mbox{$\cP \subset \cP^{\ast}$ where $\cP^{\ast}$ is defined by \eqref{Pstar}},
$$
$\cM^{\ast} = \cN \times \cD \times \cP^{\ast}$ contains the cone $\cM = \cN \times \cD \times \cP$ used in the more general situation of Theorem \ref{thm:CP_SMp}. Hence $\wt{\cM^{\ast}} \subset \wt{\cM}$ and, since $\wt{\cM}$-subharmonic functions satisfy the (ZMP) , for each $z \in \wt{\cM^{\ast}}(\Omega)$ ( the set of  $\USC(\overline{\Omega})$ functions which are $\wt{\cM^{\ast}}$-submarmonic on $\Omega$) one also has the (ZMP) 
\begin{equation}\label{ZMPstar}
z \leq 0 \ \ \text{on} \ \partial \Omega \ \ \Rightarrow \ \ 
z \leq 0 \ \ \text{on} \ \Omega.
\end{equation}
This will be used for the function $z:= u + v + \veps \varphi$ where $\veps > 0$ will be chosen small and $\varphi \in \USC(\overline{\Omega})$ is defined by
$$
\varphi(x,t) =  - \frac{1}{T - t} \ \ \text{for} \ \  t < T \quad \text{and} \quad  \varphi(x,t) =
- \infty \ \ \text{for} \ \ t = T.
$$
In $\Omega$, where $\varphi$ is smooth, one has 
$$
 \veps \varphi = - \frac{\veps}{T - t} < 0, \ \ D_t (\veps \varphi) = - \frac{\veps}{(T - t)^2} < 0, \ \ D_x (\veps  \varphi) = 0\ \text{and} \ \  D^2_x (\veps \varphi) = 0
$$
so  $J^2_{(x,t)} (\veps \varphi) \in \cM^{\ast}$ for each $(x,t) \in \Omega$ and hence $\veps \varphi$ is $\cM^{\ast}$-subharmonic in $\Omega$. 

Since the subequation $\F_c$ is $\cM^{\ast}$-monotone, one has $u + v \in \wt{\cM^{\ast}}(\Omega)$ by the Subharmonic Addition Theorem (Lemma \ref{lem:AT_Cones}). Similarly, since $\cM^{\ast} + \cM^{\ast} \subset \cM^{\ast}$ one also has
$$ \cM^{\ast} + \wt{\cM^{\ast}} \subset \cM^{\ast} \quad \text{and} \quad \cM^{\ast}(\Omega) + \wt{\cM^{\ast}}(\Omega) \subset \wt{\cM^{\ast}}(\Omega).
$$
Then, since $u + v \in \wt{\cM^{\ast}}(\Omega)$ and $\veps \varphi \in \cM^{\ast}(\Omega)$, one has
\begin{equation}\label{parabolic_addition}
	z = (u + v) + \veps \varphi \in \wt{\cM^{\ast}}(\Omega),
\end{equation}
and one can apply \eqref{ZMPstar} to $z$.

Now, if the comparison principle \eqref{PCP} fails, there must be an interior point $(x_0, t_0) \in \Omega$ such that $(u + v)(x_0, t_0) > 0$. Since $u + v \leq 0$ on the parabolic boundary $\partial^{-} \Omega$ and since $\varphi(x,t) < 0$ for each $t < T$ and $\varphi(x,T) \equiv - \infty$ one has 
$$
	z \leq 0 \ \text{on} \ \partial \Omega \quad \text{and} \quad z(x_0,t_0) = (u + v)(x_0, t_0) - \frac{\veps}{T - t} > 0
$$
if $\veps > 0$ is chosen sufficiently small. This contradicts the validity of \eqref{ZMPstar}. 
\end{proof}

Next, we present a parabolic version of the constrained case of Theorem \ref{thm:directionality}. The proof is a simple adaptation of the proof of Theorem \ref{thm:CP_P1} and hence will be left to the reader.

\begin{thm}\label{thm:PCP_CC} Consider the operator defined by 
\begin{equation}\label{F_P1}
	F(r,(p, \tau), A) = \tau \, G(r,A').
\end{equation}	
	with $G:  \G \subsetneq \R \times \cS(n-1) \to \R$ continuous and $\G$ closed and non-empty. Suppose that the pair $(G, \G)$ is $\cN \times \{0\} \times \cP'$-monotone; that is:
\begin{equation}\label{GG1}
	(r,A') \in \G \ \ \Rightarrow \ \ (r + s, A' + P') \in \G \ \ \text{for every} \ \ s \leq 0, P' \in \cP'; 
\end{equation}
\begin{equation}\label{GG2}
	G(r + s, A' + P') \geq G(r,A')  \ \text{for every} \ \ (r,A') \in \G, s \leq 0, P' \in \cP'.
\end{equation}
Suppose also that
\begin{equation}\label{GG3}
	\inf_{\G} G = 0 \ \ \text{and} \ \ \partial \G = \{ (r,A') \in \G: \ G(r,A') = 0 \}.
\end{equation}
Then $(F, \F)$ is a (constrained case) compatible proper elliptic pair for the subequation
\begin{equation}\label{FFpair}
	\F := \{ (r,(p, \tau), A) \in \J^2: \ (p, \tau) \in \cD, (r,A') \in \G, F(r,(p, \tau), A) \geq 0 \},
\end{equation}
where $\cD \subsetneq \R^n$ is the directional cone (half-space)
\begin{equation}\label{Dcone2}
	\cD := \{ (p, \tau) \in \R^{n-} \times \R: \ \tau \geq 0\}.
\end{equation}
and the pair is $\cM^{\ast}$-monotone for the monotonicity cone subequation
	\begin{equation}\label{M_D1}
	\cM^{\ast} = \cN \times \cD \times \cP^{\ast} := \{ (s,q,P) \in \J^2: \ s \leq 0, (q, \sigma) \in \cD, P \in \cP^{\ast} \}.
	\end{equation}
	where $\cP^{\ast}$ is defined by \eqref{Pstar}. Consequently, for every $c \in F(\F)$ and for every bounded domain $\Omega \subset \R^n$ one has the parabolic comparison principle 
	\begin{equation}\label{PCP3}
	\mbox{$u \leq w$ on $\partial^{-} \Omega \ \ \Rightarrow \ \  u \leq w$ on $\Omega$}
	\end{equation}
	for $u \in \USC(\overline{\Omega})$ and $w \in \LSC(\overline{\Omega})$ which are respectively $\F_c$-subharmonic and $\F_c$-superharmonic in $\Omega$ where $
	\F_c := \{ (r,p,A) \in \F: \ F(r,p,A) \geq c \}$. 
	
	If one also requires that $F$ is topologically tame on $\F$, then the comparison principle \eqref{PCP3} equivalently holds if $u$ and $w$ are respectively  an $\F$-admissible subsolution and an $\F$-admissible supersolution to $F(u,Du,D^2u) = c$ on $\Omega$.
\end{thm}


\appendix

\section{Existence Holds and Uniqueness Implies Comparison}\label{sec:esiuni}
(Theorem 12.7 of \cite{HL11}) which in rough language says that for any constant coefficient subequation that ``existence aways holds'' and that ``uniqueness bounds always hold'' for the Dirichlet problem.  Then we show how this can be used to prove that ``uniqueness implies comparison'', which was not stated in \cite{HL11} but follows easily from this Theorem 12.7 of \cite{HL11}.

Given a subequation $\F$ and a bounded domain $\Omega$ in Euclidian space $\R^n$, the Dirichlet Problem for $\F$ on $\Omega$ can be formulated as follows.

\begin{defn}\label{defn:Perron}
	Given a boundary function $\varphi \in C(\partial \Omega)$ a {\em solution to the Dirichlet Problem {\rm (DP)} for $\F$ on $\Omega$} is a function $h \in C(\overline{\Omega})$ satisfying:
	\begin{enumerate}
		\item $h_{|\Omega}$ is $\F$-harmonic on $\Omega$,
		\item $h_{|\partial \Omega} = \varphi$.
		\end{enumerate}
		
		Recalling that $\F(\overline{\Omega}) := \{ u \in \USC(\overline{\Omega}): \ u \ \text{is} \ \text{$\F$-subharmonic on} \ \Omega \},$
		the {\em associated Perron family} is defined by
		\begin{equation}\label{Perron_family}
		\mathfrak{F}_{\F, \varphi}:= \left\{ u \in \F(\overline{\Omega}): \ u_{|\partial \Omega} \leq \varphi \right\}.
		\end{equation}
		and the associated {\em Perron function} is defined on $\overline{\Omega}$ by
			\begin{equation}\label{Perron_function}
		U_{\F, \varphi}(x):= \sup \{ u(x): u \in \mathfrak{F}_{\F, \varphi}\}, \ \ x \in \overline{\Omega}.
		\end{equation}
	\end{defn}

There are two natural invariants associated with a subequation $\F$. The first is the maximal monotonicity cone $\cM_{\F}$ which was discussed in subsection \ref{subsec:MMC}. The second is the {\em asymptotic interior} $\EC{\F}$, which is crucial for defining {\em $\F$-boundary convexity of $\Omega$}. We refer the reader to Section 11 of \cite{HL11} for a discussion of this. Now we state the strong form of ``existence always holds'', where ``always'' refers to the fact that there is no restriction on the subequation $\F \subset \J^2$ if the boundary is suitably convex.

\begin{thm}\label{thm:existence}
	Suppose that $\F$ is a constant coefficient subequation and that $\Omega \subset \subset \R^n$ has a $C^2$ boundary which is both strictly $\EC{\F}$-convex and strictly $\EC{\wt{\F}}$-convex. Suppose that a boundary function $\varphi \in C(\partial \Omega)$ is given. Then the following two statements hold.
		\begin{itemize}
		\item[(a)] If a solution $h$ to the (DP) of Definition \ref{defn:Perron} exits, then $h$ lies between minus the dual Perron function $-U_{\wt{\F}, -\varphi}$ and the Perron function $ U_{\F, \varphi}$; that is,
		\begin{equation}\label{upper_lower_Perron}
		-U_{\wt{\F}, -\varphi} \leq h \leq U_{\F, \varphi} \ \ \text{on} \ \overline{\Omega}.
		\end{equation}
		\item[(b)] The bounding functions $-U_{\wt{\F}, -\varphi}$ and  $ U_{\F, \varphi}$ always solve the (DP) of Definition \ref{defn:Perron} .
		\end{itemize}
	\end{thm}

\begin{cor}\label{cor:existence1}
The negative $-U_{\wt{\F}, -\varphi}$ of the Perron function $U_{\wt{\F}, -\varphi}$ (which solves the (DP)  for $\wt{\F}$ and boundary data $-\varphi$) also solves the (DP) for $\F$ and boundary data $\varphi \in C(\partial \Omega)$.
\end{cor}

\begin{proof}
	Apply Theorem \ref{thm:existence} to the subequation $\wt{\F}$ and the boundary function $-\varphi$. Then $U_{\wt{\F}, -\varphi}$ satisfies
		\begin{enumerate}
		\item $U_{\wt{\F}, -\varphi}$ restricted to $\Omega$ is $\wt{\F}$-subharmonic and $-U_{\wt{\F}, -\varphi}$ restricted to $\Omega$ is $\F=\wt{\wt{\F}}$-subharmonic $\Omega$;
		\item $U_{\wt{\F}, -\varphi}$ restricted to $\partial \Omega$ is $- \varphi$.
	\end{enumerate}
	\end{proof}

\begin{cor}\label{cor:existence2}
If $h \in C(\overline{\Omega})$ is any other solution of (DP) for $\F$ on $\Omega$, then $h$ lies between $-U_{\wt{\F}, -\varphi}$ and $U_{\F, \varphi}$; that is,

with equality on $\partial \Omega$.
\end{cor}

\begin{proof} 
	$h$ belongs to the Perron family $\mathfrak{F}_{\F, \varphi}$ and hence $h \leq U_{\F, \varphi}$, while $-h$ belongs to the Perron family $\mathfrak{F}_{\wt{F}, -\varphi}$ and hence  $-h \leq U_{\wt{\F}, -\varphi}$.
	\end{proof}

Now ``uniqueness implies comparison'' is straightforward.

\begin{thm}\label{thm:comparion_uniqueness}
	Assume that the hypotheses of Theorem \ref{thm:existence} hold. Then if there is at most one solution to {\rm (DP)} for $\F$ on $\Omega$ for each fixed boundary function $\varphi \in C(\overline{\Omega})$, it follows that comparison holds for $\F$ on $\overline{\Omega}$.
\end{thm}

\begin{proof}
	By Theorem \ref{thm:existence}, Corollary \ref{cor:existence1} and the uniqueness assumption one has 
\begin{equation}\label{CU1}
		-U_{\wt{\F}, -\varphi} = U_{\F, \varphi} \ \ \text{on} \ \overline{\Omega}.
\end{equation}
Suppose that $u \in \F(\overline{\Omega})$ and $v \in \wt{\F}(\overline{\Omega})$ with $u + v \leq 0$ on $\partial \Omega$. 

Assume, for the moment, that at least one of the functions $u$ or $v$ belongs to $C(\overline{\Omega})$, say $u \in C(\overline{\Omega})$. Set $\varphi:= u_{|\partial \Omega} \in C(\partial \Omega)$. Then $u$ belongs to the Perron family $\mathfrak{F}_{\F, \varphi}$ and hence 
$$
	u \leq U_{\F, \varphi} \ \ \text{on} \ \overline{\Omega}.
$$
Now, the hypothesis $u + v \leq 0$ on $\partial \Omega$ is equivalent to $v_{|\partial \Omega} \leq - \varphi$. Hence $v$ belongs to the Perron family  $\mathfrak{F}_{\wt{\F}, -\varphi}$ and thus 
$$
	v \leq U_{\wt{F}, -\varphi} \ \ \text{on} \ \overline{\Omega}
	$$
Therefore
$$
	u + v \leq U_{\F, \varphi} + U_{\wt{F}, -\varphi} \ \ \text{on} \ \overline{\Omega},
$$
where this last sum is zero by \eqref{CU1} which used the assumption that uniqueness holds.

Finally, the provisional assumption that one of the functions $u$ or $v$ is continuous on $\overline{\Omega}$ can be removed. In fact, for fixed $u \in \USC(\overline{\Omega})$ given $\veps > 0$ there exists $u_{\veps} \in C(\overline{\Omega})$ such that 
\begin{equation}\label{approx1}
	u \leq u_{\veps} \leq u + \veps \ \ \text{on} \ \overline{\Omega}. 
\end{equation}
Take $\varphi_{\veps}:= {u_{\veps}}_{|\partial \Omega} \in C(\partial \Omega)$ and consider the Perron function $U_{\F, \varphi_{\veps}}$. By the negativity property for $\F$ and \eqref{approx1} one has that
\begin{equation}\label{approx2}
	U_{\F, \varphi_{\veps}} - \veps \in \F(\overline{\Omega}) \cap C(\overline{\Omega}) \ \ \text{and} \ \ U_{\F, \varphi_{\veps}} - \veps  = \varphi_{\veps} - \veps  \leq u  \ \text{on} \ \partial \Omega.
\end{equation}
Since $u + v \leq 0$ on $\partial \Omega$ by hypothesis, $	U_{\F, \varphi_{\veps}} - \veps + v \leq 0$ on $\partial \Omega$ (since we are assuming uniqueness also holds for the boundary function $\varphi_{\veps}$), by the previous step it follows that 
$$
	U_{\F, \varphi_{\veps}} - \veps + v \leq 0 \ \ \text{on} \ \overline{\Omega},
$$
but $u \leq U_{\F, \varphi_{\veps}}$ and $\veps > 0$ is arbitrary.
\end{proof}

\begin{rem}\label{rem:esiuni}
All of the results in this Appendix have natural extensions from $\R^n$ to any Riemannian homogeneous space $X \equiv K/G$ with a	subequation $\F$ which is invariant under the natural action of the Lie group $K$ on the 2-jet bundel $\J^2(X)$. See Theorem 13.5 in \cite{HL11}.
	\end{rem}

The following result is useful in identifying the maximal solution $U_{\F, \varphi}$ in \eqref{upper_lower_Perron}, especially in examples. No boundary convexity or smoothness iof $\partial \Omega$ is required.

\begin{lem}\label{lem:approx_upper_soln}
	Suppose that $h$ is a solution to the (DP) for $\F$ on $\Omega$ with boundary values $\varphi \in C(\partial \Omega)$. That is, $h \in C(\overline{\Omega}), h_{| \Omega}$ is $\F$-harmonic on $\Omega$ and $h_{|\partial \Omega} = \varphi$. If $h$ can be pointwise approximated by $\{h_j\}_{j \in \N} \subset C(\overline{\Omega}) \cap C^2(\Omega)$ with each $-h_j$ being strictly $\wt{\F}$-subharmonic on $\Omega$, then $h$ is the Perron function $U_{\F, \varphi}$.
	\end{lem}

\begin{proof}
	Choose any $u$ in the Perron family $\mathfrak{F}_{\F, \varphi}$. We need to show that $u \leq h$ on $\overline{\Omega}$. To this end, set $c_j:= \max \left\{ 0, \sup_{\partial \Omega} (\varphi - h_j) \right\}$ so that
$$
	0 \leq c_j \quad \text{and} \quad (\varphi - h_j)_{|\partial \Omega} \leq c_j. 
$$
Then, on $\partial \Omega$, $u - c_j - h_j \leq 0$. Since $c_j \geq 0$, by the negativity property (N) for $\F$, we have $u - c_j \in \F(\overline{\Omega})$. Since $-h_j$ is stictly $\wt{\F}$-subharmonic, we can apply definitional comparison (Lemma \ref{lem:DCP}) to $(u - c_j)$ and $-h_j$ to conclude that
$$
	u - c_j - h_j \leq 0 \ \ \text{on} \ \overline{\Omega}.
$$  
Since $c_j \to 0$ and $h_j \to h$ as $j \to +\infty$, we have $u - h \leq 0$ on $\overline{\Omega}$, as desired.
\end{proof}

Of course, Lemma \ref{lem:approx_upper_soln} also identifies when a solution $h$ to the (DP) for $\F$ with boundary data $\varphi$ is the minimal solution $-U_{\wt{\F}, -\varphi}$ in \eqref{upper_lower_Perron}, by identifying when $-h$ is the maximal solution to the (DP) for $\wt{\F}$ on $\Omega$ with boundary data $-\varphi$.

\begin{cor}\label{cor:approx_lower_soln}
	Suppose that $h$ is a solution to the (DP) for $\F$ on $\Omega$ with boundary values $\varphi \in C(\partial \Omega)$. If $h$ can be pointwise approximated by $\{h_j\}_{j \in \N} \subset C(\overline{\Omega}) \cap C^2(\Omega)$ with each $h_j$ being strictly $\F$-subharmonic on $\Omega$, then 
	$$
		h = -U_{\wt{\F}, -\varphi} \ \  \text{on} \ \overline{\Omega},
	$$
	is the minimal solution.
	
	Moreover, if $-h$ can also be pointwise approximated by a sequence of functions in $C(\overline{\Omega} \cap C^2(\Omega)$ which are strictly $\wt{\F}$-subharmonic on $\Omega$, then
	$$
	 -U_{\wt{\F}, -\varphi} = h = U_{\F, \varphi}  \ \  \text{on} \ \overline{\Omega},
	$$
	and hence comparison holds for $\F$ on $\Omega$.
	\end{cor}

\section{Failure of Comparison on Small Balls: Radial Proof}\label{sec:failure}

Making use of the considerations of Appendix \ref{sec:esiuni} on maximal and minimal solutions, we will prove Theorem \ref{thm:failure} concerning the failure of comparison on arbitrarily small balls for the (reduced) subequations 
\begin{equation}\label{exeF1B}
\F := \{ (p,A) \in \R^n \times \Symn:  \lambda_{\rm min} (B(p, A)) \geq 0 \} 
\end{equation}
and
\begin{equation}\label{exeG1B}
\cG := \{ (p,A) \in \R^n \times \Symn:  \lambda_{\rm max} (B(p, A)) \geq 0 \}.
\end{equation}
where $\alpha \in (1, +\infty)$ is fixed and 
\begin{equation}\label{BmapG}
B(p,A) := A +  |p|^{\frac{\alpha - 1}{\alpha}} \left( P_{p^{\perp}} + \alpha P_{p} \right)) \ \ \text{if} \ \ p \neq 0 \ \ \text{and} \ \ B(0,A) := A.
\end{equation}

\begin{proof}[Proof of Theorem \ref{thm:failure}]
With  $R \in (0, +\infty)$ and $B_R \subset \R^n$ the open $R$-ball about $0$, it suffices to show that the $C^2$ functions defined by 
\begin{equation}\label{zh_defn}
z(x) := 0 \ \ \text{and} \ \  h(x) := -\frac{|x|^{1 + \alpha}}{1 + \alpha} + \frac{R^{1 + \alpha}}{1 + \alpha} , \ \ x \in \R^n
\end{equation}
are both $\F$ and $\cG$ harmonic on all of $\R^n$. Obviously, they both take on the boundary values $\varphi = 0$ on $\partial B_R$.

For $z \equiv 0$, one has $(p,A) = (Dz, D^2z) \equiv (0,0) \in \partial \F \cap \partial \G$ by \eqref{FG_0}.

Since $\alpha > 1$,  $h$ is $C^{2,\alpha - 1}(\R^n)$. At the origin
$(Dh(0), D^2h(0)) = (0,0 \in \partial \F \cap \partial \G$. Hence $h$ is both $\F$ and $\cG$-harmonic at the origin. 

To calculate derivatives of $h$ away from the origin, we use the fact that $h$ is radial. It is enlightening to calculate the associated {\em radial subequation} $\cR$ for $\F$, with dual $\cRt$ the radial subequation for $\wt{\F}$ (leaving the analogous calculations for $\cG$ to the reader) which apply even if $h$ is not $C^2$. 

\begin{lem}[Radial subharmonics]\label{lem:radial_subeq}
	Suppose that $u \in C^2(\R^n)$ is radial with profile $\psi$; that is,  $u(x)= \psi(|x|)$. Then $u$ is $\F$-subharmonic on $\R^n \setminus \{0\}$ if and only if $\psi$ is $\cR$-subharmonic on $(0, +\infty)$ where $\cR$ is defined by the conditions
	\begin{equation}\label{R_subeq}
		\frac{\psi^{\prime}(t)}{t} + |\psi^{\prime}(t)|^{\frac{\alpha -1}{\alpha}} \geq 0 \quad \text{and} \quad \psi^{\prime \prime}(t) + \alpha |\psi^{\prime}(t)|^{\frac{\alpha -1}{\alpha}} \geq 0 \ \ \text{with} \ t > 0.
		\end{equation}
	Similarly, $u$ is $\wt{\F}$-subharmonic on $\R^n \setminus \{0\}$ if and only if $\psi$ is $\cRt$-subharmonic on $(0, +\infty)$ where $\cRt$ is determined by the conditions
	\begin{equation}\label{Rt_subeq}
	\frac{\psi^{\prime}(t)}{t} - |\psi^{\prime}(t)|^{\frac{\alpha -1}{\alpha}} \geq 0 \quad \text{or} \quad \psi^{\prime \prime}(t) - \alpha |\psi^{\prime}(t)|^{\frac{\alpha -1}{\alpha}} \geq 0  \ \ \text{with} \ t > 0.
	\end{equation}
	\end{lem}

\begin{proof} Using the radial calculation \eqref{radial_calc_3} of Remark \ref{rem:radial_calculus}, the reduced $2$-jet of $u$ is
$$
(Du(x), D^2u(x)) = \left(  \psi^{\prime}(|x|) \frac{x}{|x|}, \frac{\psi^{\prime}(|x|)}{|x|}  P_{x^{\perp}} + \psi^{\prime \prime}(|x|) P_{x} \right),
$$	
where, as always, $P_x$ and $P_{x^{\perp}}$ are the orthogonal projections onto the subspaces $[x]$ and $[x^{\perp}]$ for $x \neq 0$.
Denoting by	$(p,A):= (Du(x), D^2u(x))$ one has
	$$
	P_{p} = P_{x}, \ \ P_{p^{\perp}} = P_{x^{\perp}} \ \ \text{and} \ \ |p| = |\psi^{\prime}(|x|)|. 
	$$
as well as
\begin{eqnarray*}
		A + |p|^{\frac{\alpha -1}{\alpha}} \left( P_{p^{\perp}} + \alpha P_{p} \right)  & =   &  \left( \frac{\psi^{\prime}(|x|)}{|x|} + |\psi^{\prime}(|x|)|^{\frac{\alpha -1}{\alpha}} \right)  P_{x^{\perp}} \\ & \ \ + & \left( \psi^{\prime \prime}(|x|) + \alpha   |\psi^{\prime}(|x|)|^{\frac{\alpha -1}{\alpha}} \right) P_{x},
\end{eqnarray*}
where this gives an element of $\cP$ ($\lambda_{\rm min} \geq 0$) if an only if $\psi$ satisfies \eqref{R_subeq}. Similarly,
\begin{eqnarray*}
	A - |p|^{\frac{\alpha -1}{\alpha}} \left( P_{p^{\perp}} + \alpha P_{p} \right)  & =   &  \left( \frac{\psi^{\prime}(|x|)}{|x|} - |\psi^{\prime}(|x|)|^{\frac{\alpha -1}{\alpha}} \right)  P_{x^{\perp}} \\ & \ \ + & \left( \psi^{\prime \prime}(|x|) - \alpha   |\psi^{\prime}(|x|)|^{\frac{\alpha -1}{\alpha}} \right) P_{x},
\end{eqnarray*}
will be an element of $\cPt$ ($\lambda_{\rm max} \geq 0$) if and only if $\psi$ satisfies \eqref{Rt_subeq}.
	\end{proof}

\begin{rem}\label{rem:radial_fact} Lemma \ref{lem:radial_subeq} also holds for $u \in \USC(\R^n)$, where we note that if $u(x):=\psi(|x|)$, then $u \in \USC(\R^n)$ if and only if $\psi \in \USC([0, +\infty))$. For the proof of Lemma \ref{lem:radial_subeq} for radial $\USC$ functions, see section 2 of \cite{HL18a}.
	\end{rem}

We now return to the proof of Theorem \ref{thm:failure}. Applying Lemma \ref{lem:radial_subeq} to the profile $\psi(t) = - \frac{t^{1 + \alpha}}{1 + \alpha}$, one computes to find:
	$$
	\frac{\psi^{\prime}(t)}{t} + |\psi^{\prime}(t)|^{\frac{\alpha -1}{\alpha}} =  
		- \frac{t^{\alpha}}{t} + \left| - t^{\alpha} \right|^{\frac{\alpha -1}{\alpha}} = 0;
	$$
	$$
	\psi^{\prime \prime}(t) + \alpha |\psi^{\prime}(t)|^{\frac{\alpha -1}{\alpha}} = -\alpha t^{\alpha - 1} + \alpha |-t^{\alpha}|^{\frac{\alpha -1}{\alpha}} = 0.
	$$
	Note, in fact, that $A + |p|^{\frac{\alpha -1}{\alpha}} \left(  P_{p^{\perp}} + P_{p} \right) = 0$ has all eigenvalues zero. Thus $u$ is $\F$-harmonic on $\R^n \setminus \{0\}$ and so is $h$ (which is smooth and differs from $u$ by a constant). This completes the proof of Theorem \ref{thm:failure}.
	\end{proof}

We can also characterize these two distinct $\F$-harmonics as the extremals in Theorem \ref{thm:existence}.

\begin{prop}\label{prop2:exeFG} For $z$ and $h$ as in Theorem \ref{thm:failure} (and recalled in \eqref{zh_defn}),  the following hold.
	\begin{itemize}
		\item[(a)] The zero function $z$ equals $-U_{\wt{\F}, -\varphi}$
		where $U_{\wt{\F}, -\varphi}$ is the Perron function 
		for $\wt{\F}$ with boundary data $\varphi = 0$. 
		\item[(b)] The function $h$ equals $U_{\F, -\varphi}$
		where $U_{\F, -\varphi}$ is the Perron function for $\F$ on $B_R$ with boundary data $\varphi = 0$. 
	\end{itemize}
Moreover, the same statements hold with $\F$ replaced by $\cG$.
	\end{prop}

\begin{proof} For both claims (a) and (b), we will use definitional comparison by way of Lemma  \ref{lem:approx_upper_soln} and Corollary \ref{cor:approx_lower_soln}.
	
	For part (a), by Corollary \ref{cor:approx_lower_soln} it suffices to show that for each (small) $\veps > 0$ the function defined by
	$$
		z_{\veps}(x):= z(x) + \frac{\veps}{2}|x|^2 = \frac{\veps}{2}|x|^2
	$$ 
	is strictly $\F$ subharmonic (since $z_{\veps}$ is regular and pointwise approximates $z$). 
	Now, $\psi_{\veps}(t) = \frac{\veps}{2} t^2$ defines $z_{\veps}(x) = \psi_{\veps}(|x|)$ with $\psi_{\veps}^{\prime}(t) = \veps t$  \and $\psi_{\veps}^{\prime \prime}(t) = \veps$ so that
	$$
		\frac{\psi_{\veps}^{\prime}(t)}{t} + \left| \psi_{\veps}^{\prime}(t) \right|^{\frac{\alpha - 1}{\alpha}} = \veps + (\veps t)^{\frac{\alpha - 1}{\alpha}} > 0,
	$$
	and $\psi_{\veps}^{\prime \prime}(t) + \left| \psi_{\veps}^{\prime}(t) \right|^{\frac{\alpha - 1}{\alpha}}$ gives the same (positive) quantity. Thus we have the needed strict inequalities in \eqref{R_subeq}. Said differently, 
	$$
		A + |p|^{\frac{\alpha - 1}{\alpha}} \left( P_{p^{\perp}} 	
		+ \alpha P_{p} \right) = \left( \veps  + (\veps t)^{\frac{\alpha - 1}{\alpha}} \right)I \in \Int \, \cP. 
	$$
	Notice that since $\F \subset \cG$, $z_{\veps}$ will also be strictly $\cG$-subharmonic and we have part (a) for $\cG$.

	For part (b), since $\wt{\cG} \subset \wt{\F}$, by Lemma \ref{lem:approx_upper_soln} and the form of \eqref{Rt_subeq}, it suffices to find a $\wt{\cG}$-strict $C^2$ approximation to $-h(x) + \frac{R^{1 + \alpha}}{1 + \alpha} = \frac{|x|^{1+ \alpha}}{1 + \alpha}$. Set $\psi(t):= \frac{t^{1 + \alpha}}{1 + \alpha}$ and define $\psi_{\veps}(t) := \frac{(1 + \veps)(t + \veps)^{1 + \alpha}}{1 + \alpha}$. It remains to show that $\psi_{\veps}(t)$ is a strict subharmonic for the radial subequation associated to $\wt{\cG}$; that is 
	\begin{equation}\label{radial_Gt}
	\frac{\psi_{\veps}^{\prime}(t)}{t} - |\psi_{\veps}^{\prime}(t)|^{\frac{\alpha -1}{\alpha}} >  0 \quad \text{and} \quad \frac{1}{\alpha} \psi_{\veps}^{\prime \prime}(t) -  |\psi_{\veps}^{\prime}(t)|^{\frac{\alpha -1}{\alpha}} > 0.
	\end{equation}
	Now, $\psi_{\veps}^{\prime}(t) = (1 + \veps)(t + \veps)^{\alpha}$ and 
	$$
	\frac{1}{\alpha} \psi_{\veps}^{\prime \prime}(t) = (1 + \veps)(t + \veps)^{\alpha - 1} <  \frac{\psi_{\veps}^{\prime}(t)}{t} =
	(1 + \veps)(t + \veps)^{\alpha - 1}\left(1 + \frac{\veps}{t}\right).
	$$
	Hence it suffices to verify the second inequality in \eqref{radial_Gt}, but
	$$
	\frac{1}{\alpha} \psi_{\veps}^{\prime \prime}(t) -  |\psi_{\veps}^{\prime}(t)|^{\frac{\alpha -1}{\alpha}} = 	(1 + \veps)(t + \veps)^{\alpha - 1}\left( 1 - \frac{1}{(1 + \veps)^{1/\alpha}} \right) > 0.
	$$

	\end{proof}

\section{Equivalent Definitions of $\F$-subharmonic Functions}\label{sec:A}

Here we include the elementary facts in Appendix A of \cite{HL11} , but presented in a different manner more closely related to the notion of a function being $\F$-subharmonic at a point.

There are at least four different possibilities for defining the space of (upper) test jets for $u$ at $x_0$. Given a 2-jet $J=(r,p,A)$, let
\begin{equation}\label{QJ}
	Q_J(x):= r + \langle p, x - x_0 \rangle + \frac{1}{2} \langle A(x - x_0), x - x_0 \rangle 
\end{equation}
denote the quadratic function with 2-jet $J$ at $x_0$.

\begin{lem}\label{lem:A} Suppose that $u\in\USC(X)$ and $x_0 \in X$. The following sets of \underline{test jets for $u$ at $x_0$} all have the same closure in $\J^2$.
	\begin{itemize}
		\item[(J1)] \underline{Strict quadratic test jets}:
		\begin{multline*}
		J_1(x_0, u) = \{J \in \J^2 : \text{for some $\epsilon > 0$, $u(x)-Q_J(x) \le -\epsilon|x-x_0|^2$ near $x_0$}, \\ \text{with equality at $x = x_0$} \}.
		\end{multline*}
		\item[(J2)]  \underline{Quadratic test jets}:
		\[
		J_2(x_0, u) = \{J \in \J^2 : \text{$u(x)-Q_J(x) \le 0$ near $x_0$, \ with equality at $x = x_0$} \}.
		\]
		\item[(J3)] \underline{$C^2$-test jets}:
		\begin{multline*}
		J_3(x_0, u) = \{J^2_{x_0} \varphi : \text{$\varphi \in C^2$ near $x_0$, $u(x)-\varphi(x) \le 0$ near $x_0$}, \\ \text{with equality at $x = x_0$} \}.
		\end{multline*}
		\item[(J4) ]\underline{Little-$o$ quadratic test jets}:
		\begin{multline*}
		J_4(x_0, u) = \{J \in \J^2 : \text{$u(x)-Q_J(x) \le o(|x-x_0|^2)$ near $x_0$}, \\ \text{with equality at $x = x_0$} \}.
		\end{multline*}
	\end{itemize}
\end{lem}

\begin{proof}
First note that 
\begin{equation}\label{A2}
	J_1(x_0, u) \subset J_2(x_0, u) \subset J_3(x_0, u) \subset J_4(x_0, u)
\end{equation}
and hence it suffices to show that 
\begin{equation}\label{A3}
	J_4(x_0, u)\subset \bar{J}_1(x_0, u).
\end{equation}
Suppose that $J \in J_4(x_0,u)$; that is, 
$$
u(x_0) - Q_J(x_0) = 0  \ \text{and} \ u(x)  - Q_J(x) \leq o(|x -  x_0|^2) \ \text{as}   \ x \to x_0 
$$
Hence, for each $\veps > 0$ there exists a neighborhood $B_{\delta}(x_0)$ with $\delta = \delta(\veps)$  such that
\begin{equation}\label{A4}
u(x)  - Q_J(x) \leq \veps |x - x_0|^2 \ \ \text{in} \ B_{\delta}(x_0) \ \ \text{and} \ \ u(x_0)  - Q_J(x_0) = 0.
\end{equation}
Denoting by $J_{\alpha} := J + (0,0,\alpha I)$, since $- Q_{J_{4 \veps}}(x) = - Q_J(x) - 2 \veps |x - x_0|^2$, \eqref{A4} can be written as 
\begin{equation}\label{A5}
u(x)  - Q_{J_{4 \veps}}(x) \leq - \veps |x - x_0|^2 \ \ \text{in} \ B_{\delta}(x_0) \ \ \text{and} \ \ u(x_0)  - Q_{J_{4 \veps}}(x_0) = 0.
\end{equation} 
Hence $J_{4 \veps} \in J_1(x_0,u)$ for each $\veps > 0$ and taking the limit as $\veps \to 0^+$ gives $J \in \bar{J}_1(x_o. u)$ and hence \eqref{A3}.
\end{proof}

\begin{cor}\label{cor:A}
Let $\F \subset \J^2$ be an arbitrary closed subset. Given $u \in \USC(X)$ and $x_0 \in X$ the conditions
\begin{equation}\label{TJI}
J_1(x_0, u) \subset \F, \ J_2(x_0, u) \subset \F, \ J_3(x_0, u) \subset \F \ and \  J_4(x_0, u) \subset \F
\end{equation}
are all equivalent.
	\end{cor}

\begin{proof}
	If $J_k(x_0,u) \subset \F$ for some $k \in \{1,2,3,4\}$, then $\bar{J}_k \subset \F$ since $\F$ is closed, but $\bar{J}_j(x_0,u) = \bar{J}_k(x_0,u)$ for each $j \neq k$ by Lemma \ref{lem:A}.
\end{proof}

\begin{defn}\label{defn:FSH_equivalences}
	Given $\F$ a subequation constraint set. A function $u \in \USC(X)$ is {\em $\F$-suharmonic at $x_0 \in X$} if 
\begin{equation}\label{FSHE}
\mbox{$J \in \F$ for all test jets $J$ for $u$ at $x_0$,}
\end{equation}  
where one may adopt any of the four defintions of (upper) test jets for $u$ at $x_0$ contained in Lemma \ref{lem:A}.
	\end{defn}

As noted in section \ref{sec:constraints}, Corollary \ref{cor:A} shows that the equivalent ways of defining $\F$-subharmonicity in $x_0$ do not depend on the subequation constraint properties (P) and (N), but only on the fact that $\F$ is closed which follows from property (T), which also ensures that $\F$ is non empty and hence the conditions \eqref{TJI} are non trivial.

\section{Elementary Properties of $\F$-subharmonic Functions}\label{sec:B}

We consider $\F$-subharmonic functions on an open set $X \subset \R^n$ with $\F \subset \J^2$ a subequation constraint set (see Definitions \ref{defn:constraint} and \ref{defn:FSH}). While these are known properties (see Theorem 2.6 and Appendix B of \cite{HL09}), for the convenience of the reader, we reproduce the proofs here making use of Lemma \ref{lem:A} which has been tailored to the pointwise notion of $\F$-subharmonicity. 

\begin{prop}[Elementary Properties of $\F(X)$]\label{prop:B} For $\F$ and $X$ as above, the following hold. 

\vspace{2ex}
{\rm (A) (Local Property)} $u \in \USC(X)$ is locally $\F$-subharmonic $ \iff u \in \F(X)$.

\vspace{2ex}
{\rm (B) (Maximum Property)} If $u,v \in \F(X)$ then $\max \{u,v \} \in \F(X)$.

\vspace{2ex}
{\rm (C) (Coherence Property)} If $u \in \USC(X)$ is twice differentiable in $x_0 \in X$, then $u$ is $\F$-subharmonic in $x_0$ if and only if $J^2_{x_0} u \in \F$.

\vspace{2ex}
{\rm (D) (Translation Property)}  $u \in \F(X) \iff u_y \in F(X+y)$ where $u_y(x) := u(x-y)$.

\vspace{2ex}
{\rm (E) (Decreasing Sequence Property)}  If $\{ u_k \}_{k \in \N} \subset \F(X)$ is a decreasing sequence of functions, then the limit $\displaystyle{u :=\lim_{k\to\infty}u_k \in \F(X)}$.

\vspace{2ex}
{\rm (F) (Uniform Limits Property)}  If $\{ u_k \}_{k \in \N} \subset \F(X)$ is a sequence of functions which converges uniformly to $u$ on compact subsets of $X$, then $u \in \F(X)$. 

\vspace{2ex}
{\rm (G) (Families Locally Bounded Above)}  Suppose $\mathfrak{F} \subset \F(X)$ is a family of functions which are locally uniformly bounded above.  Then the upper semicontinuous regularization $u^*$ of the upper envelope
$$
u(x)\ =\ \sup_{v \in \mathfrak{F}} v(x)
$$
belongs to $\F(X)$. 
\end{prop}

\begin{proof} (A): This is built into the Definition \ref{defn:FSH} where locally $\F$-subharmonic just means that for each $x_0 \in X$, $u$ is $\F$-subharmonic on some neighborhood of $x_0$.

\vspace{2ex}
\noindent (B): The condition that $\max\{u,v\} -\varphi \leq 0$ near $x_0$ with equality at $x_0$ implies that for one of the functions $u,v$, say $u$, we have $u(x_0)= \varphi(x_0)$.  In this case, $u-\varphi \leq 0$ near $x_0$ with equality at $x_0$. Hence, $J^2_{x_0}\varphi \in F$.

\vspace{2ex}
\noindent (C): This is the content of Remark \ref{rem:coherence} which makes use of the little-$o$ quadratic jet formulation (J4) of Lemma \ref{lem:A} for one direction and property (P) for the other. 

\vspace{2ex}
\noindent (D): This is obvious since the fibers $\F_{x_0}$ do not depend on $x_0 \in \R^n$.

The remaining properties are all proved by contradiction of the fact that $\F$ is closed, using the Bad Test Jet Lemma \ref{lem:nonFSH} (which comes from negating the strict quadratic jet formulation (J1) of Lemma \ref{lem:A}). More precisely, by Lemma \ref{lem:nonFSH} if $u \in \USC(X)$ is not $\F$-subharmonic, then there exist $x_0 \in X, \veps > 0, \rho >0$ and a bad test jet $J=(r,p,A)\notin \F$ with
\begin{equation}\label{BTJ1}
\left\{ \begin{array}{rlcc}
u(x) - Q_J(x) & \leq \  -\veps|x-x_0|^2 & & {\rm on}\ B_{\rho}(x_0) \subset X  \\
& = \  0  &  & {\rm at}\ x_0 \end{array} \right\}
\end{equation}
where $Q_J(x) = r+\langle p, x-x_0\rangle +\half \langle A(x-x_0), x-x_0 \rangle$ has $J^2_{x_0}Q_J = J = (r,p,A) \notin \F$ with $r = u(x_0)$. Given a bad test jet $J$ for $u$ at $x_0 \in X$,  the idea is to exhibit a sequence of upper text jets $\{J_k = (r_k, p_k, A_k)\}_{k \in \N}$ (for a suitable sequence of $\F$-subharmonic functions) for which
\begin{equation}\label{BTJ2}
\{J_k\}_{k \in \N} \subset \F \ \ \text{and} \ \ \ \lim_{k \to +\infty} J_k = J \notin \F,
\end{equation}
a contradiction to $\F$ being closed.

\vspace{2ex}
\noindent (E): We begin by recalling that if $\{ u_k \}_{k \in \N} \subset \USC(X)$ is a decreasing sequence, then the limit $\displaystyle{u = \inf_{k \in \N} u_k}$ is automatically $\USC(X)$. Hence if $u \in \USC(X)$ is not $\F$-subharmonic, then the bad test jet lemma applies to $u$ and there exist $x_0 \in X, \veps > 0, \rho > 0$ and $J = (r,p,A) \notin \F$ such that \eqref{BTJ1} holds. By reducing $\rho > 0$ if necessary, we can assume that the compact neighborhood $\overline{B}_{\rho}(x_0)$ is contained in the open set $X$ and then \eqref{BTJ1} shows that the upper semicontinuous function $u - Q_J$ has a unique strict maximum on $\overline{B}_{\rho}(x_0)$ at $x = x_0$ with $u(x_0) - Q_J(x_0) = 0$. 

We will construct the needed sequence of test jets (satisfying \eqref{BTJ2}) from the decreasing sequence of functions $\{ v_k \}_{k \in \N} \subset \USC(X)$ defined by
\begin{equation}\label{DSv1}
	v_k(x) := u_k(x) - Q_J(x)
\end{equation}
which have non-negative maxima
\begin{equation}\label{DSv2}
m_k := \sup_{\overline{B}_{\rho}(x_0)} v_k = \sup_{\overline{B}_{\rho}(x_0)} \left( u_k - Q_J \right) \geq  \sup_{\overline{B}_{\rho}(x_0)} \left( u - Q_J \right) = u(x_0) - Q_J(x_0) = 0.
\end{equation}
We recall the elementary fact that for any decreasing sequence $\{ v_k \}_{k \in \N} \subset \USC(X)$ with limit $v$ and for any compact subset $K$ of $X$ one has
\begin{equation}\label{DSPL}
	\lim_{k \to + \infty} \left\{ \sup_K v_k \right\} = \sup_K v \ \ \text{in} \ \R. 
\end{equation}
To see this, let $\alpha > 0$ be arbitrary and consider the sequence of sets $\{K_k\}_{k \in \N}$ defined by 
$$
	K_k := \{ x \in K: \ v_k(x) \geq \sup_K v + \alpha \}
$$
Each $K_k$ is compact by the upper semicontinuity of $v_k$ and $\{K_k\}_{k \in \N}$ is a decreasing sequence of sets ($K_{k+1} \subset K_k$ for each $k \in \N$) since $\{ v_k \}_{k \in \N}$ is a decreasing sequence. By the pointwise convergence of $v_k$ to $v$, one must have that $\bigcap_{k \in \N} K_k = \emptyset$ and hence each $K_k$ must be empty for large $k$. More precisely,  there exists $k_0 = k_0(\alpha)\in \N$ such that
for each $k \geq k_0$ one has $v_k(x) < \sup_K v + \alpha$ for every $x \in K$ and hence
$$
	\sup_{K} v \leq \sup_{K} v_k \leq \sup_{K} v + \alpha \ \ \text{for every} \ k \geq k_0(\alpha),
$$
which yields \eqref{DSPL} since $\alpha > 0$ is arbitrary.

Returning to the construction, let $\delta \in (0, \rho)$ and consider the compact annulus $K_{\delta}:= \overline{B}_{\rho}(x_0) \setminus B_{\delta}(x_0)$. By \eqref{BTJ1} one has $u(x) - Q_J(x) < 0$ on $K_{\delta}$ for each $\delta \in (0, \rho)$
and then applying \eqref{DSPL} to the sequence $\{v_k\}_{k \in \N}$ in \eqref{DSv1} on $K_{\delta}$ shows that there exists $k_1 = k_1(\delta) \in \N$ such that
\begin{equation}\label{DSvk}
\sup_{K_{\delta}} \left( u_k(x) - Q_J(x) \right) < 0 \ \ \text{for every} \ k \geq k_1. 
\end{equation}
Consequently, for $k \geq k_1$ the non-negative maximum $m_k$ of \eqref{DSv2} can only occur at points $x_k \in B_{\rho}(x_0)$. However, since $\delta$ can be made arbitrarily small there is a sequence $\{x_k\}_{k \geq k_1}$ such that
\begin{equation}\label{xk_prop}
0 \leq m_k := \sup_{\overline{B}_{\rho}(x_0)} (u_k - Q_J) = u_k(x_k) - Q_J(x_k) \ \ \text{and} \ \ \lim_{k \to + \infty}x_k = x_0.
\end{equation}
Hence for $k \geq k_1$ one has
$$
\mbox{$u_k(x)- (Q_J(x) + m_k) \leq 0$ \ \ on $B_{\rho}(x_0)$ \quad and \quad $u_k(x_k)- (Q_J(x_k) + m_k) =  0$,}
$$
so that $Q_j + m_k$ is an upper test function for $u \in \F(X)$ at $x_k$ and hence
\begin{equation}\label{TJS}
	J_k = (r_k, p_k, A_k) = D^2_{x_k} (Q_J + m_k)=  (r + m_k, p + A(x - x_k), A) \in \F, 
\end{equation} 
where $J = (r,p,A) \notin \F$ is the bad test jet with $r = u(x_0)$. Taking the limit as $k \to + \infty$, one has $x_k \to x_0$ by construction \eqref{xk_prop} and also $m_k \to 0$ by \eqref{DSPL} and \eqref{DSv2}; that is, 
$$
	\lim_{k \to + \infty} \sup_{\overline{B}_{\rho}(x_0)}(u_k - Q_J) = \sup_{\overline{B}_{\rho}(x_0)}(u - Q_J) = 0.
$$
Hence one has $\{J_k\}_{k \in \N} \subset \F$ with $J_k \to J \notin \F$, which contradicts $\F$ being closed.

\vspace{2ex}
\noindent (F): For uniform limits, the proof is almost the same as that given above for decreasing limits. In particular, using the definition of uniform convergence it is easy to see that the limit $u$ is upper semicontinuous on $X$ (so that Lemma \ref{lem:nonFSH} applies to give a bad test jet $J \notin \F$ at some point $x_0 \in X$) and that for each compact subset $K$ of $X$ the limit property \eqref{DSPL}
$$
	\lim_{k \to + \infty} \left\{ \sup_K v_k \right\} = \sup_K v \ \ \text{in} \ \R. 
$$
holds if $v_k:= u_k + Q_J$ converges uniformly to $u - Q_J$ on $K$. The same construction of the sequence of jets $\{J_k\}_{k \in \N} \subset \F$ with $J_k \to J \notin \F$ carries over without change.

\vspace{2ex}
\noindent (G): We give an adaptation of the classical proof in \cite{CIL92} by contradiction. Suppose that $u^*\notin \F(X)$. In order to simplify notation, we will assume that $u$ is not $\F$-subharmonic in $x_0 = 0$ (which we may assume  by the Translation Property (D)). By the Bad Test Jet Lemma \ref{lem:nonFSH}, there exist $\veps > 0, \rho > 0$ and $(p,A) \in \R^n \times \cS(N)$ such that
\begin{equation}\label{B3}
u^*(x) - \bigl[  \langle p,x \rangle + \half   \langle Ax,x \rangle \bigr] \ \leq\ u^*(0) -\veps|x|^2
\qquad{\rm for\ } |x|  \leq \rho
\end{equation}
but
$$
(u^*(0), p,A) \ \notin\ \F.
$$
Since
$$
u^*(0) \ = \ \lim_{k\to\infty} \sup_{|y|\leq {1\over k}} \sup_{v \in \mathfrak{F}} v(y),
$$
it follows easily that there exist sequences $\{y_k\} \subset \R^n, \{u_k \} \subset \mathfrak{F}$, such that $y_k \to 0$ and
\begin{equation}\label{B4}
\lim_{k\to\infty} u_k(y_k)\ =\ u^*(0).
\end{equation}

Now choose a maximum point $x_k$ for the function $u_k(x) - \bigl[  \langle p,x \rangle + \half   \langle Ax,x \rangle \bigr] $ on $|x|\leq \rho$. Then
$$
u_k(y_k) - \bigl[  \langle p,y_k \rangle + \half   \langle Ay_k,y_k \rangle \bigr]
\
\leq\
u_k(x_k) - \bigl[  \langle p,x_k \rangle + \half   \langle Ax_k,x_k \rangle \bigr].
$$
Pick a subsequence if necessary so that $x_k\to x$. Taking\ $\displaystyle{\liminf_{k\to +\infty}}$
of both sides and using the standard fact that $\displaystyle{\limsup_{k \to + \infty}  u_k(x_k)\leq u^*(x)}$ yields
\begin{equation*}
\begin{array}{cclcl}
    u^*(0) & \leq & \displaystyle{\liminf_{k\to+\infty}u_k(x_k) - \bigl[  \langle p,x \rangle + \half   \langle Ax, x \rangle \bigr]}  & & [{\rm by} \ \eqref{B4}]    \\
    \\
    & \leq &  u^*(x) - \bigl[  \langle p,x \rangle + \half   \langle Ax, x \rangle \bigr]  & &      \\
    \\
    & \leq & u^*(0)-\veps|x|^2  & & [{\rm by} \ \eqref{B3}].  \end{array}
\end{equation*}
Consequently, we have $x=0$ and
$$
u^*(0)\ \leq\  \liminf_{k\to+\infty}  u_k(x_k).
$$
Since $x_k\to 0$, the upper semicontinuity of $u^*$ gives
$$
 \limsup_{k\to+\infty} u_k(x_k)\ \leq\ u^*(0)
$$
and therefore
 $$
 \lim_{k\to\infty} u_k(x_k)=u^*(0).
 $$
 The inequality
 $$
 u_k(x) - \bigl[  \langle p,x \rangle + \half   \langle Ax,x \rangle \bigr]
\
\leq\
u_k(x_k) - \bigl[  \langle p,x_k \rangle + \half   \langle Ax_k,x_k \rangle \bigr].
 $$
for $|x|\leq \rho$ (and hence near $x_k$) implies fairly easily that
$$
(u_k(x_k),  p  +Ax_k, A)\ \in \ \F
$$
since $u_k\in \F(X)$.  Taking the limit as $k\to+\infty$ yields
$$
(u^*(0), p, A)\ \in\ \F,
$$
a contradiction.

\end{proof}

\end{document}